\documentclass[10pt,reqno]{amsart}

\usepackage{graphicx}
\usepackage{times}
\usepackage[colorlinks=true,linkcolor=blue,citecolor=blue]{hyperref}%
\usepackage{dsfont}

\usepackage{xcolor}
\usepackage{stmaryrd}
\usepackage{pdfrender,xcolor}
\usepackage{tikz}
\usepackage{pgfplots}
\usetikzlibrary{decorations.pathmorphing}

\newtheorem{theorem}{Theorem}[section]
\newtheorem{lemma}[theorem]{Lemma}
\newtheorem{assump}[theorem]{Assumption}

\newtheorem{prop}[theorem]{Proposition}
\theoremstyle{definition}
\newtheorem{definition}[theorem]{Definition}
\newtheorem{notation}[theorem]{Notation}

\theoremstyle{remark}
\newtheorem{remark}[theorem]{Remark}
\numberwithin{equation}{section}

\usepackage[top=1.2in, bottom=1.2in, left=1.2in, right=1.2in]{geometry}

\usepackage[scr]{rsfso}
\usepackage[english]{babel}
\usepackage[utf8x]{inputenc}
\usepackage{fancyhdr}
\usepackage{amsmath}
\usepackage{amsthm}
\usepackage{amssymb}
\usepackage{comment}
\usepackage{enumitem}
\setlist{leftmargin=*}
\usepackage[integrals]{wasysym}
\usepackage{bm}

\newcommand\nc{\newcommand}
\nc{\on}{\operatorname}
\nc{\E}{\mathbb{E}}
\nc{\R}{\mathbb R}
\nc{\C}{\mathbb C}
\nc{\Q}{\mathbb Q}
\nc{\Z}{\mathbb Z}
\nc{\N}{\mathbb N}
\nc{\F}{\mathbb F}
\nc{\wt}{\widetilde}
\nc{\ol}{\overline}
\nc{\rnc}{\renewcommand}
\nc{\e}{\varepsilon}
\nc{\DMO}{\DeclareMathOperator}
\nc{\grad}{\nabla}
\nc{\fsp}{\fontdimen2\font=2.3pt}
\nc{\fspp}{\fontdimen2\font=2.17pt}
\nc{\abbr}[1]{{\sc\lowercase{#1}}}

\nc{\bphi}{{\boldsymbol{\phi}}}
\nc{\bpsi}{{\boldsymbol{\psi}}}
\nc{\bmu}{\boldsymbol{\mu}}

\nc{\ocolor}{\color}
\nc{\X}{\mathrm{X}}
\nc{\Y}{\mathrm{Y}}
\rnc{\t}{\mathrm{t}}
\nc{\x}{\mathrm{x}}
\nc{\y}{\mathrm{y}}
\nc{\s}{\mathrm{s}}
\nc{\z}{\mathrm{z}}
\nc{\w}{\mathrm{w}}
\rnc{\r}{\mathrm{r}}
\rnc{\a}{\mathrm{a}}
\rnc{\b}{\mathrm{b}}
\rnc{\k}{\mathrm{k}}
\rnc{\u}{\mathrm{u}}
\nc{\n}{\mathrm{n}}
\newcommand{\m}{\mathrm{m}}
\rnc{\L}{\mathrm{L}}
\rnc{\N}{\mathsf{N}}
\nc{\q}{\mathrm{q}}

\rnc{\leq}{\leqslant}
\rnc{\geq}{\geqslant}
\rnc{\d}{\mathrm{d}}
\rnc{\O}{\mathrm{O}}
\newenvironment{nouppercase}{%
  \renewcommand{\uppercasenonmath}[1]{}}{}
\title{\fsp\Large Half-space KPZ equation from a class of nonlinear SPDEs}

\author{Kevin Yang}\thanks{Department of Mathematics and Department of Statistics, University of California, Berkeley; yangkev@berkeley.edu}

\usepackage{setspace}
\begin{document}
\setstretch{1.0}
\fsp
\raggedbottom
\begin{nouppercase}
\maketitle
\end{nouppercase}
\begin{center}
\today
\end{center}

\begin{abstract}
\fspp We derive the half-space \abbr{KPZ} equation as a continuum limit from a general class of Ginzburg-Landau \abbr{SPDE}s under weak nonlinearity scaling. To our knowledge, this is the first half-space universality result for the \abbr{KPZ} equation beyond integrable models (and their perturbations). 
\end{abstract}

{\hypersetup{linkcolor=blue}
\setcounter{tocdepth}{1}
\tableofcontents}
\allowdisplaybreaks
\section{Introduction}\label{section:intro}
The \emph{half-space Kardar-Parisi-Zhang (\abbr{KPZ}) equation} is a stochastic \abbr{PDE} (\abbr{SPDE}) that describes the evolution of a random interface growing on the substrate {\small$[0,\infty)$}, subject to boundary data which fixes its slope at a ``wall" at {\small$0\in[0,\infty)$}. See Figure \ref{figure:kpz} for a sample interface. Precisely, it is the following \abbr{SPDE} for {\small$(\t,\X)\in[0,\infty)\times[0,\infty)$}:
\begin{align}
\partial_{\t}\mathbf{h}_{\t,\X}&=\alpha\partial_{\X}^{2}\mathbf{h}_{\t,\X}+\beta|\partial_{\X}\mathbf{h}_{\t,\X}|^{2}+\xi_{\t,\X}.\label{eq:kpzIa}
\end{align}
Above, {\small$\alpha>0$} and {\small$\beta\in\R$} are fixed constants; we will assume that {\small$\beta\neq0$} in this paper. The term {\small$\xi$} is the Gaussian space-time white noise. Since \eqref{eq:kpzIa} is defined on the half-space {\small$[0,\infty)$}, to guarantee well-posedness, we require the aforementioned boundary data given below, in which {\small$\wt{\mathbf{A}}\in\R$} is fixed:
\begin{align}
\partial_{\X}\mathbf{h}_{\t,\X}|_{\X=0}&=\wt{\mathbf{A}}\quad \text{for all } \t\in[0,\infty).\label{eq:kpzIb}
\end{align}
Strictly speaking, \eqref{eq:kpzIa}-\eqref{eq:kpzIb} is analytically ill-posed. We will adopt the \emph{Cole-Hopf} solution to \eqref{eq:kpzIa}-\eqref{eq:kpzIb} \cite{BC95,BG97,CS18}. That is, set {\small$\mathbf{h}:=\beta^{-1}\alpha\log\mathbf{Z}$}, where {\small$\mathbf{Z}$} solves the following \emph{half-space stochastic heat equation (\abbr{SHE})}:
\begin{align}
\partial_{\t}\mathbf{Z}_{\t,\X}&=\alpha\partial_{\X}^{2}\mathbf{Z}_{\t,\X}+\alpha^{-1}\beta\mathbf{Z}_{\t,\X}\xi_{\t,\X}. \label{eq:sheIa}
\end{align}
The corresponding boundary condition for \eqref{eq:sheIa} is the following Robin data with a potentially ``renormalized" parameter {\small$\mathbf{A}$} \cite{GH19} (depending on {\small$\wt{\mathbf{A}}$} and potentially other details):
\begin{align}
\partial_{\X}\mathbf{Z}_{\t,\X}|_{\X=0}&=\alpha^{-1}\beta\mathbf{A}\mathbf{Z}_{\t,0} \quad \text{for all } \t\in[0,\infty).\label{eq:sheIb}
\end{align}
A formal calculation by the chain and It\^{o} rules shows that \eqref{eq:kpzIa}-\eqref{eq:kpzIb} is indeed related to \eqref{eq:sheIa}-\eqref{eq:sheIb}. Moreover, as shown in \cite{CS18}, the solution {\small$\mathbf{Z}$} to the  \abbr{SHE} is positive with probability {\small$1$} (assuming its initial data is a non-negative, non-zero Borel measure on {\small$\R$}), so the Cole-Hopf solution is well-defined because {\small$\beta\alpha^{-1}\neq0$} by assumption. 
\begin{figure}[h]
\includegraphics[scale=0.5]{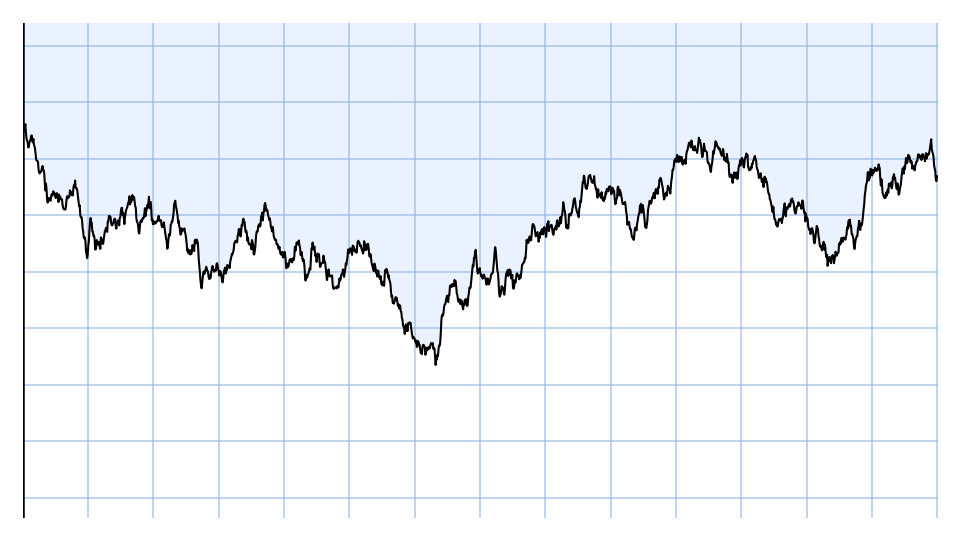}
\caption{Sample plot of a solution to \eqref{eq:kpzIa}-\eqref{eq:kpzIb} at a fixed time as a function in space, with a negative slope at the wall (in black on the left).}
\label{figure:kpz}
\end{figure}

The focus of this paper is the \emph{universality} of the half-space \abbr{KPZ} equation and \abbr{SHE}; that is, the conjecture that it is a universal large-scale limit of many random interface growth models. Our motivation comes from the rich physics displayed by these equations which come specifically from their half-space nature. Indeed, Kardar \cite{K85} predicts a \emph{phase transition} for the binding or unbinding of the half-space \abbr{KPZ} interface to the wall at {\small$0\in\llbracket0,\infty\rrbracket$} as a function of the boundary parameter in \eqref{eq:sheIb}. A realization of this phase transition was rigorously established in \cite{IMS26}. Precisely, \cite{IMS26}  computes a critical value of {\small$\mathbf{A}$} in \eqref{eq:sheIb}, above which Tracy-Widom \abbr{GSE} statistics appear, below which Gaussian statistics arise, and at which one sees Tracy-Widom \abbr{GOE} statistics. Thus, a key motivation for this work is to demonstrate the universality not only of the half-space \abbr{KPZ} equation, but also this phenomena. 

Perhaps the most \emph{fundamental} conjecture on (the ``weak") universality of half-space \abbr{KPZ}, which is motivated by the original universality proposal in \cite{KPZ86} for the full-space \abbr{KPZ} equation, is given as follows. (This is referred to as \emph{Hairer-Quastel universality} \cite{HQ18}.) Take the following class of \emph{Ginzburg-Landau} \abbr{SPDE}s:
\begin{align}
\partial_{\t}\mathbf{j}_{\t,\X}=\partial_{\X}\mathscr{U}'(\partial_{\X}\mathbf{j}_{\t,\X})+\mathbf{F}[\mathscr{U}'(\partial_{\X}\mathbf{j}_{\t,\X})]+\xi_{\t,\X}, \quad (\t,\X)\in[0,\infty)\times[0,\infty).  \label{eq:glIa}
\end{align}
Above, {\small$\mathscr{U},\mathbf{F}:\R\to\R$} are smooth functions; we refer to them as the \emph{potential} and \emph{nonlinearity}, respectively. To illustrate, if {\small$\mathscr{U}(\a)=\a^{2}$} and {\small$\mathbf{F}(\a)=\a^{2}$}, then \eqref{eq:glIa} reduces to \eqref{eq:kpzIa}. In this way, \eqref{eq:glIa} is a generalization of \eqref{eq:kpzIa} (as we will allow a much larger collection of choices of {\small$\mathscr{U},\mathbf{F}$}). Similar to \eqref{eq:kpzIa}, since \eqref{eq:glIa} lives on the half-space, we supplement it with boundary data (for fixed {\small$\wt{\mathbf{A}}\in\R$}) resembling \eqref{eq:kpzIb}:
\begin{align}
\partial_{\X}\mathbf{j}_{\t,\X}|_{\X=0}=\wt{\mathbf{A}} \quad \text{for all } \t\in[0,\infty). \label{eq:glIb} 
\end{align}
The Ginzburg-Landau \abbr{SPDE} \eqref{eq:glIa} is a general equation for modeling a non-Gaussian, time-dependent statistical mechanical system. It was systematized in \cite{HH77} for this purpose and has become a popular model in statistical mechanics due to the flexible functional parameters {\small$\mathscr{U},\mathbf{F}$}, which allow one to model a variety of physical systems using \eqref{eq:glIa}. These equations have also played a central role in the mathematical theory of hydrodynamic limits and related problems \cite{S91}.

The universality conjecture of interest is the following:
\begin{enumerate}
\item Discretize \eqref{eq:glIa}-\eqref{eq:glIb} in space, so that it becomes an \abbr{SDE} (which may be rigorously solved).
\item Take the ``continuum limit" of solutions to the discretized equations as the length-scale of the discretization vanishes. In this limit, we also apply a \emph{weak nonlinearity scaling} to the {\small$\mathbf{F}$}-nonlinearity.
\item Conjecturally \cite{KPZ86}, the limit in step (2) is the solution to \eqref{eq:kpzIa}-\eqref{eq:kpzIb}.
\end{enumerate}
\emph{The main result of this paper is a resolution of the conjecture above for a large class of potentials {\small$\mathscr{U}$}, polynomial nonlinearities {\small$\mathbf{F}$}, and non-equilibrium initial data}. \emph{To our knowledge, this is a first half-space \abbr{KPZ} universality result for a class of interfaces that are not (perturbations of) integrable ones} (which we discuss in Section \ref{subsection:background}).\footnote{Here, by ``integrable", we loosely mean a model whose height function is not exactly linearized via exponential Cole-Hopf transform.} In this paper, we choose a specific discretization which preserves invariant measures of \eqref{eq:glIa}-\eqref{eq:glIb}, and which is natural upon viewing \eqref{eq:glIa}-\eqref{eq:glIb} as a dynamical field theory.

As a consequence of these results, we are able to establish the same emergence and phase transition of random matrix statistics from the one-point distribution of \eqref{eq:glIa}-\eqref{eq:glIb}, under a double limit of \eqref{eq:glIa}-\eqref{eq:glIb} to \eqref{eq:kpzIa}-\eqref{eq:kpzIb}, then to long times. \emph{Thus, we deduce that this behavior is not an artifact of integrability and indeed holds more universally}. (We note here that the methods of \cite{IMS26} to show this phase transition for half-space \abbr{KPZ} goes through an integrable discretization.)

Additionally, through our results we also derive \emph{stationary fluctuations}, i.e. the half-space \abbr{KPZ} equation with initial data given by any of its many invariant measures that were constructed in \cite{BC23}. This is despite the fact that we do not yet have anything close to an exhaustive list of invariant measures for \eqref{eq:glIa}-\eqref{eq:glIb}. Physical importance of stationary fluctuations (as dynamics in ``non-equilibrium steady states") is discussed in \cite{S91}.

Finally, we refer to Theorems \ref{theorem:main}, \ref{theorem:mainstat}, and \ref{theorem:mainwedge} for precise statements of our results. 
\subsection{Technical aspects}\label{subsection:introtech}
Broadly speaking, the strategy we use is a generalization to the half-space of the three-step idea that was developed in \cite{Y26PTRF} for the full-space version problem. \emph{Developing such a half-space framework, which can perhaps be used more generally, is a key technical contribution of this work}. 

This three-step schematic is summarized as follows:
\begin{enumerate}
\item Establish an ansatz for the evolution equation satisfied by the exponential of (a space-discretization of) \eqref{eq:glIa}-\eqref{eq:glIb} which has the form of \eqref{eq:sheIa}-\eqref{eq:sheIb} with additional ``fluctuations". (Roughly, ``fluctuations" are spatially local functions of a space-discretized {\small$\partial_{\X}\mathbf{j}$} with vanishing projections onto model-dependent analogues, e.g. non-Gaussian versions, of the classical Wiener chaos. See Definitions \ref{definition:bulkadmissible} and \ref{definition:edgeadmissible}.)
\item Estimate these ``fluctuations" by proving the celebrated \emph{Boltzmann-Gibbs principles}, i.e. sharp \abbr{CLT}-estimates after averaging these fluctuations in space and/or time.
\item Estimate the (stochastic) heat kernel for the equation in (1) via the Boltzmann-Gibbs principles in (2).
\end{enumerate}
Indeed, the benefits of this for the full-space problem in \cite{Y26PTRF} are (at least) two-fold. First, it allows one to access non-equilibrium initial data (which the theory of energy solutions \cite{GJ14,GP18} is currently not equipped to handle). Second, it bypasses the need for the theory of regularity structures \cite{H13}, which, to our knowledge, does not apply to even the full-space version of \eqref{eq:glIa} because its linearization is a parabolic \abbr{PDE} with singular coefficients. See Section \ref{subsection:background} for a further discussion on these other methods.

Our strategy hinges on generalizing every step in the schematic above to the half-space. While a more detailed discussion of the technical aspects of this paper is deferred to Section \ref{section:method}, we emphasize here the particularly nice conceptual aspects of the Boltzmann-Gibbs principles herein for the analysis of ``bulk fluctuations", namely the fluctuations which arise from the bulk dynamics \eqref{eq:glIa}. Because bulk dynamics agree with full-space dynamics, a natural idea for proving these Boltzmann-Gibbs estimates is to directly implement the strategy in the full-space setting from \cite{Y26PTRF}. However, the analysis in \cite{Y26PTRF} depends crucially on dynamics at a variety of (sufficiently large) mesoscopic space-time scales, which now differ in a (not small!) neighborhood of the edge.

To address this, we develop a \emph{multi-scale localization-gluing scheme}, in which we construct the dynamics in a mesoscopic neighborhood of the edge by gluing together a collection of full-space dynamics; this allows us to essentially inherit the \abbr{BGP} estimates from the full-space setting in \cite{Y26PTRF} and establish the bulk \abbr{BGP}s. \emph{Remarkably, the key that drives this scheme is what seems to be a first application of Nash's and Davies' parabolic estimates for singular \abbr{SPDE}s}. In particular, the key obstacle to this end is the singular first-order term {\small$\mathbf{F}$} in \eqref{eq:glIa}. \emph{However, and fortunately, we only need Nash-Davies methods for {local} space-time scales, at which {\small$\mathbf{F}$} is ultimately shown to be perturbative}. This indicates the importance of a \emph{sub-criticality} of {\small$\mathbf{F}$}, in parallel with regularity structures.

The conceptual benefits of localization-gluing include the ability to inherit the full-space analysis in \cite{Y26PTRF} as a blackbox, thereby (in some sense) ``decoupling" the analysis for bulk fluctuations from the edge dynamics. That is, the methods in this paper directly build on the full-space theory, so we avoid having to adapt the (already very difficult) work in \cite{Y26PTRF} to domains with boundary. We comment that this type of adaptation has been done in \cite{GPS20} for special equilibrium initial data. But even in this special case, the adaptations are certainly sophisticated. 

Finally, we comment briefly on steps (1) and (3). For step (1), a natural idea is to derive and analyze separately the bulk and edge dynamics of the ``microscopic" Cole-Hopf transform (that is, a space-discretized exponential of \eqref{eq:glIa}-\eqref{eq:glIb}); this is what is done in earlier works on derivations of the \abbr{KPZ} equations with boundary conditions from integrable particle systems and their perturbations \cite{CS18,P19,Y25}. \emph{However, this does not work, as we will need to show a crucial cancellation between the two which ultimately produces a renormalized boundary parameter}. Once we establish steps (1) and (2), the analysis for step (3) is based on a parametrix-type ansatz for the stochastic heat kernel of interest around that for (a perturbation of a space-discretized) \eqref{eq:sheIa}-\eqref{eq:sheIb}. A key contribution here is streamlined version of the parametrix analysis in \cite{Y26PTRF} that reduces the number of topologies in which we must estimate said parametrix, which helps us deal with the additional ``edge terms" herein.
\subsection{Background}\label{subsection:background}
As mentioned earlier, this paper gives (to our knowledge) a first Hairer-Quastel universality result for half-space \abbr{KPZ}. In general, it provides a first derivation of half-space \abbr{KPZ} from a model which is not (a perturbation of) an integrable one. In this direction, though, we mention \cite{CS18,P19,Y25}, which derive the half-space \abbr{KPZ} equation from an interacting particle system known as the ``half-space \abbr{ASEP}", as well as \cite{W20}, which does the same for a collection of half-space directed polymer models with fairly general noise. 

We also mention the related work \cite{GPS20}, which derives the \abbr{KPZ} equation on a compact interval (with Neumann boundary data) from a version of the \abbr{ASEP} on said interval (i.e. open \abbr{ASEP}). However, key to \cite{GPS20} is the theory of energy solutions for \abbr{KPZ} \cite{GJ14,GP18}, which studies the issue of convergence through martingale problems. While it is a powerful technique, the main limitation is a crucial dependence on equilibrium or invariant measure initial data; since our interest is in a large class of non-equilibrium initial data, it is unclear how the method of energy solutions can be applied herein. (Uniqueness of energy solutions in the half-space is also not yet shown, at least to our knowledge.) We also highlight the works \cite{CS18,P19,Y22PTRF}, which derive the \abbr{KPZ} equation on a compact interval from the open \abbr{ASEP} and its long-range generalizations. 

Next, we mention the long list of works \cite{DGP17,GP16,HQ18,HS17,HX19,KWX24,KZ25,YEJP23,Y25EJP,Y26PTRF} on Hairer-Quastel universality of the \abbr{KPZ} equation (for \abbr{SPDE} models of interface growth) on the torus and the full-space. 

Finally, we mention the work \cite{GH19}, which uses the theory of regularity structures to show well-posedness for the open \abbr{KPZ} equation (on a compact interval). The analysis in \cite{GH19} has a conceptual benefit of building directly on the theory of regularity structures as developed on the torus. 
\subsection{Acknowledgements}
We are very thankful for Jimmy He pointing us to the reference \cite{W20}.
\subsection{AI disclosure}
\abbr{LLM}s\footnote{Here, \abbr{LLM} refers to Claude, Opus 4.6 and Opus 5.5, the latter of which was granted access to via the Claude-for-scientists program.} were never used to generate, check, or produce any mathematics in this paper. (They were specifically asked not to do so.) We emphasize that this paper is entirely human-written. \abbr{LLM}s were asked to give us (purely non-mathematical) feedback on an earlier version of the introduction, abstract, and disclosure, which led to the rewriting of several sentences (for clarity's sake). An \abbr{LLM} was asked to double check the proof of (1.5) in \cite{FS86} (after we had already checked it by hand); moreover, only the statement of this estimate was used (not its proof). Only after the mathematics and body of this paper were completed, an \abbr{LLM} was consulted about a refinement of the proof of Proposition \ref{prop:nashdavies} in the case where {\small$\boldsymbol{\Gamma}^{\mathrm{full}}$} therein is a probability kernel. (This specific case is \emph{not} relevant, and in fact orthogonal, to the scope of this paper, and the estimate is quite different than that in Proposition \ref{prop:nashdavies}. This was for intellectual curiosity; the \abbr{LLM} was only asked to double check our computations for this special case, not to generate any ideas.) Again after the mathematics and body were completed, \abbr{LLM}s were asked to clarify \cite{GH19} to help our discussion of it be more accurate. (We clarify that the mathematical ideas in \cite{GH19} are \emph{not} relevant for this paper.) This led to a few edits in the presentation of the introduction. Finally, an \abbr{LLM} was used to help generate code for the plot in Figure \ref{figure:kpz}.
%
%
%
\section{Main results}
\subsection{Model}
We start by introducing a space-discretization of \eqref{eq:glIa} with the aforementioned weak-nonlinearity scaling. It will take the form of an infinite-dimensional \abbr{SDE}. (In particular, the model at hand can be interpreted as a system of interacting diffusions). To this end, let us introduce notation to be used throughout the paper.
\begin{itemize}
\item Throughout, we denote elements in {\small$\R^{\Z}$} by {\small$\bphi=(\bphi_{\x})_{\x\in\Z}$}. For any {\small$\x\in\Z$}, we also define {\small$\tau_{\x}:\R^{\Z}\to\R^{\Z}$} to be the shift map, so that for any {\small$\bphi\in\R^{\Z}$} and {\small$\w\in\Z$}, we have {\small$(\tau_{\x}\bphi)_{\w}:=\bphi_{\x+\w}$}. In words, this shift map re-centers the configuration {\small$\bphi$} to be around {\small$\x$}. Note that {\small$\tau_{0}$} is the identity map. (We sometimes consider only projections of {\small$\bphi\in\R^{\Z}$} onto {\small$\R^{\mathbb{I}}$} for some {\small$\mathbb{I}\subseteq\Z$}. We use the same {\small$\bphi$} notation in order to avoid more complicated notation.)
\item We let {\small$\N>0$} denote the (main) scaling parameter. It will always be an integer, and we will be interested in the ``large-{\small$\N$} limit" obtained by sending {\small$\N\to\infty$}.
\item For any {\small$\a,\b\in\R$} such that {\small$\a\leq\b$}, we define {\small$\llbracket\a,\b\rrbracket:=[\a,\b]\cap\Z$}.
\end{itemize}
The space-discretization of \eqref{eq:glIa} will be an \abbr{SDE} for {\small$\t\geq0$} and {\small$\x\in\llbracket0,\infty\rrbracket$}, whose dynamics in the interior {\small$\llbracket1,\infty\rrbracket$} and at the edge {\small$\{0\}$} will be different from each other (to ultimately account for inserting boundary data). For the interior, the dynamics will resemble \eqref{eq:glIa}, whereas at the edge, the \abbr{SDE} describes instead a ``flux" coming from an external ``heat bath" reservoir (or Glauber dynamic) at the edge; see Remark \ref{remark:phi}. This mirrors the ``heat bath" rules that are used to define the edge dynamics in open \abbr{ASEP} in \cite{CS18,P19,Y25}, for instance.

We now introduce the edge and interior dynamics as the following system, which uses notation to be explained afterwards, and in which {\small$\x\in\llbracket1,\infty\rrbracket$} in the first line \eqref{eq:currI} below:
\begin{align}
\d\mathbf{j}^{\N}_{\t,\x}&=\N^{\frac32}\grad^{\mathbf{X}}_{1}\mathscr{U}'[\bphi_{\t,\x}]\d\t+\N\mathbf{F}[\tau_{\x}\bphi_{\t}]+\sqrt{2}\N^{\frac12}\d\mathbf{b}_{\t,\x},\label{eq:currI}\\
\d\mathbf{j}^{\N}_{\t,0}&=\N^{\frac32}\mathscr{U}'[\bphi_{\t,1}]\d\t+\sqrt{2}\N^{\frac12}\d{\mathbf{b}_{\t,0}}+\N\mathbf{F}[\bphi_{\t}]\d\t.\label{eq:currII}
\end{align}
%
\begin{itemize}
\item The smooth function {\small$\mathscr{U}:\R\to\R$} is the ``potential", on which we make assumptions in {\color{black}Assumption \ref{assump:potential}}.
\item For any {\small$\mathfrak{l}\in\Z$}, let {\small$\grad^{\mathbf{X}}_{\mathfrak{l}}$} be the discrete gradient on {\small$\Z$} of length {\small$\mathfrak{l}$}. That is, for {\small$\mathsf{f}:\Z\to\R$}, we set {\small$\grad^{\mathbf{X}}_{\mathfrak{l}}\mathsf{f}_{\x}=\mathsf{f}_{\x+\mathfrak{l}}-\mathsf{f}_{\x}$}.  
\item We have used the standard field-theoretic notation (see \cite{S91}) {\small$\bphi_{\t,\x}:=-\N^{1/2}\grad^{\mathbf{X}}_{-1}\mathbf{j}^{\N}_{\t,\x}=\N^{-1/2}(\mathbf{j}^{\N}_{\t,\x}-\mathbf{j}^{\N}_{\t,\x-1})$} for all {\small$\x\in\llbracket1,\infty\rrbracket$}. It will be convenient to extend {\small$\bphi_{\t,\x}$} to every {\small$\x\in\Z$} via {\small$\mathscr{U}'[\bphi_{\t,\x}]\equiv0$} for all {\small$\x\in\llbracket-\infty,0\rrbracket$}, and to package this data as {\small$\bphi_{\t}:=(\bphi_{\t,\x})_{\x\in\Z}\in\R^{\Z}$}. (This will be well-defined by assumptions on the potential {\small$\mathscr{U}$} to be given in {\color{black}Assumption \ref{assump:potential}}.)
\item The processes {\small$\t\mapsto\mathbf{b}_{\t,\x}$} are independent standard Brownian motions for {\small$\x\in\llbracket0,\infty\rrbracket$}.
\item We now define the nonlinearity {\small$\mathbf{F}$} as follows. For any {\small$\bphi\in\R^{\Z}$}, we set
\begin{align}
\mathbf{F}[\bphi]:=\mathbf{F}_{2}[\bphi]+\mathbf{F}_{\geq3}[\bphi]&:=\tfrac13\beta_{2}(\mathscr{U}'[\bphi_{-1}]\mathscr{U}'[\bphi_{0}]+\mathscr{U}'[\bphi_{0}]\mathscr{U}'[\bphi_{1}]+\mathscr{U}'[\bphi_{1}]\mathscr{U}'[\bphi_{2}])\nonumber\\
&+\sum_{\d\in\llbracket3,\deg\rrbracket}\beta_{\d}\Big(\prod_{\ell=1}^{\d}\mathscr{U}'[\bphi_{\ell}]+\ldots+\prod_{\ell=1}^{\d}\mathscr{U}'[\bphi_{-\ell+1}]\Big).\label{eq:nl}
\end{align}
Above, {\small$\deg\geq2$} is a fixed positive integer, and {\small$\beta_{2},\ldots,\beta_{\deg}\in\R$} are fixed constants, with the assumption that {\small$\beta_{2}\neq0$}, which ultimately corresponds to {\small$\beta\neq0$} in the limit \abbr{SPDE} \eqref{eq:kpzIa}. We clarify that \eqref{eq:nl} is a polynomial in the {\small$\mathscr{U}'[\bphi_{\cdot}]$}-variables, and that there are no constant or linear terms. Indeed, a constant term may be removed from \eqref{eq:currI} by a spatially global shift in {\small$\mathbf{j}^{\N}$}, and a linear term would yield a diverging-speed characteristic that does not make sense to follow in the half-space. We also clarify that when we evaluate {\small$\mathbf{F}[\tau_{\x}\bphi_{\t}]$} for {\small$\x=0$} (or any point close to the edge), some terms in \eqref{eq:nl} will vanish by nature of the full-space extension of {\small$\bphi_{\t}$}.
\end{itemize}
To clarify the structure of \eqref{eq:currI}-\eqref{eq:currII}, we record the \abbr{SDE} satisfied by {\small$\bphi_{\t}$}. By using {\small$\bphi_{\t,\x}=-\N^{1/2}\grad^{\mathbf{X}}_{-1}\mathbf{j}^{\N}_{\t,\x}$} for each {\small$\x\in\llbracket1,\infty\rrbracket$}, and by \eqref{eq:currI}-\eqref{eq:currII}, we arrive at the following system, where {\small$\x\in\llbracket2,\infty\rrbracket$} in the first line \eqref{eq:phiI} below:
\begin{align}
\d\bphi_{\t,\x}&=\N^{2}\Delta\mathscr{U}'[\bphi_{\t,\x}]\d\t+\N^{\frac32}\wt{\mathbf{F}}[\tau_{\x}\bphi_{\t}]-\sqrt{2}\N\grad^{\mathbf{X}}_{-1}\d\mathbf{b}_{\t,\x},\label{eq:phiI}\\
\d\bphi_{\t,1}&=-\N^{2}\mathscr{U}'[\bphi_{\t,1}]\d\t-\sqrt{2}\N^{}\d\mathbf{b}_{\t,0}+\N^{2}\grad^{\mathbf{X}}_{1}\mathscr{U}'[\bphi_{\t,1}]\d\t+\sqrt{2}\N\d\mathbf{b}_{\t,1}+\N^{\frac32}\wt{\mathbf{F}}[\tau_{1}\bphi_{\t}]\d\t.\label{eq:phiII}
\end{align}
Above, we used notation for the discrete Laplacian {\small$\Delta=-\grad^{\mathbf{X}}_{-1}\grad^{\mathbf{X}}_{1}$} on the lattice {\small$\Z$}. We have also introduced the function {\small$\wt{\mathbf{F}}[\bphi]:=\mathbf{F}[\bphi]-\mathbf{F}[\tau_{-1}\bphi]$}; more precisely, by \eqref{eq:nl}, this function is given by the formula
\begin{align}
\wt{\mathbf{F}}[\bphi]=\wt{\mathbf{F}}_{2}[\bphi]+\wt{\mathbf{F}}_{\geq3}[\bphi]&:=\tfrac13\beta_{2}(\mathscr{U}'[\bphi_{2}]\mathscr{U}'[\bphi_{1}]-\mathscr{U}'[\bphi_{-1}]\mathscr{U}'[\bphi_{-2}])\nonumber\\
&+\sum_{\d\in\llbracket3,\deg\rrbracket}\beta_{\d}\Big(\prod_{\ell=1,\ldots,\d}\mathscr{U}'[\bphi_{\ell}]-\prod_{\ell=1,\ldots,\d}\mathscr{U}'[\bphi_{-\ell}]\Big)\label{eq:wtf}
\end{align}
%
\begin{remark}\label{remark:phi}
\fsp The first two terms on the \abbr{RHS} of \eqref{eq:phiII} correspond to a natural Langevin or Glauber (or ``heat bath") dynamic associated to the potential {\small$\mathscr{U}$}. These are exactly what contribute to the ``flux" \eqref{eq:currII}. The last three terms on the \abbr{RHS} of \eqref{eq:phiII}, on the other hand, correspond to ``interior spin" interactions between {\small$\bphi_{\t,1},\bphi_{\t,2}$} (namely, that these two spins are ``exchanging" with each other, or alternatively, that the last three terms in \eqref{eq:phiII} correspond to a [stochastic] conservation law for {\small$\bphi_{\t}$} near the edge).
\end{remark}
We now introduce a family of probability measures with which the coefficients {\small$\alpha,\beta$} in the limiting half-space \abbr{KPZ} equation will ultimately be computed. 
\begin{definition}\label{definition:gc}
\fsp Fix any {\small$\sigma\in\R$}. Consider the product probability measure below for {\small$\bphi\in(\bphi_{\x})_{\x\in\llbracket1,\infty\rrbracket}\in\R^{\llbracket1,\infty\rrbracket}$}, written in terms of a Radon-Nikodym derivative with respect to Lebesgue measure on {\small$\R^{\llbracket1,\infty\rrbracket}$}:
\begin{align}
\mathbb{P}^{\sigma}[\d\bphi]:=\prod_{\x\in\llbracket1,\infty\rrbracket}\boldsymbol{\mathcal{Z}}_{\sigma,\mathscr{U}}^{-1}\cdot\exp(-\mathscr{U}[\bphi_{\x}]+\upsilon_{\sigma}\bphi_{\x})\d\bphi_{\x}.\label{eq:gcI}
\end{align}
The constant {\small$\upsilon_{\sigma}\in\R$} is chosen so that {\small$\E^{\sigma}\bphi_{\x}=\sigma$} for each {\small$\x\in\llbracket1,\infty\rrbracket$}; existence and uniqueness of such a {\small$\upsilon_{\sigma}$} will ultimately follow by our assumptions on the potential {\small$\mathscr{U}$} in {\color{black}Assumption \ref{assump:potential}}. The measure {\small$\mathbb{P}^{\sigma}$} is often referred to as the \emph{grand-canonical ensemble} or \emph{grand-canonical measure} of density {\small$\sigma$}. We also clarify that the {\small$\mathbb{P}^{\sigma}$}-measures extend naturally to {\small$\R^{\Z}$} upon replacing {\small$\llbracket1,\infty\rrbracket$} by {\small$\Z$} in the formula above. (The same is true for any subset {\small$\mathbb{B}\subseteq\Z$}.)
\end{definition}
We clarify that the {\small$\mathbb{P}^{\sigma}$} measures are \emph{not} all invariant measures of \eqref{eq:phiI}-\eqref{eq:phiII}. Indeed, the Langevin component \eqref{eq:phiII} picks out a specific {\small$\sigma\in\R$} (namely, the density for which {\small$\upsilon_{\sigma}=0$}, which we will later assume to be {\small$\sigma=0$} in {\color{black}Assumption \ref{assump:potential}} for the sake of a clearer presentation.) The invariance of {\small$\mathbb{P}^{0}$} for \eqref{eq:phiI}-\eqref{eq:phiII} is obtained in Section \ref{subsection:generator}. However, the entire {family} of {\small$\mathbb{P}^{\sigma}$}-ensembles is also important, because the dynamics \eqref{eq:phiI}-\eqref{eq:phiII} in the bulk resemble those in the full-space, which have all ensembles in Definition \ref{definition:gc} as invariant measures. 

The final object we must introduce is a Cole-Hopf transform of \eqref{eq:currI}-\eqref{eq:currII}. We define
\begin{align}
\mathbf{Z}^{\N}_{\t,\x}:=\exp(\lambda\mathbf{j}^{\N}_{\t,\x}-\lambda\mathscr{R}_{\lambda}\t), \label{eq:ch}
\end{align}
for all {\small$(\t,\x)\in[0,\infty)\times\llbracket0,\infty\rrbracket$}, where the constants {\small$\lambda,\mathscr{R}_{\lambda}$} are defined as follows. First, we set
\begin{align}
\alpha&:=\partial_{\sigma}\E^{\sigma}\mathscr{U}'[\bphi_{1}]|_{\sigma=0},\label{eq:alpha}\\
\beta&:=\tfrac12\partial_{\sigma}^{2}\E^{\sigma}(\mathbf{F}_{2}[\tau_{\mathrm{M}}\bphi])|_{\sigma=0}.\label{eq:beta}
\end{align}
Above, {\small$\mathrm{M}>0$} is a large but fixed integer depending only on {\small$\deg$} from \eqref{eq:nl}. (It does not matter so much what {\small$\mathrm{M}$} is as long as it is large enough such that none of the terms in \eqref{eq:nl}, after shifting {\small$\bphi\mapsto\tau_{\mathrm{M}}\bphi$} therein, trivially vanish, i.e. all of the spatial points therein live in {\small$\llbracket1,\infty\rrbracket$}. This follows by exchangeability of the one-dimensional marginals of \eqref{eq:gcI}. By the same token, we can evaluate at any point in {\small$\llbracket1,\infty\rrbracket$} in \eqref{eq:alpha}.) Now, we set
\begin{align}
\lambda&:=\alpha^{-1}\beta \quad\text{and}\quad \mathscr{R}_{\lambda}:=\tfrac{1}{12}\lambda^{3}\E^{0}(\mathscr{U}'[\bphi_{1}]\bphi_{1}^{3})-\tfrac16\beta_{2}\lambda^{2}. \label{eq:rg}
\end{align}
It turns out that {\small$\alpha>0$} and that {\small$\beta\neq0$} as long as {\small$\beta_{2}\neq0$}; see Appendix A.2 of \cite{Y26PTRF}. As before, the expectation in \eqref{eq:rg} can be evaluated at any spatial point in {\small$\llbracket1,\infty\rrbracket$} without changing the value of {\small$\mathscr{R}_{\lambda}$}.
\subsection{Precise statements of theorems}
Before we present the main results of this paper, we shall first introduce two assumptions on the model at hand. We start with assumptions on {\small$\mathscr{U}$} in \eqref{eq:currI}-\eqref{eq:currII} and \eqref{eq:phiI}-\eqref{eq:phiII}.
\begin{assump}\label{assump:potential}
\fsp We have {\small$\mathscr{U}''[\cdot]\in[\boldsymbol{\Lambda}^{-1},\boldsymbol{\Lambda}]$} for some fixed constant {\small$\boldsymbol{\Lambda}\geq1$}, and {\small$\upsilon_{0}=0$} (see Definition \ref{definition:gc}). 
\end{assump}
In particular, the potential is uniformly convex (so it is within the scope of Bakry-\'{E}mery theory), and it admits a technical {\small$\mathscr{C}^{2}$}-type upper bound (which can certainly be relaxed with some more work). The other assumption on {\small$\upsilon_{0}$} is there only for convenience (so that we do not have to track this additional parameter and replace {\small$\mathscr{U}$} in all of our calculations by {\small$\mathscr{U}[\cdot]-\upsilon_{0}[\cdot]$}).

The second assumption concerns the class of non-equilibrium initial data that we allow in this paper. 
\begin{assump}\label{assump:data}
\fsp Let {\small$\mathfrak{P}$} denote the Radon-Nikodym derivative for the law of the initial data {\small$(\bphi_{0,\x})_{\x\in\llbracket1,\infty\rrbracket}$} with respect to {\small$\mathbb{P}^{0}$} from Definition \ref{definition:gc}. We assume that for some (small) constant {\small$\gamma_{\mathrm{data}}>0$} independent of {\small$\N$}, we have
\begin{align}
\|\mathfrak{P}\|_{\mathrm{L}^{\infty}(\R^{\llbracket1,\infty\rrbracket})}\leq\N^{\gamma_{\mathrm{data}}}. \label{eq:dataI0}
\end{align}
(If {\small$\mathfrak{P}_{\t}$} denotes the density for the law of {\small$\bphi_{\t}$} with respect to {\small$\mathbb{P}^{0}$} at any deterministic time {\small$\t\geq0$}, then since {\small$\mathbb{P}^{0}$} is an invariant measure for {\small$\bphi_{\t}$}, standard convexity implies that \eqref{eq:dataI0} also holds for {\small$\mathfrak{P}_{\t}$} in place of {\small$\mathfrak{P}$}.)
\end{assump}
We clarify that \eqref{eq:dataI0} will only be used for estimates over the local theory for \eqref{eq:currI}-\eqref{eq:currII}. We also clarify that Assumption \ref{assump:data} is \emph{far from perturbative}, since it allows for non-equilibrium initial data. Indeed, the divergence of the \abbr{RHS} of \eqref{eq:dataI0} in the large-{\small$\N$} limit is what allows us to derive the half-space \abbr{KPZ} equation with essentially any deterministic initial data which it can make sense of. Finally, we remark that Assumption \ref{assump:data} is not optimal, and it can perhaps be replaced by the following much more general {\small$\mathrm{L}^{2}$}-estimate (again for {\small$\gamma_{\mathrm{data}}>0$} small):
\begin{align}
\E^{0}|\mathfrak{P}|^{2}\leq\exp(\N^{\gamma_{\mathrm{data}}}). \label{eq:dataI}
\end{align}
Indeed, in the most of our estimates, we will use \eqref{eq:dataI} instead of \eqref{eq:dataI0}. We will only use \eqref{eq:dataI0} for convenience in a couple of instances in Section \ref{section:stochII}, where the more general estimate \eqref{eq:dataI} (in all likelihood) suffices, though only after a much more complicated and technical argument based on prior work \cite{Y25EJP} (that requires only a priori estimates on the \emph{relative entropy} {\small$\E^{0}\mathfrak{P}\log\mathfrak{P}$} that, under the assumption of \eqref{eq:dataI}, is {\small$\lesssim\N^{\gamma_{\mathrm{data}}}$}.) However, in order to avoid an overly technical presentation (without losing the non-equilibrium flavor of this paper), we take \eqref{eq:dataI0} for convenience. (Indeed, \cite{Y25EJP} is very complicated and lengthy. Moreover, this paper is already quite technically involved.) We defer an explanation of how to go from \eqref{eq:dataI0} to \eqref{eq:dataI} in a future manuscript.
\subsubsection{Universality for continuous initial data}
The first result of this paper focuses on ``nice" continuous initial data for the limit half-space \abbr{KPZ} equation. In particular, besides the a priori estimate in \eqref{eq:dataI}, we impose only standard estimates on the Cole-Hopf transform \eqref{eq:ch} as in \cite{CS18,P19}. First, let us introduce important notation for the entire paper. We write {\small$\a=\mathrm{O}(\b)$} if {\small$|\a|\leq\mathrm{C}|\b|$} for some {\small$\mathrm{C}>0$}. We will also write {\small$\a\lesssim\b$} to mean the same thing. Subscripts in the {\small$\mathrm{O}$} and {\small$\lesssim$} notation indicate parameters on which the implied constant {\small$\mathrm{C}$} depends.
\begin{theorem}\label{theorem:main}
\fsp Suppose that in addition to Assumptions \ref{assump:potential} and \ref{assump:data}, we have the following.
\begin{enumerate}
\item For any {\small$p\geq1$} and {\small$\zeta>0$}, there exists {\small$\kappa_{p,\zeta}\geq0$} such that for any {\small$\x,\y\in\llbracket0,\infty\rrbracket$} with {\small$\mathfrak{l}:=\y-\x\lesssim\N$}, we have
\begin{align}
\E|\mathbf{Z}^{\N}_{0,\x}|^{2p}\lesssim_{p}\exp(\tfrac{\kappa_{p,\zeta}|\x|}{\N}) \quad\text{and}\quad \E|\grad^{\mathbf{X}}_{\mathfrak{l}}\mathbf{Z}^{\N}_{0,\x}|^{2p}\lesssim_{p,\zeta}\N^{-p+p\zeta}|\mathfrak{l}|^{p-p\zeta}\exp(\tfrac{\kappa_{p,\zeta}|\x|}{\N}).\label{eq:mainI}
\end{align}
\item The initial data for \eqref{eq:currI}-\eqref{eq:currII} is independent of the Brownian motions {\small$\mathbf{b}$}. Moreover, {\small$\X\mapsto\mathbf{Z}^{\N}_{0,\N\X}$} converges locally uniformly in {\small$[0,\infty)$} to a continuous {\small$\mathbf{Z}_{0,\cdot}$} in probability. (We extend {\small$\X\mapsto\mathbf{Z}^{\N}_{0,\N\X}$} from {\small$\X\in[0,\infty)\cap\N^{-1}\Z$} to all {\small$\X\in[0,\infty)$} by linear interpolation, though this specific choice is unimportant.) 
\end{enumerate}
Then, there exists a coupling between {\small$\mathbf{Z}^{\N}$} and the solution {\small$\mathbf{Z}$} to \eqref{eq:sheIa}-\eqref{eq:sheIb}, with initial data {\small$\mathbf{Z}_{0,\cdot}$} above and with parameter {\small$\mathbf{A}$} determined below, such that {\small$\mathbf{Z}^{\N}_{\t,\N\X}-\mathbf{Z}_{\t,\X}\to0$} locally uniformly in {\small$(\t,\X)\in[0,\infty)\times[0,\infty)$}:
\begin{align}
\mathbf{A}:=-\tfrac12\alpha^{-1}\lambda.\label{eq:mainII}
\end{align}
\end{theorem}
\begin{remark}
\fsp The constant {\small$\mathbf{A}$} also appears in \cite{GH19} as a shift for the Neumann-Robin parameters (though in \cite{GH19}, one has {\small$\alpha=1$} since the potential is taken to be {\small$\mathscr{U}(\a)=\frac12\a^{2}$} therein).
\end{remark}
Many examples of initial data satisfying Assumption \ref{assump:data} and \eqref{eq:mainI} may be constructed fairly easily. Indeed, consider any (possibly deterministic) initial data for \abbr{KPZ} whose exponential, after rescaling space by {\small$\N$}, satisfies \eqref{eq:mainI}. Then, we take initial data for {\small$\mathbf{j}^{\N}$} whose gradients are sampled according to the invariant measure {\small$\mathbb{P}^{0}$} but conditioned so that the resulting initial data for {\small$\mathbf{j}^{\N}$} is within {\small$\e(\N)$} of the original (possibly deterministic) data for spatially rescaled \abbr{KPZ} on the interval {\small$\llbracket0,\Lambda(\N)\cdot\N\rrbracket$}, for some {\small$\e(\N)\to0$} and {\small$\Lambda(\N)\to\infty$}, both sufficiently slowly. This slightly randomized initial data converges to the data of interest as {\small$\N\to\infty$} on a domain that is much larger than the macroscopic scale {\small$\N$}. It is essentially a consequence of density of the support of Wiener measure that if {\small$\e(\N)\to0$} and {\small$\Lambda(\N)\to\infty$} slowly enough, then the estimate \eqref{eq:dataI} is satisfied. We emphasize the importance that the \abbr{RHS} of \eqref{eq:dataI} diverges in the large-{\small$\N$} limit; if it did not, we could not take {\small$\e(\N)\to0$} or {\small$\Lambda(\N)\to\infty$} at all.

As an illustration of this construction, we will now present a derivation of statistically stationary solutions to the half-space \abbr{KPZ} equation from \eqref{eq:currI}-\eqref{eq:currII}. (Remarkably, this result does \emph{not} use anything about the invariant measures for \eqref{eq:currI}-\eqref{eq:currII} except that it has {\small$\mathbb{P}^{0}$} as an invariant measure.) Before we state the result, for the sake of clear presentation, we first introduce the family of invariant measures for half-space \abbr{KPZ} from \cite{BC23} as follows.
\begin{definition}\label{definition:statkpz}
\fsp Fix any {\small$\mathbf{A}\in\R$} and {\small${v}\leq\min(0,\mathbf{A})$}. Define the process
\begin{align*}
\mathbf{h}^{\mathbf{A},{v}}_{\X}:=\log\Big({\omega}^{-1}\int_{0}^{\X}\mathrm{e}^{{B}^{(1)}_{\s}+{B}^{(2)}_{\X}-{B}^{(2)}_{\s}}+\mathrm{e}^{{B}^{(2)}_{\X}}\d\s\Big), \quad \X\in[0,\infty),
\end{align*}
in which {\small${B}^{(1)},{B}^{(2)}$} are independent standard Brownian motions with constant drifts given by {\small${-v},{v}$}, respectively, and where {\small$\omega$} is an independent random variable with law {\small$\mathrm{Gamma}^{-1}(\mathbf{A}-v)$}. (As in \cite{BC23}, we adopt the convention that {\small$\omega=0$} if {\small$\mathbf{A}=v$}.) We will denote by {\small$\bmu_{\mathbf{A},v}$} the law of the above (as a process in {\small$\X\in[0,\infty)$}).
\end{definition}
\begin{remark}
\fsp We emphasize that {\small$\bmu_{\mathbf{A},v}$} is a stationary measure for \eqref{eq:kpzIa}-\eqref{eq:kpzIb} if {\small$\alpha=\beta=\frac12$} therein, not necessarily in general. (Though, by a standard rescaling of parameters, we can use Definition \ref{definition:statkpz} to find invariant measures for general {\small$\alpha,\beta$}. We do not perform this calculation in this paper, as it is not the main goal.)
\end{remark}
\begin{theorem}\label{theorem:mainstat}
\fsp Take any measure {\small$\bmu_{\mathbf{A},v}$} from Definition \ref{definition:statkpz}. Let {\small$\mathbf{Z}$} be the solution to \eqref{eq:sheIa}-\eqref{eq:sheIb} with initial data given by {\small$\exp(\lambda\mathbf{h}^{\mathbf{A},v})$}, where {\small$\mathbf{h}^{\mathbf{A},v}$} is sampled according to {\small$\bmu_{\mathbf{A},v}$}. There exists initial data for \eqref{eq:currI}-\eqref{eq:currII} so that there is a coupling between {\small$\mathbf{Z}^{\N}$} and {\small$\mathbf{Z}$} that satisfies {\small$\mathbf{Z}^{\N}_{\t,\N\X}-\mathbf{Z}_{\t,\X}\to0$} locally uniformly in {\small$(\t,\X)\in[0,\infty)\times[0,\infty)$}. (Again, we extend functions defined for {\small$\X\in[0,\infty)\cap\N^{-1}\Z$} to all {\small$\X\in[0,\infty)$} by linear interpolation.)
\end{theorem}
Let us briefly comment on an interesting feature of Theorem \ref{theorem:mainstat}. A natural approach for obtaining stationary fluctuations would be to use the theory of energy solutions as developed in \cite{GJ14,GP18} in the full-space (but in the half-space). However, to our knowledge, it is not known how to construct space-discretizations of the invariant measures in Definition \ref{definition:statkpz} that are invariant for \eqref{eq:currI}-\eqref{eq:currII}, let alone how to prove any estimates for them. This obstructs the use of energy solutions; on the other hand, Theorem \ref{theorem:mainstat} (and its proof) tells us that this information is not so important, and that one only needs to understand the simple invariant measure coming from {\small$\mathbb{P}^{0}$}.
\subsubsection{Universality for wedge-type initial data}
We now address a special case of singular initial data that is far from (statistical \emph{and} hydrodynamic) equilibrium. The relevant solution in this case is the \emph{narrow-wedge solution} to \eqref{eq:sheIa}-\eqref{eq:sheIb}. To be precise, this is the continuous, adapted process {\small$\mathbf{Z}^{\mathrm{NW}}:[0,\infty)\times[0,\infty)\to\R$} such that 
\begin{align}
\mathbf{Z}^{\mathrm{NW}}_{\t,\X}=\mathbf{H}^{\mathbf{A}}_{\t,\X,0}+\int_{0}^{\t}\int_{\R}\mathbf{H}^{\mathbf{A}}_{\t-\s,\x,\Y}\lambda\mathbf{Z}^{\mathrm{NW}}_{\s,\Y}\xi_{\s,\Y}\d\Y\d\s,
\end{align}
where the integral on the \abbr{RHS} is in the It\^{o}-Walsh sense, and {\small$\mathbf{H}^{\mathbf{A}}$} is the Robin heat kernel with boundary parameter {\small$\lambda\mathbf{A}$}. (That is, we have {\small$\partial_{\t}\mathbf{H}^{\mathbf{A}}_{\t,\X,\Y}=\alpha\partial_{\X}^{2}\mathbf{H}^{\mathbf{A}}_{\t,\X,\Y}$} with {\small$\partial_{\X}\mathbf{H}^{\mathbf{A}}|_{\X=0}=\lambda\mathbf{A}\mathbf{H}^{\mathbf{A}}$}. The extra factor of {\small$\lambda$} takes into account the presence of {\small$\lambda$} inside the exponential in \eqref{eq:ch}, and it is purely a formal convention.)
\begin{theorem}\label{theorem:mainwedge}
\fsp There exists an initial data for \eqref{eq:currI}-\eqref{eq:currII} such that for some deterministic sequence {\small$\mathcal{R}_{\N}\to\infty$}, there exists a coupling between {\small$\mathbf{Z}^{\N}$} and {\small$\mathbf{Z}^{\mathrm{NW}}$} such that {\small$\mathcal{R}_{\N}\mathbf{Z}^{\N}_{\t,\N\X}-\mathbf{Z}^{\mathrm{NW}}_{\t,\X}\to0$} locally uniformly in {\small$[\tau,\infty)\times[0,\infty)$} for any fixed {\small$\tau>0$}. (Again, we extend functions of {\small$\X\in[0,\infty)\cap\N^{-1}\Z$} to {\small$\X\in[0,\infty)$} by linear interpolation.)
\end{theorem}
The initial data in Theorem \ref{theorem:mainwedge} is ultimately built by taking a smooth approximation to a delta function that blows up slowly as {\small$\N\to\infty$}, and then applying a slight randomization as explained after Theorem \ref{theorem:main}. Moreover, by Theorem 1.7 in \cite{IMS26} and Theorem \ref{theorem:mainwedge}, we obtain a Tracy-Widom \abbr{GOE} or \abbr{GSE} distributional limit (depending on the value of the limiting Neumann parameter \eqref{eq:mainII}) under the double limit given by convergence to {\small$\log\mathbf{Z}^{\mathrm{NW}}_{\t,0}$} then to Tracy-Widom. We emphasize that both \abbr{GOE} and \abbr{GSE} statistics can be obtained from the model \eqref{eq:currI}-\eqref{eq:currII} by varying the coupling constant {\small$\beta_{2}$} in \eqref{eq:nl}.
\begin{remark}\label{remark:horizon}
\fsp For clarity of presentation, we show locally uniform convergence on {\small$[0,1]\times[0,\infty)$} in Theorems \ref{theorem:main}, \ref{theorem:mainstat}, and \ref{theorem:mainwedge} (and, in particular, we restrict to a time-horizon of {\small$1$}). There is nothing special about {\small$1$}, though. It is just a convenient choice that can be replaced by any finite horizon.
\end{remark}
\subsection{Plan for the paper}
In Section \ref{section:method}, we give a guide through the technical aspects of this paper. In Section \ref{section:sde}, we compute an appropriate (ansatz for the) \abbr{SDE} satisfied by the Cole-Hopf transform. In Section \ref{section:shk}, we introduce the framework of stochastic heat kernels to analyze the Cole-Hopf \abbr{SDE}, breaking the proof of the main theorems into two results (Proposition \ref{prop:hkzeta} and Lemma \ref{lemma:smallzeta}). In Section \ref{section:stochI}, we record preliminary (of large-deviations-type) stochastic estimates for some coefficients in the Cole-Hopf \abbr{SDE} (and for regularity of {\small$\mathbf{j}^{\N}$}). In Section \ref{section:stochII}, we record key \emph{Boltzmann-Gibbs} estimates for the most singular coefficients in the Cole-Hopf \abbr{SDE}. In Section \ref{section:hkzeta}, we show Proposition \ref{prop:hkzeta}. In Section \ref{section:smallzeta}, we show Lemma \ref{lemma:smallzeta}. The appendix collects technical, standard results.
\subsection{Notation}\label{subsection:notation}
Below, we present a list of notation to be used throughout the paper; we focus on notation that is used in multiple sections, whereas essentially all other notation is introduced in the section it is used. The reader is invited to come back to this list during the reading of this paper.
\begin{enumerate}
\item We use the standard notation {\small$\a=\mathrm{O}(\b)$} if {\small$|\a|\leq\mathrm{C}|\b|$} for some constant {\small$\mathrm{C}>0$}. If {\small$\mathrm{C}$} depends on any number of parameters, said parameters will be recorded as subscripts in the {\small$\mathrm{O}$}. We will at times use {\small$\a\lesssim\b$} to mean {\small$\a=\mathrm{O}(\b)$}. Moreover, by {\small$\a=\mathrm{o}(\b)$}, we mean that {\small$|\a|/|\b|\to0$} in the large-{\small$\N$} limit. Finally, by {\small$\a\asymp\b$}, we mean both {\small$\a\lesssim\b$} and {\small$\b\lesssim\a$} (with different implied constants).
\item For any {\small$\a,\b\in\R$}, we use {\small$\llbracket\a,\b\rrbracket:=[\a,\b]\cap\Z$} both as an indexing set (for summations) and a spatial domain.
\item We say that an event {\small$\mathscr{E}$} holds with high probability if {\small$\mathbb{P}(\mathscr{E})=1-\mathrm{o}(1)$}.
\item The process {\small$\mathbf{Z}^{\N}$} is the Cole-Hopf transform defined in \eqref{eq:ch}. The notation {\small$\mathbf{S}^{\N}$} is a smoothened version of {\small$\mathbf{Z}^{\N}$} (on a mesoscopic scale), and {\small$\mathbf{R}^{\N}:=\mathbf{Z}^{\N}\cdot(\mathbf{S}^{\N})^{-1}$} is the corresponding ratio; these are introduced in Definition \ref{definition:heat}. Relatedly, the parameter {\small$\delta_{\mathbf{S}}>0$} is a small exponent that determines the smoothing scale for {\small$\mathbf{S}^{\N}$}; this is also introduced in Definition \ref{definition:heat}. Moreover:
\begin{itemize}
\item Any bold English letter with one space and one time input is a modification of {\small$\mathbf{Z}^{\N}$} or {\small$\mathbf{S}^{\N}$}. For instance, {\small$\mathbf{S}^{\N,\zeta}$} is a modification of {\small$\mathbf{S}^{\N}$} obtained by cutting off terms in its evolution equation in space (in a fashion depending on the parameter {\small$\zeta>0$}); see Section \ref{section:shk}.
\item Bold caligraphic letters (such as {\small$\boldsymbol{\mathcal{K}}$} or {\small$\boldsymbol{\mathcal{S}}$} in Proposition \ref{prop:sde} and Lemma \ref{lemma:sdepre}) are smoothing kernels that are supported on a mesoscopic spatial scale.
\item Bold Greek letters are used to package important terms in our analysis. (For example, see the notation in Lemma \ref{lemma:duhamel}.) 
\end{itemize}
\item Any bold English letter with two time-variables and space-variables is a (heat) kernel. For example, {\small$\mathbf{H}^{\N,\mathbf{a}}_{\s,\t,\x,\y}$} is the heat kernel for the diffusively scaled Robin Laplacian on {\small$\llbracket0,\infty\rrbracket$} (with  parameter determined by {\small$\mathbf{a}\in\R$}). Similarly, the kernels {\small$\mathbf{K}^{\N,\zeta}_{\s,\t,\x,\y}$} in Section \ref{section:shk} denote fundamental solutions for the \abbr{SDE} satisfied by {\small$\mathbf{S}^{\N,\zeta}$}-terms. (We will use {\small$\mathbf{H}^{\N}$} to mean {\small$\mathbf{H}^{\N,\mathbf{a}}$} for the special value {\small$\mathbf{a}=\mathbf{A}=\eqref{eq:mainII}$} for convenience.)
\begin{itemize}
\item Relatedly, the Robin Laplacian that defines the heat kernel {\small$\mathbf{H}^{\N,\mathbf{a}}$} is denoted by {\small$\mathscr{T}_{\N,\mathbf{a}}$}; similarly, we use {\small$\mathscr{T}_{\N}$} to mean {\small$\mathscr{T}_{\N,\mathbf{a}}$} for {\small$\mathbf{a}=\mathbf{A}=\eqref{eq:mainII}$}. We let {\small$\t\mapsto\exp(\t\mathscr{T}_{\N,\mathbf{a}})$} be the associated semigroup (with kernel {\small$\mathbf{H}^{\N,\mathbf{a}}$}).
\end{itemize}
\item We let {\small$\boldsymbol{\chi}^{(\zeta)}:\llbracket0,\infty\rrbracket\to\R$} denote some smooth bump function (for any {\small$\zeta>0$}) such that {\small$\boldsymbol{\chi}\equiv1$} on {\small$\llbracket0,\N^{1+\zeta}\rrbracket$} and {\small$\boldsymbol{\chi}\equiv0$} on {\small$\llbracket\mathrm{O}(\N^{1+\zeta}),\infty\rrbracket$}. This is the bump function that we use to cutoff the evolution equation for {\small$\mathbf{S}^{\N}$} in space; see Section \ref{section:shk}. We will distinguish a separate but fixed large parameter {\small$\zeta_{\mathrm{large}}>0$}.
\item We will let fraktur notation (such as {\small$\mathfrak{q},\mathfrak{b}$}) denote one of a class of ``error terms" in our analysis of {\small$\mathbf{Z}^{\N}$}; these are functions of {\small$\bphi\in\R^{\llbracket1,\infty\rrbracket}$}, and the exact type of functions that we will consider to be error terms are classified in Definitions \ref{definition:bulkadmissible} and \ref{definition:edgeadmissible}. (These are ``bulk admissible" and ``edge admissible" terms, respectively.)
\item The space-time averaging operator {\small$\mathbf{Av}^{\mathbf{T},\mathbf{X}}$}, the space-averaging operator {\small$\mathbf{Av}^{\mathbf{X}}$}, and the time-averaging operator {\small$\mathbf{Av}^{\mathbf{T}}$} are introduced in Definition \ref{definition:bulkaverage} and \ref{definition:edgeaverage}. We refer to these environments for the exact space-time, spatial, and time scales on which we average. The related time-gradient {\small$\grad^{\mathbf{T},\mathrm{av}}$} is built in Definition \ref{definition:grad}.
\end{enumerate}

%
%
%
\section{Discussion of the methods}\label{section:method}
The goal for this section is provide a detailed yet still intuitive discussion of the proof.

The general proofs for Theorems \ref{theorem:main}, \ref{theorem:mainstat}, and \ref{theorem:mainwedge} are essentially the same (except with different ``last steps"), so let us focus on Theorem \ref{theorem:main}. In a nutshell, the general strategy consists of the following steps:
\begin{enumerate}
\item Establish an ansatz for the \abbr{SDE} satisfied by the Cole-Hopf transform \eqref{eq:ch}, which has the form of (a spatial-discretization of) \eqref{eq:sheIa}-\eqref{eq:sheIb} with additional terms that are linear in \eqref{eq:ch} with ``fluctuating coefficients". The main ideas here are to glue together an equation in the bulk, which comes from \cite{Y26PTRF}, and one at the edge, which is to be derived in Section \ref{section:sde}.
\item Establish probabilistic \emph{Boltzmann-Gibbs estimates} for the fluctuations in the ansatz above. \emph{This is where the localization scheme via Nash-Davies estimates is implemented.}
\item Use the Boltzmann-Gibbs estimates to estimate the \emph{stochastic heat kernel} associated to the \abbr{SDE} in step (1).
\end{enumerate}
The vast majority of this paper is dedicated to challenges that are specific to the half-space. The (relatively small) rest of this paper is dedicated to presenting a streamlined version of ideas from \cite{Y26PTRF}.

Let us now describe each step above in a detailed yet still intuitive fashion.
\subsection{SDE for the Cole-Hopf transform}\label{subsection:method1}
We start with an evolution equation for the Cole-Hopf transform \eqref{eq:ch}. This is computed by the formula \eqref{eq:ch} in terms of {\small$\mathbf{j}^{\N}$}, the evolution equation \eqref{eq:currI}-\eqref{eq:currII} for {\small$\mathbf{j}^{\N}$}, and the It\^{o} formula, and it produces the following \abbr{SDE}, which is a lattice version of the continuum \abbr{SHE} \eqref{eq:sheIa}-\eqref{eq:sheIb} but with a number of additional terms to be described below:
\begin{align}
\d\mathbf{Z}^{\N}_{\t,\x}\approx\alpha\N^{2}\Delta_{\mathbf{A}}\mathbf{Z}^{\N}_{\t,\x}\d\t+\sqrt{2}\lambda\N^{\frac12}\mathbf{Z}^{\N}_{\t,\x}\d\mathbf{b}_{\t,\x}+\mathfrak{e}_{\mathrm{bulk}}\d\t+\mathfrak{e}_{\mathrm{edge}}\d\t.\label{eq:method1a}
\end{align}
Here, {\small$\Delta_{\mathbf{A}}$} denotes a lattice approximation to the Robin Laplacian with edge parameter {\small$\mathbf{A}$} (namely, the Laplacian with boundary data \eqref{eq:sheIb}). The quantities {\small$\mathfrak{e}_{\mathrm{bulk}}$} and {\small$\mathfrak{e}_{\mathrm{edge}}$} are supported in the bulk and edge, respectively, and up to leading order, they essentially have the form
\begin{align*}
\mathfrak{e}_{\mathrm{bulk}}&\approx\N^{\frac32}\grad^{\mathbf{X}}_{1}((\mathscr{U}'[\bphi_{\t,\x}]-\alpha\bphi_{\t,\x})\mathbf{Z}^{\N}_{\t,\x})+\N\mathfrak{q}_{\x}[\bphi_{\t}]\mathbf{Z}^{\N}_{\t,\x},\\
\mathfrak{e}_{\mathrm{edge}}&\approx\mathbf{1}_{\x=0}\N^{\frac32}(\mathscr{U}'[\bphi_{\t,\x}]-\alpha\bphi_{\t,\x})\mathbf{Z}^{\N}_{\t,\x}+\mathbf{1}_{\x=0}\N\alpha\lambda\mathbf{A}\mathbf{Z}^{\N}_{\t,\x}+\mathbf{1}_{\x\in\llbracket0,1\rrbracket}\N\mathfrak{b}_{\x}[\bphi_{\t}]\mathbf{Z}^{\N}_{\t,\x}.
\end{align*}
The quantity {\small$\mathfrak{q}_{\x}$} encodes, among many other things, the fluctuations of the {\small$\mathbf{F}$}-nonlinearity in \eqref{eq:nl} about a suitable quadratic one, whereas the first gradient term on the \abbr{RHS} of the first line above encodes fluctuations of {\small$\mathscr{U}$} around a Gaussian quadratic potential. The second line has a similar flavor, with an additional constant from the boundary condition in {\small$\Delta_{\mathbf{A}}$}. Thus, we anticipate that {\small$\mathfrak{e}_{\mathrm{bulk}},\mathfrak{e}_{\mathrm{edge}}$} are error terms and have a vanishing contribution (in some topology) in the large-{\small$\N$} limit. 

The key challenge in analyzing the equation \eqref{eq:method1a} is the singular nature of the fluctuations {\small$\mathfrak{e}_{\mathrm{bulk}},\mathfrak{e}_{\mathrm{edge}}$}. In order to handle these singular terms, it will be convenient to introduce a mesoscopic spatial smoothing of {\small$\mathbf{Z}^{\N}$}. That is, consider the semigroup {\small$\exp(\mathfrak{t}\N^{2}\Delta_{\mathbf{A}})$} for some {\small$\mathfrak{t}\sim\N^{-\delta}$} and some {\small$\delta>0$} small, and set {\small$\mathbf{S}^{\N,\sim}:=\exp(\mathfrak{t}\N^{2}\Delta_{\mathbf{A}})\mathbf{Z}^{\N}$}. Since we anticipate {\small$\mathbf{Z}^{\N}\approx$} \abbr{SHE}, and since {\small$\mathfrak{t}\ll1$}, we also anticipate {\small$\mathbf{S}^{\N,\sim}\approx\mathbf{Z}^{\N}$}, because \abbr{SHE} has H\"{o}lder regularity in space with positive exponent \cite{CS18,P19}. (This is circular, but we make it rigorous by directly showing regularity for {\small$\mathbf{Z}^{\N}$}.) On the other hand, {\small$\mathfrak{t}$} is not too small, so we obtain a considerable amount of smoothness. This particular choice of smoothing is important because it commutes with the Laplacian {\small$\Delta_{\mathbf{A}}$} in \eqref{eq:method1a}. Thus, we have
\begin{align*}
\d\mathbf{S}^{\N,\sim}\approx\alpha\N^{2}\Delta_{\mathbf{A}}\mathbf{S}^{\N,\sim}\d\t+\sqrt{2}\lambda\N^{\frac12}\exp(\mathfrak{t}\N^{2}\Delta_{\mathbf{A}})(\mathbf{S}^{\N,\sim}\d\mathbf{b})+\exp(\mathfrak{t}\N^{2}\Delta_{\mathbf{A}})\mathfrak{e}_{\mathrm{bulk}}\d\t+\exp(\mathfrak{t}\N^{2}\Delta_{\mathbf{A}})\mathfrak{e}_{\mathrm{edge}}\d\t.
\end{align*}
(The other factors of {\small$\mathbf{Z}^{\N}$} on the \abbr{RHS} of \eqref{eq:method1a} are replaced by {\small$\mathbf{S}^{\N,\sim}$} because we anticipate the two to be essentially the same. Let us also clarify that the actual smoothing we choose is a spatially-cutoff version of the semigroup {\small$\exp(\mathfrak{t}\N^{2}\Delta_{\mathbf{A}})$} since we will only be able to establish local regularity estimates for {\small$\mathbf{Z}^{\N}$} to justify the approximation of {\small$\mathbf{Z}^{\N}$} by its smoothing.) An explicit advantage of the spatial smoothing is that summation-by-parts allows us to move the gradient in {\small$\mathfrak{e}_{\mathrm{bulk}}$} away from the singular fluctuation and to the smooth {\small$\exp(\mathfrak{t}\N^{2}\Delta_{\mathbf{A}})$}-kernel. \emph{Moreover, it turns out that the resulting boundary term at {\small$0\in\llbracket0,\infty\rrbracket$} cancels with the order {\small$\N^{3/2}$}-term in {\small$\mathfrak{e}_{\mathrm{edge}}$}, up to a term which fluctuates around {\small$\N\alpha\lambda\mathbf{A}$}.} This explains the renormalization of the boundary parameter  \eqref{eq:mainII}.

At this point, we will replace the fluctuations in {\small$\mathfrak{e}_{\mathrm{bulk}},\mathfrak{e}_{\mathrm{edge}}$} with their space-time- and time-averages. Indeed, the Boltzmann-Gibbs estimates for these fluctuations will hold only after performing such averaging. The scales for this averaging will be mesoscopic, and therefore, the errors behind introducing this averaging may ultimately be absorbed by the smoothness of the {\small$\exp(\mathfrak{t}\N^{2}\Delta_{\mathbf{A}})$}-kernels. This produces a formula of the type
\begin{align}
\d{\mathbf{S}}^{\N,\sim}_{\t,\x}&\approx\alpha\N^{2}\Delta_{\mathbf{A}}\mathbf{S}^{\N,\sim}\d\t+\sqrt{2}\lambda\N^{\frac12}\exp(\mathfrak{t}\N^{2}\Delta_{\mathbf{A}})(\mathbf{S}^{\N,\sim}_{\t,\cdot}\d\mathbf{b}_{\t,\cdot})_{\t,\x}\label{eq:method1c}\\
&+\sum_{\w\in\llbracket0,\infty\rrbracket}\boldsymbol{\mathcal{K}}_{\x,\w}\cdot\N\mathbf{Av}^{\mathbf{T},\mathbf{X}}_{\t,\w}\cdot{\mathbf{S}}^{\N,\sim}_{\t,\w}\d\t+\sum_{\w\in\llbracket0,1\rrbracket}\boldsymbol{\mathcal{K}}^{\partial}_{\x,\w}\cdot\N\mathbf{Av}^{\mathbf{T}}_{\t,\w}\cdot{\mathbf{S}}^{\N,\sim}_{\t,\w}\d\t\nonumber
\end{align}
for some nice kernels {\small$\boldsymbol{\mathcal{K}},\boldsymbol{\mathcal{K}}^{\partial}$} which are built out of the {\small$\N^{2}\Delta_{\mathbf{A}}$}-kernel and its gradients, as mentioned above. The {\small$\mathbf{Av}^{\mathbf{T},\mathbf{X}}$} and {\small$\mathbf{Av}^{\mathbf{T}}$} are the aforementioned space-time- and time-averages. This is essentially the \abbr{SDE} that we will work with. The majority of the work in this first step that remains is identifying what it means for a function (like {\small$\mathfrak{q}$} or {\small$\mathfrak{b}$} above) to be \emph{fluctuating}, with a different notion required for the bulk terms (like {\small$\mathfrak{q}$}) and the edge terms (like {\small$\mathfrak{b}$}). What also remains is identifying the scales at which we must perform the space-time- and time-averages, as well as a careful bookkeeping of all the other terms which are omitted via the {\small$\approx$}-notation above.
\subsection{Probabilistic estimates}\label{subsection:method2}
The main goal in this step is to show estimates for the {\small$\mathbf{Av}^{\mathbf{T},\mathbf{X}}$}- and {\small$\mathbf{Av}^{\mathbf{T}}$}-quantities appearing in \eqref{eq:method1c}. Precisely, the goal is to show
\begin{align}
\int_{0}^{\mathrm{T}}\sum_{\x\in\mathbb{I}}|\mathbf{Av}^{\mathbf{T},\mathbf{X}}_{\t,\x}|^{2}\d\t\lesssim\N^{-2-\beta_{\mathrm{CLT}}}|\mathbb{I}| \quad\text{and}\quad \int_{0}^{\mathrm{T}}\sum_{\x\in\llbracket0,1\rrbracket}|\mathbf{Av}^{\mathbf{T}}_{\t,\x}|^{2}\d\t\lesssim\N^{-\beta_{\mathrm{CLT}}}\label{eq:method2a}
\end{align}
for some {\small$\beta_{\mathrm{CLT}}>0$} fixed (reflecting the fact that the estimates above come from \abbr{CLT}-type cancellations), for any interval {\small$\mathbb{I}\subseteq\llbracket0,\infty\rrbracket$}, and for any {\small$\mathrm{T}\lesssim1$}. We clarify that the above estimates are to be shown with high probability given any preliminary choice of interval {\small$\mathbb{I}$}, not for all intervals simultaneously. In particular, the estimates above hold in a very weak topology. We discuss the utility of these estimates for analyzing \eqref{eq:method1c} in Section \ref{subsection:method3}.

We now describe the localization alluded to Section \ref{subsection:introtech}. A concrete representation of the {\small$\mathbf{Av}^{\mathbf{T},\mathbf{X}}$}-term above is given as follows, in which {\small$\mathbb{U}\subseteq[0,\infty)$} is a finite time-interval, in which {\small$\mathbb{I}_{\x}\subseteq\llbracket0,\infty\rrbracket$} contains {\small$\x$}, in which {\small$\boldsymbol{\mathcal{H}}_{\x,\w}$} is another nice kernel, and in which {\small$\mathfrak{q}_{\w}$} is ``fluctuating" (in a sense made precise in the construction of our ansatz \eqref{eq:method1c}, as discussed above) and localized near {\small$\w\in\llbracket0,\infty\rrbracket$}:
\begin{align}
\mathbf{Av}^{\mathbf{T},\mathbf{X}}_{\t,\x}=|\mathbb{U}|^{-1}\int_{\mathbb{U}}\sum_{\w\in\mathbb{I}_{\x}}\boldsymbol{\mathcal{H}}_{\x,\w}\cdot\mathfrak{q}_{\w}[\bphi_{\t+\s}]\d\s.\nonumber
\end{align}
Strictly speaking, we need to include other factors inside the space-time integration above that take into account the dynamics of {\small$\mathbf{Z}^{\N}$} itself, though let us forget about this technical detail for now. We now decompose the \abbr{RHS} of the previous display as follows:
\begin{itemize}
\item Let us decompose {\small$\mathbb{I}_{\x}=\mathbb{I}_{\x,1}\cup\ldots\cup\mathbb{I}_{\x,\mathrm{K}}$} into a finite collection (with {\small$\mathrm{K}\lesssim1$}) of consecutive intervals with sizes {\small$\asymp\N^{\gamma_{1}},\N^{\gamma_{2}},\ldots$}, where {\small$\gamma_{\k+1}=\gamma_{\k}+\e$} for some small {\small$\e>0$}. 
\item For every {\small$\k$}-index, we then decompose {\small$\mathbb{U}=\mathbb{U}_{\k,1}\cup\ldots\cup\mathbb{U}_{\k,\mathrm{L}}$} into a collection of adjacent time-intervals with equal lengths (or time-scales) {\small$\mathfrak{t}_{\k}$}. This lets us rewrite the previous display as 
\begin{align}
\mathbf{Av}^{\mathbf{T},\mathbf{X}}_{\t,\x}=\sum_{\k\in\llbracket1,\mathrm{K}\rrbracket}\mathrm{L}^{-1}\sum_{\ell\in\llbracket1,\mathrm{L}\rrbracket}\Big(|\mathbb{U_{\k,\ell}}|^{-1}\int_{\mathbb{U}_{\k,\ell}}\sum_{\w\in\mathbb{I}_{\x,\k}}\boldsymbol{\mathcal{H}}_{\x,\w}\cdot\mathfrak{q}_{\w}[\bphi_{\t+\s}]\d\s\Big).\label{eq:method2b}
\end{align}
\end{itemize}
The key aspects behind the decomposition \eqref{eq:method2b} are the following:
\begin{enumerate}
\item By construction, {\small$\mathbb{I}_{\x,\k}$} is separated from the edge by {\small$\gtrsim\N^{\gamma_{\k-1}}$}, at least for {\small$\k=2$}. Thus, assuming that we can establish estimates for the spatial propagation speed of \eqref{eq:currI}-\eqref{eq:currII}, then for small enough {\small$\mathfrak{t}_{\k}$}, information at the edge does not propagate into a neighborhood of {\small$\mathbb{I}_{\x,\k}$} with very high probability. This implies that we can compare {\small$\bphi$} to a full-space version and essentially inherit estimates from the full-space problem in \cite{Y26PTRF}.
\item However, we cannot choose {\small$\mathfrak{t}_{\k}$} to be too small, otherwise the Boltzmann-Gibbs cancellations that we obtain for the corresponding summand in \eqref{eq:method2b} will not be strong enough for the purposes of \eqref{eq:method2a}. This is especially relevant for smaller indices {\small$\k$}, where {\small$\gamma_{\k-1}$} is smaller. On the other hand, if {\small$\k$} is small, then the contribution of the corresponding summands in \eqref{eq:method2b} is smaller since {\small$|\mathbb{I}_{\x,\k}|\asymp\N^{\gamma_{\k}}$} is smaller.
\item Ultimately, the competition in point (2) essentially balances (which is \emph{good} for the purposes of proving \eqref{eq:method2a}, since it implies that the cancellation effects leading to \eqref{eq:method2a} do not deteriorate towards the edge).
\end{enumerate}
An edge-specific analysis for {\small$\k=1$} will also be developed separately, though this amounts to a separate (simpler) proof. What remains is the speed-of-propagation estimates for \eqref{eq:currI} in space. The basic principle starts with the full-space version of \eqref{eq:currI}, written as
\begin{align*}
\d\mathbf{j}^{\Z}_{\t,\x}=\N^{\frac32}\grad^{\mathbf{X}}_{1}(\mathscr{U}'[\bphi^{\Z}_{\t,\x}])\d\t+\N\mathbf{F}[\tau_{\x}\bphi^{\Z}_{\t}]\d\t+\sqrt{2}\N^{\frac12}\d\mathbf{b}_{\t,\x},
\end{align*}
where {\small$\bphi^{\Z}$} is the associated gradient field for {\small$\mathbf{j}^{\Z}$}, and the equation above is posed for {\small$(\t,\x)\in[0,\infty)\times\Z$} instead of {\small$[0,\infty)\times\llbracket0,\infty\rrbracket$}. We want to couple {\small$\mathbf{j}^{\N}$} to {\small$\mathbf{j}^{\Z}$}. By \eqref{eq:currI} and the previous display, we know that {\small$\mathbf{g}:=\mathbf{j}^{\Z}-\mathbf{j}^{\N}$} solves the following \abbr{PDE} away from the edge:
\begin{align*}
\partial_{\t}\mathbf{g}_{\t,\x}=-\N^{2}\grad^{\mathbf{X}}_{1}(\mathscr{U}''[\bpsi_{\t,\x}]\grad^{\mathbf{X}}_{-1}\mathbf{g}_{\t,\x})+\mathscr{D}\mathbf{g}_{\t,\x}=:\mathscr{P}\mathbf{g}_{\t,\x}.
\end{align*}
Above, {\small$\mathscr{D}$} is short-hand for some first-order gradient operator of speed {\small$\N^{3/2}$}, and {\small$\bpsi$} is a random field depending on {\small$\mathbf{j}^{\Z},\mathbf{j}^{\N}$} in a complicated way. Indeed, the Brownian motions have cancelled each other, and the previous display is a divergence-form parabolic \abbr{PDE}. Because the \abbr{PDE} above is linear, we can solve for {\small$\mathbf{g}$} via the corresponding heat kernel {\small$\boldsymbol{\Gamma}$}, at which point speed-of-propagation estimates follow by a heat kernel estimate of the form
\begin{align}
|\boldsymbol{\Gamma}_{\s,\t,\x,\y}|\lesssim\mathrm{e}^{\N|\t-\s|}\exp(-\tfrac{\upsilon|\x-\y|}{\N|\t-\s|^{1/2}}).\label{eq:method2c}
\end{align}
Above, {\small$\upsilon>0$} is fixed. Also, the denominator on the \abbr{RHS} in the exponential should be {\small$\N\max(\N^{-1},|\t-\s|^{1/2})$}, but this is a minor detail, as is the fact that the various small powers of {\small$\N$} are missing. The main point is instead that the \abbr{RHS} resembles a diffusive kernel behavior. We clarify that \eqref{eq:method2c} might seem to be a bad estimate (because of the first exponential factor on the \abbr{RHS}). However, we will only need \eqref{eq:method2c} for time-scales of essentially{\small$|\t-\s|\lesssim\N^{-1}$}, for which this exponential factor is harmless. \emph{This reflects the importance of a sub-criticality for the nonlinearity {\small$\mathbf{F}$}} (and the corresponding operator {\small$\mathscr{D}$}), \emph{in that at small scales, it is ultimately perturbative}.

The proof of \eqref{eq:method2c} is based on ideas of Nash and Davies \cite{N58,D89} (see also \cite{FS86} for a more thorough treatments) applied to {\small$\mathscr{P}$}. Although standard applications of these methods treat {\small$\mathscr{P}-\mathscr{D}$}, we can control the contribution of {\small$\mathscr{D}$} to get an estimate like \eqref{eq:method2c}. To this point, the first exponential on the \abbr{RHS} of \eqref{eq:method2c} is of the form {\small$\exp(\N|\t-\s|)$}, not the naive guess of {\small$\exp(\N^{3/2}|\t-\s|)$} (coming from the strength of {\small$\mathscr{D}$} and Gronwall). It is crucial for us to establish \eqref{eq:method2c} as written, not with {\small$\exp(\N^{3/2}|\t-\s|)$}; the fact that we can show \eqref{eq:method2c} as written reflects the importance of Nash's and Davies' methods in this work. Moreover, because this use of Nash-Davies theory is a technical crux of this paper, we develop a detailed analysis to prove \eqref{eq:method2c} which works not only for spatial-discretizations of \eqref{eq:glIa}-\eqref{eq:glIb}, but also continuous regularizations.

Let us complete this subsection by discussing a bit more context around \eqref{eq:method2c} and its proof. A similar estimate was obtained in \cite{Y25EJP} in the case where {\small$\mathbf{F}$} is degree {\small$1$}, in which case {\small$\boldsymbol{\Gamma}$} has the interpretation of a continuous-time random walk on {\small$\Z$}, from which \eqref{eq:method2c} follows from random walk theory. For general degree {\small$\mathbf{F}$} (like in \eqref{eq:nl}), this random walk mapping fails. In particular, {\small$\mathscr{P}$} is not the generator of a random walk, and it fails to even admit any pointwise positivity (e.g. maximum principle) properties. However, Nash-Davies methods do not require these inputs, rather only an {\small$\ell^{2}$}-positivity property satisfied by {\small$\mathscr{P}$}. We mention \cite{MY22}, which uses a similar ``{\small$\ell^{2}$}-positivity" to analyze parabolic \abbr{PDE}s with no maximum principle or probabilistic interpretation (in the completely different context of eigenvectors of random matrices).
\subsection{Stochastic heat kernels}\label{subsection:method3}
We will now use \eqref{eq:method2a} to study \eqref{eq:method1c}. The main (and immediate) challenge is that \eqref{eq:method2a} provides coefficient estimates in only a very weak topology; in fact, it is a local topology, because the first estimate in \eqref{eq:method2a} can be thought of as summing terms of {\small$\lesssim\N^{-2-\beta_{\mathrm{CLT}}}$} a total of {\small$|\mathbb{I}|$}-many times. Therefore, to get global estimates on \eqref{eq:method1c} becomes difficult, e.g. through the methods of weighted spaces as in \cite{HL18}.

The key idea we use to overcome this challenge is to consider the following hierarchy of \abbr{SDE}s parameterized by {\small$\zeta>0$}. Let {\small$\boldsymbol{\chi}^{(\zeta)}$} be a smooth cutoff function that is supported on {\small$\llbracket0,\mathrm{O}(\N^{1+\zeta})\rrbracket$} (i.e. a block that is a factor of {\small$\N^{\zeta}$}-larger than the macroscopic scale). We let {\small$\mathbf{S}^{\N,\sim,\zeta}$} be the solution to
\begin{align}
\d\mathbf{S}^{\N,\sim,\zeta}_{\t,\x}&\approx\alpha\N^{2}\Delta_{\mathbf{A}}\mathbf{S}^{\N,\sim,\zeta}_{\t,\x}\d\t+\sqrt{2}\lambda\N^{\frac12}\exp(\mathfrak{t}\N^{2}\Delta_{\mathbf{A}})(\boldsymbol{\chi}^{(\zeta)}_{\cdot}\mathbf{S}^{\N,\sim,\zeta}_{\t,\cdot}\d\mathbf{b}_{\t,\cdot})\label{eq:method3a}\\
&+\sum_{\w\in\llbracket0,\infty\rrbracket}\boldsymbol{\mathcal{K}}_{\x,\w}\cdot\boldsymbol{\chi}^{(\zeta)}_{\w}\cdot\N\mathbf{Av}^{\mathbf{T},\mathbf{X}}_{\t,\w}\cdot{\mathbf{S}}^{\N,\sim,\zeta}_{\t,\w}\d\t+\sum_{\w\in\llbracket0,1\rrbracket}\boldsymbol{\mathcal{K}}^{\partial}_{\x,\w}\cdot\N\mathbf{Av}^{\mathbf{T}}_{\t,\w}\cdot{\mathbf{S}}^{\N,\sim,\zeta}_{\t,\w}\d\t.\nonumber
\end{align}
In words, the \abbr{SDE} \eqref{eq:method3a} is the same as \eqref{eq:method1c} but with the spatial cutoff via {\small$\boldsymbol{\chi}^{(\zeta)}$}.\footnote{Technically, we use a separate cutoff exponent $\zeta_{\mathrm{large}}$ for the noise term in \eqref{eq:method3a}, but this is a minor detail.} Because the \abbr{RHS} of \eqref{eq:method1c} is local in space with speed {\small$\lesssim\N^{2}$}, it is not hard to prove that {\small$\mathbf{S}^{\N,\sim}\approx\mathbf{S}^{\N,\sim,\zeta}$} on domains of the form {\small$[0,\mathrm{O}(1)]\times\llbracket0,\mathrm{O}(\N)\rrbracket$} (which are exactly the domains of interest in Theorem \ref{theorem:main}), as long as {\small$\zeta>0$} is large. Our goal is then to compare {\small$\mathbf{S}^{\N,\sim,\zeta}$} between different values of {\small$\zeta$}, allowing us to compare {\small$\mathbf{S}^{\N,\sim}\approx\mathbf{S}^{\N,\sim,\zeta}$} for small {\small$\zeta>0$}. In particular, we will reduce the analysis to the ``essentially compact" setting, for which we can use fixed-point arguments to compare \eqref{eq:method3a} to a spatially-discretized half-space \abbr{SHE}. 

To implement the aforementioned comparison between different {\small$\zeta$}, we let {\small$\mathbf{K}^{\N,\sim,\zeta}$} be the fundamental solution (or stochastic heat kernel) associated to the linear equation satisfied by {\small$\mathbf{S}^{\N,\sim,\zeta}$}. Our key estimate is 
\begin{align}
\sup_{0\leq\s\leq\t\lesssim1}\sup_{\x\in\llbracket0,\infty\rrbracket}\sum_{\y\in\llbracket0,\infty\rrbracket}\exp(\tfrac{\kappa|\x-\y|}{\N})\cdot|{\mathbf{K}}^{\N,\sim,\zeta}_{\s,\t,\x,\y}|\lesssim_{\kappa}\N^{\delta}\exp(\N^{-\beta_{\star}}\N^{\zeta}),\label{eq:method3b}
\end{align}
where {\small$\kappa>0$} is any arbitrary constant, where {\small$\delta>0$} is small, and where {\small$\beta_{\star}>0$} is a fixed constant. The factor of {\small$\N^{\zeta}$} comes from the fact the fluctuations in the second line of \eqref{eq:method3a} are supported on length-scales which are {\small$\lesssim\N^{\zeta}$} larger than the macroscopic scale {\small$\N$}. The factor of {\small$\N^{-\beta_{\star}}$} comes from the power-saving for said fluctuations from \eqref{eq:method2a}. The exponential nature of the \abbr{RHS} comes from a Gronwall inequality. The exponential weight on the \abbr{LHS} comes from that of the deterministic {\small$\Delta_{\mathbf{A}}$}-heat kernel. The factor of {\small$\N^{\delta}$} is technical and not so important. 

The estimate \eqref{eq:method3b} implies that information beyond a length {\small$\N^{1+\zeta-\e}$} of the origin contributes an exponentially small amount to values of {\small$\mathbf{S}^{\N,\sim,\zeta}$} on {\small$[0,\mathrm{O}(1)]\times\llbracket0,\mathrm{O}(\N)\rrbracket$}), where {\small$\e$} is small but fixed (and independent of {\small$\zeta>0$}). Thus, {\small$\mathbf{S}^{\N,\sim,\zeta}\approx\mathbf{S}^{\N,\sim,\zeta-\e}$} on {\small$[0,\mathrm{O}(1)]\times\llbracket0,\mathrm{O}(\N)\rrbracket$}, hence the desired comparison.
%
%
%
\section{The main \abbr{SDE}}\label{section:sde}
The goal of this section is an \abbr{SDE} for the evolution of the Cole-Hopf transform \eqref{eq:ch} whose (stochastic) heat kernel is amenable to analysis. We start by writing the final form of this \abbr{SDE} in Proposition \ref{prop:sde}, after which we clarify its structure. Then, we prove Proposition \ref{prop:sde}.

Before we can present the \abbr{SDE} of interest, we will need to introduce preliminary constructions. The first is an analogue of (orthogonal complements of) the Gaussian Wiener chaos for more general potentials {\small$\mathscr{U}$}.
\begin{definition}\label{definition:jet}
\fsp We say that {\small$\mathsf{F}:\R^{\llbracket1,\infty\rrbracket}\to\R^{\llbracket1,\infty\rrbracket}$} is a \emph{local function} if there exists a finite subset {\small$\mathbb{I}\subseteq\llbracket1,\infty\rrbracket$} such that for any {\small$\bphi\in\R^{\llbracket1,\infty\rrbracket}$}, the value {\small$\mathsf{F}[\bphi]$} depends only on {\small$\bphi_{\x}$} for {\small$\x\in\mathbb{I}$}. We refer to the smallest such interval {\small$\mathbb{I}$} as the support of {\small$\mathsf{F}$}. Next, for any {\small$\mathfrak{k}=0,1,2$}, we define {\small$\mathrm{Jet}_{\mathfrak{k}}^{\perp}$} to be the space of local functions whose {\small$\mathfrak{k}$}-jet vanishes in the following sense:
\begin{align*}
\partial_{\sigma}^{\ell}\E^{\sigma}\mathsf{F}[\bphi]|_{\sigma=0}=0 \quad\text{for all} \ \ell=0,\ldots,\mathfrak{k}.
\end{align*}
\end{definition}
Intuitively, functions in {\small$\mathrm{Jet}_{\mathfrak{k}}^{\perp}$} should be thought of as carrying a factor of {\small$\N^{-(\mathfrak{k}+1)/2}$}. Indeed, local equilibrium suggests {\small$\mathsf{F}[\tau_{\x}\bphi_{\t}]\approx\E^{\sigma[\t,\x]}\mathsf{F}$}, where {\small$\sigma[\t,\x]$} is the average of {\small$\bphi_{\t,\w}$} for {\small$\w$} in a neighborhood of {\small$\x$} of radius {\small$\mathrm{o}(\N)$}, so by standard \abbr{CLT}-cancellations, we (intuitively) have {\small$|\sigma[\t,\x]|\lesssim\N^{-1/2}$}. Taylor expansion then yields the heuristic.

We now present two classes of nonlinearities, ones for which a bulk analysis will be developed, and ones for which a new edge analysis is required. Intuitive clarifications for both classes will be given in Remark \ref{remark:admissible}.
\begin{definition}\label{definition:bulkadmissible}
\fsp We say that a local function {\small$\mathfrak{q}:\R^{\llbracket1,\infty\rrbracket}\to\R$} is \emph{bulk admissible with respect to a point {\small$\x\in\llbracket0,\infty\rrbracket$}} if the following collection of conditions are satisfied:
\begin{enumerate}
\item There exists {\small$\mathfrak{k}\in\llbracket0,2\rrbracket$} such that {\small$\mathfrak{q}=\N^{-1+\mathfrak{k}/2}\mathfrak{f}$}, where {\small$\mathfrak{f}\in\mathrm{Jet}_{\mathfrak{k}}^{\perp}$}.
\item Exactly one of the following two conditions is satisfied for any {\small$\bphi\in\R^{\llbracket1,\infty\rrbracket}$} and for some {\small$0<\mathfrak{l}=\mathrm{O}(1)$}:
\begin{itemize}
\item The value {\small$\mathfrak{f}[\bphi]$} depends only on {\small$\bphi_{\w}$} for {\small$\w\in\llbracket\x+1,\x+\mathfrak{l}\rrbracket$};
\item The value {\small$\mathfrak{f}[\bphi]$} depends only on {\small$\bphi_{\w}$} for {\small$\w\in\llbracket\x-\mathfrak{l},\x\rrbracket\cap\llbracket1,\infty\rrbracket$}.
\end{itemize}
(That is, the ``support" of {\small$\mathfrak{f}$} has finite size uniformly in {\small$\N$}, and it is either to the right or to the left of the point {\small$\x\in\llbracket1,\infty\rrbracket$} at hand.) In the former case (with support in {\small$\llbracket\x+1,\infty\rrbracket$}), we say that {\small$\mathfrak{f}$} (and {\small$\mathfrak{q}$}) is \emph{forwards-oriented}. In the other case, we say it is \emph{backwards oriented}.
\item The function {\small$\mathfrak{f}$} admits the following deterministic estimate for any {\small$\bphi\in\R^{\llbracket1,\infty\rrbracket}$} and some fixed {\small$0<\mathrm{C}\lesssim1$}:
\begin{align}
|\mathfrak{f}[\bphi]|+\sup_{\y\in\llbracket1,\infty\rrbracket}|\partial_{\bphi_{\y}}\mathfrak{f}[\bphi]|\lesssim1+\sum_{|\w-\x|\lesssim1}|\bphi_{\w}|^{\mathrm{C}}.\label{eq:bulkadmissibleI}
\end{align}
\end{enumerate}
\end{definition}
\begin{definition}\label{definition:edgeadmissible}
\fsp We say that a function {\small$\mathfrak{b}:\R^{\llbracket1,\infty\rrbracket}$} is \emph{edge admissible} if the following are satisfied:
\begin{enumerate}
\item We have {\small$\E^{0}\mathfrak{b}=0$}.
\item Given any {\small$\bphi\in\R^{\llbracket1,\infty\rrbracket}$}, the value {\small$\mathfrak{b}[\bphi]$} depends only on {\small$\bphi_{\w}$} for {\small$\w\in\llbracket0,\mathrm{O}(1)\rrbracket$}.
\item The function {\small$\mathfrak{b}$} satisfies \eqref{eq:bulkadmissibleI} for {\small$\x=0$}.
\end{enumerate}
\end{definition}
\begin{remark}\label{remark:admissible}
\fsp Let us briefly clarify Definition \ref{definition:bulkadmissible} to start. Condition (1) suggests that {\small$\mathfrak{q}$} is roughly of order {\small$\N^{-3/2}$}, more or less, given the discussion after Definition \ref{definition:jet}. Condition (2) is an important support assumption; it states that the Brownian motions which drive the evolution of {\small$\mathfrak{q}[\bphi_{\t}]$} are basically independent of those which drive that of {\small$\mathbf{j}^{\N}_{\t,\x}$}, so that we ultimately do not have to estimate correlations between these two dynamics. Condition (3) is a technical a priori estimate to control the size of {\small$\mathfrak{q}$}.

Similarly, an edge-admissible function should be thought of as small, although the support assumption (2) in Definition \ref{definition:edgeadmissible} suggests that this is because of the Langevin-Glauber dynamic \eqref{eq:phiII} at the edge, not because of a local smoothing procedure as discussed after Definition \ref{definition:jet}.

We emphasize that despite the heuristic power-counting in {\small$\N$} that has been given, we will \emph{not} actually prove estimates as sharp as what this power-counting suggests, since this is highly technical and not necessary.
\end{remark}
The basic mechanism of cancellations in bulk- and edge-admissible functions comes by taking their averages in space and time. In what follows, we construct these averages and their space-time scales. To this end, we will require some notation for deterministic heat kernels and a slight smoothing of the Cole-Hopf transform \eqref{eq:ch} via said heat kernels (the latter of which we will not use for now, but is anyways convenient to introduce below).
\begin{definition}\label{definition:heat}
\fsp Fix any {\small$\mathbf{a}\in\R$}. We let {\small$\mathbf{H}^{\N,\mathbf{a}}_{\s,\t,\x,\y}$} be the function of {\small$(\s,\t,\x,\y)\in[0,\infty)^{2}\times\llbracket0,\infty\rrbracket^{2}$}, supported on the slice {\small$\s\leq\t$}, which solves the following deterministic \abbr{PDE}:
\begin{align}
\partial_{\t}\mathbf{H}^{\N,\mathbf{a}}_{\s,\t,\x,\y}=\mathscr{T}_{\N,\mathbf{a}}\mathbf{H}^{\N,\mathbf{a}}_{\s,\t,\x,\y}:=(\alpha\N^{2}+\tfrac14\N\lambda)\Delta_{\mathbf{a}}\mathbf{H}^{\N,\mathbf{a}}_{\s,\t,\x,\y} \quad \text{and} \quad \mathbf{H}^{\N,\mathbf{a}}_{\s,\s,\x,\y}=\mathbf{1}_{\x=\y}. \label{eq:heatI}
\end{align}
Above, {\small$\Delta_{\mathbf{a}}$} is the Robin Laplacian on {\small$\llbracket0,\infty\rrbracket$} with parameter {\small$\mathbf{a}$}. In particular, given any {\small$\mathsf{f}:\llbracket0,\infty\rrbracket\to\R$}, we have {\small$\Delta_{\mathbf{a}}\mathsf{f}=\Delta\mathsf{f}$} after extending {\small$\mathsf{f}$} to {\small$\llbracket-1,\infty\rrbracket$} by {\small$\mathsf{f}_{0}-\mathsf{f}_{-1}=\N^{-1}\mathbf{a}\lambda\mathsf{f}_{0}$}; see \eqref{eq:rg} for what {\small$\lambda\in\R$} is. In \eqref{eq:heatI}, the operator {\small$\mathscr{T}_{\N,\mathbf{a}}$} acts on {\small$\x$}. We also let {\small$\t\mapsto\mathbf{e}^{\t\mathscr{T}_{\N,\mathbf{a}}}$} be the associated semigroup. That is, for any {\small$\mathsf{f}:\llbracket0,\infty\rrbracket\to\R$}, we set
\begin{align}
[\mathbf{e}^{\t\mathscr{T}_{\N,\mathbf{a}}}(\mathsf{f})]_{\x}:=[\mathbf{e}^{\t\mathscr{T}_{\N,\mathbf{a}}}(\mathsf{f}_{\cdot})]_{\x}:=\sum_{\w\in\llbracket0,\infty\rrbracket}\mathbf{H}^{\N,\mathbf{a}}_{0,\t,\x,\w}\mathsf{f}_{\w}.\label{eq:heatII}
\end{align}
(We briefly remark here that standard estimates for the heat kernel, like those that are established in Proposition \ref{prop:hkestimates}, show that {\small$\mathbf{e}^{\t\mathscr{T}_{\N,\mathbf{a}}}:\ell^{2}(\llbracket0,\infty\rrbracket)\to\ell^{2}(\llbracket0,\infty\rrbracket)$} is a bounded operator.) Now, with {\small$\mathbf{A}$} from \eqref{eq:mainII}, we set
\begin{align}
\mathbf{S}^{\N}_{\t,\x}:=[\mathbf{e}^{\mathfrak{t}_{\N}\mathscr{T}_{\N,\mathbf{A}}}(\mathbf{1}_{|\x-\cdot|\lesssim(\log\N)^{2}\N^{}\mathfrak{t}_{\N}^{1/2}}\cdot\mathbf{Z}^{\N}_{\t,\cdot})]_{\x} \quad\text{and}\quad \mathbf{R}^{\N}_{\t,\x}:=\mathbf{Z}^{\N}_{\t,\x}\cdot(\mathbf{S}^{\N}_{\t,\x})^{-1}.\label{eq:heatIII}
\end{align}
Above, we introduced the mesoscopic time-scale {\small$\mathfrak{t}_{\N}=\N^{-\delta_{\mathbf{S}}}$} (with corresponding diffusive length-scale {\small$\N\mathfrak{t}_{\N}^{1/2}$}), where {\small$\delta_{\mathbf{S}}>0$} is a small, fixed constant (depending only on the parameters from \eqref{eq:dataI0}-\eqref{eq:dataI}). (We note {\small$\mathbf{S}^{\N}$} and {\small$\mathbf{R}^{\N}$} stand for ``smoothing" and ``ratio", respectively.) It will be convenient to denote the indicator in \eqref{eq:heatIII} by {\small$\mathbf{1}^{(\x)}_{\cdot}$} throughout this paper (for presentation's sake).
\end{definition}
We will later show that the smoothing in \eqref{eq:heatIII} is negligible in the large-{\small$\N$} limit of {\small$\mathbf{Z}^{\N}$}, because it occurs on a mesoscopic scale. Moreover, we use the heat semigroup with parameter {\small$\mathbf{A}$} to perform smoothing in \eqref{eq:heatIII}, since it will commute (and in general be compatible) with the operator {\small$\Delta_{\mathbf{A}}$} (that, in some sense, must show up in the evolution of {\small$\mathbf{Z}^{\N}$}). We also clarify that the indicator in \eqref{eq:heatIII} is almost entirely technical, as the heat kernel therein is exponentially small in {\small$\N$} on the complement of the support of the indicator function, though at the same time, the indicator helps restrict our regularity estimates for {\small$\mathbf{Z}^{\N}$}, which are needed to  obtain {\small$\mathbf{R}^{\N}\approx1$} in a sufficiently strong sense, to a local spatial scale instead of a global one.

We now present convenient notation for local averages of bulk-admissible functions (which we will need to record the desired evolution equation for \eqref{eq:ch}). We explain the constructions below in Remark \ref{remark:bulkaverage}.
\begin{definition}\label{definition:bulkaverage}
\fsp Fix the time-scale {\small$\mathfrak{t}_{\mathbf{Av}}:=\N^{-2/3-\mathrm{C}_{0}\delta_{\mathbf{S}}}$} and the length-scale {\small$\mathfrak{l}_{\mathbf{Av}}:=\N^{1-3\delta_{\mathbf{S}}/2}$}, where {\small$\delta_{\mathbf{S}}>0$} is the constant from Definition \ref{definition:heat}, and {\small$\mathrm{C}_{0}>0$} is a large but fixed constant. We also define {\small$\tau_{\mathbf{Av}}:=\N^{-2}\mathfrak{l}_{\mathbf{Av}}^{2}=\N^{-3\delta_{\mathbf{S}}}$} to be the diffusive time-scale associated to {\small$\mathfrak{l}_{\mathbf{Av}}$}.

Now, consider a collection {\small$\mathscr{Q}:=\{\mathfrak{q}_{\x}\}_{\x\in\llbracket0,\infty\rrbracket}$} of local functions. Let us define the space-time average
\begin{align}
\boldsymbol{\Upsilon}^{\mathbf{T},\mathbf{X},\mathscr{Q}}_{\t,\x}:=\mathfrak{t}_{\mathbf{Av}}^{-1}\int_{0}^{\mathfrak{t}_{\mathbf{Av}}}[\mathbf{e}^{\tau_{\mathbf{Av}}\mathscr{T}_{\N,0}}(\mathbf{1}_{|\x-\cdot|\lesssim(\log\N)^{2}\mathfrak{l}_{\mathbf{Av}}}\cdot\mathfrak{q}_{\cdot}[\bphi_{\t-\s}]\cdot\mathbf{Z}^{\N}_{\t-\s,\cdot})]_{\x}\d\s.\label{eq:bulkaverageI}
\end{align}
In words, we time-average on scale {\small$\mathfrak{t}_{\mathbf{Av}}$} in the backwards direction (to avoid stochastic calculus issues of adaptability), and then we average in space on a length-scale of (essentially) {\small$\mathfrak{l}_{\mathbf{Av}}$} with the {\small$\mathscr{T}_{\N,0}$}-kernel (for reasons to be explained in Remark \ref{remark:bulkaverage}). (In the previous formula, we adopt the convention that the integral excludes the set of times {\small$\s\in[0,\mathfrak{t}_{\mathbf{Av}}]$} for which {\small$\t-\s\leq0$}, though this choice of convention is not so important.)

We will now consider two special cases. First, assume that {\small$\mathfrak{q}_{\x}$} is bulk-admissible with respect to {\small$\x\in\llbracket0,\infty\rrbracket$}.
\begin{enumerate}
\item Suppose that {\small$\mathscr{Q}$} consists of only forwards-oriented functions (see Definition \ref{definition:bulkadmissible}). We set
\begin{align}
\mathbf{Av}^{\mathbf{T},\mathbf{X},\mathscr{Q}}_{\t,\x}:=\boldsymbol{\Upsilon}^{\mathbf{T},\mathbf{X},\mathscr{Q}}_{\t,\x}\cdot(\mathbf{Z}^{\N}_{\t,\x_{-}})^{-1}, \quad \text{where} \ \x_{-}=\max\{0,\x-\lfloor(\log\N)^{3}\mathfrak{l}_{\mathbf{Av}}\rfloor\}.\label{eq:bulkaverageII}
\end{align}
\item Suppose that {\small$\mathscr{Q}$} consists of only backwards-oriented functions (see Definition \ref{definition:bulkadmissible}). We set
\begin{align}
\mathbf{Av}^{\mathbf{T},\mathbf{X},\mathscr{Q}}_{\t,\x}:=\boldsymbol{\Upsilon}^{\mathbf{T},\mathbf{X},\mathscr{Q}}_{\t,\x}\cdot(\mathbf{Z}^{\N}_{\t,\x_{+}})^{-1}, \quad \text{where} \ \x_{+}=\x+\lfloor(\log\N)^{3}\mathfrak{l}_{\mathbf{Av}}\rfloor.\label{eq:bulkaverageIII}
\end{align}
\end{enumerate}
Finally, note that \eqref{eq:bulkaverageI}-\eqref{eq:bulkaverageIII} make sense if we formally replace {\small$\mathfrak{t}_{\mathbf{Av}}$} with {\small$0$}, in which the time-averages go away and we replace {\small$\s$} by {\small$0$} everywhere. In particular, we define the following which only involve a spatial-average:
\begin{align}
\boldsymbol{\Upsilon}^{\mathbf{X},\mathscr{Q}}_{\t,\x}&:=[\mathbf{e}^{\tau_{\mathbf{Av}}\mathscr{T}_{\N,0}}(\mathbf{1}_{|\x-\cdot|\lesssim(\log\N)^{2}\mathfrak{l}_{\mathbf{Av}}}\cdot\mathfrak{q}_{\cdot}[\bphi_{\t}]\cdot\mathbf{Z}^{\N}_{\t,\cdot})]_{\x},\label{eq:bulkaverageIV}\\
\mathbf{Av}^{\mathbf{X},\mathscr{Q}}_{\t,\x}&:=\boldsymbol{\Upsilon}^{\mathbf{X},\mathscr{Q}}_{\t,\x}\cdot(\mathbf{Z}^{\N}_{\t,\x_{\pm}})^{-1}.\label{eq:bulkaverageV}
\end{align}
\end{definition}
\begin{remark}\label{remark:bulkaverage}
\fsp Let us now clarify the constructions in Definition \ref{definition:bulkaverage}. The gadget {\small$\boldsymbol{\Upsilon}^{\mathbf{T},\mathbf{X},\mathscr{Q}}$} performs a space-time average of the entire quantity {\small$\mathfrak{q}_{\cdot}[\bphi_{\cdot}]\mathbf{Z}^{\N}_{\cdot,\cdot}$}. This gadget ultimately appears in the evolution equation for {\small$\mathbf{Z}^{\N}$}, and in order to view this equation as a multiplicative and linear one in {\small$\mathbf{Z}^{\N}$} (similar to \eqref{eq:sheIa}-\eqref{eq:sheIb} itself), we will need to factor out a copy of {\small$\mathbf{Z}^{\N}$} to arrive at the {\small$\mathbf{Av}^{\mathbf{T},\mathbf{X},\mathscr{Q}}$}-coefficient. The detail behind factoring out either {\small$\mathbf{Z}^{\N}_{\t,\x_{+}}$} or {\small$\mathbf{Z}^{\N}_{\t,\x_{-}}$} depending on the orientation of {\small$\mathscr{Q}$}-functions can be explained as follows. 
\begin{itemize}
\item In case (1) above, when we factor out {\small$\mathbf{Z}^{\N}_{\t,\x_{-}}$} from {\small$\boldsymbol{\Upsilon}^{\mathbf{T},\mathbf{X},\mathscr{Q}}_{\t,\x}$}, we end up with a ratio of the form {\small$\mathbf{Z}^{\N}_{\t-\s,\w}\cdot(\mathbf{Z}^{\N}_{\t,\x_{-}})^{-1}$} with {\small$\w\geq\x$}. By \eqref{eq:ch}, this ratio involves {\small$\bphi_{\tau,\y}$}-variables for times {\small$\tau$}, though only for points {\small$\y\leq\w$}. In particular, these extra factors do not ``correlate" with the forwards-oriented {\small$\mathfrak{q}_{\w}$} and do not destroy its ``fluctuating" nature.
\item In case (2), the same holds after reflecting about {\small$\w$}.
\end{itemize}

We now note that the relevance of \eqref{eq:bulkaverageIV}-\eqref{eq:bulkaverageV} comes from the fact that when we introduce averaging of bulk admissible functions, we will first introduce a spatial-average and then a time-average. Finally, the use of {\small$\mathbf{H}^{\N,0}$}-smoothing (as opposed to {\small$\mathbf{H}^{\N,\mathbf{a}}$} for other {\small$\mathbf{a}\in\R$}) is because this kernel is a probability measure on {\small$\llbracket0,\infty\rrbracket$} (with respect to either spatial variable). Thus, the {\small$\mathbf{H}^{\N,0}$}-kernel is actually an \emph{averaging} one, which will be convenient.
\end{remark}
We now present an averaging scheme for edge admissible functions, which we clarify in Remark \ref{remark:edgeaverage}.
\begin{definition}\label{definition:edgeaverage}
\fsp Suppose that {\small$\mathfrak{b}$} is an edge-admissible function. Setting {\small$\mathfrak{t}_{\partial}:=\N^{-3/2}$}, we define
\begin{align}
\boldsymbol{\Upsilon}^{\mathbf{T},\mathfrak{b}}_{\t,\x}:=\mathfrak{t}_{\partial}^{-1}\int_{0}^{\mathfrak{t}_{\partial}}\mathfrak{b}[\bphi_{\t-\s}]\mathbf{Z}^{\N}_{\t-\s,\x}\d\s \quad\text{and}\quad \mathbf{Av}^{\mathbf{T},\mathfrak{b}}_{\t,\x}:=\boldsymbol{\Upsilon}^{\mathbf{T},\mathfrak{b}}_{\t,\x}\cdot(\mathbf{Z}^{\N}_{\t,\x})^{-1}. \label{eq:edgeaverageI}
\end{align}
\end{definition}
\begin{remark}\label{remark:edgeaverage}
\fsp We do not include (and, in fact, will not need) a spatial-average for the analysis of edge-admissible functions. Indeed, {\color{black}as we explained in Section \ref{section:method}}, we only need power-saving estimates for {\small$\mathbf{Av}^{\mathbf{T},\mathfrak{b}}$}-quantities that are of the form {\small$\lesssim\N^{-\rho}$} with a possibly small but still fixed {\small$\rho>0$}. Because of this, there is no issue of {\small$\x_{\pm}$}-points as in Definition \ref{definition:bulkaverage}.  Moreover, since {\small$\mathfrak{t}_{\partial}=\N^{-3/2}$} is small, we can ultimately remove all the {\small$\mathbf{Z}^{\N}$}-factors from \eqref{eq:edgeaverageI} by showing time-regularity for such objects.

We also clarify that the choice of {\small$\mathfrak{t}_{\partial}=\N^{-3/2}$} is technical. On one hand, it is small in {\small$\N$}, which will allow us to ultimately justify a replacement of {\small$\mathfrak{b}[\bphi_{\cdot}]\mathbf{Z}^{\N}$} by a time-average on this scale. On the other hand, it is large enough to establish sufficient cancellations in the time-averages in \eqref{eq:edgeaverageI}. 
\end{remark}
Finally, before we state the \abbr{SDE} for the evolution of \eqref{eq:ch}, we will need notation for a time-gradient operator which encodes the errors in replacing various quantities by their time averages. 
\begin{definition}\label{definition:grad}
\fsp For any time-scale {\small$\mathfrak{t}\geq0$} and function {\small$\mathsf{g}:[0,\infty)\to\R$},  we define the averaged time-gradient
\begin{align}
\grad^{\mathbf{T},\mathrm{av}}_{\mathfrak{t}}\mathsf{g}_{\t}:=\mathbf{1}_{\mathfrak{t}>0}\mathfrak{t}^{-1}\int_{0}^{\mathfrak{t}}(\mathsf{g}_{\t}-\mathsf{g}_{\t-\r})\d\r.\label{eq:gradII}
\end{align}
(In order to make sense of the expression \eqref{eq:gradII}, we again drop all {\small$\r$} such that {\small$\t-\r\leq0$}.)
\end{definition}
We will now present the \abbr{SDE} evolution for \eqref{eq:ch} (or rather, for the slightly regularized version {\small$\mathbf{S}^{\N}$} from \eqref{eq:heatIII}). Before we do so, we briefly clarify its structure as a lattice version of the \abbr{SHE} \eqref{eq:sheIa}-\eqref{eq:sheIb} (with mesoscopic-scale smoothing as in \eqref{eq:heatIII}), plus additional error terms that are either (essentially) deterministically small, or that are small after space-time averaging as in Definitions \ref{definition:bulkaverage} and \ref{definition:edgeaverage}. We give more detailed clarification of Proposition \ref{prop:sde} in Remark \ref{remark:sde}.
\begin{prop}\label{prop:sde}
\fsp First, we recall {\small$\mathfrak{t}_{\N},\Delta_{\mathbf{A}}$}, and {\small$\mathscr{T}_{\N,\mathbf{A}}$} from Definition \ref{definition:heat} (in which {\small$\mathbf{A}$} is the parameter in \eqref{eq:mainII}). Second, recall the notation from \eqref{eq:heatIII}. Third, recall the notation from Definitions \ref{definition:bulkaverage}, \ref{definition:edgeaverage}, and \ref{definition:grad}. We have the following \abbr{SDE} for any {\small$(\t,\x)\in[0,\infty)\times\llbracket0,\infty\rrbracket$}, which uses additional notation to be explained afterwards:
\begin{align}
&\d\mathbf{S}^{\N}_{\t,\x}=\mathscr{T}_{\N}\mathbf{S}^{\N}_{\t,\x}\d\t+[\mathbf{e}^{\mathfrak{t}_{\N}\mathscr{T}_{\N}}(\mathbf{1}^{(\x)}_{\cdot}\sqrt{2}\lambda\N^{\frac12}\mathbf{R}^{\N}_{\t,\cdot}\mathbf{S}^{\N}_{\t,\cdot}\d\mathbf{b}_{\t,\cdot})]_{\x}+\mathrm{Err}[\mathbf{R}^{\N}\mathbf{S}^{\N}]_{\t,\x}\d\t\label{eq:sdeI1}\\
&+\sum_{\ell\in\llbracket1,\mathrm{L}_{1}\rrbracket}\sum_{\w\in\llbracket0,\infty\rrbracket}\boldsymbol{\mathcal{K}}^{(\ell)}_{\x,\w}\cdot\N\cdot\mathbf{Av}^{\mathbf{T},\mathbf{X},\mathscr{Q}_{\ell}}_{\t,\w}\mathbf{R}^{\N}_{\t,\w_{\pm}}\mathbf{S}^{\N}_{\t,\w_{\pm}}\d\t\label{eq:sdeI3}\\
&+\sum_{\ell\in\llbracket1,\mathrm{L}_{1}\rrbracket}\sum_{\w\in\llbracket0,\infty\rrbracket}\boldsymbol{\mathcal{K}}^{(\ell)}_{\x,\w}\cdot\N\cdot\grad^{\mathbf{T},\mathrm{av}}_{\mathfrak{t}_{\mathbf{Av}}}(\mathbf{Av}^{\mathbf{X},\mathscr{Q}_{\ell}}_{\t,\w}\mathbf{R}^{\N}_{\t,\w_{\pm}}\mathbf{S}^{\N}_{\t,\w_{\pm}})\d\t\label{eq:sdeI4}\\
&+\sum_{\ell\in\llbracket1,\mathrm{L}_{2}\rrbracket}\sum_{\w\in\llbracket0,\infty\rrbracket}\boldsymbol{\mathcal{K}}^{\partial,\ell}_{\x,\w}\cdot\N\mathbf{Av}^{\mathbf{T},\mathfrak{b}_{\ell}}_{\t,\w}\mathbf{R}^{\N}_{\t,\w}\mathbf{S}^{\N}_{\t,\w}\d\t+\sum_{\ell\in\llbracket1,\mathrm{L}_{2}\rrbracket}\sum_{\w\in\llbracket0,\infty\rrbracket}\boldsymbol{\mathcal{K}}^{\partial,\ell}_{\x,\w}\cdot\N\grad^{\mathbf{T},\mathrm{av}}_{\mathfrak{t}_{\partial}}(\mathfrak{b}_{\ell}[\bphi_{\t}]\mathbf{R}^{\N}_{\t,\w}\mathbf{S}^{\N}_{\t,\w})\d\t.\label{eq:sdeI5}
\end{align}
%
\begin{enumerate}
\item We recall {\small$\mathbf{1}^{(\x)}_{\cdot}:=\mathbf{1}[|\x-\cdot|\lesssim(\log\N)^{2}\N\mathfrak{t}_{\N}^{1/2}]$} from Definition \ref{definition:heat}. We also recall {\small$\mathscr{T}_{\N,\mathbf{A}}$} from Definition \ref{definition:heat} and set {\small$\mathscr{T}_{\N}:=\mathscr{T}_{\N,\mathbf{A}}$} (since this parameter {\small$\mathbf{A}$} is of special interest).
\item The quantity {\small$\mathrm{Err}[\mathbf{R}^{\N}\mathbf{S}^{\N}]_{\t,\x}$} is the linear integral operator
\begin{align}
\mathrm{Err}[\mathbf{R}^{\N}\mathbf{S}^{\N}]_{\t,\x}&:=\sum_{\w\in\llbracket0,\infty\rrbracket}\mathbf{E}^{\N}_{\t,\x,\w}\cdot\mathbf{R}^{\N}_{\t,\w}\mathbf{S}^{\N}_{\t,\w},\label{eq:sdeII1}
\end{align}
where the kernel {\small$\mathbf{E}^{\N}$} satisfies the deterministic estimate below for any {\small$\kappa>0$} and some {\small$0<\mathrm{C}\lesssim1$}:
\begin{align}
|\mathbf{E}^{\N}_{\t,\x,\w}|\lesssim_{\kappa}\Big(\N^{-\frac65}+\mathbf{1}_{\w\in\llbracket0,1\rrbracket}\N^{-\frac13}\Big)\cdot\exp(-\tfrac{\kappa|\x-\w|}{\N})\cdot\Big\{1+\sum_{|\z-\w|\lesssim1}|\bphi_{\t,\z}|^{\mathrm{C}}\Big\}.\label{eq:sdeII2}
\end{align}
\item The kernel {\small$\boldsymbol{\mathcal{K}}^{(\ell)}$} is deterministic and satisfies the following estimate for any {\small$\kappa>0$} and some {\small$\mathrm{C}\lesssim1$}, where {\small$\delta_{\mathbf{S}}>0$} is the small constant from Definition \ref{definition:heat}:
\begin{align}
|\boldsymbol{\mathcal{K}}^{(\ell)}_{\x,\w}|\lesssim_{\kappa}\N^{-1+\mathrm{C}\delta_{\mathbf{S}}}\exp(-\tfrac{\kappa|\x-\w|}{\N}).\label{eq:sdeIII}
\end{align}
\item The integers {\small$\mathrm{L}_{1},\mathrm{L}_{2}$} are large but {\small$\mathrm{O}(1)$}.
\item For any index {\small$\ell\in\llbracket1,\mathrm{L}_{1}\rrbracket$}, the collection {\small$\mathscr{Q}_{\ell}=(\mathfrak{q}_{\ell,\x})_{\x\in\llbracket0,\infty\rrbracket}$} of functions is such that {\small$\mathfrak{q}_{\ell,\x}$} is bulk-admissible with respect to {\small$\x\in\llbracket0,\infty\rrbracket$} (see Definition \ref{definition:bulkadmissible}). Moreover, either all functions in {\small$\mathscr{Q}_{\ell}$} are forwards-oriented, or they are all backwards oriented. In the former case, for the corresponding sum over {\small$\w$}, we set {\small$\w_{\pm}=\w_{-}$}. In the latter case, we instead set {\small$\w_{\pm}=\w_{+}$}.
\item For every {\small$\ell\in\mathrm{L}_{2}$}, the function {\small$\mathfrak{b}_{\ell}$} is edge-admissible (see Definition \ref{definition:edgeadmissible}).
\item The kernel {\small$\boldsymbol{\mathcal{K}}^{\partial,\ell}_{\x,\w}$} is supported either at {\small$\w=0$} or {\small$\w=1$} (depending on the index {\small$\ell\in\llbracket1,\mathrm{L}_{2}\rrbracket$}). Moreover, we have the following estimate for any {\small$\kappa>0$} and for some {\small$\mathrm{C}=\mathrm{O}(1)$}, again where {\small$\delta_{\mathbf{S}}$} is from Definition \ref{definition:heat}:
\begin{align}
|\boldsymbol{\mathcal{K}}^{\partial,\ell}_{\x,\w}|\lesssim_{\kappa}\N^{-1+\mathrm{C}\delta_{\mathbf{S}}}\exp(-\tfrac{\kappa|\x-\w|}{\N})\mathbf{1}_{\w=0,1}.\label{eq:sdeIV}
\end{align}
\end{enumerate}
\end{prop}
\begin{remark}\label{remark:sde}
\fsp Let us clarify the structure of Proposition \ref{prop:sde}, {\color{black}keeping in mind the discussion from Section \ref{section:method}}.
\begin{itemize}
\item The first line \eqref{eq:sdeI1} is essentially a lattice \abbr{SHE}. Indeed, as shown later, we have {\small$\mathbf{R}^{\N}\approx1$} in a sufficiently strong sense, and the heat kernel smoothing in \eqref{eq:sdeI1} occurs on a mesoscopic scale that vanishes in the large-{\small$\N$} limit. Moreover, the estimate \eqref{eq:sdeII2} suggests that that the last term in \eqref{eq:sdeI1}, provided a priori estimates on {\small$\bphi$}-spins that will eventually be shown, is scaled too small to persist in the large-{\small$\N$} limit. (The meaning of {\small$-6/5$} is not important, just that it is less than {\small$-1$}; similarly, any fixed negative exponent would be okay in place of {\small$-1/3$}.) 
\item Let us explain \eqref{eq:sdeI3}. The {\small$\mathbf{Av}^{\mathbf{T},\mathbf{X},\mathscr{Q}_{\ell}}$}-coefficients {\color{black}were explained in Section \ref{section:method} as what we obtain after replacing bulk-admissible coefficients by their space-time averages} (then factor out {\small$\mathbf{Z}^{\N}=\mathbf{R}^{\N}\mathbf{S}^{\N}$} like in Definition \ref{definition:bulkaverage}). The kernels {\small$\boldsymbol{\mathcal{K}}^{(\ell)}$} are essentially {\small$\mathbf{H}^{\N,\mathbf{A}}$}-kernels at time-scale {\small$\mathfrak{t}_{\N}=\N^{-\delta_{\mathbf{S}}}$} which generate the smoothing in \eqref{eq:heatIII}, but perhaps with decorations. For example, some {\small$\boldsymbol{\mathcal{K}}^{(\ell)}$}-kernels come with a finite number of space-gradients to absorb error terms coming from the replacement of {\small$\mathfrak{q}$}-functions by their space-time averages, {\color{black}as discussed in Section \ref{section:method}}. We emphasize that {\small$\mathbf{H}^{\N,\mathbf{A}}$}-kernels at time-scale {\small$\mathfrak{t}_{\N}=\N^{-\delta_{\mathbf{S}}}$}, as well as their gradients, satisfy the necessary estimate \eqref{eq:sdeIII} by Proposition \ref{prop:hkestimates}. 

Lastly, the functions in {\small$\mathscr{Q}_{\ell}=(\mathfrak{q}_{\ell,\x})_{\x\in\llbracket0,\infty\rrbracket}$} are essentially {\small$\tau$}-spatial-shifts of a local function {\small$\R^{\llbracket1,\infty\rrbracket}\to\R$}, though near the boundary, these shifts do not quite make sense anymore, hence the more abstract dependence of {\small$\mathfrak{q}_{\ell,\x}$} on {\small$\x\in\llbracket0,\infty\rrbracket$}. In any case, this means that if one of the functions in {\small$\mathscr{Q}_{\ell}$} is forwards- (resp. backwards-) oriented, then all of them are forwards- (resp backwards-) oriented.
\item The terms in \eqref{eq:sdeI4} have the same structure as \eqref{eq:sdeI3}, though only with spatial-averages and a time-averaged gradient. Indeed, this line accounts for the cost in replacing the spatial-average by a space-time average.
\item The first term in \eqref{eq:sdeI5} has the same structure as \eqref{eq:sdeI3} but without a space-average, as everything is localized to the edge. The second term in \eqref{eq:sdeI5}, similar to \eqref{eq:sdeI4}, accounts for the cost in replacing an edge-admissible coefficient by a time-average. We clarify that the kernels {\small$\boldsymbol{\mathcal{K}}^{\partial,\ell}$} are also restriction of {\small$\mathbf{H}^{\N,\mathbf{A}}$}-kernels to the edge (perhaps with additional decorations as noted above). \emph{A key aspect of \eqref{eq:sdeI5}, however, is that the scaling is {\small$\N$}, not {\small$\N^{3/2}$}, which is perhaps the natural guess, because the scaling in \eqref{eq:currII} is of order {\small$\N^{3/2}$}.} {\color{black}This reflects (more concretely) the important Neumann cancellation that was discussed in Section \ref{section:method}.}
\end{itemize}
\end{remark}
The remainder of this section is dedicated to the proof of Proposition \ref{prop:sde}. The details of this proof, however, have no bearing on the rest of this paper, so the reader is invited to skip to Section \ref{section:shk} below.
\subsection{Interior dynamics}
By \eqref{eq:ch} and \eqref{eq:currI}, we may directly inherit the bulk dynamics of {\small$\mathbf{Z}^{\N}$} from the equation in the full-space case from \cite{Y26PTRF}. For this, we will need to introduce the necessary notation below.
\begin{notation}\label{notation:mshe}
\fsp With notation to be explained afterwards, we introduce the following ``error term operator":
\begin{align}
\mathfrak{e}[\tau_{\x}{\boldsymbol{\phi}}_{\t}]\mathbf{Z}^{\N}_{\t,\x}:=\mathfrak{e}_{2}[\tau_{\x}{\boldsymbol{\phi}}_{\t}]\mathbf{Z}^{\N}_{\t,\x}+\mathfrak{e}_{1}[\tau_{\x}{\boldsymbol{\phi}}_{\t}]\mathbf{Z}^{N}_{\t,\x}+\mathfrak{e}_{0}[\tau_{\x}{\boldsymbol{\phi}}_{\t}]\mathbf{Z}^{\N}_{\t,\x}+\mathfrak{f}_{\mathrm{total}}[\tau_{\x}{\boldsymbol{\phi}}_{\t}]\mathbf{Z}^{\N}_{\t,\x}+\N\Delta(\mathfrak{f}_{0}[\tau_{\x}\bphi_{\t}]\mathbf{Z}^{\N}_{\t,\x}).\label{eq:msheII}
\end{align}
In what follows, the subscript for the {\small$\mathfrak{e}_{\ell}$}-terms, for {\small$\ell=0,1,2$}, indicates the power of the fluctuation-scaling {\small$\N^{1/2}$}. (It also turns out that the coefficients in {\small$\mathfrak{e}_{\ell}$}-terms belong to {\small$\mathrm{Jet}_{\ell}^{\perp}$} from Definition \ref{definition:jet} and, after checking the other technical conditions in Definition \ref{definition:bulkadmissible}, are bulk admissible.) 
\begin{enumerate}
\item Before we proceed, throughout this construction, we will restrict to {\small$\x\in\llbracket2,\infty\rrbracket$}. Indeed, it is for these points that all functions of {\small$\bphi_{\t,\cdot}$} below depend only on {\small$\bphi_{\t,\w}$} for {\small$\w\in\llbracket1,\infty\rrbracket$}.
\item We first introduce {\small$\mathfrak{e}_{2}$}. Recall {\small$\lambda$} from \eqref{eq:rg}. We define the following renormalization of a quadratic nonlinearity in {\small$\mathscr{U}'$} by a generalization of a degree-{\small$2$} Hermite polynomial to non-Gaussian potentials:
\begin{align}
\mathscr{Q}[{\boldsymbol{\phi}}]:=\tfrac12\beta_{2}\mathscr{U}'[{\boldsymbol{\phi}}_{0}]\mathscr{U}'[{\boldsymbol{\phi}}_{1}]-\tfrac12\lambda(\mathscr{U}'[{\boldsymbol{\phi}}_{0}]{\boldsymbol{\phi}}_{0}-1).\label{eq:quadratic_renorm}
\end{align}
The first term on the \abbr{RHS} of \eqref{eq:msheII} is given as follows, in which {\small$\deg\geq3$} is the degree of {\small$\mathbf{F}$} from \eqref{eq:nl}:
\begin{align}
\mathfrak{e}_{2}[\tau_{\x}{\boldsymbol{\phi}}_{\t}]\mathbf{Z}^{\N}_{\t,\x}&:=\lambda \N\mathscr{Q}[\bphi_{\t,\x+1}]\mathbf{Z}^{\N}_{\t,\x}+\lambda \N\mathscr{Q}[\bphi_{\t,\x-1}]\mathbf{Z}^{\N}_{\t,\x}\label{eq:msheII(2)a}\\
&+\lambda \N\mathbf{F}_{\geq3}[\tau_{\x+2\deg}\bphi_{\t}]\mathbf{Z}^{\N}_{\t,\x}\d\t\label{eq:msheII(2)aa}\\
&-\tfrac16\lambda^{2} \N(\Delta\{\mathscr{U}'[\bphi_{\t,\x+2}]\mathscr{U}'[\bphi_{\t,\x+3}]\}\mathbf{Z}^{\N}_{\t,\x})\d\t.\label{eq:msheII(2)b}
\end{align}
The \abbr{RHS} of \eqref{eq:msheII(2)a} can be seen as a non-Gaussian analogue of how the exponential Cole-Hopf transform for \abbr{KPZ} itself introduces a renormalization by a quadratic Hermite polynomial. The term \eqref{eq:msheII(2)aa} is, up to spatial shift, the contribution of the degree {\small$3$} and higher terms in \eqref{eq:nl}, and \eqref{eq:msheII(2)b} is an error term to account for all the spatial-shifting in \eqref{eq:msheII(2)a}-\eqref{eq:msheII(2)b}.
\item To define the second term on the \abbr{RHS} of \eqref{eq:msheII}, let us introduce the function {\small$\mathscr{V}[{\boldsymbol{\phi}}_{\x}]:=\mathscr{U}[{\boldsymbol{\phi}}_{\x}]-\frac12\alpha{\boldsymbol{\phi}}_{\x}^{2}$}, so that {\small$\mathscr{V}'[\bphi_{\x}]=\mathscr{U}'[\bphi_{\x}]-\alpha\bphi_{\x}$}. We will also need the function
\begin{align}
\mathscr{G}[\bphi]:=\bphi_{2}\mathscr{U}'[\bphi_{2}]\cdot(\mathscr{U}'[\bphi_{1}]-\mathscr{U}'[\bphi_{3}])+\bphi_{2}\cdot(\mathscr{U}'[\bphi_{4}]\mathscr{U}'[\bphi_{5}]-\mathscr{U}'[\bphi_{3}]\mathscr{U}'[\bphi_{4}]).\label{eq:funnygradient}
\end{align}
The second term on the \abbr{RHS} of \eqref{eq:msheII} is now given as follows, where {\small$c_{\ell}\in\R$} are fixed constants:
\begin{align}
\mathfrak{e}_{1}[\tau_{\x}{\boldsymbol{\phi}}_{\t}]\mathbf{Z}^{\N}_{\t,\x}&:=\tfrac12\lambda\N^{\frac32}\grad^{\mathbf{X}}_{1}(\mathscr{V}'[{\boldsymbol{\phi}}_{\t,\x}]\mathbf{Z}^{\N}_{\t,\x})-\tfrac12\lambda\N^{\frac32}\grad^{\mathbf{X}}_{-1}(\mathscr{V}'[{\boldsymbol{\phi}}_{\t,\x+1}]\mathbf{Z}^{N}_{\t,\x})\label{eq:msheII(1)a}\\
&-\tfrac16\lambda^{3} \N^{\frac12}\mathscr{G}[\tau_{\x}\bphi_{\t}]\mathbf{Z}^{\N}_{\t,\x}\d\t+\N^{\frac12}\Big(\mathbf{F}_{\geq3}[\tau_{\x+2\deg}\bphi_{\t}]\cdot\sum_{\ell\in\llbracket1,2\deg\rrbracket}c_{\ell}\bphi_{\t,\x+\ell}\Big)\mathbf{Z}^{\N}_{\t,\x}\d\t\label{eq:msheII(1)c}
\end{align}
The exact form of the constants {\small$c_{\ell}$} is not important, just that they are deterministic and {\small$\mathrm{O}(1)$}. Intuitively, the \abbr{RHS} of \eqref{eq:msheII(1)a} accounts for ``how non-quadratic" the potential {\small$\mathscr{U}$} is. (If {\small$\mathscr{U}$} is quadratic, then {\small$\mathscr{V}\equiv0$}.) The term \eqref{eq:msheII(1)c} is essentially another error term obtained from shifting the {\small$\mathbf{F}$}-nonlinearity from \eqref{eq:nl} in space.
\item Third, we have the following, where {\small$c_{\ell_{1},\ell_{2}}\in\R$} are fixed constants:
\begin{align}
\mathfrak{e}_{0}[\tau_{\x}{\boldsymbol{\phi}}_{\t}]\mathbf{Z}^{\N}_{\t,\x}&:=\tfrac{1}{12}\beta_{2}^{4}\cdot(\mathscr{U}'[{\boldsymbol{\phi}}_{\t,\x+1}]{\boldsymbol{\phi}}_{\t,\x+1}^{3}-\E^{0}\{\mathscr{U}'[{\boldsymbol{\phi}}_{0}]{\boldsymbol{\phi}}_{0}^{3}\})\mathbf{Z}^{\N}_{\t,\x}\label{eq:msheII(0)a}\\
&-\tfrac16\lambda^{4}(\mathscr{G}[\tau_{\x}\bphi_{\t}]\bphi_{\t,\x+1}-1)\mathbf{Z}^{\N}_{\t,\x}\d\t\label{eq:msheII(0)a1}\\
&-\lambda\N\grad^{\mathbf{X}}_{1}(\mathbf{F}_{\geq3}[\tau_{\x+2\deg}\bphi_{\t}]\mathbf{Z}^{\N}_{\t,\x})\d\t\label{eq:msheII(0)a3}\\
&+\tfrac12\lambda^{2}\N\grad^{\mathbf{X}}_{1}([{\boldsymbol{\phi}}_{\t,\x}^{2}-1]\cdot\mathbf{Z}^{\N}_{\t,\x})\label{eq:msheII(0)b}\\
&+\tfrac12\lambda^{2}\N\grad^{\mathbf{X}}_{-1}([{\boldsymbol{\phi}}_{\t,\x+1}^{2}-1]\cdot\mathbf{Z}^{\N}_{\t,\x})\label{eq:msheII(0)c}\\
&+\tfrac13\lambda^{2} \N\grad^{\mathbf{X}}_{1}(\Delta\{\mathscr{U}'[\bphi_{\t,\x+2}]\mathscr{U}'[\bphi_{\t,\x+3}]\}\mathbf{Z}^{\N}_{\t,\x})\label{eq:msheII(0)d}\\
&-\tfrac{1}{12}\lambda^{4}(\Delta\{\mathscr{U}'[\bphi_{\t,\x+2}]\mathscr{U}'[\bphi_{\t,\x+3}]\})\bphi_{\t,\x+1}^{2}\mathbf{Z}^{\N}_{\t,\x}\d\t.\label{eq:msheII(0)e}\\
&-\tfrac{1}{12}\lambda^{4}(\Delta\{\mathscr{U}'[\bphi_{\t,\x+2}]\mathscr{U}'[\bphi_{\t,\x+3}]\})\bphi_{\t,\x+2}^{2}\mathbf{Z}^{\N}_{\t,\x}\d\t.\label{eq:msheII(0)e1}\\
&+\tfrac{1}{12}\lambda^{4}(\grad^{\mathbf{X}}_{1}\{\mathscr{U}'[{\boldsymbol{\phi}}_{\t,\x+1}]{\boldsymbol{\phi}}_{\t,\x+1}^{3}\})\mathbf{Z}^{\N}_{\t,\x}\d\t\label{eq:msheII(0)f}\\
&+\Big(\mathbf{F}_{\geq3}[\tau_{\x+2\deg}\bphi_{\t}]\cdot\sum_{\ell_{1},\ell_{2}\in\llbracket1,2\deg\rrbracket}c_{\ell_{1},\ell_{2}}\bphi_{\t,\x+\ell_{1}}\bphi_{\t,\x+\ell_{2}}\Big)\cdot\mathbf{Z}^{\N}_{\t,\x}\d\t.\label{eq:msheII(0)a2}
\end{align}
All of these terms are essentially error terms obtained in the same way as \eqref{eq:msheII(1)c}.
 \item The first error term is given by
\begin{align}
\mathfrak{f}_{\mathrm{total}}[\tau_{\x}{\boldsymbol{\phi}}_{\t}]\mathbf{Z}^{\N}_{\t,\x}&=\N^{\frac12}\grad^{\mathbf{X}}_{1}\left(\mathfrak{f}_{1}[\tau_{\x}{\boldsymbol{\phi}}_{\t}]\mathbf{Z}^{\N}_{\t,\x}\right)\label{eq:msheIIfa}\\
&+\N^{\frac12}\grad^{\mathbf{X}}_{-1}\left(\mathfrak{f}_{2}[\tau_{\x}{\boldsymbol{\phi}}_{\t}]\mathbf{Z}^{\N}_{\t,\x}\right)\label{eq:msheIIfb}\\
&+\N^{-\frac12}\mathfrak{f}_{3}[\tau_{\x}{\boldsymbol{\phi}}_{\t}]\mathbf{Z}^{\N}_{\t,\x},\label{eq:msheIIfc}
\end{align}
in which the {\small$\mathfrak{f}_{i}:\R^{\Z}\to\R$} satisfy the following a priori estimate (which is identical to the pointwise estimate, but not the regularity estimate, in \eqref{eq:bulkadmissibleI}) for some $0<\mathrm{c},\mathrm{C}=\mathrm{O}(1)$:
\begin{align}
|\mathfrak{f}_{i}[\bphi]|\lesssim1+\sum_{|\w|\leq\mathrm{c}}|\bphi_{\w}|^{\mathrm{C}}.
\label{eq:msheIIff}
\end{align}
The exact form of these functions is similarly not important. Indeed, every term in \eqref{eq:msheIIfa}-\eqref{eq:msheIIfc} is scaled too weakly in {\small$\N$} to have any contribution in the large-{\small$\N$} limit (given that the functions {\small$\mathfrak{f}_{i}$} therein are controlled in size, in some sense, because of \eqref{eq:msheIIff}).
\item The last error {\small$\mathfrak{f}_{0}$} is a function, whose exact form is also not important, which satisfies the estimate \eqref{eq:msheIIff}.
\end{enumerate}
\end{notation}
Before we proceed, observe that each coefficient in Notation \ref{notation:mshe}, with the exception of {\small$\mathfrak{f}_{\mathrm{total}}$} and {\small$\mathfrak{f}_{0}$} possibly, is either forwards-oriented or backwards-oriented (see Definition \ref{definition:bulkaverage}) with respect to {\small$\x\in\llbracket2,\infty\rrbracket$}. Now, with this notation in hand, we may record the evolution equation for {\small$\mathbf{Z}^{\N}$} in {\small$\llbracket2,\infty\rrbracket$}.
\begin{lemma}\label{lemma:msheinterior1}
\fsp Fix any {\small$\x\in\llbracket2,\infty\rrbracket$}. We have the following evolution equation for \eqref{eq:ch}, where we use notation to be explained afterwards:
\begin{align}
\d\mathbf{Z}^{\N}_{\t,\x}=\mathscr{T}_{\N,\mathbf{A}}\mathbf{Z}^{\N}_{\t,\x}\d\t+\sqrt{2}\lambda\N^{\frac12}\mathbf{Z}^{\N}_{\t,\x}\d\mathbf{b}_{\t,\x}+\mathfrak{e}[\tau_{\x}\bphi_{\t}]\mathbf{Z}^{\N}_{\t,\x}\d\t.\label{eq:msheinterior1I}
\end{align}
The Brownian motions {\small$\mathbf{b}$} are the ones from \eqref{eq:currI}, and {\small$\mathscr{T}_{\N}$} is a rescaled discrete Laplacian as in Proposition \ref{prop:sde}. The rightmost term in \eqref{eq:msheinterior1I}, we recall, is given in \eqref{eq:msheII}.
\end{lemma}
\begin{proof}
We can directly cite Lemma 5.1 in \cite{Y26PTRF} and its proof therein by the following reasons. Given {\small$\x\in\llbracket1,\infty\rrbracket$}, we have the following. First, we have {\small$\mathbf{j}^{\N}_{\t,\x}-\mathbf{j}^{\N}_{\t,\x-1}=\N^{-1/2}\bphi_{\t,\x}$}. Second, the \abbr{SDE} \eqref{eq:currI} for {\small$\t\mapsto\mathbf{j}^{\N}_{\t,\x}$} is the same as the full-space \abbr{SDE} (2.6) in \cite{Y26PTRF} {in a neighborhood of radius {\small$1$} around {\small$\x\in\llbracket2,\infty\rrbracket$}}. These are the only inputs used in the proof of Lemma 5.1 in \cite{Y26PTRF}, so the result follows. (Again, all coefficients in Notation \ref{notation:mshe} depend only on {\small$\bphi_{\t,\x}$} for {\small$\x\in\llbracket1,\infty\rrbracket$}.)
\end{proof}
We now record an \abbr{SDE} for {\small$\mathbf{Z}^{\N}$} at the point {\small$\x=1$}, at which the dynamics \eqref{eq:currI} are still those of the full-space, but at which some of the coefficients (e.g. the final term in \eqref{eq:msheII(2)a}) are not well-defined. (In what follows, we are only concerned with terms modulo {\small$\N^{1/2}$}-scaling, because the terms are localized to a single point in space, and {\small$\N^{1/2}$}-scaling is too weak to persist in the large-{\small$\N$} limit after we integrate against a heat kernel.)
\begin{lemma}\label{lemma:msheinterior2}
\fsp We have the following formula, where {\small$\mathfrak{f}$} satisfies \eqref{eq:msheIIff}:
\begin{align}
\d\mathbf{Z}^{\N}_{\t,1}&=\mathscr{T}_{\N}\mathbf{Z}^{\N}_{\t,1}\d\t+\sqrt{2}\lambda\N^{\frac12}\mathbf{Z}^{\N}_{\t,1}\d\mathbf{b}_{\t,1}\label{eq:msheinterior2Ia}\\
&+\tfrac12\lambda\N^{\frac32}\grad^{\mathbf{X}}_{1}(\mathscr{V}'[\bphi_{\t,1}]\mathbf{Z}^{\N}_{\t,1})\d\t-\tfrac12\lambda\N^{\frac32}\grad^{\mathbf{X}}_{-1}(\mathscr{V}'[\bphi_{\t,2}]\mathbf{Z}^{\N}_{\t,1})\d\t\label{eq:msheinterior2Ib}\\
&+\lambda\N(\mathbf{F}[\tau_{1}\bphi_{\t}]-\tfrac12\lambda(\mathscr{U}'[\bphi_{\t,1}]\bphi_{\t,1}-1)-\tfrac12\lambda(\mathscr{U}'[\bphi_{\t,2}]\bphi_{\t,2}-1))\mathbf{Z}^{\N}_{\t,1}\d\t+\N^{\frac12}\mathfrak{f}[\bphi_{\t}]\mathbf{Z}^{\N}_{\t,1}\d\t.\label{eq:msheinterior2Id}
\end{align}
\end{lemma}
\begin{proof}
By the It\^{o} formula with \eqref{eq:ch} and \eqref{eq:currI} at {\small$\x=1$}, we have 
\begin{align}
&\d\mathbf{Z}^{\N}_{\t,1}=\lambda\mathbf{Z}^{\N}_{\t,1}\d\mathbf{j}^{\N}_{\t,1}+\lambda^{2}\N\mathbf{Z}^{\N}_{\t,1}\d\t\nonumber\\
&=\lambda\N^{\frac32}\grad^{\mathbf{X}}_{1}\mathscr{U}'[\bphi_{\t,1}]\mathbf{Z}^{\N}_{\t,1}\d\t+\lambda\N\mathbf{F}[\tau_{1}\bphi_{\t}]\mathbf{Z}^{\N}_{\t,1}\d\t+\lambda^{2}\N\mathbf{Z}^{\N}_{\t,1}\d\t-\lambda^{2}\mathscr{R}_{\lambda}\mathbf{Z}^{\N}_{\t,1}\d\t+\sqrt{2}\lambda\N^{\frac12}\mathbf{Z}^{\N}_{\t,1}\d\mathbf{b}_{\t,1}.\label{eq:msheinterior2I1}
\end{align}
We now recall the discrete Leibniz rule {\small$\grad^{\mathbf{X}}_{\mathfrak{l}}(\mathsf{f}\mathsf{g})_{\x}=\mathsf{f}_{\x}\grad^{\mathbf{X}}_{\mathfrak{l}}\mathsf{g}_{\x}+\mathsf{g}_{\x+\mathfrak{l}}\grad^{\mathbf{X}}\mathsf{f}_{\x}$} for any {\small$\mathfrak{l}\in\Z$} and functions {\small$\mathsf{f},\mathsf{g}:\Z\to\R$}. Combining this for {\small$\mathfrak{l}=1$} with a few elementary manipulations (and {\small$\mathscr{V}'[\bphi_{\x}]=\mathscr{U}'[\bphi_{\x}]-\alpha\bphi_{\x}$}) gives
\begin{align}
\lambda\N^{\frac32}\grad^{\mathbf{X}}_{1}\mathscr{U}'[\bphi_{\t,1}]\mathbf{Z}^{\N}_{\t,1}&=\lambda\N^{\frac32}\grad^{\mathbf{X}}_{1}(\mathscr{U}'[\bphi_{\t,1}]\mathbf{Z}^{\N}_{\t,1})-\lambda\N^{\frac32}\mathscr{U}'[\bphi_{\t,2}]\grad^{\mathbf{X}}_{1}\mathbf{Z}^{\N}_{\t,1}\nonumber\\
&=\alpha\lambda\N^{\frac32}\grad^{\mathbf{X}}_{1}(\bphi_{\t,1}\mathbf{Z}^{\N}_{\t,1})+\lambda\N^{\frac32}\grad^{\mathbf{X}}_{1}(\mathscr{V}'[\bphi_{\t,1}]\mathbf{Z}^{\N}_{\t,1})-\lambda\N^{\frac32}\mathscr{U}'[\bphi_{\t,2}]\grad^{\mathbf{X}}_{1}\mathbf{Z}^{\N}_{\t,1}.\nonumber
\end{align}
On the other hand, we also have {\small$\grad^{\mathbf{X}}_{1}\mathsf{f}_{\x}=-\grad^{\mathbf{X}}_{-1}\mathsf{f}_{\x+1}$}, so we can also write
\begin{align}
\lambda\N^{\frac32}\grad^{\mathbf{X}}_{1}\mathscr{U}'[\bphi_{\t,1}]\mathbf{Z}^{\N}_{\t,1}&=-\lambda\N^{\frac32}\grad^{\mathbf{X}}_{-1}\mathscr{U}'[\bphi_{\t,2}]\mathbf{Z}^{\N}_{\t,1}\nonumber\\
&=-\alpha\lambda\N^{\frac32}\grad^{\mathbf{X}}_{-1}(\bphi_{\t,2}\mathbf{Z}^{\N}_{\t,1})-\lambda\N^{\frac32}\grad^{\mathbf{X}}_{-1}(\mathscr{V}'[\bphi_{\t,2}]\mathbf{Z}^{\N}_{\t,1})+\lambda\N^{\frac32}\mathscr{U}'[\bphi_{\t,1}]\grad^{\mathbf{X}}_{-1}\mathbf{Z}^{\N}_{\t,1},\nonumber
\end{align}
where the second line follows by the same considerations as before. Averaging the previous two displays gives
\begin{align}
\lambda\N^{\frac32}\grad^{\mathbf{X}}_{1}\mathscr{U}'[\bphi_{\t,1}]\mathbf{Z}^{\N}_{\t,1}&=\tfrac12\alpha\lambda\N^{\frac32}\grad^{\mathbf{X}}_{1}(\bphi_{\t,1}\mathbf{Z}^{\N}_{\t,1})-\tfrac12\alpha\lambda\N^{\frac32}\grad^{\mathbf{X}}_{-1}(\bphi_{\t,2}\mathbf{Z}^{\N}_{\t,1})\nonumber\\
&+\tfrac12\lambda\N^{\frac32}\grad^{\mathbf{X}}_{1}(\mathscr{V}'[\bphi_{\t,1}]\mathbf{Z}^{\N}_{\t,1})-\tfrac12\lambda\N^{\frac32}\grad^{\mathbf{X}}_{-1}(\mathscr{V}'[\bphi_{\t,2}]\mathbf{Z}^{\N}_{\t,1})\nonumber\\
&-\tfrac12\lambda\N^{\frac32}\mathscr{U}'[\bphi_{\t,2}]\grad^{\mathbf{X}}_{1}\mathbf{Z}^{\N}_{\t,1}+\tfrac12\lambda\N^{\frac32}\mathscr{U}'[\bphi_{\t,1}]\grad^{\mathbf{X}}_{-1}\mathbf{Z}^{\N}_{\t,1}.\label{eq:msheinterior2I2}
\end{align}
On the other hand, by \eqref{eq:ch}, the rule {\small$\bphi_{\t,\x}=\N^{1/2}(\mathbf{j}^{\N}_{\t,\x}-\mathbf{j}^{\N}_{\t,\x-1})$} for {\small$\x\in\llbracket1,\infty\rrbracket$}, and Taylor expansion, we have, as in (5.27)-(5.30) of \cite{Y26PTRF}, the following formulas, in which {\small$\approx$} indicates (throughout this argument) an identity up to terms of the form {\small$\N^{1/2}\mathfrak{f}$}, where {\small$\mathfrak{f}$} satisfies the estimate \eqref{eq:msheIIff}:
\begin{align*}
\N^{2}\grad^{\mathbf{X}}_{1}\mathbf{Z}^{\N}_{\t,\x}&\approx\lambda\N^{\frac32}\bphi_{\t,\x+1}\mathbf{Z}^{\N}_{\t,\x}+\tfrac12\lambda^{2}\N^{}\bphi_{\t,\x+1}^{2}\mathbf{Z}^{\N}_{\t,\x},\\
\N^{2}\grad^{\mathbf{X}}_{-1}\mathbf{Z}^{\N}_{\t,\x}&\approx-\lambda\N^{\frac32}\bphi_{\t,\x}\mathbf{Z}^{\N}_{\t,\x}+\tfrac12\lambda^{2}\N^{}\bphi_{\t,\x}^{2}\mathbf{Z}^{\N}_{\t,\x}.
\end{align*}
We use the above formulas in two different ways. First, we evaluate the last line in \eqref{eq:msheinterior2I2} as follows:
\begin{align}
-\lambda\N^{\frac32}\mathscr{U}'[\bphi_{\t,2}]\grad^{\mathbf{X}}_{1}\mathbf{Z}^{\N}_{\t,1}+\lambda\N^{\frac32}\mathscr{U}'[\bphi_{\t,1}]\grad^{\mathbf{X}}_{-1}\mathbf{Z}^{\N}_{\t,1}&\approx-\lambda^{2}\N\mathscr{U}'[\bphi_{\t,2}]\bphi_{\t,2}\mathbf{Z}^{\N}_{\t,1}-\lambda^{2}\N\mathscr{U}'[\bphi_{\t,1}]\bphi_{\t,1}\mathbf{Z}^{\N}_{\t,1}\nonumber\\
&\approx-(\lambda^{2}\N\mathscr{U}'[\bphi_{\t,1}]\bphi_{\t,1}+\lambda^{2}\N\mathscr{U}'[\bphi_{\t,2}]\bphi_{\t,2})\mathbf{Z}^{\N}_{\t,1}.\label{eq:msheinterior2I3}
\end{align}
The last line follows by the gradient formulas above. Next, we observe {\small$\mathscr{T}_{\N}\mathbf{Z}^{\N}_{\t,1}\approx\alpha\N^{2}\Delta\mathbf{Z}^{\N}_{\t,1}$}, since the difference between the two is {\small$\propto\N\Delta\mathbf{Z}^{\N}_{\t,1}$}, which has the form {\small$\N^{1/2}\mathfrak{f}$} for some {\small$\mathfrak{f}$} satisfying \eqref{eq:msheIIff} by the gradient formulas. Now, we observe that {\small$\Delta=-\grad^{\mathbf{X}}_{1}\grad^{\mathbf{X}}_{-1}=-\grad^{\mathbf{X}}_{-1}\grad^{\mathbf{X}}_{1}$}. This gives us two ways of computing {\small$\Delta\mathbf{Z}^{\N}$}, both of which are recorded as follows:
\begin{align}
\alpha\N^{2}\Delta\mathbf{Z}^{\N}_{\t,1}&=-\alpha\grad^{\mathbf{X}}_{1}\grad^{\mathbf{X}}_{-1}\mathbf{Z}^{\N}_{\t,1}\approx\grad^{\mathbf{X}}_{1}(\alpha\lambda\N^{\frac32}\bphi_{\t,1}\mathbf{Z}^{\N}_{\t,1})-\grad^{\mathbf{X}}_{1}(\tfrac12\alpha\lambda^{2}\N\bphi_{\t,1}^{2}\mathbf{Z}^{\N}_{\t,1}),\label{eq:msheinterior2I4a}\\
\alpha\N^{2}\Delta\mathbf{Z}^{\N}_{\t,1}&=-\alpha\grad^{\mathbf{X}}_{-1}\grad^{\mathbf{X}}_{1}\mathbf{Z}^{\N}_{\t,1}\approx-\grad^{\mathbf{X}}_{-1}(\alpha\lambda\N^{\frac32}\bphi_{\t,2}\mathbf{Z}^{\N}_{\t,1})-\grad^{\mathbf{X}}_{-1}(\tfrac12\alpha\lambda^{2}\N\bphi_{\t,2}^{2}\mathbf{Z}^{\N}_{\t,1}).\label{eq:msheinterior2I4b}
\end{align}
Moreover, we have {\small$\N\Delta\mathbf{Z}^{\N}_{\t,1}\approx0$}, since one gradient acting on {\small$\mathbf{Z}^{\N}_{\t,1}$} already introduces a power of {\small$\N^{-1/2}$}. We now combine this with \eqref{eq:msheinterior2I2}, \eqref{eq:msheinterior2I3}, and \eqref{eq:msheinterior2I4a}-\eqref{eq:msheinterior2I4b} to get
\begin{align*}
\lambda\N^{\frac32}\grad^{\mathbf{X}}_{1}\mathscr{U}'[\bphi_{\t,1}]\mathbf{Z}^{\N}_{\t,1}&\approx(\alpha\N^{2}+\tfrac14\N\lambda)\Delta\mathbf{Z}^{\N}_{\t,1}+\tfrac14\alpha\lambda^{2}\N\grad^{\mathbf{X}}_{1}(\bphi_{\t,1}^{2}\mathbf{Z}^{\N}_{\t,1})+\tfrac14\alpha\lambda^{2}\N\grad^{\mathbf{X}}_{-1}(\bphi_{\t,2}^{2}\mathbf{Z}^{\N}_{\t,1})\\
&+\tfrac12\lambda\N^{\frac32}\grad^{\mathbf{X}}_{1}(\mathscr{V}'[\bphi_{\t,1}]\mathbf{Z}^{\N}_{\t,1})-\tfrac12\lambda\N^{\frac32}\grad^{\mathbf{X}}_{-1}(\mathscr{V}'[\bphi_{\t,2}]\mathbf{Z}^{\N}_{\t,1})\\
&-\tfrac12\lambda^{2}\N(\mathscr{U}'[\bphi_{\t,1}]\bphi_{\t,1}+\mathscr{U}'[\bphi_{\t,2}]\bphi_{\t,2})\mathbf{Z}^{\N}_{\t,1}.
\end{align*}
We now plug the previous display into \eqref{eq:msheinterior2I1}. When do so, the following happens. First, in the last two terms in the first line above, we move {\small$\mathbf{Z}^{\N}_{\t,1}$} outside the gradients. We can do this because, by the gradient formulas above, the cost is of the form {\small$\N^{1/2}\mathfrak{f}$} for some {\small$\mathfrak{f}$} satisfying \eqref{eq:msheIIff}. Then, we use that the resulting coefficients cancel, i.e. {\small$\grad^{\mathbf{X}}_{1}\bphi_{\t,1}^{2}+\grad^{\mathbf{X}}_{-1}\bphi_{\t,2}^{2}=0$}. In particular, the last two terms in the first line cancel. Next, we combine the last line in the previous display with {\small$\lambda\N\mathbf{F}[\tau_{1}\bphi_{\t}]\mathbf{Z}^{\N}_{\t,1}\d\t+\lambda^{2}\N\mathbf{Z}^{\N}_{\t,1}\d\t$}. This yields the first term in \eqref{eq:msheinterior2Id}. Finally, we use {\small$\mathscr{T}_{\N}=(\alpha\N^{2}+\frac14\N\lambda)\Delta$} when acting at {\small$\x=1$} (there is no boundary data here). Thus, the first term on the \abbr{RHS} above matches the first term on the \abbr{RHS} of \eqref{eq:msheinterior2Ia}. Ultimately, we get \eqref{eq:msheinterior2Ia}-\eqref{eq:msheinterior2Id}, so the proof is done.
\end{proof}
\subsection{Edge dynamics}
We will now record an evolution equation for {\small$\mathbf{Z}^{\N}$} at the edge {\small$\x=0$}. It is here where the boundary condition in the {\small$\Delta_{\mathbf{A}}$}-operator from Definition \ref{definition:heat} plays an important role.
\begin{lemma}\label{lemma:msheedge}
\fsp We have the following \abbr{SDE}, in which {\small$\wt{\mathfrak{f}}$} satisfies the estimate \eqref{eq:msheIIff}:
\begin{align}
\d\mathbf{Z}^{\N}_{\t,0}&=\mathscr{T}_{\N}\mathbf{Z}^{\N}_{\t,0}\d\t+\sqrt{2}\lambda\N^{\frac12}\mathbf{Z}^{\N}_{\t,0}\d\mathbf{b}_{\t,0}+\lambda\N^{\frac32}\mathscr{V}'[\bphi_{\t,1}]\mathbf{Z}^{\N}_{\t,0}\d\t\label{eq:msheedgeI}\\
&+\N(\lambda\mathbf{F}[\bphi_{\t}]+\lambda^{2}-\tfrac12\alpha\lambda^{2}\bphi_{\t,1}^{2}+\alpha\lambda\mathbf{A})\mathbf{Z}^{\N}_{\t,0}\d\t+\N^{\frac12}\wt{\mathfrak{f}}[\bphi_{\t}]\mathbf{Z}^{\N}_{\t,0}\d\t.\nonumber
\end{align}
\end{lemma}
\begin{proof}
We use the It\^{o} formula with \eqref{eq:ch} and \eqref{eq:currII} (as in \eqref{eq:msheinterior2I1}) to get
\begin{align}
\d\mathbf{Z}^{\N}_{\t,0}&=\lambda\N^{\frac32}\mathscr{U}'[\bphi_{\t,1}]\mathbf{Z}^{\N}_{\t,0}\d\t+\lambda\N\mathbf{F}[\bphi_{\t}]\mathbf{Z}^{\N}_{\t,0}+\sqrt{2}\lambda\N^{\frac12}\mathbf{Z}^{\N}_{\t,0}\d\mathbf{b}_{\t,0}+\lambda^{2}\N\mathbf{Z}^{\N}_{\t,0}\d\t-\lambda^{2}\mathscr{R}_{\lambda}\mathbf{Z}^{\N}_{\t,0}\d\t.\nonumber
\end{align}
As in the proof of Lemma \ref{lemma:msheinterior2}, we note that {\small$\mathscr{T}_{\N}\mathbf{Z}^{\N}_{\t,0}\approx\alpha\N^{2}\Delta_{\mathbf{A}}\mathbf{Z}^{\N}_{\t,0}$}, where in this proof, {\small$\approx$} means identity up to {\small$\N^{1/2}\wt{\mathfrak{f}}$} for some function {\small$\wt{\mathfrak{f}}$} satisfying \eqref{eq:msheIIff}. Moreover, with the gradient formulas in the proof of Lemma \ref{lemma:msheinterior2}, we may compute {\small$\alpha\N^{2}\Delta_{\mathbf{A}}\mathbf{Z}^{\N}_{\t,0}:=\alpha\N^{2}\grad^{\mathbf{X}}_{1}\mathbf{Z}^{\N}_{\t,0}-\alpha\lambda\N\mathbf{A}\mathbf{Z}^{\N}_{\t,0}$} (see Definition \ref{definition:heat}) as
\begin{align}
\alpha\N^{2}\Delta_{\mathbf{A}}\mathbf{Z}^{\N}_{\t,0}&\approx\alpha\lambda\N^{\frac32}\bphi_{\t,1}\mathbf{Z}^{\N}_{\t,0}+\tfrac12\alpha\lambda^{2}\N\bphi_{\t,1}^{2}\mathbf{Z}^{\N}_{\t,0}-\alpha\lambda\N\mathbf{A}\mathbf{Z}^{\N}_{\t,0}.\nonumber
\end{align}
By the same token, we also have {\small$\N\Delta_{\mathbf{A}}\mathbf{Z}^{\N}_{\t,0}\approx0$}. Combining this and the last two displays gives \eqref{eq:msheedgeI}.
\end{proof}
\subsection{Putting it together}
As an immediate consequence of Lemmas \ref{lemma:msheinterior1}, \ref{lemma:msheinterior2}, and \ref{lemma:msheedge}, we have 
\begin{align*}
\d\mathbf{Z}^{\N}_{\t,\x}&=\mathscr{T}_{\N}\mathbf{Z}^{\N}_{\t,\x}\d\t+\sqrt{2}\lambda\N^{\frac12}\mathbf{Z}^{\N}_{\t,\x}\d\mathbf{b}_{\t,\x}+\mathbf{1}_{\x\in\llbracket2,\infty\rrbracket}\mathfrak{e}[\tau_{\x}\bphi_{\t}]\mathbf{Z}^{\N}_{\t,\x}\d\t\\
&+\mathbf{1}_{\x=1}\cdot\{\tfrac12\lambda\N^{\frac32}\grad^{\mathbf{X}}_{1}(\mathscr{V}'[\bphi_{\t,1}]\mathbf{Z}^{\N}_{\t,1})\d\t-\tfrac12\lambda\N^{\frac32}\grad^{\mathbf{X}}_{-1}(\mathscr{V}'[\bphi_{\t,2}]\mathbf{Z}^{\N}_{\t,1})\d\t\}\\
&+\mathbf{1}_{\x=1}\cdot\{\lambda\N(\mathbf{F}[\tau_{1}\bphi_{\t}]-\tfrac12\lambda(\mathscr{U}'[\bphi_{\t,1}]\bphi_{\t,1}-1)-\tfrac12\lambda(\mathscr{U}'[\bphi_{\t,2}]\bphi_{\t,2}-1))\mathbf{Z}^{\N}_{\t,1}\d\t\}\\
&+\mathbf{1}_{\x=0}\cdot\{\lambda\N^{\frac32}\mathscr{V}'[\bphi_{\t,1}]\mathbf{Z}^{\N}_{\t,0}\d\t\}+\mathbf{1}_{\x=0}\cdot\{\N(\lambda\mathbf{F}[\bphi_{\t}]+\lambda^{2}-\tfrac12\alpha\lambda^{2}\bphi_{\t,1}^{2}+\alpha\lambda\mathbf{A})\mathbf{Z}^{\N}_{\t,0}\d\t\}\\
&+\mathbf{1}_{\x=1}\cdot\{\N^{\frac12}\mathfrak{f}[\bphi_{\t}]\mathbf{Z}^{\N}_{\t,1}\d\t\}+\mathbf{1}_{\x=0}\cdot\{\N^{\frac12}\wt{\mathfrak{f}}[\bphi_{\t}]\mathbf{Z}^{\N}_{\t,0}\d\t\},
\end{align*}
where {\small$\mathfrak{f},\wt{\mathfrak{f}}$} satisfy \eqref{eq:msheIIff}. We now use the above \abbr{SDE} to get an equation for {\small$\d\mathbf{S}^{\N}$}. First, we rewrite \eqref{eq:heatIII} as 
\begin{align*}
\mathbf{S}^{\N}_{\t,\x}=[\mathbf{e}^{\mathfrak{t}_{\N}\mathscr{T}_{\N}}(\mathbf{Z}^{\N}_{\t,\cdot})]_{\x}-[\mathbf{e}^{\mathfrak{t}_{\N}\mathscr{T}_{\N}}((1-\mathbf{1}^{(\x)}_{\cdot})\cdot\mathbf{Z}^{\N}_{\t,\cdot})]_{\x}.
\end{align*}
We now take time-differentials on both sides of the \abbr{SDE} above. In the first term, said differential only hits {\small$\mathbf{Z}^{\N}$}, at which point we obtain the following by plugging in the \abbr{SDE} above for {\small$\d\mathbf{Z}^{\N}$} and rewriting some of the {\small$\mathbf{Z}^{\N}$}-factors as {\small$\mathbf{R}^{\N}\mathbf{S}^{\N}$}:
\begin{align*}
&\d[\mathbf{e}^{\mathfrak{t}_{\N}\mathscr{T}_{\N}}(\mathbf{Z}^{\N}_{\t,\cdot})]_{\x}=[\mathbf{e}^{\mathfrak{t}_{\N}\mathscr{T}_{\N}}(\d\mathbf{Z}^{\N}_{\t,\cdot})]_{\x}\\
&=[\mathbf{e}^{\mathfrak{t}_{\N}\mathscr{T}_{\N}}(\mathscr{T}_{\N}\mathbf{Z}^{\N}_{\t,\cdot})]_{\x}\d\t+[\mathbf{e}^{\mathfrak{t}_{\N}\mathscr{T}_{\N}}(\sqrt{2}\lambda\N^{\frac12}\mathbf{R}^{\N}_{\t,\cdot}\mathbf{S}^{\N}_{\t,\cdot}\d\mathbf{b}_{\t,\cdot})]_{\x}+[\mathbf{e}^{\mathfrak{t}_{\N}\mathscr{T}_{\N}}(\mathbf{1}_{\cdot\in\llbracket2,\infty\rrbracket}\mathfrak{e}[\tau_{\cdot}\bphi_{\t}]\mathbf{Z}^{\N}_{\t,\cdot})]_{\x}\d\t\\
&+[\mathbf{e}^{\mathfrak{t}_{\N}\mathscr{T}_{\N}}(\mathbf{1}_{\cdot=1}\cdot\{\tfrac12\lambda\N^{\frac32}\grad^{\mathbf{X}}_{1}(\mathscr{V}'[\bphi_{\t,1}]\mathbf{Z}^{\N}_{\t,1})-\tfrac12\lambda\N^{\frac32}\grad^{\mathbf{X}}_{-1}(\mathscr{V}'[\bphi_{\t,2}]\mathbf{Z}^{\N}_{\t,1})\})]_{\x}\d\t\\
&+[\mathbf{e}^{\mathfrak{t}_{\N}\mathscr{T}_{\N}}(\mathbf{1}_{\cdot=1}\cdot\{\lambda\N(\mathbf{F}[\tau_{1}\bphi_{\t}]-\tfrac12\lambda(\mathscr{U}'[\bphi_{\t,1}]\bphi_{\t,1}-1)-\tfrac12\lambda(\mathscr{U}'[\bphi_{\t,2}]\bphi_{\t,2}-1))\mathbf{Z}^{\N}_{\t,1})\})]_{\x}\d\t\\
&+[\mathbf{e}^{\mathfrak{t}_{\N}\mathscr{T}_{\N}}(\mathbf{1}_{\cdot=0}\cdot\{\lambda\N^{\frac32}\mathscr{V}'[\bphi_{\t,1}]\mathbf{Z}^{\N}_{\t,0}\})]_{\x}\d\t\\
&+[\mathbf{e}^{\mathfrak{t}_{\N}\mathscr{T}_{\N}}(\mathbf{1}_{\cdot=0}\cdot\{\N(\lambda\mathbf{F}[\bphi_{\t}]+\lambda^{2}-\tfrac12\alpha\lambda^{2}\bphi_{\t,1}^{2}+\alpha\lambda\mathbf{A})\mathbf{Z}^{\N}_{\t,0}\})]_{\x}\d\t\\
&+[\mathbf{e}^{\mathfrak{t}_{\N}\mathscr{T}_{\N}}(\mathbf{1}_{\cdot=0}\cdot\N^{\frac12}\wt{\mathfrak{f}}[\bphi_{\t}]\mathbf{R}^{\N}_{\t,0}\mathbf{S}^{\N}_{\t,0}+\mathbf{1}_{\cdot=1}\cdot\N^{\frac12}\mathfrak{f}[\bphi_{\t}]\mathbf{R}^{\N}_{\t,1}\mathbf{S}^{\N}_{\t,1})]_{\x}\d\t.
\end{align*}
Note that {\small$[\mathbf{e}^{\mathfrak{t}_{\N}\mathscr{T}_{\N}}(\mathscr{T}_{\N}\mathbf{Z}^{\N}_{\t,\cdot})]_{\x}=\mathscr{T}_{\N}[\mathbf{e}^{\mathfrak{t}_{\N}\mathscr{T}_{\N}}(\mathbf{Z}^{\N}_{\t,\cdot})]_{\x}$}, since the {\small$\mathscr{T}_{\N}$}-semigroup commutes with {\small$\mathscr{T}_{\N}$} itself. Therefore, 
\begin{align*}
[\mathbf{e}^{\mathfrak{t}_{\N}\mathscr{T}_{\N}}(\mathscr{T}_{\N}\mathbf{Z}^{\N}_{\t,\cdot})]_{\x}=\mathscr{T}_{\N}\mathbf{S}^{\N}_{\t,\x}+\mathscr{T}_{\N}[\mathbf{e}^{\mathfrak{t}_{\N}\mathscr{T}_{\N}}((1-\mathbf{1}^{(\x)}_{\cdot})\mathbf{Z}^{\N}_{\t,\cdot})]_{\x}.
\end{align*}
Next, let us compute {\small$\d[\mathbf{e}^{\mathfrak{t}_{\N}\mathscr{T}_{\N}}((1-\mathbf{1}^{(\x)}_{\cdot})\cdot\mathbf{Z}^{\N}_{\t,\cdot})]_{\x}=[\mathbf{e}^{\mathfrak{t}_{\N}\mathscr{T}_{\N}}((1-\mathbf{1}^{(\x)}_{\cdot})\cdot\d\mathbf{Z}^{\N}_{\t,\cdot})]_{\x}$}. An application of the It\^{o} formula with \eqref{eq:ch} and \eqref{eq:currI}-\eqref{eq:currII} shows that {\small$\d\mathbf{Z}^{\N}_{\t,\y}=\N^{3/2}\mathfrak{h}_{\y}[\bphi_{\t}]\d\mathbf{Z}^{\N}_{\t,\y}\d\t+\sqrt{2}\lambda\N^{1/2}\mathbf{Z}^{\N}_{\t,\y}\d\mathbf{b}_{\t,\y}$} for any {\small$\y\in\llbracket0,\infty\rrbracket$}, where {\small$\mathfrak{h}_{\y}$} satisfies the estimate \eqref{eq:msheIIff} (but replacing {\small$|\w|\leq\mathrm{c}$} by {\small$|\w-\y|\leq\mathrm{c}$}). Thus, we have 
\begin{align*}
\d[\mathbf{e}^{\mathfrak{t}_{\N}\mathscr{T}_{\N}}((1-\mathbf{1}^{(\x)}_{\cdot})\cdot\mathbf{Z}^{\N}_{\t,\cdot})]_{\x}&=[\mathbf{e}^{\mathfrak{t}_{\N}\mathscr{T}_{\N}}((1-\mathbf{1}^{(\x)}_{\cdot})\cdot\sqrt{2}\lambda\N^{\frac12}\mathbf{Z}^{\N}_{\t,\cdot}\d\mathbf{b}_{\t,\cdot})]_{\x}+\N^{\frac32}[\mathbf{e}^{\mathfrak{t}_{\N}\mathscr{T}_{\N}}((1-\mathbf{1}^{(\x)}_{\cdot})\cdot\mathfrak{h}_{\cdot}[\bphi_{\t}]\mathbf{Z}^{\N}_{\t,\cdot})]_{\x}.
\end{align*}
If we now combine the previous four displays, then we obtain
\begin{align}
\d\mathbf{S}^{\N}_{\t,\x}&=\mathscr{T}_{\N}\mathbf{S}^{\N}_{\t,\x}\d\t+[\mathbf{e}^{\mathfrak{t}_{\N}\mathscr{T}_{\N}}(\mathbf{1}^{(\x)}_{\cdot}\cdot\sqrt{2}\lambda\N^{\frac12}\mathbf{R}^{\N}_{\t,\cdot}\mathbf{S}^{\N}_{\t,\cdot}\d\mathbf{b}_{\t,\cdot})]_{\x}+[\mathbf{e}^{\mathfrak{t}_{\N}\mathscr{T}_{\N}}(\mathbf{1}_{\cdot\in\llbracket2,\infty\rrbracket}\mathfrak{e}[\tau_{\cdot}\bphi_{\t}]\mathbf{Z}^{\N}_{\t,\cdot})]_{\x}\d\t\label{eq:sdebasic1}\\
&+[\mathbf{e}^{\mathfrak{t}_{\N}\mathscr{T}_{\N}}(\mathbf{1}_{\cdot=1}\cdot\{\tfrac12\lambda\N^{\frac32}\grad^{\mathbf{X}}_{1}(\mathscr{V}'[\bphi_{\t,1}]\mathbf{Z}^{\N}_{\t,1})-\tfrac12\lambda\N^{\frac32}\grad^{\mathbf{X}}_{-1}(\mathscr{V}'[\bphi_{\t,2}]\mathbf{Z}^{\N}_{\t,1})\})]_{\x}\d\t\label{eq:sdebasic2}\\
&+[\mathbf{e}^{\mathfrak{t}_{\N}\mathscr{T}_{\N}}(\mathbf{1}_{\cdot=1}\cdot\{\lambda\N(\mathbf{F}[\tau_{1}\bphi_{\t}]-\tfrac12\lambda(\mathscr{U}'[\bphi_{\t,1}]\bphi_{\t,1}-1)-\tfrac12\lambda(\mathscr{U}'[\bphi_{\t,2}]\bphi_{\t,2}-1))\mathbf{Z}^{\N}_{\t,1}\})]_{\x}\d\t\label{eq:sdebasic4}\\
&+[\mathbf{e}^{\mathfrak{t}_{\N}\mathscr{T}_{\N}}(\mathbf{1}_{\cdot=0}\cdot\{\lambda\N^{\frac32}\mathscr{V}'[\bphi_{\t,1}]\mathbf{Z}^{\N}_{\t,0}\})]_{\x}\d\t\label{eq:sdebasic5}\\
&+[\mathbf{e}^{\mathfrak{t}_{\N}\mathscr{T}_{\N}}(\mathbf{1}_{\cdot=0}\cdot\{\N(\lambda\mathbf{F}[\bphi_{\t}]+\lambda^{2}-\tfrac12\alpha\lambda^{2}\bphi_{\t,1}^{2}+\alpha\lambda\mathbf{A})\mathbf{Z}^{\N}_{\t,0}\})]_{\x}\d\t\label{eq:sdebasic6}\\
&+[\mathbf{e}^{\mathfrak{t}_{\N}\mathscr{T}_{\N}}(\mathbf{1}_{\x=0}\cdot\N^{\frac12}\wt{\mathfrak{f}}[\bphi_{\t}]\mathbf{R}^{\N}_{\t,0}\mathbf{S}^{\N}_{\t,0}+\mathbf{1}_{\x=1}\cdot\N^{\frac12}\mathfrak{f}[\bphi_{\t}]\mathbf{R}^{\N}_{\t,1}\mathbf{S}^{\N}_{\t,1})]_{\x}\d\t\label{eq:sdebasic7}\\
&+\mathscr{T}_{\N}[\mathbf{e}^{\mathfrak{t}_{\N}\mathscr{T}_{\N}}((1-\mathbf{1}^{(\x)}_{\cdot})\mathbf{Z}^{\N}_{\t,\cdot})]_{\x}-\N^{\frac32}[\mathbf{e}^{\mathfrak{t}_{\N}\mathscr{T}_{\N}}((1-\mathbf{1}^{(\x)}_{\cdot})\cdot\mathfrak{h}_{\cdot}[\bphi_{\t}]\mathbf{Z}^{\N}_{\t,\cdot})]_{\x}.\label{eq:sdebasic8}
\end{align}
We clarify that the {\small$\mathbf{Z}^{\N}$}-factors that we have not rewritten as {\small$\mathbf{R}^{\N}\mathbf{S}^{\N}$} were left untouched because we will to perform extra gymnastics to obtain the {\small$\mathbf{Av}$}-terms in Definitions \ref{definition:bulkaverage} and \ref{definition:edgeaverage}. We now show that the last two terms, those in \eqref{eq:sdebasic8}, are ultimately error terms because of the factors of {\small$1-\mathbf{1}^{(\x)}_{\cdot}$} which restrict the heat kernels therein beyond their natural diffusive scale.
\begin{lemma}\label{lemma:triv}
\fsp The terms {\small$\mathscr{T}_{\N}[\mathbf{e}^{\mathfrak{t}_{\N}\mathscr{T}_{\N}}((1-\mathbf{1}^{(\x)}_{\cdot})\mathbf{Z}^{\N}_{\t,\cdot})]_{\x}$} and {\small$\N^{\frac32}[\mathbf{e}^{\mathfrak{t}_{\N}\mathscr{T}_{\N}}((1-\mathbf{1}^{(\x)}_{\cdot})\cdot\mathfrak{h}_{\cdot}[\bphi_{\t}]\mathbf{Z}^{\N}_{\t,\cdot})]_{\x}$} are both of the form {\small$\mathrm{Err}[\mathbf{R}^{\N}\mathbf{S}^{\N}]_{\t,\x}$} from \eqref{eq:sdeII1}-\eqref{eq:sdeII2}.
\end{lemma}
\begin{proof}
By the off-diagonal decay in Proposition \ref{prop:hkestimates}, we have (recall {\small$\mathfrak{t}_{\N}=\N^{-\delta_{\mathbf{S}}}$} from Definition \ref{definition:heat})
\begin{align*}
|\mathscr{T}_{\N}\mathbf{H}^{\N,\mathbf{A}}_{0,\mathfrak{t}_{\N},\x,\w}|\mathbf{1}_{|\x-\w|\gtrsim(\log\N)^{2}\N\mathfrak{t}_{\N}^{1/2}}&\lesssim_{\kappa}\N^{2}\cdot\N^{-1}\mathfrak{t}_{\N}^{-1/2}\exp(-\tfrac{2\kappa|\x-\w|}{\N\mathfrak{t}_{\N}^{1/2}})\mathbf{1}_{|\x-\w|\gtrsim(\log\N)^{2}\N\mathfrak{t}_{\N}^{1/2}}\\
&\lesssim\N^{1+\frac12\delta_{\mathbf{S}}}\exp(-\tfrac{\kappa|\x-\w|}{\N\mathfrak{t}_{\N}^{1/2}})\cdot\mathbf{1}_{|\x-\w|\gtrsim(\log\N)^{2}\N\mathfrak{t}_{\N}^{1/2}}\cdot\exp(-\tfrac{\kappa|\x-\w|}{\N})\\
&\lesssim_{\kappa,\mathrm{D}}\N^{-\mathrm{D}}\exp(-\tfrac{\kappa|\x-\w|}{\N})
\end{align*}
for any {\small$\kappa,\mathrm{D}>0$}. The same estimate holds without {\small$\mathscr{T}_{\N}$} as well. Combining this with the deterministic bound for {\small$\mathfrak{h}_{\w}$} (mentioned prior to \eqref{eq:sdebasic1}-\eqref{eq:sdebasic8}) completes the proof.
\end{proof}
We will now further manipulate \eqref{eq:sdebasic1}-\eqref{eq:sdebasic8} in the following ways.
\begin{enumerate}
\item We observe that the term in \eqref{eq:sdebasic2} is restricted not just to {\small$\cdot=1$} but to all {\small$\cdot\in\llbracket1,\infty\rrbracket$}, because it also appears in {\small$\mathfrak{e}$} (see \eqref{eq:msheII} and \eqref{eq:msheII(1)a}). In order to analyze this term, we will need to move the gradients in \eqref{eq:sdebasic2} onto the {\small$\mathbf{H}^{\N,\mathbf{A}}$}-heat kernel therein. However, the presence of the edge at {\small$0$} leads to boundary terms when we use summation-by-parts to move the gradient.  It turns out this boundary term cancels the term in \eqref{eq:sdebasic5}, modulo other error terms that are lower order (with scaling factor of {\small$\N$} instead of {\small$\N^{3/2}$}). Similarly, we also develop a summation-by-parts calculation for the other terms in {\small$\mathfrak{e}$} in \eqref{eq:sdebasic1} that carry a discrete gradient as well.
\item We identify the Laplacian term in {\small$\mathfrak{e}$} (see \eqref{eq:msheII}) as a contribution to the {\small$\mathrm{Err}$}-term in \eqref{eq:sdeII1}. Indeed, small-ness of this Laplacian term comes from the fact that it is scaled too weakly in {\small$\N$} to contribute anything.
\item We identify the remaining coefficients as either bulk-admissible or edge-admissible terms. Then, we replace these coefficients by their space-time and time averages to eventually get the form \eqref{eq:sdeI3}-\eqref{eq:sdeI5}.
\end{enumerate}
\subsubsection{Summation-by-parts lemmas}
We embark on making precise step (1) above. Key to this is the following general summation-by-parts calculation. Before we state the lemma, we emphasize that it is basically a discrete-version of integration-by-parts, turning one bulk sum into another with a cost of an edge term.
\begin{lemma}\label{lemma:sbp}
\fsp Consider a (possibly random and time-dependent) function {\small$\mathsf{f}:\llbracket0,\infty\rrbracket\to\R$}. We have the following, where we use notation to be explained after:
\begin{align}
[\mathbf{e}^{\mathfrak{t}_{\N}\mathscr{T}_{\N}}(\mathbf{1}_{\cdot\in\llbracket1,\infty\rrbracket}\cdot\grad^{\mathbf{X}}_{1}(\mathsf{f}_{\cdot}\mathbf{Z}^{\N}_{\t,\cdot}))]_{\x}&=\sum_{\w\in\llbracket2,\infty\rrbracket}\grad^{\mathbf{X}}_{-1}\mathbf{H}^{\N,\mathbf{A}}_{0,\mathfrak{t}_{\N},\x,\w} \cdot\mathsf{f}_{\w}\mathbf{Z}^{\N}_{\t,\w}-\mathbf{H}^{\N,\mathbf{A}}_{0,\mathfrak{t}_{\N},\x,1}\cdot\mathsf{f}_{1}\mathbf{Z}^{\N}_{\t,1}.\nonumber
\end{align}
Similarly, we also have the formula
\begin{align}
[\mathbf{e}^{\mathfrak{t}_{\N}\mathscr{T}_{\N}}(\mathbf{1}_{\cdot\in\llbracket1,\infty\rrbracket}\cdot\grad^{\mathbf{X}}_{-1}(\mathsf{f}_{\cdot+1}\mathbf{Z}^{\N}_{\t,\cdot}))]_{\x}&=\sum_{\w\in\llbracket0,\infty\rrbracket}\grad^{\mathbf{X}}_{1}\mathbf{H}^{\N,\mathbf{A}}_{0,\mathfrak{t}_{\N},\x,\w}\cdot\mathsf{f}_{\w+1}\mathbf{Z}^{\N}_{\t,\w}+\mathbf{H}^{\N,\mathbf{A}}_{0,\mathfrak{t}_{\N},\x,0}\mathsf{f}_{1}\mathbf{Z}^{\N}_{\t,0}.\nonumber
\end{align}
\end{lemma}
\begin{proof}
On the \abbr{LHS} of the first desired identity, we can extend the sum over {\small$\llbracket0,\infty\rrbracket$} in the {\small$\mathbf{e}^{\mathfrak{t}_{\N}\mathscr{T}_{\N}}$} operator to all of {\small$\Z$} because of the indicator function {\small$\mathbf{1}_{\cdot\in\llbracket1,\infty\rrbracket}$}. Thus, since the adjoint of {\small$\grad^{\mathbf{X}}_{\mathfrak{l}}$} with respect to the flat measure on {\small$\Z$} is equal to {\small$\grad^{\mathbf{X}}_{-\mathfrak{l}}$}, we deduce the following, in which all gradients act on the {\small$\w$}-variable:
\begin{align*}
[\mathbf{e}^{\mathfrak{t}_{\N}\mathscr{T}_{\N}}(\mathbf{1}_{\cdot\in\llbracket1,\infty\rrbracket}\cdot\grad^{\mathbf{X}}_{1}(\mathsf{f}_{\cdot}\mathbf{Z}^{\N}_{\t,\cdot}))]_{\x}=\sum_{\w\in\Z}\grad^{\mathbf{X}}_{-1}(\mathbf{H}^{\N,\mathbf{A}}_{0,\mathfrak{t}_{\N},\x,\w}\mathbf{1}_{\w\in\llbracket1,\infty\rrbracket})\cdot\mathsf{f}_{\w}\mathbf{Z}^{\N}_{\t,\w}.
\end{align*}
We now apply the discrete Leibniz rule from the proof of Lemma \ref{lemma:msheinterior2} to observe that 
\begin{align*}
\grad^{\mathbf{X}}_{-1}(\mathbf{H}^{\N,\mathbf{A}}_{0,\mathfrak{t}_{\N},\x,\w}\mathbf{1}_{\w\in\llbracket1,\infty\rrbracket})&=\mathbf{1}_{\w-1\in\llbracket1,\infty\rrbracket}\grad^{\mathbf{X}}_{-1}\mathbf{H}^{\N,\mathbf{A}}_{0,\mathfrak{t}_{\N},\x,\w}+\mathbf{H}^{\N,\mathbf{A}}_{0,\mathfrak{t}_{\N},\x,\w}\grad^{\mathbf{X}}_{-1}\mathbf{1}_{\w\in\llbracket1,\infty\rrbracket}\\
&=\mathbf{1}_{\w\in\llbracket2,\infty\rrbracket}\grad^{\mathbf{X}}_{-1}\mathbf{H}^{\N,\mathbf{A}}_{0,\mathfrak{t}_{\N},\x,\w}-\mathbf{1}_{\w=1}\mathbf{H}^{\N,\mathbf{A}}_{0,\mathfrak{t}_{\N},\x,\w}.
\end{align*}
Combining the previous two displays gives the first identity. For the second identity, it is enough to use instead
\begin{align*}
[\mathbf{e}^{\mathfrak{t}_{\N}\mathscr{T}_{\N}}(\mathbf{1}_{\cdot\in\llbracket1,\infty\rrbracket}\cdot\grad^{\mathbf{X}}_{-1}(\mathsf{f}_{\cdot+1}\mathbf{Z}^{\N}_{\t,\cdot}))]_{\x}=\sum_{\w\in\Z}\grad^{\mathbf{X}}_{1}(\mathbf{H}^{\N,\mathbf{A}}_{0,\mathfrak{t}_{\N},\x,\w}\mathbf{1}_{\w\in\llbracket1,\infty\rrbracket})\cdot\mathsf{f}_{\w+1}\mathbf{Z}^{\N}_{\t,\w},
\end{align*}
which also follows by the same summation-by-parts considerations, as well as the Leibniz rule identity
\begin{align*}
\grad^{\mathbf{X}}_{1}(\mathbf{H}^{\N,\mathbf{A}}_{0,\mathfrak{t}_{\N},\x,\w}\mathbf{1}_{\w\in\llbracket1,\infty\rrbracket})&=\mathbf{1}_{\w+1\in\llbracket1,\infty\rrbracket}\grad^{\mathbf{X}}_{1}\mathbf{H}^{\N,\mathbf{A}}_{0,\mathfrak{t}_{\N},\x,\w}+\mathbf{H}^{\N,\mathbf{A}}_{0,\mathfrak{t}_{\N},\x,\w}\cdot\grad^{\mathbf{X}}_{1}\mathbf{1}_{\w\in\llbracket1,\infty\rrbracket},
\end{align*}
while noting that {\small$\mathbf{1}_{\w+1\in\llbracket1,\infty\rrbracket}=\mathbf{1}_{\w\in\llbracket0,\infty\rrbracket}$} and {\small$\grad^{\mathbf{X}}_{1}\mathbf{1}_{\w\in\llbracket1,\infty\rrbracket}=\mathbf{1}_{\w=0}$}. This completes the proof.
\end{proof}
\subsubsection{Laplacian term}
The next step, namely controlling the Laplacian term in \eqref{eq:msheII}, is of a similar spirit, but it uses a more subtle cancellation that is reflected by the fact that the {\small$\mathscr{T}_{\N}$}-semigroup commutes with {\small$\mathscr{T}_{\N}$} itself.
\begin{lemma}\label{lemma:lap}
\fsp The quantity {\small$[\mathbf{e}^{\mathfrak{t}_{\N}\mathscr{T}_{\N}}(\mathbf{1}_{\cdot\in\llbracket2,\infty\rrbracket}\cdot\N\Delta\mathfrak{f}_{0}[\tau_{\cdot}\bphi_{\t}]\mathbf{Z}^{\N}_{\t,\cdot})]_{\x}$}, where the function {\small$\mathfrak{f}_{0}$} satisfies the deterministic estimate \eqref{eq:msheIIff}, has the form of {\small$\mathrm{Err}[\mathbf{R}^{\N}\mathbf{S}^{\N}]_{\t,\x}$} from \eqref{eq:sdeII1}.
\end{lemma}
\begin{proof}
Extend the function {\small$\x\mapsto\mathfrak{f}_{0}[\tau_{\x}\bphi_{\t}]$} from {\small$\x\in\llbracket2,\infty\rrbracket$} to {\small$\x\in\llbracket0,\infty\rrbracket$} by {\small$\Delta_{\mathbf{A}}\mathfrak{f}_{0}[\tau_{\x}\bphi_{\t}]=0$} for {\small$\x=0,1$}. (This can be done since it requires solving a linear system of two equation in two variables, conditioning on the values {\small$\mathfrak{f}_{0}[\tau_{\x}\bphi_{\t}]$}.) Then, we have {\small$[\mathbf{e}^{\mathfrak{t}_{\N}\mathscr{T}_{\N}}(\mathbf{1}_{\cdot\in\llbracket2,\infty\rrbracket}\cdot\N\Delta\mathfrak{f}_{0}[\tau_{\cdot}\bphi_{\t}]\mathbf{Z}^{\N}_{\t,\cdot})]_{\x}=[\mathbf{e}^{\mathfrak{t}_{\N}\mathscr{T}_{\N}}(\N\Delta_{\mathbf{A}}\mathfrak{f}_{0}[\tau_{\cdot}\bphi_{\t}]\mathbf{Z}^{\N}_{\t,\cdot})]_{\x}$}. However, {\small$\mathscr{T}_{\N}$} is a constant multiple of {\small$\Delta_{\mathbf{A}}$}, thus the {\small$\mathscr{T}_{\N}$}-semigroup commutes with {\small$\Delta_{\mathbf{A}}$}, and we get {\small$[\mathbf{e}^{\mathfrak{t}_{\N}\mathscr{T}_{\N}}(\mathbf{1}_{\cdot\in\llbracket2,\infty\rrbracket}\cdot\N\Delta\mathfrak{f}_{0}[\tau_{\cdot}\bphi_{\t}]\mathbf{Z}^{\N}_{\t,\cdot})]_{\x}=\N\Delta_{\mathbf{A}}[\mathbf{e}^{\mathfrak{t}_{\N}\mathscr{T}_{\N}}(\mathfrak{f}_{0}[\tau_{\cdot}\bphi_{\t}]\mathbf{Z}^{\N}_{\t,\cdot})]_{\x}$}. We now expand the last quantity as follows, where {\small$\Delta_{\mathbf{A}}$} acts with respect to {\small$\x$}:
\begin{align}
\N\Delta_{\mathbf{A}}[\mathbf{e}^{\mathfrak{t}_{\N}\mathscr{T}_{\N}}(\mathfrak{f}_{0}[\tau_{\cdot}\bphi_{\t}]\mathbf{Z}^{\N}_{\t,\cdot})]_{\x}=\sum_{\w\in\llbracket0,\infty\rrbracket}\N\Delta_{\mathbf{A}}\mathbf{H}^{\N,\mathbf{A}}_{0,\mathfrak{t}_{\N},\x,\w}\mathfrak{f}_{0}[\tau_{\w}\bphi_{\t}]\mathbf{Z}^{\N}_{\t,\w}.
\end{align}
By the heat kernel estimates in Proposition \ref{prop:hkestimates}, we get {\small$|\N\Delta_{\mathbf{A}}\mathbf{H}^{\N,\mathbf{A}}_{0,\mathfrak{t}_{\N},\x,\w}|\lesssim_{\kappa}\N^{-2}\mathfrak{t}_{\N}^{-3}\exp(-\frac{\kappa|\x-\w|}{\N})$} for all {\small$\kappa,\x,\w$}, at which point we can plug the bound \eqref{eq:msheIIff} for {\small$\mathfrak{f}_{0}$}, use {\small$\mathfrak{t}_{\N}=\N^{-\delta_{\mathbf{S}}}$}, and use {\small$\mathbf{Z}^{\N}=\mathbf{R}^{\N}\mathbf{S}^{\N}$} to finish the proof.
\end{proof}
\subsubsection{Identification of coefficients}
We proceed towards step (3) after \eqref{eq:sdebasic1}-\eqref{eq:sdebasic8}. For this, we require a list of functions that are bulk-admissible and another of edge-admissible ones. We start with the following lemma, which collects a list of functions in jet spaces and is taken directly from Lemma 5.4 of \cite{Y26PTRF}:
\begin{lemma}\label{lemma:classify}
\fsp Recall the notation in \eqref{eq:msheII} to \eqref{eq:msheIIff}. We have the following (where {\small$\mathsf{F}$} is a generic local function):
\begin{enumerate}
\item The function {\small$\mathscr{Q}[\bphi]$} is in {\small$\mathrm{Jet}_{2}^{\perp}$}.
\item The function {\small$\mathscr{V}'[{\boldsymbol{\phi}}_{1}]=\mathscr{U}'[\bphi_{1}]-\alpha\bphi_{1}$} is in {\small$\mathrm{Jet}_{1}^{\perp}$}.
\item The function {\small$\mathscr{U}'[\bphi_{1}]\bphi_{1}-1$} is in {\small$\mathrm{Jet}_{1}^{\perp}$}.
\item The function {\small$\alpha{\boldsymbol{\phi}}_{1}^{2}-1$} is in {\small$\mathrm{Jet}_{0}^{\perp}$}. 
\item The function {\small$\mathscr{U}'[\bphi_{1}]$} is in {\small$\mathrm{Jet}_{0}^{\perp}$}
\item Any function of the form {\small$\mathsf{F}[{\boldsymbol{\phi}}]-\E^{0}\mathsf{F}[{\boldsymbol{\phi}}]$} is in {\small$\mathrm{Jet}_{0}^{\perp}$}.
\item Any function of the form {\small$\mathsf{F}[\boldsymbol{\phi}]-\mathsf{F}[\tau_{\ell}\boldsymbol{\phi}]$} is in {\small$\mathrm{Jet}_{2}^{\perp}$}, where {\small$\ell$} is any integer so that {\small$\mathsf{F}[\tau_{\ell}\bphi]$} has support in {\small$\llbracket1,\infty\rrbracket$}.
\item The function {\small$\mathscr{G}[\bphi]$} is in {\small$\mathrm{Jet}_{1}^{\perp}$}.
\item The function {\small$\mathscr{G}[\bphi]\bphi_{1}-1$} is in {\small$\mathrm{Jet}_{0}^{\perp}$}.
\item The function {\small$\mathbf{F}_{\geq3}[\bphi]$} is in {\small$\mathrm{Jet}_{2}^{\perp}$}.
\item For any {\small$\w,\w'\in\llbracket1,\infty\rrbracket$}, the function {\small$\mathbf{F}_{\geq3}[\bphi]\bphi_{\w}$} is in {\small$\mathrm{Jet}_{1}^{\perp}$}, and {\small$\mathbf{F}_{\geq3}[\bphi]\bphi_{\w}\bphi_{\w'}$} is in {\small$\mathrm{Jet}_{0}^{\perp}$}.
\end{enumerate}
\end{lemma}
Lemma \ref{lemma:classify} and direct inspection (of supports and a priori estimates for functions on {\small$\R^{\llbracket1,\infty\rrbracket}$}) imply that the coefficients appearing in the {\small$\mathfrak{e}_{\ell}$}-terms for {\small$\ell=0,1,2$} (see Notation \ref{notation:mshe}) are all bulk-admissible; we will return to this in more detail when it becomes more relevant. We now present a list of functions that are edge-admissible functions, which is new to this paper and does not involve any support conditions.
\begin{lemma}\label{lemma:classifyedge}
\fsp Recall the {Robin parameter {\small$\mathbf{A}=-\frac12\alpha^{-1}\lambda$} from \eqref{eq:mainII}}. We have the following:
\begin{enumerate}
\item The function {\small$\bphi\mapsto\mathbf{F}[\bphi]+\lambda^{2}-\frac12\alpha\lambda^{2}\bphi_{1}^{2}+\alpha\lambda\mathbf{A}$} is edge-admissible.
\item The function {\small$\bphi\mapsto\mathscr{U}'[\bphi_{1}]\bphi_{1}-\alpha\bphi_{1}^{2}=\mathscr{V}'[\bphi_{1}]\bphi_{1}$} is edge-admissible.
\item The function {\small$\bphi\mapsto\mathscr{U}'[\bphi_{1}]-\bphi_{1}$} is edge-admissible.
\end{enumerate}
\end{lemma}
\begin{proof}
To prove point (1), we observe that {\small$\E^{0}\mathbf{F}[\bphi]=0$}. This follows by \eqref{eq:nl}, the property of {\small$\mathbb{P}^{0}$} being a product measure over sites in {\small$\llbracket1,\infty\rrbracket$}, and {\small$\E^{0}\mathscr{U}'[\bphi_{\x}]=0$} as shown above. On the other hand, we get {\small$\E^{0}\bphi_{1}^{2}=\alpha^{-1}$} by point (4) in Lemma \ref{lemma:classify}. This yields point (1).

To show (2), we first note that the following generalization of Gaussian-integration-by-parts to non-quadratic potentials {\small$\mathscr{U}$} holds for any smooth and local function {\small$\mathsf{F}$} and any {\small$\x\in\llbracket1,\infty\rrbracket$}:
\begin{align}
\E^{0}(\mathscr{U}'[\bphi_{\x}]\mathsf{F}[\bphi])=\E^{0}(\partial_{\bphi_{\x}}\mathsf{F}[\bphi]).\label{eq:classifybdry1}
\end{align}
If we plug in {\small$\mathsf{F}[\bphi]=\bphi_{\x}$} and set {\small$\x=1$}, then we obtain {\small$\E^{0}(\mathscr{U}'[\bphi_{\x}]\bphi_{\x})=1$}. Combining this with {\small$\E^{0}\bphi_{1}^{2}=\alpha^{-1}$} (see point (4) in Lemma \ref{lemma:classify}) yields point (2). To prove (3), we use \eqref{eq:classifybdry1} with {\small$\mathsf{F}\equiv1$} and {\small$\x=1$} to get {\small$\E^{0}(\mathscr{U}'[\bphi_{1}])=0$}. Combining this with {\small$\E^{0}\bphi_{1}=0$} (which follows by construction in Definition \ref{definition:gc}), we get point (3). To prove (4), we use again that {\small$\mathbf{F}$} is edge-admissible. Moreover, using \eqref{eq:classifybdry1} with {\small$\mathsf{F}[\bphi_{\x}]=\bphi_{\x}$}, we obtain that {\small$\mathscr{U}'[\bphi_{\x}]\bphi_{\x}-1$} is edge-admissible. This completes the proof.
\end{proof}
The next step is to combine Lemmas \ref{lemma:triv}, \ref{lemma:sbp}, \ref{lemma:lap}, \ref{lemma:classify}, and \ref{lemma:classifyedge} to organize the \abbr{SDE} \eqref{eq:sdebasic1}-\eqref{eq:sdebasic8} into a form that more closely resembles the desired \abbr{SDE} \eqref{eq:sdeI1}-\eqref{eq:sdeI5} (but without all the averaging of terms yet).
\begin{lemma}\label{lemma:sdepre}
\fsp Recall the notation from Proposition \ref{prop:sde}. We have
\begin{align}
&\d\mathbf{S}^{\N}_{\t,\x}=\mathscr{T}_{\N}\mathbf{S}^{\N}_{\t,\x}\d\t+[\mathbf{e}^{\mathfrak{t}_{\N}\mathscr{T}_{\N}}(\sqrt{2}\lambda\N^{\frac12}\mathbf{R}^{\N}_{\t,\cdot}\mathbf{S}^{\N}_{\t,\cdot}\d\mathbf{b}_{\t,\cdot})]_{\x}+\mathrm{Err}[\mathbf{R}^{\N}\mathbf{S}^{\N}]_{\t,\x}\d\t\label{eq:sdepreI1}\\
&+\sum_{\ell\in\llbracket1,\mathrm{L}_{1}\rrbracket}\sum_{\w\in\llbracket0,\infty\rrbracket}\boldsymbol{\mathcal{S}}^{(\ell)}_{\x,\w}\cdot\N\mathfrak{q}_{\ell,\w}[\bphi_{\t}]\mathbf{Z}^{\N}_{\t,\w}+\sum_{\ell\in\llbracket1,\mathrm{L}_{2}\rrbracket}\sum_{\w\in\llbracket0,\infty\rrbracket}\boldsymbol{\mathcal{K}}^{\partial,\ell}_{\x,\w}\cdot\N\mathfrak{b}_{\ell}[\bphi_{\t}]\mathbf{Z}^{\N}_{\t,\w}\d\t.\label{eq:sdepreI3}
\end{align}
The kernels {\small$\boldsymbol{\mathcal{S}}^{(\ell)}$} satisfy the following upgraded version of \eqref{eq:sdeIII} to higher regularity with respect to the close-to-macroscopic length-scale {\small$\N^{1-\delta_{\mathbf{S}}}$}. In particular, it is deterministic and satisfies the following estimate for any {\small$|\mathfrak{l}_{1}|,|\mathfrak{l}_{2}|\lesssim1$}, any {\small$\kappa>0$}, and any non-negative integers {\small$\a,\b\lesssim1$}, in which {\small$\mathrm{C}\lesssim1$} and {\small$\delta_{\mathbf{S}}>0$} is small:
\begin{align}
|(\grad^{\mathbf{X}}_{\mathfrak{l}_{1}})^{\a}(\grad^{\mathbf{X}}_{\mathfrak{l}_{2}})^{\b}\boldsymbol{\mathcal{S}}^{(\ell)}_{\x,\w}|\lesssim\N^{-1-(\mathrm{a}+\mathrm{b})+(\mathrm{a}+\mathrm{b}+\mathrm{C})\delta_{\mathbf{S}}}\exp(-\tfrac{\kappa|\x-\w|}{\N}).\label{eq:sdepreI}
\end{align}
(The first iterated discrete gradient acts on the {\small$\x$}-variable, and the second acts on the {\small$\w$}-variable; we clarify that we will only restrict to {\small$\mathfrak{l}_{1},\mathfrak{l}_{2},\x,\w,\a,\b$} such that the \abbr{LHS} only involves values of {\small$\boldsymbol{\mathcal{S}}^{(\ell)}$} on {\small$\llbracket0,\infty\rrbracket\times\llbracket0,\infty\rrbracket$}.)
\end{lemma}
Before we start the proof, we make a clarifying point. The kernels {\small$\boldsymbol{\mathcal{S}}^{(\ell)}$} and {\small$\boldsymbol{\mathcal{K}}^{\partial,\ell}$} are ultimately {\small$\mathbf{H}^{\N,\mathbf{A}}$}-kernels at time-scale {\small$\mathfrak{t}_{\N}=\N^{-\delta_{\mathbf{S}}}$}, possibly with a finite number of length-{\small$\mathrm{O}(1)$} discrete gradients. The necessary estimates \eqref{eq:sdeIV} and \eqref{eq:sdepreI} will thus follow by standard heat kernel estimates in Proposition \ref{prop:hkestimates}. (See Remark \ref{remark:sde}.)
\begin{proof}
Take the \abbr{SDE} \eqref{eq:sdebasic1}-\eqref{eq:sdebasic8}. The first two terms on the \abbr{RHS} match to the \abbr{RHS} of \eqref{eq:sdeI1}. By Lemmas \ref{lemma:classify} and \ref{lemma:classifyedge}, the terms in \eqref{eq:sdebasic4} and \eqref{eq:sdebasic6} can be absorbed by the last term in \eqref{eq:sdepreI3} (as the coefficients are of order {\small$\N$}, are edge-admissible as in Definition \ref{definition:edgeadmissible}, and localized to {\small$0,1\in\llbracket0,\infty\rrbracket$}). The terms in \eqref{eq:sdebasic7} may be absorbed by the {\small$\mathrm{Err}$} term in \eqref{eq:sdepreI1}. The terms in \eqref{eq:sdebasic8} are addressed in Lemma \ref{lemma:triv}. Thus, we are left with the last term in \eqref{eq:sdebasic1}, \eqref{eq:sdebasic2}, and \eqref{eq:sdebasic5}. By Notation \ref{notation:mshe}, this can be separated into the following cases. 
\begin{enumerate}
\item Consider the last term in \eqref{eq:sdebasic1}, and extract from that the quantity {\small$\N^{-1/2}\mathfrak{f}_{3}[\tau_{\x}\bphi_{\t}]\mathbf{Z}^{\N}_{\t,\x}$}; see Notation \ref{notation:mshe}. The contribution of this term is {\small$[\mathbf{e}^{\mathfrak{t}_{\N}\mathscr{T}_{\N}}(\mathbf{1}_{\cdot\in\llbracket2,\infty\rrbracket}\N^{-1/2}\mathfrak{f}_{3}[\tau_{\cdot}\bphi_{\t}]\mathbf{R}^{\N}_{\t,\cdot}\mathbf{S}^{\N}_{\t,\cdot})]]_{\x}$}. Because {\small$\mathfrak{f}_{3}$} satisfies the bound in \eqref{eq:msheIIff}, heat kernel estimates in Proposition \ref{prop:hkestimates} show that this contribution {\small$[\mathbf{e}^{\mathfrak{t}_{\N}\mathscr{T}_{\N}}(\mathbf{1}_{\cdot\in\llbracket2,\infty\rrbracket}\N^{-1/2}\mathfrak{f}_{3}[\tau_{\cdot}\bphi_{\t}]\mathbf{R}^{\N}_{\t,\cdot}\mathbf{S}^{\N}_{\t,\cdot})]]_{\x}$} has the form of {\small$\mathrm{Err}[\mathbf{R}^{\N}\mathbf{S}^{\N}]_{\t,\x}$} from \eqref{eq:sdeII1}-\eqref{eq:sdeII2}, due to the factor of {\small$\N^{-1/2}$}.
\item Consider the last term in \eqref{eq:sdebasic1}; extract from that each term in the formula \eqref{eq:msheII} for {\small$\mathfrak{e}$} which is not attached with a gradient. By Lemma \ref{lemma:classify}, every such term has the form {\small$\mathbf{1}_{\x\in\llbracket2,\infty\rrbracket}\N\mathfrak{q}[\tau_{\x}\bphi_{\t}]\mathbf{Z}^{\N}_{\t,\x}$} with a bulk-admissible {\small$\mathfrak{q}$}. So the contribution of each such term, which is {\small$[\mathbf{e}^{\mathfrak{t}_{\N}\mathscr{T}_{\N}}(\mathbf{1}_{\cdot\in\llbracket2,\infty\rrbracket}\N\mathfrak{q}[\tau_{\cdot}\bphi_{\t}]\mathbf{Z}^{\N}_{\t,\cdot})]_{\x}$}, is absorbed by the first term of \eqref{eq:sdepreI3}. (To be precise, we consider the collection {\small$\mathscr{Q}:=(\mathfrak{q}_{\x})_{\x\in\llbracket0,\infty\rrbracket}$} given by {\small$\mathfrak{q}_{\x}[\bphi]:=\mathbf{1}_{\x\in\llbracket2,\infty\rrbracket}\mathfrak{q}[\tau_{\x}\bphi]$}.)
\item Let us now consider again the last term in \eqref{eq:sdebasic1}, and extract from the formula \eqref{eq:msheII} for {\small$\mathfrak{e}$} the \abbr{RHS} of \eqref{eq:msheII(1)a}. We will combine this with \eqref{eq:sdebasic2} and \eqref{eq:sdebasic5}. Thus, the quantity that we must fit into \eqref{eq:sdepreI1}-\eqref{eq:sdepreI3} is
\begin{align*}
\boldsymbol{\Phi}:=[\mathbf{e}^{\mathfrak{t}_{\N}\mathscr{T}_{\N}}(\mathbf{1}_{\cdot\in\llbracket1,\infty\rrbracket}\cdot\tfrac12\lambda\N^{\frac32}\grad^{\mathbf{X}}_{1}(\mathscr{V}'[\bphi_{\t,\cdot}]\mathbf{Z}^{\N}_{\t,\cdot}))]_{\x}&-[\mathbf{e}^{\mathfrak{t}_{\N}\mathscr{T}_{\N}}(\mathbf{1}_{\cdot\in\llbracket1,\infty\rrbracket}\cdot\tfrac12\lambda\N^{\frac32}\grad^{\mathbf{X}}_{-1}(\mathscr{V}'[\bphi_{\t,\cdot+1}]\mathbf{Z}^{\N}_{\t,\cdot}))]_{\x}\\
&+[\mathbf{e}^{\mathfrak{t}_{\N}\mathscr{T}_{\N}}(\mathbf{1}_{\cdot=0}\cdot\{\lambda\N^{\frac32}\mathscr{V}'[\bphi_{\t,1}]\mathbf{Z}^{\N}_{\t,0}\})]_{\x}.
\end{align*}
We use Lemma \ref{lemma:sbp} to rewrite this as 
\begin{align}
\boldsymbol{\Phi}&=\sum_{\w\in\llbracket0,\infty\rrbracket}\N\grad^{\mathbf{X}}_{-1}\mathbf{H}^{\N,\mathbf{A}}_{0,\mathfrak{t}_{\N},\x,\w}\cdot\mathbf{1}_{\w\in\llbracket2,\infty\rrbracket}\tfrac12\lambda\N^{\frac12}\mathscr{V}'[\bphi_{\t,\w}]\mathbf{Z}^{\N}_{\t,\w}\label{eq:sdepreI1000}\\
&-\sum_{\w\in\llbracket0,\infty\rrbracket}\N\grad^{\mathbf{X}}_{1}\mathbf{H}^{\N,\mathbf{A}}_{0,\mathfrak{t}_{\N},\x,\w}\cdot\tfrac12\lambda\N^{\frac12}\mathscr{V}'[\bphi_{\t,\w+1}]\mathbf{Z}^{\N}_{\t,\w}\nonumber\\
&-\tfrac12\lambda\N^{\frac32}\mathbf{H}^{\N,\mathbf{A}}_{0,\mathfrak{t}_{\N},\x,1}\cdot\mathscr{V}'[\bphi_{\t,1}]\mathbf{Z}^{\N}_{\t,1}-\tfrac12\lambda\N^{\frac32}\mathbf{H}^{\N,\mathbf{A}}_{0,\mathfrak{t}_{\N},\x,0}\cdot\mathscr{V}'[\bphi_{\t,1}]\mathbf{Z}^{\N}_{\t,0}+\lambda\N^{\frac32}\mathbf{H}^{\N,\mathbf{A}}_{0,\mathfrak{t}_{\N},\x,0}\cdot\mathscr{V}'[\bphi_{\t,1}]\mathbf{Z}^{\N}_{\t,0}.\nonumber
\end{align}
By Definition \ref{definition:bulkadmissible} and Lemma \ref{lemma:classify}, both of {\small$\mathbf{1}_{\w\in\llbracket2,\infty\rrbracket}\N^{-1/2}\mathscr{V}'[\bphi_{\t,\w}]$} and {\small$\mathbf{1}_{\w\in\llbracket1,\infty\rrbracket}\N^{-1/2}\mathscr{V}'[\bphi_{\t,\w+1}]$} are bulk-admissible (see Notation \ref{notation:mshe} for the definition of the {\small$\mathscr{V}$} function). Therefore, the sums over {\small$\w\in\llbracket0,\infty\rrbracket$} are absorbed into the first term in \eqref{eq:sdepreI3}. Next, we further unfold the last term in the previous line. First, in the first term in the last line, we replace {\small$\mathbf{Z}^{\N}_{\t,1}$} by {\small$\mathbf{Z}^{\N}_{\t,0}$}. That is, we write
\begin{align*}
-\tfrac12\lambda\N^{\frac32}\mathbf{H}^{\N,\mathbf{A}}_{0,\mathfrak{t}_{\N},\x,1}\cdot\mathscr{V}'[\bphi_{\t,1}]\mathbf{Z}^{\N}_{\t,1}&=-\tfrac12\lambda\N^{\frac32}\mathbf{H}^{\N,\mathbf{A}}_{0,\mathfrak{t}_{\N},\x,1}\cdot\mathscr{V}'[\bphi_{\t,1}]\mathbf{Z}^{\N}_{\t,0}+\tfrac12\lambda\N^{\frac32}\mathbf{H}^{\N,\mathbf{A}}_{0,\mathfrak{t}_{\N},\x,1}\cdot\mathscr{V}'[\bphi_{\t,1}]\grad^{\mathbf{X}}_{-1}\mathbf{Z}^{\N}_{\t,1}.
\end{align*}
We plug this display into the last line of \eqref{eq:sdepreI1000}; observe that the first term on the \abbr{RHS} above cancels exactly with the last two terms in \eqref{eq:sdepreI1000}. (This is the Neumann cancellation from Section \ref{section:method}.) Thus, we are left with the last term in the previous display. To this end, we use the gradient formula from the proof of Lemma \ref{lemma:msheinterior2} to show {\small$\tfrac12\lambda\N^{\frac32}\mathbf{H}^{\N,\mathbf{A}}_{0,\mathfrak{t}_{\N},\x,1}\cdot\mathscr{V}'[\bphi_{\t,1}]\grad^{\mathbf{X}}_{-1}\mathbf{Z}^{\N}_{\t,1}=-\mathbf{H}^{\N,\mathbf{A}}_{0,\mathfrak{t}_{\N},\x,1}\cdot\frac12\lambda^{2}\N\mathscr{V}'[\bphi_{\t,1}]\bphi_{\t,1}\mathbf{Z}^{\N}_{\t,0}+\mathbf{H}^{\N,\mathbf{A}}_{0,\mathfrak{t}_{\N},\x,1}\cdot\N^{\frac12}\mathfrak{f}[\bphi_{\t}]$}, in which {\small$\mathfrak{f}[\bphi]$} satisfies the deterministic estimate \eqref{eq:bulkadmissibleI}. By Lemma \ref{lemma:classifyedge}, the first term on the \abbr{RHS} of this identity may be absorbed by the last term in \eqref{eq:sdepreI3}. Moreover, the last term in this identity is absorbed by the {\small$\mathrm{Err}$}-term in \eqref{eq:sdepreI1}.
\item Now take the other gradient terms in the formula \eqref{eq:msheII} for {\small$\mathfrak{e}$}, starting with terms that have the following form of {\small$\mathbf{1}_{\x\in\llbracket2,\infty\rrbracket}\cdot\N\grad^{\mathbf{X}}_{1}(\mathfrak{q}[\tau_{\x}\bphi_{\t}]\mathbf{Z}^{\N}_{\t,\x})$}, where by Lemma \ref{lemma:classify}, we know that {\small$\mathfrak{q}$} is bulk-admissible. For these terms, the contribution is{\small$[\mathbf{e}^{\mathfrak{t}_{\N}\mathscr{T}_{\N}}(\mathbf{1}_{\cdot\in\llbracket2,\infty\rrbracket}\N\grad^{\mathbf{X}}_{1}(\mathfrak{q}[\tau_{\cdot}\bphi_{\t}]\mathbf{Z}^{\N}_{\t,\cdot})]$}, for which we use a summation-by-parts argument as in Lemma \ref{lemma:sbp}. Ultimately, this produces two terms, one of which moves the gradient onto the {\small$\mathbf{H}^{\N,\mathbf{A}}$}-kernel (and can thus be absorbed by the first term in \eqref{eq:sdepreI3}), and a boundary term that can be absorbed by the second term in \eqref{eq:sdepreI3}. We omit these details (which have already been illustrated). 
\item The only other terms that are left to deal with have the form {\small$\mathbf{1}_{\x\in\llbracket2,\infty\rrbracket}\cdot\N^{1/2}\grad^{\mathbf{X}}_{1}(\mathfrak{f}[\tau_{\x}\bphi_{\t}]\mathbf{Z}^{\N}_{\t,\x})$}, where {\small$\mathfrak{f}$} satisfies the a priori estimate \eqref{eq:msheIIff}. For these terms, the same summation-by-parts argument shows that their contributions {\small$[\mathbf{e}^{\mathfrak{t}_{\N}\mathscr{T}_{\N}}(\mathbf{1}_{\cdot\in\llbracket2,\infty\rrbracket}\N^{1/2}\grad^{\mathbf{X}}_{1}(\mathfrak{f}[\tau_{\cdot}\bphi_{\t}]\mathbf{Z}^{\N}_{\t,\cdot})]$} can be absorbed by the error term in \eqref{eq:sdepreI1}, since the scaling with respect to {\small$\N$} is too weak to persist in the large-{\small$\N$} limit. As before, we omit these calculations.
\end{enumerate} 
This completes the proof.
\end{proof}
Equipped with Lemma \ref{lemma:sdepre}, let us now introduce space-time-averages and time-averages for the quantities in \eqref{eq:sdepreI3}. We start with spatial-averages for the first set of summations in \eqref{eq:sdepreI3}.
\begin{lemma}\label{lemma:xav}
\fsp Consider any {\small$\ell\in\llbracket1,\mathrm{L}_{1}\rrbracket$} from \eqref{eq:sdepreI3} and retain the notation therein. We have
\begin{align}
\sum_{\w\in\llbracket0,\infty\rrbracket}\boldsymbol{\mathcal{S}}^{(\ell)}_{\x,\w}\cdot\N\mathfrak{q}_{\ell,\w}[\bphi_{\t}]\mathbf{Z}^{\N}_{\t,\w}&=\sum_{\w\in\llbracket0,\infty\rrbracket}\boldsymbol{\mathcal{K}}^{(\ell)}_{\x,\w}\cdot\N\boldsymbol{\Upsilon}^{\mathbf{X},\mathscr{Q}_{\ell}}_{\t,\w}+\boldsymbol{\mathcal{E}}_{\t,\x}.\label{eq:xavI}
\end{align}
Here, the kernels {\small$\boldsymbol{\mathcal{K}}^{(\ell)}$} satisfy the estimate \eqref{eq:sdeIII}. The quantity {\small$\boldsymbol{\Upsilon}^{\mathbf{X},\mathscr{Q}_{\ell}}$} is defined in \eqref{eq:bulkaverageIV} with {\small$\mathscr{Q}_{\ell}=(\mathfrak{q}_{\ell,\x})_{\x\in\llbracket0,\infty\rrbracket}$}. Finally, the term {\small$\boldsymbol{\mathcal{E}}_{\t,\x}$} has the form of {\small$\mathrm{Err}[\mathbf{R}^{\N}\mathbf{S}^{\N}]_{\t,\x}$} from \eqref{eq:sdeII1}-\eqref{eq:sdeII2}.
\end{lemma}
\begin{proof}
With notation to be recalled afterwards, we first have the identity
\begin{align}
\N\mathfrak{q}_{\ell,\w}[\bphi_{\t}]\mathbf{Z}^{\N}_{\t,\w}&=[\mathbf{e}^{\tau_{\mathbf{Av}}\mathscr{T}_{\N,0}}(\N\mathfrak{q}_{\ell,\cdot}[\bphi_{\t}]\mathbf{Z}^{\N}_{\t,\cdot})]_{\w}+[(\mathrm{Id}-\mathbf{e}^{\tau_{\mathbf{Av}}\mathscr{T}_{\N,0}})(\N\mathfrak{q}_{\ell,\cdot}[\bphi_{\t}]\mathbf{Z}^{\N}_{\t,\cdot})]_{\w}.\label{eq:xav1}
\end{align}
Above, {\small$\tau_{\mathbf{Av}}:=\N^{-2}\mathfrak{l}_{\mathbf{Av}}^{2}$} and {\small$\mathfrak{l}_{\mathbf{Av}}$} are from Definition \ref{definition:bulkaverage}. We iterate this identity in the last term, so for any integer {\small$\mathrm{M}>0$}, we have the following, in which we set {\small$\mathscr{X}:=\mathrm{Id}-\mathbf{e}^{\tau_{\mathbf{Av}}\mathscr{T}_{\N,0}}$} for convenience (just for this proof):
\begin{align}
\N\mathfrak{q}_{\ell,\w}[\bphi_{\t}]\mathbf{Z}^{\N}_{\t,\w}&=\sum_{\m\in\llbracket0,\mathrm{M}\rrbracket}[\mathscr{X}^{\m}\mathbf{e}^{\tau_{\mathbf{Av}}\mathscr{T}_{\N,0}}(\N\mathfrak{q}_{\ell,\cdot}[\bphi_{\t}]\mathbf{Z}^{\N}_{\t,\cdot})]_{\w}+[\mathscr{X}^{\mathrm{M}+1}(\N\mathfrak{q}_{\ell,\cdot}[\bphi_{\t}]\mathbf{Z}^{\N}_{\t,\cdot})]_{\w}.\label{eq:xav2}
\end{align}
We multiply both sides by {\small$\boldsymbol{\mathcal{S}}^{(\ell)}_{\x,\w}$} and sum over {\small$\w\in\llbracket0,\infty\rrbracket$}. We first track the contribution of the last term in \eqref{eq:xav2}. Note that the {\small$\mathscr{X}$}-operator above is self-adjoint on {\small$\llbracket0,\infty\rrbracket$} with respect to the flat measure, as the {\small$\mathscr{T}_{\N,0}$}-semigroup is self-adjoint. Thus, we get 
\begin{align}
\sum_{\w\in\llbracket0,\infty\rrbracket}\boldsymbol{\mathcal{S}}^{(\ell)}_{\x,\w}\cdot[\mathscr{X}^{\mathrm{M}+1}(\N\mathfrak{q}_{\ell,\cdot}[\bphi_{\t}]\mathbf{Z}^{\N}_{\t,\cdot})]_{\w}=\sum_{\w\in\llbracket0,\infty\rrbracket}\mathscr{X}^{\mathrm{M}+1}\boldsymbol{\mathcal{S}}^{(\ell)}_{\x,\w}\cdot\N\mathfrak{q}_{\ell,\w}[\bphi_{\t}]\mathbf{Z}^{\N}_{\t,\w},\label{eq:xav2a}
\end{align}
where {\small$\mathscr{X}$} on the \abbr{RHS} acts on the {\small$\w$}-variable. Next, we claim that for any {\small$\mathrm{D}>0$}, we can take {\small$\mathrm{M}>0$} large enough so that for any {\small$\kappa>0$}, we have the following pointwise estimate:
\begin{align*}
|\mathscr{X}^{\mathrm{M}+1}\boldsymbol{\mathcal{S}}^{(\ell)}_{\x,\w}|\lesssim_{\kappa,\mathrm{D}}\N^{-\mathrm{D}}\exp(-\tfrac{\kappa|\x-\w|}{\N}).
\end{align*}
A proof is given in Lemma \ref{lemma:kernelestimate}, but intuitively, the kernel {\small$\boldsymbol{\mathcal{S}}^{(\ell)}$} has spatial-regularity with respect to the length-scale {\small$\N^{1-\delta_{\mathbf{S}}}$} by \eqref{eq:sdepreI}, and each {\small$\mathscr{X}$} above is essentially an average of gradients at the much smaller scale {\small$\lesssim\mathfrak{l}_{\mathbf{Av}}=\N^{1-3\delta_{\mathbf{S}}/2}$}. In any case, we deduce that the sum in \eqref{eq:xav2a} has the form of {\small$\mathrm{Err}[\mathbf{R}^{\N}\mathbf{S}^{\N}]_{\t,\x}$} from \eqref{eq:sdeII1}-\eqref{eq:sdeII2}.

We now track the contribution of the first term on the \abbr{RHS} of \eqref{eq:xav2} after multiplying by {\small$\boldsymbol{\mathcal{S}}^{(\ell)}_{\x,\w}$} and summing over {\small$\w\in\llbracket0,\infty\rrbracket$}. Again by the self-adjoint property of the {\small$\mathscr{X}$}-operator, we have (for all {\small$\m\in\llbracket0,\mathrm{M}\rrbracket$})
\begin{align}
\sum_{\w\in\llbracket0,\infty\rrbracket}\boldsymbol{\mathcal{S}}^{(\ell)}_{\x,\w}\cdot[\mathscr{X}^{\m}\mathbf{e}^{\tau_{\mathbf{Av}}\mathscr{T}_{\N,0}}(\N\mathfrak{q}_{\ell,\cdot}[\bphi_{\t}]\mathbf{Z}^{\N}_{\t,\cdot})]_{\w}&=\sum_{\w\in\llbracket0,\infty\rrbracket}\mathscr{X}^{\m}\boldsymbol{\mathcal{S}}^{(\ell)}_{\x,\w}\cdot[\mathbf{e}^{\tau_{\mathbf{Av}}\mathscr{T}_{\N,0}}(\N\mathfrak{q}_{\ell,\cdot}[\bphi_{\t}]\mathbf{Z}^{\N}_{\t,\cdot})]_{\w}. \nonumber
\end{align}
The {\small$\mathscr{X}$}-operator on the \abbr{RHS} again acts {\small$\w$}. Let us take {\small$\boldsymbol{\mathcal{K}}^{(\ell)}_{\x,\w}$} in the statement of the lemma to be
\begin{align*}
\boldsymbol{\mathcal{K}}^{(\ell)}_{\x,\w}:=\sum_{\m\in\llbracket0,\mathrm{M}\rrbracket}\mathscr{X}^{\m}\boldsymbol{\mathcal{S}}^{(\ell)}_{\x,\w}.
\end{align*}
The necessary estimate in \eqref{eq:sdeIII} for {\small$\boldsymbol{\mathcal{K}}^{(\ell)}$} follows from Lemma \ref{lemma:kernelestimate} as well. Next, recall \eqref{eq:bulkaverageIV}; write
\begin{align*}
[\mathbf{e}^{\tau_{\mathbf{Av}}\mathscr{T}_{\N,0}}(\N\mathfrak{q}_{\ell,\cdot}[\bphi_{\t}]\mathbf{Z}^{\N}_{\t,\cdot})]_{\w}&=\N\cdot\boldsymbol{\Upsilon}^{\mathbf{X},\mathscr{Q}_{\ell}}_{\t,\w}+[\mathbf{e}^{\tau_{\mathbf{Av}}\mathscr{T}_{\N,0}}(\mathbf{1}_{|\w-\cdot|\gtrsim(\log\N)^{2}\mathfrak{l}_{\mathbf{Av}}}\cdot\N\mathfrak{q}_{\ell,\cdot}[\bphi_{\t}]\mathbf{Z}^{\N}_{\t,\cdot})]_{\w}.
\end{align*}
If we use the previous five displays (starting with \eqref{eq:xav2}), then to complete the proof, it suffices to show that 
\begin{align*}
\sum_{\w\in\llbracket0,\infty\rrbracket}\boldsymbol{\mathcal{K}}^{(\ell)}_{\x,\w}\cdot[\mathbf{e}^{\tau_{\mathbf{Av}}\mathscr{T}_{\N,0}}(\mathbf{1}_{|\w-\cdot|\gtrsim(\log\N)^{2}\mathfrak{l}_{\mathbf{Av}}}\cdot\N\mathfrak{q}_{\ell,\cdot}[\bphi_{\t}]\mathbf{Z}^{\N}_{\t,\cdot})]_{\w}
\end{align*}
has the form of {\small$\mathrm{Err}[\mathbf{R}^{\N}\mathbf{S}^{\N}]_{\t,\x}$} from \eqref{eq:sdeII1}-\eqref{eq:sdeII2}. To this end, we unfold the previous display as 
\begin{align*}
&\sum_{\w\in\llbracket0,\infty\rrbracket}\boldsymbol{\mathcal{K}}^{(\ell)}_{\x,\w}\cdot\sum_{\y\in\llbracket0,\infty\rrbracket}\mathbf{H}^{\N,0}_{0,\tau_{\mathbf{Av}},\w,\y}\cdot\mathbf{1}_{|\w-\y|\gtrsim(\log\N)^{2}\mathfrak{l}_{\mathbf{Av}}}\cdot\N\mathfrak{q}_{\ell,\y}[\bphi_{\t}]\mathbf{Z}^{\N}_{\t,\y}\\
&=\sum_{\y\in\llbracket0,\infty\rrbracket}\Big(\sum_{\w\in\llbracket0,\infty\rrbracket}\boldsymbol{\mathcal{K}}^{(\ell)}_{\x,\w}\cdot\mathbf{H}^{\N,0}_{0,\tau_{\mathbf{Av}},\w,\y}\cdot\mathbf{1}_{|\w-\y|\gtrsim(\log\N)^{2}\mathfrak{l}_{\mathbf{Av}}}\Big)\cdot\N\mathfrak{q}_{\ell,\y}[\bphi_{\t}]\mathbf{Z}^{\N}_{\t,\y}.
\end{align*}
By the estimates \eqref{eq:sdeIII} and off-diagonal heat kernel estimates in Proposition \ref{prop:hkestimates}, the term in parentheses above is {\small$\lesssim_{\mathrm{D},\kappa}\N^{-\mathrm{D}}\exp(-\frac{\kappa|\w-\y|}{\N})$} for any {\small$\mathrm{D},\kappa>0$}. Indeed, the requirement for {\small$\w,\y$} to be separated by {\small$\gtrsim(\log\N)^{2}\mathfrak{l}_{\mathbf{Av}}$} forces the time-scale {\small$\tau_{\mathbf{Av}}=\N^{-2}\mathfrak{l}_{\mathbf{Av}}^{2}$} heat kernel to have beyond-polynomial decay in {\small$\N$}. Finally, if we now use the bulk-admissible estimate in \eqref{eq:bulkadmissibleI} for {\small$\mathfrak{q}_{\ell,\cdot}$}-functions and {\small$\mathbf{Z}^{\N}=\mathbf{R}^{\N}\mathbf{S}^{\N}$}, then the previous display is of the form {\small$\mathrm{Err}[\mathbf{R}^{\N}\mathbf{S}^{\N}]_{\t,\x}$} from \eqref{eq:sdeII1}-\eqref{eq:sdeII2}. As stated above, this completes the proof.
\end{proof}
We will now replace both the first term on the \abbr{RHS} of \eqref{eq:xavI} and each {\small$\w$}-summation in the last term in \eqref{eq:sdepreI3} by time-averages on appropriate time-scales.
\begin{lemma}\label{lemma:tav}
\fsp Retain the notation of Proposition \ref{prop:sde} as well as that of Lemmas \ref{lemma:sdepre} and \ref{lemma:xav}. For any index {\small$\ell\in\llbracket1,\mathrm{L}_{1}\rrbracket$}, as well as any points {\small$\x,\w\in\llbracket0,\infty\rrbracket$}, we have 
\begin{align}
\boldsymbol{\mathcal{K}}^{(\ell)}_{\x,\w}\cdot\N\cdot\boldsymbol{\Upsilon}^{\mathbf{X},\mathscr{Q}_{\ell}}_{\t,\w}&=\boldsymbol{\mathcal{K}}^{(\ell)}_{\x,\w}\cdot\N\cdot\mathbf{Av}^{\mathbf{T},\mathbf{X},\mathscr{Q}_{\ell}}_{\t,\w}\mathbf{R}^{\N}_{\t,\w_{\pm}}\mathbf{S}^{\N}_{\t,\w_{\pm}}+\grad^{\mathbf{T},\mathrm{av}}_{\mathfrak{t}_{\mathbf{Av}}}(\boldsymbol{\mathcal{K}}^{(\ell)}_{\x,\w}\cdot\N\cdot\mathbf{Av}^{\mathbf{X},\mathscr{Q}_{\ell}}_{\t,\w}\mathbf{R}^{\N}_{\t,\w_{\pm}}\mathbf{S}^{\N}_{\t,\w_{\pm}}). \label{eq:tavI}
\end{align}
Moreover, given any {\small$\ell\in\llbracket1,\mathrm{L}_{2}\rrbracket$}, we also have 
\begin{align}
\boldsymbol{\mathcal{K}}^{\partial,\ell}_{\x,\w}\cdot\N\mathfrak{b}_{\ell}[\bphi_{\t}]\mathbf{Z}^{\N}_{\t,\w}&=\boldsymbol{\mathcal{K}}^{\partial,\ell}_{\x,\w}\cdot\N\mathbf{Av}^{\mathbf{T},\mathfrak{b}_{\ell}}_{\t,\w}\mathbf{R}^{\N}_{\t,\w}\mathbf{S}^{\N}_{\t,\w}+\grad^{\mathbf{T},\mathrm{av}}_{\mathfrak{t}_{\partial}}(\boldsymbol{\mathcal{K}}^{\partial,\ell}_{\x,\w}\cdot\N\mathfrak{b}_{\ell}[\bphi_{\t}]\mathbf{R}^{\N}_{\t,\w}\mathbf{S}^{\N}_{\t,\w}).\label{eq:tavII}
\end{align}
\end{lemma}
\begin{proof}
We first observe that by construction in Definitions \ref{definition:bulkaverage} and \ref{definition:grad}, we have
\begin{align}
\boldsymbol{\Upsilon}^{\mathbf{X},\mathscr{Q}_{\ell}}_{\t,\w}&=\boldsymbol{\Upsilon}^{\mathbf{T},\mathbf{X},\mathscr{Q}_{\ell}}_{\t,\w}+\grad^{\mathbf{T},\mathrm{av}}_{\mathfrak{t}_{\mathbf{Av}}}(\boldsymbol{\Upsilon}^{\mathbf{X},\mathscr{Q}_{\ell}}_{\t,\w}).
\end{align}
We multiply by {\small$\boldsymbol{\mathcal{K}}^{(\ell)}_{\x,\w}$} on both sides, pull it through the time-gradient on the \abbr{RHS} (since it is independent of time), then use \eqref{eq:bulkaverageII}-\eqref{eq:bulkaverageIII} to rewrite {\small$\boldsymbol{\Upsilon}^{\mathbf{T},\mathbf{X},\mathscr{Q}_{\ell}}_{\t,\w}=\mathbf{Av}^{\mathbf{T},\mathbf{X},\mathscr{Q}_{\ell}}_{\t,\w}\mathbf{Z}^{\N}_{\t,\w_{\pm}}$} and {\small$\boldsymbol{\Upsilon}^{\mathbf{X},\mathscr{Q}_{\ell}}_{\t,\w}=\mathbf{Av}^{\mathbf{X},\mathscr{Q}_{\ell}}_{\t,\w}\mathbf{Z}^{\N}_{\t,\w_{\pm}}$}. Using {\small$\mathbf{Z}^{\N}=\mathbf{R}^{\N}\mathbf{S}^{\N}$} then yields \eqref{eq:tavI}. The proof of \eqref{eq:tavII} is identical.
\end{proof}
We are now in position to prove the main result, Proposition \ref{prop:sde}, of this section.
\begin{proof}[Proof of Proposition \ref{prop:sde}]
Take \eqref{eq:sdepreI1}-\eqref{eq:sdepreI3}. We use Lemmas \ref{lemma:xav} and \ref{lemma:tav} to rewrite the first term in \eqref{eq:sdepreI3} in the form of \eqref{eq:sdeI3}-\eqref{eq:sdeI4}. We use Lemma \ref{lemma:tav} to rewrite the second term in \eqref{eq:sdepreI3} in the form of \eqref{eq:sdeI5}. This completes the proof.
\end{proof}
%
%
%
\section{Stochastic heat kernels}\label{section:shk}
We now view the object {\small$\mathbf{S}^{\N}$} not from the perspective of \eqref{eq:ch} and \eqref{eq:heatIII}, but as the solution to the linear equation \eqref{eq:sdeI1}-\eqref{eq:sdeI5}. To this end, we will consider a more general family of related \abbr{SDE}s parameterized by {\small$\zeta>0$}, whose solutions are denoted by {\small$\mathbf{S}^{\N,\zeta}$}. These \abbr{SDE}s are essentially copies of \eqref{eq:sdeI1}-\eqref{eq:sdeI5}, but with adjustments that are marked in blue font (for clarity), provide useful a priori estimates, introduce the {\small$\zeta$}-cutoff-scale {\color{black}as discussed in Section \ref{section:method}}, and will be explained immediately afterwards:
\begin{align}
\d\mathbf{S}^{\N,\zeta}_{\t,\x}&=\mathscr{T}_{\N}\mathbf{S}^{\N,\zeta}_{\t,\x}\d\t+[\mathbf{e}^{\mathfrak{t}_{\N}\mathscr{T}_{\N}}(\mathbf{1}^{(\x)}_{\cdot}{\color{blue}\boldsymbol{\chi}^{(\zeta_{\mathrm{large}})}_{\cdot}}\cdot\sqrt{2}\lambda\N^{\frac12}{\color{blue}\mathbf{R}^{\N,\wedge}_{\t,\cdot}}\mathbf{S}^{\N,\zeta}_{\t,\cdot}\d\mathbf{b}_{\t,\cdot})]_{\x}+\mathrm{Err}[{\color{blue}\boldsymbol{\chi}^{(\zeta_{})}_{}\mathbf{R}^{\N,\wedge}}\mathbf{S}^{\N,\zeta}]_{\t,\x}\d\t\label{eq:sdezetaI1}\\
&+\sum_{\ell\in\llbracket1,\mathrm{L}_{1}\rrbracket}{}\sum_{\w\in\llbracket0,\infty\rrbracket}\boldsymbol{\mathcal{K}}^{(\ell)}_{\x,\w}\cdot{\color{blue}\boldsymbol{\chi}^{(\zeta_{})}_{\w}}\cdot\N\cdot\mathbf{Av}^{\mathbf{T},\mathbf{X},\mathscr{Q}_{\ell}}_{\t,\w}{\color{blue}\mathbf{R}^{\N,\wedge}_{\t,\w_{\pm}}}\mathbf{S}^{\N,\zeta}_{\t,\w_{\pm}}\d\t\label{eq:sdezetaI3}\\
&+\sum_{\ell\in\llbracket1,\mathrm{L}_{1}\rrbracket}{}\grad^{\mathbf{T},\mathrm{av}}_{\mathfrak{t}_{\mathbf{Av}}}\Big(\sum_{\w\in\llbracket0,\infty\rrbracket}\boldsymbol{\mathcal{K}}^{(\ell)}_{\x,\w}\cdot{\color{blue}\boldsymbol{\chi}^{(\zeta_{})}_{\w}}\cdot\N\cdot\mathbf{Av}^{\mathbf{X},\mathscr{Q}_{\ell}}_{\t,\w}{\color{blue}\mathbf{R}^{\N,\wedge}_{\t,\w_{\pm}}}\mathbf{S}^{\N,\zeta}_{\t,\w_{\pm}}\Big)\d\t\label{eq:sdezetaI4}\\
&+\sum_{\ell\in\llbracket1,\mathrm{L}_{2}\rrbracket}\sum_{\w\in\llbracket0,\infty\rrbracket}\boldsymbol{\mathcal{K}}^{\partial,\ell}_{\x,\w}\cdot\N\mathbf{Av}^{\mathbf{T},\mathfrak{b}_{\ell}}_{\t,\w}{\color{blue}\mathbf{R}^{\N,\wedge}_{\t,\w}}\mathbf{S}^{\N,\zeta}_{\t,\w}\d\t\label{eq:sdezetaI5}\\
&+\sum_{\ell\in\llbracket1,\mathrm{L}_{2}\rrbracket}\grad^{\mathbf{T},\mathrm{av}}_{\mathfrak{t}_{\partial}}\Big(\sum_{\w\in\llbracket0,\infty\rrbracket}\boldsymbol{\mathcal{K}}^{\partial,\ell}_{\x,\w}\cdot\N\mathfrak{b}_{\ell}[\bphi_{\t}]{\color{blue}\mathbf{R}^{\N,\wedge}_{\t,\w}}\mathbf{S}^{\N,\zeta}_{\t,\w}\Big)\d\t.\label{eq:sdezetaI6}
\end{align}
The initial data for {\small$\mathbf{S}^{\N,\zeta}$} is given by that of {\small$\mathbf{S}^{\N}$} (which is fixed by \eqref{eq:heatIII}). Well-posedness of this \abbr{SDE} follows from a Picard iteration, which converges because \eqref{eq:sdezetaI1}-\eqref{eq:sdezetaI6} is a linear equation whose coefficients have all moments bounded with respect to the measure {\small$\mathbb{P}^{0}$} from Definition \ref{definition:gc} (see Assumption \ref{assump:potential}, which yields sub-Gaussianity of the {\small$\bphi_{\x}$}-spins under {\small$\mathbb{P}^{0}$}). This gives well-posedness almost surely given {\small$\mathbb{P}^{0}$}-initial data for \eqref{eq:phiI}-\eqref{eq:phiII}, at which point well-posedness for general initial data follows by the absolute continuity estimate \eqref{eq:dataI}.

Let us explain the notation above. The function {\small$\boldsymbol{\chi}^{(\zeta)}:\llbracket0,\infty\rrbracket\to\R$} satisfies the following assumptions:
\begin{enumerate}
\item Uniform smoothness: for any {\small$\a\geq0$} and {\small$\gamma_{1},\ldots,\gamma_{\a}\in\{1,-1\}$}, we have {\small$\|(\prod_{\ell=1}^{\a}\N\grad^{\mathbf{X}}_{\gamma_{\ell}})\boldsymbol{\chi}^{(\zeta)}_{\cdot}\|_{\mathrm{L}^{\infty}(\llbracket0,\infty\rrbracket)}\lesssim_{\a}1$}.
\item Support: we have {\small$\boldsymbol{\chi}^{(\zeta)}_{\w}=1$} for all {\small$\w\in\llbracket0,\N^{1+\zeta}\rrbracket$} and {\small$\boldsymbol{\chi}^{(\zeta)}_{\w}=0$} for all {\small$\w\in\llbracket2\N^{1+\zeta},\infty\rrbracket$}.
\end{enumerate}
In words, the function {\small$\boldsymbol{\chi}^{(\zeta)}$} is a smooth version of the indicator function of {\small$\llbracket0,\N^{1+\zeta}\rrbracket$}. (We clarify that there is no cutoff function hitting the {\small$\boldsymbol{\mathcal{K}}^{\partial,\ell}$} because, by Proposition \ref{prop:sde}, it is already supported near the edge {\small$0\in\llbracket0,\infty\rrbracket$} in the forwards spatial variable.) Moreover, the parameter {\small$\zeta_{\mathrm{large}}>0$} is a large but {\small$\mathrm{O}(1)$} constant. In \eqref{eq:sdezetaI3}, the term {\small$\mathrm{Err}[{\color{blue}\boldsymbol{\chi}^{(\zeta)}}\mathbf{R}^{\N,\wedge}\mathbf{S}^{\N,\zeta}]_{\t,\x}$} is given by \eqref{eq:sdeII1}-\eqref{eq:sdeII2} after replacing {\small$\mathbf{S}^{\N}$} therein by {\small${\color{blue}\boldsymbol{\chi}^{(\zeta)}}\mathbf{S}^{\N,\zeta}$} and {\small$\mathbf{R}^{\N}$} by {\small$\mathbf{R}^{\N,\wedge}$}. Finally, recalling {\small$\delta_{\mathbf{S}},\mathbf{R}^{\N}$} from Definition \ref{definition:heat}, we introduced
\begin{align}
\mathbf{R}^{\N,\wedge}_{\t,\x}:=\mathbf{R}^{\N}_{\t,\x}\cdot\mathbf{1}(|\mathbf{R}^{\N}_{\t,\x}-1|\lesssim\N^{-\frac16\delta_{\mathbf{S}}}),\label{eq:rwedge}
\end{align}
which, as discussed after Definition \ref{definition:heat}, will be later shown to equal {\small$\mathbf{R}^{\N}_{\t,\x}$} in a sufficiently strong sense. That is, we have the following result, {whose proof is delegated to Section \ref{section:stochI} (which contains a host of other preliminary stochastic estimates)}. Before we state the result, however, we declare the following notion of high probability events in this paper.
\begin{definition}\label{definition:hp}
\fsp We say an event {\small$\mathscr{E}$} holds with high probability if {\small$\mathbb{P}(\mathscr{E})=1-\mathrm{o}(1)$}.
\end{definition}
\begin{lemma}\label{lemma:rwedge}
\fsp Fix any {\small$\mathrm{L}>0$}. With high probability, we have {\small$\mathbf{R}^{\N,\wedge}_{\t,\x}=\mathbf{R}^{\N}_{\t,\x}$} for all {\small$\t\in[0,1]$} and {\small$\x\in\llbracket0,\mathrm{L}\N\rrbracket$}.
\end{lemma}
In words, the \abbr{SDE} \eqref{eq:sdezetaI1}-\eqref{eq:sdezetaI6} is a spatially cutoff modification of \eqref{eq:sdeI1}-\eqref{eq:sdeI5} with cosmetic adjustments. A key feature about \eqref{eq:sdezetaI1}-\eqref{eq:sdezetaI6} is that it cuts off (in space) the error terms and nonlinearities that we need to control. {\color{black}As discussed in Section \ref{section:method}, our aim is to compare {\small$\mathbf{S}^{\N}$} to {\small$\mathbf{S}^{\N,\zeta}$} for {\small$\zeta>0$} small, starting with a comparison for {\small$\zeta>0$} large.} This first step is the content of the following, whose proof is a standard consequence of the spatial locality, polynomial-speed, and well-posedness of \eqref{eq:sdezetaI1}-\eqref{eq:sdezetaI6}.
\begin{lemma}\label{lemma:largezeta}
\fsp Fix any {\small$\mathrm{L}>0$}. For large enough {\small$0<\zeta,\zeta_{\mathrm{large}}=\mathrm{O}(1)$}, we have the following with high probability:
\begin{align}
\sup_{\t\in[0,1]}\sup_{\x\in\llbracket0,\mathrm{L}\N\rrbracket}|\mathbf{S}^{\N}_{\t,\x}-\mathbf{S}^{\N,\zeta}_{\t,\x}|\lesssim_{\mathrm{L}}\N^{-\gamma_{\mathrm{L}}}, \quad\text{where } \gamma_{\mathrm{L}}>0.\nonumber
\end{align}
\end{lemma}
In order to reduce the parameter {\small$\zeta>0$}, we will proceed by analyzing the heat kernel associated to \eqref{eq:sdezetaI1}-\eqref{eq:sdezetaI6}. In particular, we will let {\small$\mathbf{K}^{\N,\zeta}_{\s,\t,\x,\y}$} be a function of {\small$(\s,\t,\x,\y)\in[0,\infty)^{2}\times\llbracket0,\infty\rrbracket^{2}$}, restricted to {\small$\s\leq\t$}, which satisfies the following \abbr{SDE} in {\small$(\t,\x)$}-variables (so that {\small$(\s,\y)$} are fixed) with initial data {\small$\mathbf{K}^{\N,\zeta}_{\s,\s,\x,\y}=\mathbf{1}_{\x=\y}$}:
\begin{align}
\d\mathbf{K}^{\N,\zeta}_{\s,\t,\x,\y}&=\mathscr{T}_{\N}\mathbf{K}^{\N,\zeta}_{\s,\t,\x,\y}\d\t+[\mathbf{e}^{\mathfrak{t}_{\N}\mathscr{T}_{\N}}(\mathbf{1}^{(\x)}_{\cdot}{\color{blue}\boldsymbol{\chi}^{(\zeta_{\mathrm{large}})}_{\cdot}}\sqrt{2}\lambda\N^{\frac12}{\color{blue}\mathbf{R}^{\N,\wedge}_{\t,\cdot}}\mathbf{K}^{\N,\zeta}_{\s,\t,\cdot,\y}\d\mathbf{b}_{\t,\cdot})]_{\x}+\mathrm{Err}[{\color{blue}\boldsymbol{\chi}^{(\zeta_{})}_{}\mathbf{R}^{\N,\wedge}}\mathbf{K}^{\N,\zeta}_{\s,\cdot,\cdot,\y}]_{\t,\x}\d\t\label{eq:hkzetaI1}\\
&+\sum_{\ell\in\llbracket1,\mathrm{L}_{1}\rrbracket}{}\sum_{\w\in\llbracket0,\infty\rrbracket}\boldsymbol{\mathcal{K}}^{(\ell)}_{\x,\w}\cdot{\color{blue}\boldsymbol{\chi}^{(\zeta_{})}_{\w}}\cdot\N\cdot\mathbf{Av}^{\mathbf{T},\mathbf{X},\mathscr{Q}_{\ell}}_{\t,\w}{\color{blue}\mathbf{R}^{\N,\wedge}_{\t,\w_{\pm}}}\mathbf{K}^{\N,\zeta}_{\s,\t,\w_{\pm},\y}\d\t\label{eq:hkzetaI3}\\
&+\sum_{\ell\in\llbracket1,\mathrm{L}_{1}\rrbracket}{}\grad^{\mathbf{T},\mathrm{av}}_{\mathfrak{t}_{\mathbf{Av}}}\Big(\sum_{\w\in\llbracket0,\infty\rrbracket}\boldsymbol{\mathcal{K}}^{(\ell)}_{\x,\w}\cdot{\color{blue}\boldsymbol{\chi}^{(\zeta_{})}_{\w}}\cdot\N\cdot\mathbf{Av}^{\mathbf{X},\mathscr{Q}_{\ell}}_{\t,\w}{\color{blue}\mathbf{R}^{\N,\wedge}_{\t,\w_{\pm}}}\mathbf{K}^{\N,\zeta}_{\s,\t,\w_{\pm},\y}\Big)\d\t\label{eq:hkzetaI4}\\
&+\sum_{\ell\in\llbracket1,\mathrm{L}_{2}\rrbracket}\sum_{\w\in\llbracket0,\infty\rrbracket}\boldsymbol{\mathcal{K}}^{\partial,\ell}_{\x,\w}\cdot\N\mathbf{Av}^{\mathbf{T},\mathfrak{b}_{\ell}}_{\t,\w}{\color{blue}\mathbf{R}^{\N,\wedge}_{\t,\w}}\mathbf{K}^{\N,\zeta}_{\s,\t,\w,\y}\d\t\label{eq:hkzetaI5}\\
&+\sum_{\ell\in\llbracket1,\mathrm{L}_{2}\rrbracket}\grad^{\mathbf{T},\mathrm{av}}_{\mathfrak{t}_{\partial}}\Big(\sum_{\w\in\llbracket0,\infty\rrbracket}\boldsymbol{\mathcal{K}}^{\partial,\ell}_{\x,\w}\cdot\N\mathfrak{b}_{\ell}[\bphi_{\t}]{\color{blue}\mathbf{R}^{\N,\wedge}_{\t,\w}}\mathbf{K}^{\N,\zeta}_{\s,\t,\w,\y}\Big)\d\t.\label{eq:hkzetaI6}
\end{align}

In particular, the stochastic heat kernel {\small$\mathbf{K}^{\N,\zeta}$} satisfies the same \abbr{SDE} \eqref{eq:sdezetaI1}-\eqref{eq:sdezetaI6} but with additional parameters {\small$(\s,\y)$} and different initial data. A standard calculation for linear evolution equations then gives the formula below for any {\small$(\t,\x)\in[0,\infty)\times\llbracket0,\infty\rrbracket$} (since the initial data of {\small$\mathbf{S}^{\N,\zeta}$} is that of {\small$\mathbf{S}^{\N}$} by construction):
\begin{align}
\mathbf{S}^{\N,\zeta}_{\t,\x}=\sum_{\y\in\llbracket0,\infty\rrbracket}\mathbf{K}^{\N,\zeta}_{0,\t,\x,\y}\mathbf{S}^{\N}_{0,\y}.\label{eq:hkzetaI7}
\end{align}
Note that reducing the parameter {\small$\zeta$} in \eqref{eq:sdezetaI1}-\eqref{eq:sdezetaI6} (or \eqref{eq:hkzetaI1}-\eqref{eq:hkzetaI6}) means showing that the contribution of terms far from the origin is negligible in the large-{\small$\N$} limit. This will be done ultimately via off-diagonal estimates for the {\small$\mathbf{K}^{\N,\zeta}$}-kernel.
\begin{prop}\label{prop:hkzeta}
\fsp There exists {\small$\beta_{\star}>0$} that is independent of all other parameters such that for any finite {\small$\zeta,\kappa,\delta>0$}, we have the following estimate with high probability:
\begin{align}
\sup_{0\leq\s\leq\t\leq1}\sup_{\x\in\llbracket0,\infty\rrbracket}\sum_{\y\in\llbracket0,\infty\rrbracket}\exp(\tfrac{\kappa|\x-\y|}{\N})\cdot|\mathbf{K}^{\N,\zeta}_{\s,\t,\x,\y}|\lesssim_{\kappa,\delta}\exp(\N^{-\beta_{\star}}\N^{\zeta})\N^{\delta}.\label{eq:hkzetaI}
\end{align}
\end{prop}
{\color{black}The proof of Proposition \ref{prop:hkzeta} is deferred to Section \ref{section:hkzeta}, as the argument requires resolving essentially all of the technical challenges and novelties, as discussed in Sections \ref{section:intro} and \ref{section:method}, of this paper.} We emphasize the importance of {\small$\beta_{\star}>0$} being independent of {\small$\zeta>0$}, as discussed below. (We think of {\small$\beta_{\star}>0$} as small but independent of {\small$\N$}.) 

Since {\small$\delta>0$} is small, the {\small$\N^{\delta}$} in \eqref{eq:hkzetaI} is ultimately harmless. (All estimates in this paper will be quantitative in powers of {\small$\N$}, so small powers of {\small$\N$} are negligible for us.) Moreover, if we restrict summands on the \abbr{LHS} of \eqref{eq:hkzetaI} to points {\small$\x\in\llbracket0,\mathrm{O}(\N)\rrbracket$} and {\small$\y\gtrsim\N^{1+\zeta-\beta_{\star}/2}$}, then by \eqref{eq:hkzetaI}, we get {\small$|\mathbf{K}^{\N,\zeta}_{\s,\t,\x,\y}|\lesssim\exp(-\N^{\zeta-\beta_{\star}/2})\exp(\N^{\zeta-\beta_{\star}})\N^{\delta}$}, which is exponentially small in {\small$\N$} unless {\small$\zeta>0$} is small enough so that \eqref{eq:hkzetaI1}-\eqref{eq:hkzetaI6} is sufficiently close to having compactly supported error terms. By using (essentially) this estimate, we will be able to show that all terms in \eqref{eq:hkzetaI1}-\eqref{eq:hkzetaI6} which are supported outside the domain {\small$\llbracket0,\mathrm{O}(\N^{1+\zeta-\beta_{\star}/2})\rrbracket$} ultimately have an exponentially-small-in-{\small$\N$} contribution to {\small$\mathbf{S}^{\N,\zeta}$}. This lets us then modify \eqref{eq:hkzetaI1}-\eqref{eq:hkzetaI6} outside the domain {\small$\llbracket0,\mathrm{O}(\N^{1+\zeta-\beta_{\star}/2})\rrbracket$} and slightly decrease the {\small$\zeta>0$} parameter therein without changing much the value of {\small$\mathbf{S}^{\N,\zeta}$}.

The paragraph above, assuming a rigorous implementation of it, allows us to ultimately lower the {\small$\zeta$} parameter in Lemma \ref{lemma:largezeta} to any small but fixed value. We will then be able to ultimately compare {\small$\mathbf{S}^{\N,\zeta}$} (for small {\small$\zeta>0$}) to the following lattice \abbr{SHE} for {\small$(\t,\x)\in[0,\infty)\times\llbracket0,\infty\rrbracket$} (with initial data given by that of {\small$\mathbf{Z}^{\N}$}):
\begin{align}
\d\mathbf{Q}^{\N}_{\t,\x}=\mathscr{T}_{\N}\mathbf{Q}^{\N}_{\t,\x}+\sqrt{2}\lambda\N^{\frac12}\mathbf{Q}^{\N}_{\t,\x}\d\mathbf{b}_{\t,\x}.\label{eq:qsde}
\end{align}
The reasoning in the previous paragraphs will lead to the following result.
\begin{lemma}\label{lemma:smallzeta}
\fsp Fix any {\small$\mathrm{L}>0$} and any {\small$\zeta_{1}>\zeta_{2}>0$}. With high probability, we have 
\begin{align}
\sup_{\t\in[0,1]}\sup_{\x\in\llbracket0,\mathrm{L}\N\rrbracket}|\mathbf{S}^{\N,\zeta_{1}}_{\t,\x}-\mathbf{S}^{\N,\zeta_{2}}_{\t,\x}|\lesssim_{\mathrm{L}}\N^{-\gamma_{\mathrm{L},\zeta_{1},\zeta_{2}}}, \quad\text{where } \gamma_{\mathrm{L},\zeta_{1},\zeta_{2}}>0.\label{eq:smallzetaI}
\end{align}
Moreover, there exists {\small$\zeta_{\star}>0$} independent of {\small$\N$} such that if {\small$\zeta\in(0,\zeta_{\star})$}, then with high probability, we have
\begin{align}
\sup_{\t\in[0,1]}\max_{\x\in\llbracket0,\mathrm{L}\N\rrbracket}|\mathbf{S}^{\N,\zeta}_{\t,\x}-\mathbf{Q}^{\N}_{\t,\x}|\lesssim_{\mathrm{L}}\N^{-\gamma_{\zeta_{\star},\mathrm{L}}}, \quad\text{where } \gamma_{\zeta_{\star},\mathrm{L}}>0.\label{eq:smallzetaII}
\end{align}
\end{lemma}
{\color{black}The proof of Lemma \ref{lemma:smallzeta} is deferred to Section \ref{section:smallzeta}, as it will use ingredients in the proof of Proposition \ref{prop:hkzeta}.}
\begin{proof}[Proof of Theorem \ref{theorem:main}]
By the formula {\small$\mathbf{Z}^{\N}=\mathbf{R}^{\N}\mathbf{S}^{\N}$}, the estimate in Lemma \ref{lemma:rwedge}, as well as Lemmas \ref{lemma:largezeta} and \ref{lemma:smallzeta}, it is enough to show the desired convergence in Theorem \ref{theorem:main} but with {\small$\mathbf{Q}^{\N}$} in place of {\small$\mathbf{Z}^{\N}$} therein. This is a standard stability of the half-space \abbr{SHE} under spatial-discretizations; the proof is identical to that of Theorem 2.1 in \cite{BG97} after standard modifications to go from the full-space to the half-space.
\end{proof}
\begin{proof}[Proof of Theorem \ref{theorem:mainstat}]
It suffices to construct an initial data for {\small$\mathbf{Z}^{\N}$} such that the assumptions in Theorem \ref{theorem:main} are satisfied, and {\small$\mathbf{Z}^{\N}_{0,\N\X}\to\mathbf{Z}^{\mathrm{stat}}_{0,\X}$} uniformly in {\small$\X\in\llbracket0,\mathrm{L}\rrbracket$} as {\small$\N\to\infty$} for any fixed {\small$\mathrm{L}>0$}, Here, {\small$\mathbf{Z}^{\mathrm{stat}}_{0,\X}=\exp(\lambda\mathbf{h}^{\mathrm{stat}}_{0,\X})$} with {\small$\X\mapsto\mathbf{h}^{\mathrm{stat}}_{0,\X}$} distributed according to any of the stationary measures of the open \abbr{KPZ} equation \eqref{eq:kpzIa}-\eqref{eq:kpzIb} from Definition \ref{definition:statkpz} with {\small$\mathbf{h}^{\mathrm{stat}}_{0,0}=0$}. To this end, let us first define {\small$\mathbf{Z}^{\mathrm{stat},\N}_{0,\x}:=\mathbf{Z}^{\mathrm{stat}}_{0,\N^{-1}\x}$} for {\small$\x\in\llbracket0,\infty\rrbracket$}. We claim that this function satisfies the estimates in \eqref{eq:mainI}. Since this function is a discretization of {\small$\mathbf{Z}^{\mathrm{stat}}$} at scale {\small$\N^{-1}$}, this would follow from the following estimate for any {\small$p\geq1$}, and for some {\small$\kappa_{p}\geq0$}:
\begin{align}
\sup_{\X\in[0,\infty)}\mathbf{e}^{-\kappa_{p}|\X|}\E|\mathbf{Z}^{\mathrm{stat}}_{0,\X}|^{2p}+\sup_{\substack{\X,\Y\in[0,\infty)\\\X\neq\Y}}(\mathbf{e}^{-\kappa_{p}|\X|}\cdot\mathbf{e}^{-\kappa_{p}|\Y|})\cdot|\X-\Y|^{-p}\E|\mathbf{Z}^{\mathrm{stat}}_{0,\X}-\mathbf{Z}^{\mathrm{stat}}_{0,\Y}|^{2p}\lesssim_{p}1.\label{eq:mainstat1}
\end{align}
This follows from an elementary calculation {\color{black}using the formulas in Definition \ref{definition:statkpz}, moment generating function bounds for the Gaussian distribution, and moments of the inverse-Gamma distribution}; we omit this computation. We observe that \eqref{eq:mainstat1} implies that {\small$\mathbf{Z}^{\mathrm{stat}}_{0,\cdot}$} is locally H\"{o}lder-{\small$1/2$} continuous with probability {\small$1$} by the Kolmogorov criterion, and thus {\small$\mathbf{Z}^{\mathrm{stat},\N}_{0,\N\X}\to\mathbf{Z}^{\mathrm{stat}}_{0,\X}$} locally uniformly in {\small$\X\in[0,\infty)$}.

Now, let us choose our initial data for \eqref{eq:currI}-\eqref{eq:currII} as follows. First, we impose {\small$\mathbf{j}^{\N}_{0,0}=0$}, similar to {\small$\mathbf{h}^{\mathrm{stat}}_{0,0}$}. Then, choose {\small$\mathbf{j}^{\N}_{0,\cdot}$}, as a function on {\small$\llbracket0,\infty\rrbracket$}, such that the gradients {\small$\bphi_{0,\cdot}$} are distributed according to the invariant measure {\small$\mathbb{P}^{0}$} from Definition \ref{definition:gc} independently of {\small$\mathbf{h}^{\mathrm{stat}}_{0,\cdot}$}, but then conditioning on the event
\begin{align}
\mathscr{E}:=\Big\{\max_{\x\in\llbracket0,\Lambda(\N)\cdot\N\rrbracket}|\exp(\lambda\mathbf{j}^{\N}_{0,\x})-\mathbf{Z}^{\mathrm{stat},\N}_{0,\x}|\leq\e(\N)\Big\},\label{eq:mainstat2}
\end{align}
where {\small$\Lambda(\N)\to\infty$} and {\small$\e(\N)\to0$} slowly (at rates to be determined) as {\small$\N\to\infty$}. Observe that this construction is done through the {\small$\bphi_{0,\cdot}$}-variables. Moreover, by the definition of conditional probability, we have
\begin{align}
\|\mathfrak{P}\|_{\mathrm{L}^{\infty}(\R^{\llbracket1,\infty\rrbracket})}\leq\mathbb{P}(\mathscr{E})^{-1}.\nonumber
\end{align}
Thus, by the previous display, it suffices to show that for {\small$\Lambda(\N)\to\infty$} and {\small$\e(\N)\to0$} sufficiently slowly, we have {\small$\mathbb{P}(\mathscr{E})\gtrsim\N^{-\gamma_{\mathrm{data}}}$} for some small but fixed {\small$\gamma_{\mathrm{data}}>0$}. Note that the ``good" event {\small$\mathscr{G}$} on which the following estimate holds has probability {\small$1-\mathrm{o}(1)$} because of the estimates in \eqref{eq:mainstat1} and almost sure positivity of {\small$\mathbf{Z}^{\mathrm{stat}}$}:
\begin{align}
\|\mathbf{h}^{\mathrm{stat}}_{0,\cdot}\|_{\mathscr{C}^{\upsilon}([0,\Lambda(\N)])}\lesssim\Gamma(\N),\label{eq:mainstat3a}
\end{align}
Above, {\small$\upsilon\in(0,1/2)$} is fixed, and {\small$\Gamma(\N)\to\infty$} slowly if {\small$\Lambda(\N)\to\infty$} slowly. Thus, we have {\small$\mathbb{P}(\mathscr{E})\gtrsim\mathbb{P}(\mathscr{E}|\mathscr{G})$}. Now, condition on {\small$\mathbf{h}^{\mathrm{stat}}_{0,\cdot}$} (while conditioning on {\small$\mathscr{G}$}, thus the functional value that {\small$\mathbf{h}^{\mathrm{stat}}_{0,\cdot}$} takes satisfies \eqref{eq:mainstat3a}). The proof of Theorem 2.7 in Section 4.4 of \cite{Y26PTRF} (via support estimates for the Wiener measure) now shows that 
\begin{align}
\mathbb{P}^{0}\Big(\max_{\x\in\llbracket0,\Lambda(\N)\cdot\N\rrbracket}|\mathbf{j}^{\N}_{0,\x}-\mathbf{h}^{\mathrm{stat}}_{0,\N^{-1}\x}|\leq\nu(\N)\Big)\geq\delta(\N),\nonumber
\end{align}
where {\small$\nu(\N),\delta(\N)\to0$} slowly. (We clarify that this argument in \cite{Y26PTRF} requires the initial data that we approximate to be deterministic and admit an estimate like \eqref{eq:mainstat3a} for slowly growing {\small$\Gamma(\N)\to\infty$}. Moreover, we clarify that {\small$\mathbf{j}^{\N}$} above is distributed via {\small$\mathbf{j}^{\N}_{0,0}=0$} and {\small$\bphi_{0,\cdot}\sim\mathbb{P}^{0}$}.) Exponentiating the estimate in the probability above and using \eqref{eq:mainstat3a} gives {\small$\mathbb{P}(\mathscr{E}|\mathscr{G},\mathbf{h}^{\mathrm{stat}}_{0,\cdot})\gtrsim\delta(\N)$} with {\small$\delta(\N)\to0$} slowly, uniformly in whatever functional value {\small$\mathbf{h}^{\mathrm{stat}}_{0,\cdot}$} takes. We now average over all possibly functional values of {\small$\mathbf{h}^{\mathrm{stat}}_{0,\cdot}$} to get {\small$\mathbb{P}(\mathscr{E}|\mathscr{G})\gtrsim\nu(\N)$}. Since {\small$\mathbb{P}(\mathscr{E})\gtrsim\mathbb{P}(\mathscr{E}|\mathscr{G})$}, as we noted earlier, the proof is complete.
\end{proof}
\begin{proof}[Proof of Theorem \ref{theorem:mainwedge}]
The proof of Theorem 2.7 in \cite{Y26PTRF} gives a choice of initial data {\small$\mathbf{Z}^{\N}_{0,\cdot}$} so that Assumption \ref{assump:data} is satisfied, and so that the following properties hold (before we state properties, we remark that the construction gives a choice of {\small$\mathbf{Z}^{\N}_{0,\cdot}$} defined on the full-space {\small$\Z$}, but we can always restrict it to {\small$\llbracket0,\infty\rrbracket$}):
\begin{enumerate}
\item The initial data {\small$\mathbf{Z}^{\N}_{0,\cdot}$} is strictly positive.
\item We have {\small$\|\mathbf{e}^{\t\mathscr{T}_{\N}}(\mathbf{Z}^{\N}_{0,\cdot})\|_{\mathrm{L}^{\infty}(\llbracket0,\infty\rrbracket)}\lesssim\t^{-1/2}$} for any {\small$\t>0$}. (Technically, \cite{Y26PTRF} shows this in the following different setting. First, we replace {\small$\mathscr{T}_{\N}\mapsto\N^{2}\Delta_{\Z}$}, in which {\small$\Delta_{\Z}$} denotes the discrete Laplacian on {\small$\Z$}. Second, we extend {\small$\mathbf{Z}^{\N}_{0,\cdot}$} from {\small$\llbracket0,\infty\rrbracket$} to {\small$\Z$} by reflection. However, the proof of this estimate in \cite{Y26PTRF} only uses that the heat kernel for {\small$\mathbf{e}^{\t\mathscr{T}_{\N}}$} is bounded by {\small$\lesssim\N^{-1}\t^{-1/2}$}, so the estimate also holds in our setting by Proposition \ref{prop:hkestimates}.)
\item We have the estimate {\small$\N^{-1}\|\mathbf{Z}^{\N}_{0,\cdot}\|_{\ell^{1}(\llbracket0,\infty\rrbracket)}\lesssim1$}. This follows because we choose {\small$\mathbf{Z}^{\N}_{0,\cdot}$} to be the restriction of the initial data constructed in the proof of Theorem 2.7 in \cite{Y26PTRF} from {\small$\Z$} to {\small$\llbracket0,\infty\rrbracket$}, and the {\small$\ell^{1}$}-norm cannot increase after restricting the spatial domain.
\item We have convergence {\small$\mathbf{Z}^{\N}_{0,\N\cdot}\to\delta_{0}$} as Borel measures on {\small$[0,\infty)$} as {\small$\N\to\infty$}. This follows since the construction of {\small$\mathbf{Z}^{\N}_{0,\cdot}$} in the proof of Theorem 2.7 in \cite{Y26PTRF} forces {\small$\mathbf{Z}^{\N}_{0,\cdot}$} to be invariant-in-law upon reflection about the origin in {\small$\Z$}, and the constructed initial data therein converges to the deterministic {\small$\delta_{0}$}-measure as {\small$\N\to\infty$}. (Therefore, the restriction of said initial data to the half-space must also converge to {\small$\delta_{0}$}.)
\end{enumerate}
With these ingredients, it is now standard to establish the desired convergence of {\small$\mathbf{Z}^{\N}$} to the narrow-wedge solution of the open \abbr{SHE} (see, for instance, Section 6 of \cite{P19}).
\end{proof}
To summarize, we are now left with proving Proposition \ref{prop:hkzeta} and Lemma \ref{lemma:smallzeta}.
%
%
%
\section{Stochastic estimates I: preliminary estimates}\label{section:stochI}
\subsection{Regularity estimates}
We start with a priori estimates on the solution to \eqref{eq:phiI}-\eqref{eq:phiII}, which can be seen as (and will actually be used for) estimates on gradients of \eqref{eq:currI}-\eqref{eq:currII} in a variety of local and global topologies.
\begin{lemma}\label{lemma:phiap}
\fsp Fix any {\small$\mathrm{D},\delta,\mathrm{T}>0$} independent of {\small$\N$}. Fix also any collection of intervals {\small$\mathbb{I}_{1},\ldots,\mathbb{I}_{\N^{\mathrm{D}}}\subseteq\llbracket1,\infty\rrbracket$} with size {\small$|\mathbb{I}_{\j}|\lesssim\N^{\mathrm{D}}$} for all {\small$\j$}. For some {\small$\mathrm{c}>0$} depending on {\small$\delta>0$}, we have
\begin{align}
\mathbb{P}\Big(\sup_{\t\in[0,\mathrm{T}]}\sup_{\j\in\llbracket1,\N^{\mathrm{D}}\rrbracket}|\mathbb{I}_{\j}|^{-\frac12}\Big|\sum_{\x\in\mathbb{I}_{\j}}\bphi_{\t,\x}\Big|\geq\N^{\delta}\Big)\lesssim_{}\exp(-\N^{\mathrm{c}}).\label{eq:phiapI}
\end{align}
By taking {\small$\mathbb{I}_{\j}=\{\x_{\j}\}$} for distinct points {\small$\x_{1},\ldots,\x_{\N^{\mathrm{D}}}$}, we obtain {\small$|\bphi_{\t,\x}|\lesssim\N^{\delta}$} simultaneously for all {\small$(\t,\x)\in[0,\mathrm{T}]\times\mathbb{I}$} with probability {\small$1-\mathrm{O}(\exp(-\N^{\mathrm{c}}))$}, where {\small$\mathbb{I}\subseteq\llbracket0,\infty\rrbracket$} is any deterministic set of size {\small$\lesssim\N^{\mathrm{D}}$}.
\end{lemma}
\begin{proof}
Fix an interval {\small$\mathbb{I}\subseteq\llbracket1,\infty\rrbracket$}. For any {\small$\mathrm{K}\geq0$}, we have the following for some {\small$\upsilon>0$} depending only on {\small$\mathscr{U}$}:
\begin{align}
\mathbb{P}^{0}\Big(\Big|\sum_{\x\in\mathbb{I}}\bphi_{\x}\Big|\geq\mathrm{K}|\mathbb{I}|^{\frac12}\Big)\leq\exp(-\upsilon\mathrm{K}^{2}).\label{eq:phiap1}
\end{align}
This follows because under Assumption \ref{assump:potential}, the {\small$\bphi_{\x}$}-variables are i.i.d. and sub-Gaussian, at which point we can use a standard Chernoff estimate (as in \cite{V18}) to obtain sub-Gaussianity of the sum inside the \abbr{LHS} of \eqref{eq:phiap1} with variance parameter {\small$\lesssim|\mathbb{I}|$}. Now, until we mention otherwise, suppose that {\small$\t\mapsto\bphi_{\t}$} is given initial data distributed as {\small$\mathbb{P}^{0}$}. Since {\small$\mathbb{P}^{0}$} is an invariant measure (see Section \ref{subsection:generator}), we deduce from \eqref{eq:phiap1} that
\begin{align}
\sup_{\t\in[0,\mathrm{T}]}\max_{\j\in\llbracket1,\N^{\mathrm{D}}\rrbracket}\mathbb{P}\Big(\Big|\sum_{\x\in\mathbb{I}_{\j}}\bphi_{\t,\x}\Big|\geq\N^{\delta}|\mathbb{I}_{\j}|^{\frac12}\Big)\lesssim\exp(-\upsilon\N^{2\delta}).\nonumber
\end{align}
Since {\small$\delta>0$} is independent of {\small$\N$}, we can perform a union bound to obtain the following for any {\small$\mathrm{C}\lesssim1$}:
\begin{align}
\mathbb{P}\Big(\sup_{\t\in[0,\mathrm{T}]\cap\N^{-\mathrm{C}_{}}\Z}\max_{\j\in\llbracket1,\N^{\mathrm{D}}\rrbracket}|\mathbb{I}_{\j}|^{-\frac12}\Big|\sum_{\x\in\mathbb{I}_{\j}}\bphi_{\t,\x}\Big|\geq\N^{\delta}\Big)\lesssim\N^{\mathrm{C}+\mathrm{D}}\exp(-\upsilon\N^{2\delta})\lesssim\exp(-\upsilon\N^{\delta}).\nonumber
\end{align}
A standard short-time continuity argument allows us to remove the intersection with {\small$\N^{-\mathrm{C}}\Z$} on the \abbr{LHS}, assuming we change the implied constant. (Indeed, the scale {\small$\N^{-\mathrm{C}}$} is much smaller than the speed {\small$\N^{2}$} at which \eqref{eq:phiI}-\eqref{eq:phiII} evolves.) In particular, the desired result \eqref{eq:phiapI} holds assuming that the initial data of {\small$\t\mapsto\bphi_{\t}$} is sampled via {\small$\mathbb{P}^{0}$}:
\begin{align}
\mathbb{P}^{0}\Big(\sup_{\t\in[0,\mathrm{T}]}\sup_{\j\in\llbracket1,\N^{\mathrm{D}}\rrbracket}|\mathbb{I}_{\j}|^{-\frac12}\Big|\sum_{\x\in\mathbb{I}_{\j}}\bphi_{\t,\x}\Big|\geq\N^{\delta}\Big)\lesssim\exp(-\upsilon\N^{\delta}).\label{eq:phiapI1}
\end{align}
To go beyond {\small$\mathbb{P}^{0}$}-initial data, we use the a priori estimate \eqref{eq:dataI}. Indeed, the path-space measure on the \abbr{LHS} is the same as the path-space measure upon conditioning on the initial data, and then sampling the initial data with respect to {\small$\mathbb{P}^{0}$}. Since the path-space dynamics do not depend on the initial data, we can change measure only for the law of the initial data and use \eqref{eq:dataI} to control the cost in this change-of-measure. So, by the Cauchy-Schwarz inequality, we have 
\begin{align*}
\mathbb{P}^{}\Big(\sup_{\t\in[0,\mathrm{T}]}\sup_{\j\in\llbracket1,\N^{\mathrm{D}}\rrbracket}|\mathbb{I}_{\j}|^{-\frac12}\Big|\sum_{\x\in\mathbb{I}_{\j}}\bphi_{\t,\x}\Big|\geq\N^{\delta}\Big)\lesssim\exp(\N^{\gamma_{\mathrm{data}}})\cdot\mathbb{P}^{0}\Big(\sup_{\t\in[0,\mathrm{T}]}\sup_{\j\in\llbracket1,\N^{\mathrm{D}}\rrbracket}|\mathbb{I}_{\j}|^{-\frac12}\Big|\sum_{\x\in\mathbb{I}_{\j}}\bphi_{\t,\x}\Big|\geq\N^{\delta}\Big),
\end{align*}
at which point we we use \eqref{eq:phiapI1} to obtain the desired bound \eqref{eq:phiapI} (since {\small$\gamma_{\mathrm{data}}>0$} is small).
\end{proof}
As a consequence of Lemma \ref{lemma:phiap}, we are able to show Lemma \ref{lemma:rwedge} (and actually a much stronger version). In order to state the following lemma, we first recall {\small$\mathbf{R}^{\N,\wedge}$} from \eqref{eq:rwedge}.
\begin{lemma}\label{lemma:rratio}
\fsp Fix {\small$\mathrm{D}>0$} and {\small$\mathrm{T}>0$} independent of {\small$\N$}. With high probability (see Definition \ref{definition:hp}), we have 
\begin{align}
\sup_{\t\in[0,\mathrm{T}]}\max_{\x\in\llbracket0,\N^{\mathrm{D}}\rrbracket}|\mathbf{R}^{\N}_{\t,\x}-\mathbf{R}^{\N,\wedge}_{\t,\x}|=0.\label{eq:rratioI}
\end{align}
\end{lemma}
\begin{proof}
Recall that {\small$\mathbf{1}^{(\x)}_{\w}:=\mathbf{1}[|\x-\w|\lesssim(\log\N)^{2}\N\mathfrak{t}_{\N}^{1/2}]$}, and recall {\small$\mathbf{S}^{\N}$}, both from \eqref{eq:heatIII}; by definition, we have 
\begin{align}
\mathbf{S}^{\N}_{\t,\x}-\mathbf{Z}^{\N}_{\t,\x}=\sum_{\w\in\llbracket0,\infty\rrbracket}\mathbf{H}^{\N,\mathbf{A}}_{0,\mathfrak{t}_{\N},\x,\w}\cdot\mathbf{1}^{(\x)}_{\w}\cdot(\mathbf{Z}^{\N}_{\t,\w}-\mathbf{Z}^{\N}_{\t,\x})+\mathbf{Z}^{\N}_{\t,\x}\cdot([\mathbf{e}^{\mathfrak{t}_{\N}\mathscr{T}_{\N}}(\mathbf{1}^{(\x)}_{\cdot})]_{\x}-1).\label{eq:rratioI1}
\end{align}
As shown in the proof of Lemma \ref{lemma:triv}, we have {\small$[\mathbf{e}^{\mathfrak{t}_{\N}\mathscr{T}_{\N}}(\mathbf{1}^{(\x)}_{\cdot})]_{\x}=\mathbf{e}^{\mathfrak{t}_{\N}\mathscr{T}_{\N}}(\mathbf{1})+\mathrm{O}_{\mathrm{D}}(\N^{-\mathrm{D}})$} for any {\small$\mathrm{D}>0$}, where {\small$\mathbf{1}$} denotes the constant function on {\small$\llbracket0,\infty\rrbracket$} that is equal to {\small$1$} everywhere. Thus,
\begin{align}
|\mathbf{Z}^{\N}_{\t,\x}\cdot([\mathbf{e}^{\mathfrak{t}_{\N}\mathscr{T}_{\N}}(\mathbf{1}^{(\x)}_{\cdot})]_{\x}-1)|\lesssim|\mathbf{Z}^{\N}_{\t,\x}|\cdot|[\mathbf{e}^{\mathfrak{t}_{\N}\mathscr{T}_{\N}}(\mathbf{1})]_{\x}-1|+\N^{-\mathrm{D}}|\mathbf{Z}^{\N}_{\t,\x}|\lesssim\N^{-\frac12\delta_{\mathbf{S}}}|\mathbf{Z}^{\N}_{\t,\x}|,\label{eq:rratioI2}
\end{align}
where the last bound follows by Lemma \ref{lemma:hknorm} and {\small$\mathfrak{t}_{\N}=\N^{-\delta_{\mathbf{S}}}$} from Definition \ref{definition:heat}. We now address the first term on the \abbr{RHS} of \eqref{eq:rratioI1}. We recall \eqref{eq:ch} and the rule {\small$\bphi_{\t,\x}=\N^{1/2}(\mathbf{j}^{\N}_{\t,\x}-\mathbf{j}^{\N}_{\t,\x-1})$} for any {\small$\x\in\llbracket1,\infty\rrbracket$}, so
\begin{align*}
|\mathbf{Z}^{\N}_{\t,\w}-\mathbf{Z}^{\N}_{\t,\x}|\lesssim|\mathbf{Z}^{\N}_{\t,\x}|\cdot\Big|\exp\Big(\mathrm{c}_{\x,\w}\lambda\N^{-\frac12}\sum_{\y\in\mathbb{I}_{\x,\w}}\bphi_{\t,\y}\Big)-1\Big|.
\end{align*}
Above, {\small$\mathbb{I}_{\x,\w}$} is the interval with endpoints {\small$\x,\w\in\llbracket1,\infty\rrbracket$}, and {\small$\mathrm{c}_{\x,\w}\in\{\pm1\}$} depending on whether {\small$\x\leq\w$} or {\small$\w\leq\x$}. Consider the event on which the sum over {\small$\mathbb{I}_{\x,\w}$} above is {\small$\lesssim\N^{\delta}|\mathbb{I}_{\x,\w}|^{1/2}$} for all {\small$\x\in\llbracket0,\N^{\mathrm{D}}\rrbracket$} (as in the statement of the lemma) and all {\small$\w\in\llbracket0,\infty\rrbracket$} with {\small$|\x-\w|\lesssim(\log\N)^{2}\N\mathfrak{t}_{\N}^{1/2}$}. Here, {\small$\delta>0$} is small but fixed. There are {\small$\lesssim\N^{\mathrm{O}(1)}$}-many such intervals. Thus by Lemma \ref{lemma:phiap}, such an event holds with high probability. Since {\small$|\mathbb{I}_{\x,\w}|\lesssim(\log\N)^{2}\N\mathfrak{t}_{\N}^{1/2}$}, we deduce that with high probability, the term in the exponential is {\small$\lesssim\N^{-1/2}\N^{\delta}(\log\N)^{2}\N^{1/2}\mathfrak{t}_{\N}^{1/4}\lesssim\N^{\delta}(\log\N)^{2}\N^{-\delta_{\mathbf{S}}/4}$}. Since {\small$\delta>0$} is small, this is {\small$\lesssim\N^{-\delta_{\mathbf{S}}/5}$}. We use and {\small$|\exp(\a)-1|\lesssim\a$} (for {\small$\a\ll1$}) to get
\begin{align*}
\Big|\exp\Big(\mathrm{c}_{\x,\w}\lambda\N^{-\frac12}\sum_{\y\in\mathbb{I}_{\x,\w}}\bphi_{\t,\y}\Big)-1\Big|\lesssim\N^{-\frac15\delta_{\mathbf{S}}}.
\end{align*}
Combining the previous two displays, while noting that {\small$\mathbf{H}^{\N,\mathbf{A}}_{0,\mathfrak{t}_{\N},\x,\w}$} has {\small$\ell^{1}(\llbracket0,\infty\rrbracket)$}-norm of {\small$\lesssim1$} in  {\small$\w$} (see Proposition \ref{prop:hkestimates}), we get the following with high probability simultaneously over all {\small$(\t,\x)$} as in \eqref{eq:rratioI}:
\begin{align*}
\sum_{\w\in\llbracket0,\infty\rrbracket}\mathbf{H}^{\N,\mathbf{A}}_{0,\mathfrak{t}_{\N},\x,\w}\cdot\mathbf{1}^{(\x)}_{\w}\cdot(\mathbf{Z}^{\N}_{\t,\w}-\mathbf{Z}^{\N}_{\t,\x})\lesssim\N^{-\frac15\delta_{\mathbf{S}}}|\mathbf{Z}^{\N}_{\t,\x}|
\end{align*}
Combining this with \eqref{eq:rratioI1} and \eqref{eq:rratioI2} yields {\small$|\mathbf{S}^{\N}_{\t,\x}-\mathbf{Z}^{\N}_{\t,\x}|\lesssim\N^{-\delta_{\mathbf{S}}/5}|\mathbf{Z}^{\N}_{\t,\x}|$}. Dividing both sides of this inequality by {\small$\mathbf{Z}^{\N}_{\t,\x}$} shows with high probability, we have {\small$|\mathbf{R}^{\N}_{\t,\x}-1|\lesssim\N^{-\delta_{\mathbf{S}}/5}$} for all {\small$\t\in[0,\mathrm{T}]$} and {\small$\x\in\llbracket0,\N^{\mathrm{D}}\rrbracket$} simultaneously; see \eqref{eq:heatIII} for {\small$\mathbf{R}^{\N}$}. Since {\small$\mathbf{R}^{\N,\wedge}=\mathbf{R}^{\N}$} when this bound holds (see \eqref{eq:rwedge}), the proof is complete.
\end{proof}
\subsection{\abbr{CLT} estimates for spatial-averaging}
We now provide estimates for the {\small$\mathbf{Av}^{\mathbf{X},\mathscr{Q}}$}-objects in Definition \ref{definition:bulkaverage}. Intuitively, these are space-averages of bulk-admissible functions, which have finite-support properties and vanish under {\small$\E^{0}$}; see Definition \ref{definition:bulkadmissible}. Thus, averaging a large collection of these yields \abbr{CLT}-type cancellations. The following captures this precisely, with a technical hurdle being the fact that all the bulk-admissible functions are multiplied by {\small$\mathbf{Z}^{\N}$}-factors in {\small$\mathbf{Av}^{\mathbf{X},\mathscr{Q}}$}; see \eqref{eq:bulkaverageIV}.
\begin{lemma}\label{lemma:cltx}
\fsp Fix any {\small$\mathrm{D},\delta,\mathrm{T}>0$} independent of {\small$\N$}. Let {\small$\mathscr{Q}=(\mathfrak{q}_{\x})_{\x\in\llbracket0,\infty\rrbracket}$} denote a collection of functions that are bulk-admissible with respect to the corresponding point {\small$\x\in\llbracket0,\infty\rrbracket$}. With high probability, we have 
\begin{align}
\sup_{\t\in[0,\mathrm{T}]}\sup_{\x\in\llbracket0,\N^{\mathrm{D}}\rrbracket}|\mathbf{Av}^{\mathbf{X},\mathscr{Q}}_{\t,\x}|\lesssim\N^{-\frac12+\frac34\delta_{\mathbf{S}}+\delta}.\label{eq:cltxI}
\end{align}
\end{lemma}
\begin{proof}
Intuitively, recall from Definition \ref{definition:bulkaverage} that {\small$\mathbf{Av}^{\mathbf{X},\mathscr{Q}}$} is averaging on a length-scale {\small$\mathfrak{l}_{\mathbf{Av}}=\N^{1-3\delta_{\mathbf{S}}/2}$} and thus \abbr{CLT} cancellations suggest a power-saving of {\small$\mathfrak{l}_{\mathbf{Av}}^{-1/2}=\N^{-1/2+3\delta_{\mathbf{S}}/4}$}, explaining the \abbr{RHS} of \eqref{eq:cltxI}. The extra power of {\small$\N^{\delta}$} on the \abbr{RHS} of \eqref{eq:cltxI} is to go beyond this natural scale and get a high probability, non-equilibrium estimate (for initial data satisfying \eqref{eq:dataI}). 

Let us now make the above precise. The argument below proceeds by making a sequence of high probability claims. We only make finitely many of these claims, and the intersection of finitely many high probability events holds with high probability itself, so the final conclusions hold with high probability as well.

For the sake of concreteness, let us assume that {\small$\mathfrak{q}_{\w}$} is bulk-admissible with backwards-orientation (see Definition \ref{definition:bulkadmissible} for what this means). For forwards-orientation, the same argument works upon reflection about {\small$\w$}-points below. By \eqref{eq:bulkaverageIV}-\eqref{eq:bulkaverageV}, we can write 
\begin{align}
\mathbf{Av}^{\mathbf{X},\mathscr{Q}}_{\t,\x}=\sum_{\w\in\llbracket0,\infty\rrbracket}{\boldsymbol{\mathcal{H}}}_{\x,\w}\cdot\mathfrak{q}_{\w}[\bphi_{\t}]\cdot\mathbf{Z}^{\N}_{\t,\w}(\mathbf{Z}^{\N}_{\t,\x_{+}})^{-1}, \label{eq:cltxI1}
\end{align}
where {\small${\boldsymbol{\mathcal{H}}}_{\x,\w}:=\mathbf{H}^{\N,0}_{0,\tau_{\mathbf{Av}},\x,\w}\mathbf{1}_{|\x-\w|\lesssim(\log\N)^{2}\mathfrak{l}_{\mathbf{Av}}}$} is supported in a neighborhood of radius {\small$\lesssim(\log\N)^{2}\N^{1-3\delta_{\mathbf{S}}/2}$} around {\small$\x$} and is uniformly {\small$\lesssim\N^{-1+3\delta_{\mathbf{S}}/2}$} in size. This follows by Proposition \ref{prop:hkestimates} and {\small$\tau_{\mathbf{Av}}=\N^{-3\delta_{\mathbf{S}}}$} from Definition \ref{definition:bulkaverage}. We will use these properties freely in what follows. We now make two modifications to \eqref{eq:cltxI1}.
\begin{enumerate}
\item First, we introduce a technical cutoff for {\small$\mathfrak{q}_{\w}$} which introduces a deterministic upper bound on its size (coming from its deterministic polynomial estimate \eqref{eq:bulkadmissibleI} and {\small$\bphi$}-estimates from Lemma \ref{lemma:phiap}). We also arrange this modification to keep the mean-zero property with respect to {\small$\E^{0}$} for {\small$\mathfrak{q}_{\w}$}. This is up to an ultimately negligible error and holds with high probability. We make this precise as follows.

We claim that there is an event of high probability such that for all {\small$(\t,\x)\in[0,\mathrm{T}]\times\llbracket0,\N^{\mathrm{D}}\rrbracket$} and {\small$\w$} for which {\small${\boldsymbol{\mathcal{H}}}_{\x,\w}\neq0$}, we have {\small$\mathfrak{q}_{\w}[\bphi_{\t}]=\mathfrak{q}_{\w}[\bphi_{\t}]\mathbf{1}(|\mathfrak{q}_{\w}[\bphi_{\t}]|\leq\N^{\e})$} for some {\small$\e>0$} small but fixed. Indeed, use Lemma \ref{lemma:phiap} with {\small$\mathbb{I}_{\j}$} corresponding to each {\small$\w$} of interest; this gives {\small$|\bphi_{\t,\w}|\leq\N^{\nu}$} for {\small$\nu>0$} small and all {\small$(\t,\w)$}-pairs at hand, with which we can apply the polynomial estimate \eqref{eq:bulkadmissibleI} for bulk-admissible functions to deduce the claim. By the same token, the bound \eqref{eq:phiapI1} implies that {\small$\mathbb{P}^{0}(|\mathfrak{q}_{\w}[\bphi]|>\N^{\e})\lesssim\exp(-\upsilon\N^{\mu})$} for some {\small$\mu>0$} depending on {\small$\e>0$}. By all of this, we get that {\small$|\E^{0}[\mathfrak{q}_{\w}\mathbf{1}(|\mathfrak{q}_{\w}|\leq\N^{\e})]|\lesssim(\E^{0}|\mathfrak{q}_{\w}|^{2})^{1/2}\mathbb{P}^{0}(|\mathfrak{q}_{\w}|>\N^{\e})^{1/2}\lesssim_{\mathrm{D}}\N^{-\mathrm{D}}$}, for any {\small$\mathrm{D}>0$}. Indeed, if {\small$\mathfrak{q}_{\w}$} is bulk-admissible, then {\small$\E^{0}\mathfrak{q}_{\w}=0$} and {\small$\E^{0}|\mathfrak{q}_{\w}|^{2}\lesssim1$} (see \eqref{eq:bulkadmissibleI} and the sub-Gaussianity of {\small$\bphi_{\y}$}-variables with respect to {\small$\mathbb{P}^{0}$}). Ultimately, we obtain, with high probability simultaneously for all {\small$(\t,\x)$} of interest, the following, in which {\small$\wt{\mathfrak{q}}_{\w}=\mathfrak{q}_{\w}\mathbf{1}(|\mathfrak{q}_{\w}|\leq\N^{\e})-\E^{0}[\mathfrak{q}_{\w}\mathbf{1}(|\mathfrak{q}_{\w}|\leq\N^{\e})]$}:
\begin{align}
\mathbf{Av}^{\mathbf{X},\mathscr{Q}}_{\t,\x}=\sum_{\w\in\llbracket0,\infty\rrbracket}{\boldsymbol{\mathcal{H}}}_{\x,\w}\cdot\wt{\mathfrak{q}}_{\w}[\bphi_{\t}]\cdot\mathbf{Z}^{\N}_{\t,\w}(\mathbf{Z}^{\N}_{\t,\x_{+}})^{-1}+\sum_{\w\in\llbracket0,\infty\rrbracket}{\boldsymbol{\mathcal{H}}}_{\x,\w}\cdot\mathrm{O}_{\mathrm{D}}(\N^{-\mathrm{D}})\cdot\mathbf{Z}^{\N}_{\t,\w}(\mathbf{Z}^{\N}_{\t,\x_{+}})^{-1}.\label{eq:cltxI2}
\end{align}
\item We now use \eqref{eq:ch} to write {\small$\mathbf{Z}^{\N}_{\t,\w}(\mathbf{Z}^{\N}_{\t,\x_{+}})^{-1}=\exp(\lambda(\mathbf{j}^{\N}_{\t,\w}-\mathbf{j}^{\N}_{\t,\x_{+}}))$}. We will now introduce an additional cutoff in this exponential factor (that holds with high probability), essentially by using the feature that {\small$\mathbf{j}^{\N}$} has (close to) macroscopic-scale spatial regularity, and that {\small$|\w-\x_{+}|$} is mesoscopic. Let us be more precise.

By the identity {\small$\bphi_{\t,\y}=\N^{1/2}(\mathbf{j}^{\N}_{\t,\y}-\mathbf{j}^{\N}_{\t,\y-1})$} and by Lemma \ref{lemma:phiap}, we get the existence of a high probability event on which simultaneously for all {\small$\x\in\llbracket0,\N^{\mathrm{D}}\rrbracket$}, all {\small$\w$} such that {\small${\boldsymbol{\mathcal{H}}}_{\x,\w}\neq0$}, and all {\small$\t\in[0,\mathrm{T}]$}, we have the estimate {\small$|\mathbf{j}^{\N}_{\t,\w}-\mathbf{j}^{\N}_{\t,\x_{+}}|\lesssim\N^{-1/2+\delta}|\w-\x_{+}|^{1/2}$}. Now, recall that {\small${\boldsymbol{\mathcal{H}}}_{\x,\w}\neq0$} forces {\small$|\x-\w|\lesssim(\log\N)^{2}\N^{1-3\delta_{\mathbf{S}}/2}$} as mentioned after \eqref{eq:cltxI1}, and {\small$|\x-\x_{+}|\lesssim(\log\N)^{3}\mathfrak{l}_{\mathbf{Av}}\lesssim(\log\N)^{3}\N^{1-3\delta_{\mathbf{S}}/2}$} by construction in Definition \ref{definition:bulkaverage}. Thus, on this event, we have {\small$|\mathbf{j}^{\N}_{\t,\w}-\mathbf{j}^{\N}_{\t,\x_{+}}|\lesssim\N^{\delta-3\delta_{\mathbf{S}}/4}$}, and if {\small$\delta>0$} is small, we may modify \eqref{eq:cltxI2} (without changing it with high probability) to be written as
\begin{align}
\mathbf{Av}^{\mathbf{X},\mathscr{Q}}_{\t,\x}&=\sum_{\w\in\llbracket0,\infty\rrbracket}{\boldsymbol{\mathcal{H}}}_{\x,\w}\cdot\wt{\mathfrak{q}}_{\w}[\bphi_{\t}]\cdot\mathbf{G}_{\t,\x,\w}+\sum_{\w\in\llbracket0,\infty\rrbracket}{\boldsymbol{\mathcal{H}}}_{\x,\w}\cdot\mathrm{O}_{\mathrm{D}}(\N^{-\mathrm{D}})\cdot\mathbf{G}_{\t,\x,\w},\label{eq:cltxI3a}
\end{align}
where
\begin{align}
\mathbf{G}_{\t,\x,\w}:=\exp(\lambda(\mathbf{j}^{\N}_{\t,\w}-\mathbf{j}^{\N}_{\t,\x_{+}}))\cdot\mathbf{1}[\exp(\lambda(\mathbf{j}^{\N}_{\t,\w}-\mathbf{j}^{\N}_{\t,\x_{+}}))\lesssim1].\label{eq:cltxI3b}
\end{align}
\end{enumerate}
By construction, the last term in \eqref{eq:cltxI3a} is {\small$\lesssim_{\mathrm{D}}\N^{-\mathrm{D}}$}. Thus, it remains to analyze the first term on the \abbr{RHS} of \eqref{eq:cltxI3a}. To this end, we summarize two key points about the terms therein.
\begin{enumerate}
\item The kernel {\small${\boldsymbol{\mathcal{H}}}_{\x,\w}$} is deterministic, and {\small$\mathbf{G}_{\t,\x,\w}$} depends only on {\small$\bphi_{\t,\y}$} for {\small$\y\in\llbracket\w+1,\infty\rrbracket$}. 
\item The function {\small$\wt{\mathfrak{q}}_{\w}$}, being a function of only {\small$\mathfrak{q}_{\w}$}, depends only on {\small$\bphi_{\t,\y}$} for {\small$\y\in\llbracket\w-\mathfrak{l},\w\rrbracket$} for {\small$\mathfrak{l}=\mathrm{O}(1)$}, since {\small$\mathfrak{q}_{\w}$} was assumed to be backwards-oriented earlier in this proof.
\end{enumerate}
In view of the previous two points, we now write, with notation explained afterwards,
\begin{align*}
\sum_{\w\in\llbracket0,\infty\rrbracket}{\boldsymbol{\mathcal{H}}}_{\x,\w}\cdot\wt{\mathfrak{q}}_{\w}[\bphi_{\t}]\cdot\mathbf{G}_{\t,\x,\w}=\sum_{\a\in\llbracket1,2\mathfrak{l}\rrbracket}\sum_{\w\in\mathbb{I}_{\a}}{\boldsymbol{\mathcal{H}}}_{\x,\w}\cdot\wt{\mathfrak{q}}_{\w}[\bphi_{\t}]\cdot\mathbf{G}_{\t,\x,\w}.
\end{align*}
In the above, the index {\small$\a\in\llbracket1,2\mathfrak{l}\rrbracket$} parameterizes equivalence classes in {\small$\llbracket0,\infty\rrbracket$} modulo {\small$2\mathfrak{l}$}, and {\small$\w\in\mathbb{I}_{\a}$} ranges over all elements in the corresponding equivalence class. In particular, for any distinct points {\small$\w_{1},\w_{2}\in\mathbb{I}_{\a}$}, the values {\small$\wt{\mathfrak{q}}_{\w_{1}}[\bphi_{\t}],\wt{\mathfrak{q}}_{\w_{2}}[\bphi_{\t}]$} depend on {\small$\bphi_{\t,\y}$}-spins for {\small$\y$}-points in disjoint intervals. Now, arrange the sum over {\small$\w\in\mathbb{I}_{\a}$} so that the {\small$\w$}-indices are in decreasing order; we can do this since {\small${\boldsymbol{\mathcal{H}}}_{\x,\w}\neq0$} only for a finite set of {\small$\w$}-points. Let us also assume for now that {\small$\bphi_{\t}\sim\mathbb{P}^{0}$}, i.e. its initial data is distributed via {\small$\mathbb{P}^{0}$}. For any {\small$\w\in\mathbb{I}_{\a}$}, define the {\small$\sigma$}-algebra {\small$\mathscr{F}_{\w}$} to be generated by {\small$\bphi_{\t,\y}$} for {\small$\y\in\llbracket\w-\mathfrak{l},\infty\rrbracket$}. 

We now claim that in the current setting, the summands {\small${\boldsymbol{\mathcal{H}}}_{\x,\w}\cdot\wt{\mathfrak{q}}_{\w}[\bphi_{\t}]\cdot\mathbf{G}_{\t,\x,\w}$} are martingale increments (with time-parameter {\small$\w\in\mathbb{I}_{\a}$} and with respect to the filtration {\small$\mathscr{F}_{\w}$}). Indeed, we recalled earlier that {\small${\boldsymbol{\mathcal{H}}}_{\x,\w}\cdot\wt{\mathfrak{q}}_{\w}[\bphi_{\t}]\cdot\mathbf{G}_{\t,\x,\w}$} is a function of {\small$\bphi_{\t,\y}$} for {\small$\y\in\llbracket\w-\mathfrak{l},\infty\rrbracket$}, so it is {\small$\mathscr{F}_{\w}$}-measurable. On the other hand, suppose {\small$\w_{1}>\w_{2}$} in the sum over {\small$\mathbb{I}_{\a}$}, so that {\small$\w_{1}$} comes before {\small$\w_{2}$}. We now compute {\small$\E_{\mathscr{F}_{\w_{1}}}({\boldsymbol{\mathcal{H}}}_{\x,\w_{2}}\cdot\wt{\mathfrak{q}}_{\w_{2}}[\bphi_{\t}]\cdot\mathbf{G}_{\t,\x,\w_{2}})$} as follows. By the tower property of conditional expectation, we condition further on {\small$\bphi_{\t,\y}$} for {\small$\y\in\llbracket\w_{2}+1,\infty\rrbracket$}; this is indeed conditioning further, since {\small$\mathscr{F}_{\w_{1}}$} conditions on {\small$\bphi_{\t,\y}$} for {\small$\y\in\llbracket\w_{1}-\mathfrak{l},\infty\rrbracket$}, and by construction, if {\small$\w_{1}>\w_{2}$} are both in {\small$\mathbb{I}_{\a}$}, then {\small$\w_{1}\geq\w_{2}+2\mathfrak{l}$}, which implies {\small$\w_{2}+1\leq\w_{1}-\mathfrak{l}$}. After we do this further conditioning, the factor {\small$\mathbf{G}_{\t,\x,\w_{2}}$} becomes deterministic. as well. However, as {\small$\wt{\mathfrak{q}}_{\w_{2}}[\bphi_{\t}]$} depends only on {\small$\bphi_{\t,\y}$} for {\small$\y\leq\w_{2}$}, and since the one-point marginals of {\small$\mathbb{P}^{0}$} are jointly independent, after all this conditioning, the law of {\small$\wt{\mathfrak{q}}_{\w_{2}}[\bphi_{\t}]$} is still as if {\small$\bphi_{\t}\sim\mathbb{P}^{0}$}. Thus, we have {\small$\E_{\mathscr{F}_{\w_{1}}}({\boldsymbol{\mathcal{H}}}_{\x,\w_{2}}\cdot\wt{\mathfrak{q}}_{\w_{2}}[\bphi_{\t}]\cdot\mathbf{G}_{\t,\x,\w_{2}})=\E_{\mathscr{F}_{\w_{1}}}({\boldsymbol{\mathcal{H}}}_{\x,\w_{2}}\cdot\mathbf{G}_{\t,\x,\w_{2}}\cdot\E\wt{\mathfrak{q}}_{\w_{2}}[\bphi_{\t}])=0$}, where the last identity follows because {\small$\wt{\mathfrak{q}}_{\w}$} is mean-zero for all {\small$\w$} by construction (recall that we assumed {\small$\bphi_{\t}\sim\mathbb{P}^{0}$} for now).

With the martingale property in hand, we also note the deterministic bound {\small$|{\boldsymbol{\mathcal{H}}}_{\x,\w}\cdot\wt{\mathfrak{q}}_{\w}[\bphi_{\t}]\cdot\mathbf{G}_{\t,\x,\w}|\lesssim\N^{\e}{\boldsymbol{\mathcal{H}}}_{\x,\w}$}. Thus, by the Azuma-Hoeffding martingale inequality, we have (for any {\small$\mathrm{K}>0$} and some {\small$\upsilon>0$}) that
\begin{align*}
\mathbb{P}^{0}\Big(\Big(\sum_{\w\in\mathbb{I}_{\a}}{\boldsymbol{\mathcal{H}}}_{\x,\w}\cdot\wt{\mathfrak{q}}_{\w}[\bphi_{\t}]\cdot\mathbf{G}_{\t,\x,\w}\Big)^{2}\geq\mathrm{K}^{2}\N^{2\e}\sum_{\w\in\mathbb{I}_{\a}}{\boldsymbol{\mathcal{H}}}_{\x,\w}^{2}\Big)\lesssim\exp(-\upsilon\mathrm{K}^{2}).
\end{align*}
Recalling estimates for {\small${\boldsymbol{\mathcal{H}}}_{\x,\w}$} from after \eqref{eq:cltxI1} and taking {\small$\mathrm{K}=\N^{\e/2}$} turns the previous display into
\begin{align*}
\mathbb{P}^{0}\Big(\Big(\sum_{\w\in\mathbb{I}_{\a}}{\boldsymbol{\mathcal{H}}}_{\x,\w}\cdot\wt{\mathfrak{q}}_{\w}[\bphi_{\t}]\cdot\mathbf{G}_{\t,\x,\w}\Big)^{2}\geq\N^{3\e}\N^{-1+\frac32\delta_{\mathbf{S}}}(\log\N)^{\mathrm{O}(1)}\Big)\lesssim\exp(-\upsilon\N^{\e})
\end{align*}
We take a union bound over all {\small$\a\in\llbracket1,2\mathfrak{l}\rrbracket$} and {\small$\x\in\llbracket0,\N^{\mathrm{D}}\rrbracket$}. (Recall that {\small$\mathfrak{l}=\mathrm{O}(1)$}.) Then we take a union bound over all {\small$\t\in[0,\mathrm{T}]$} (first in a sub-microscopic discretization of {\small$[0,\mathrm{T}]$}, then extending by short-time continuity as in the proof of Lemma \ref{lemma:phiap}). Ultimately, we obtain
\begin{align*}
\mathbb{P}^{0}\Big(\sup_{\t\in[0,\mathrm{T}]}\max_{\x\in\llbracket0,\N^{\mathrm{D}}\rrbracket}\Big|\sum_{\w\in\llbracket0,\infty\rrbracket}{\boldsymbol{\mathcal{H}}}_{\x,\w}\cdot\wt{\mathfrak{q}}_{\w}[\bphi_{\t}]\cdot\mathbf{G}_{\t,\x,\w}\Big|\geq\N^{\frac32\e}\N^{-\frac12+\frac34\delta_{\mathbf{S}}}(\log\N)^{\mathrm{O}(1)}\Big)\lesssim\exp(-\upsilon\N^{\frac12\e})
\end{align*}
Recall that this was all under the assumption that the initial data of {\small$\t\mapsto\bphi_{\t}$} has law {\small$\mathbb{P}^{0}$}. To change measures to general initial data satisfying \eqref{eq:dataI}, we use the change-of-measure argument in the proof of Lemma \ref{lemma:phiap}:
\begin{align*}
&\mathbb{P}\Big(\sup_{\t\in[0,\mathrm{T}]}\max_{\x\in\llbracket0,\N^{\mathrm{D}}\rrbracket}\Big|\sum_{\w\in\llbracket0,\infty\rrbracket}{\boldsymbol{\mathcal{H}}}_{\x,\w}\cdot\wt{\mathfrak{q}}_{\w}[\bphi_{\t}]\cdot\mathbf{G}_{\t,\x,\w}\Big|\geq\N^{\frac32\e}\N^{-\frac12+\frac34\delta_{\mathbf{S}}}(\log\N)^{\mathrm{O}(1)}\Big)\\
&\lesssim\exp(\N^{\gamma_{\mathrm{data}}})\mathbb{P}\Big(\sup_{\t\in[0,\mathrm{T}]}\max_{\x\in\llbracket0,\N^{\mathrm{D}}\rrbracket}\Big|\sum_{\w\in\llbracket0,\infty\rrbracket}\mathbf{H}_{\x,\w}\cdot\wt{\mathfrak{q}}_{\w}[\bphi_{\t}]\cdot\mathbf{G}_{\t,\x,\w}\Big|\geq\N^{\frac32\e}\N^{-\frac12+\frac34\delta_{\mathbf{S}}}(\log\N)^{\mathrm{O}(1)}\Big)^{\frac12}\lesssim\exp(-\tfrac13\upsilon\N^{\frac12\e}).
\end{align*}
Combining this with \eqref{eq:cltxI3a} (and the fact that the last term therein is {\small$\lesssim_{\mathrm{D}}\N^{-\mathrm{D}}$} deterministically for large {\small$\mathrm{D}>0$}), we deduce that the desired estimate \eqref{eq:cltxI} holds with high probability if we set {\small$\delta=2\e$} therein.
\end{proof}
We will now bootstrap the estimate \eqref{eq:cltxI} to an estimate for {\small$\mathbf{Av}^{\mathbf{T},\mathbf{X},\mathscr{Q}}$}-quantities from Definition \ref{definition:bulkaverage}. To this end, we require a time-regularity estimate for {\small$\mathbf{j}^{\N}$}.
\begin{lemma}\label{lemma:currtreg}
\fsp Fix any {\small$\mathrm{D},\mathrm{T}>0$} independent of {\small$\N$}, and recall from Definition \ref{definition:bulkaverage} that {\small$\mathfrak{t}_{\mathbf{Av}}\leq\N^{-2/3}$}. There exists {\small$\rho>0$} such that with high probability, we have 
\begin{align}
\sup_{\s\in[0,\mathfrak{t}_{\mathbf{Av}}]}\sup_{\t\in[0,\mathrm{T}]}\sup_{\x\in\llbracket0,\N^{\mathrm{D}}\rrbracket}|\mathbf{j}^{\N}_{\t+\s,\x}-\mathbf{j}^{\N}_{\t,\x}|\lesssim\N^{-\rho}.\label{eq:currtregI}
\end{align}
\end{lemma}
\begin{proof}
Set {\small$\mathfrak{l}_{\N}=\N^{1-\delta_{\mathbf{S}}}$}, with {\small$\delta_{\mathbf{S}}>0$} the small constant from Definition \ref{definition:heat}. Define the averaging operator 
\begin{align*}
\boldsymbol{\mathcal{G}}[\mathbf{j}^{\N}_{\tau,\cdot}]_{\x}:=\sum_{\w\in\llbracket0,\infty\rrbracket}\boldsymbol{\mathcal{G}}_{\x,\w}\mathbf{j}^{\N}_{\tau,\w},
\end{align*}
where {\small$\boldsymbol{\mathcal{G}}_{\x,\w}$} is a probability measure on {\small$\w\in\llbracket0,\infty\rrbracket$} for all {\small$\x\in\llbracket0,\infty\rrbracket$} that admits the size and regularity estimates {\small$|\boldsymbol{\mathcal{G}}_{\x,\w}|\lesssim\mathfrak{l}_{\N}^{-1}\mathbf{1}[|\x-\w|\lesssim\mathfrak{l}_{\N}]$} and {\small$|\grad^{\mathbf{X}}_{\mathfrak{l}}\boldsymbol{\mathcal{G}}_{\x,\w}|\lesssim\mathfrak{l}_{\N}^{-2}\mathbf{1}[|\x-\w|\lesssim\mathfrak{l}_{\N}]$} for any {\small$|\mathfrak{l}|\lesssim1$}. (The discrete gradient is allowed to act on either spatial variable.) The triangle inequality gives the following, in which we compare each of {\small$\mathbf{j}^{\N}_{\t+\s,\x}$} and {\small$\mathbf{j}^{\N}_{\t,\x}$} to their images under the {\small$\boldsymbol{\mathcal{G}}$}-operator:
\begin{align}
\sup_{\s\in[0,\mathfrak{t}_{\mathbf{Av}}]}\sup_{\t\in[0,\mathrm{T}]}\sup_{\x\in\llbracket0,\N^{\mathrm{D}}\rrbracket}|\mathbf{j}^{\N}_{\t+\s,\x}-\mathbf{j}^{\N}_{\t,\x}|&\lesssim\sup_{\s\in[0,\mathfrak{t}_{\mathbf{Av}}]}\sup_{\t\in[0,\mathrm{T}]}\sup_{\x\in\llbracket0,\N^{\mathrm{D}}\rrbracket}|\boldsymbol{\mathcal{G}}[\mathbf{j}^{\N}_{\t+\s,\cdot}]_{\x}-\boldsymbol{\mathcal{G}}[\mathbf{j}^{\N}_{\t,\cdot}]_{\x}|\label{eq:currtregI1a}\\
&+\sup_{\t\in[0,2\mathrm{T}]}\sup_{\x\in\llbracket0,\N^{\mathrm{D}}\rrbracket}|\boldsymbol{\mathcal{G}}[\mathbf{j}^{\N}_{\t,\cdot}]_{\x}-\mathbf{j}^{\N}_{\t,\x}|. \label{eq:currtregI1b}
\end{align}
Because {\small$\boldsymbol{\mathcal{G}}_{\x,\cdot}$} is a probability measure whose support length is assumed to be {\small$\lesssim\mathfrak{l}_{\N}$}, we can write {\small$|\boldsymbol{\mathcal{G}}[\mathbf{j}^{\N}_{\t,\cdot}]_{\x}-\mathbf{j}^{\N}_{\t,\x}|=|\boldsymbol{\mathcal{G}}[\mathbf{j}^{\N}_{\t,\cdot}-\mathbf{j}^{\N}_{\t,\x}]_{\x}|\lesssim\max_{|\w-\x|\lesssim\mathfrak{l}_{\N}}|\mathbf{j}^{\N}_{\t,\w}-\mathbf{j}^{\N}_{\t,\x}|$}. Since {\small$\mathbf{j}^{\N}_{\t,\w}-\mathbf{j}^{\N}_{\t,\x}=\pm\N^{-1/2}\sum_{\y\in\mathbb{I}_{\x,\w}}\bphi_{\t,\y}$}, in which {\small$\mathbb{I}_{\x,\w}$} is the interval with endpoints given by {\small$\x,\w$} (and thus has size {\small$\lesssim\mathfrak{l}_{\N}$}),  we can use Lemma \ref{lemma:phiap} (as in the proof of \eqref{eq:cltxI3a}-\eqref{eq:cltxI3b}) to obtain that the term in \eqref{eq:currtregI1b} is {\small$\lesssim\N^{\delta}\N^{-1/2}\mathfrak{l}_{\N}^{1/2}\lesssim\N^{\delta}\N^{-\delta_{\mathbf{S}}/2}\lesssim\N^{-\delta_{\mathbf{S}}/3}$} for {\small$\delta>0$} small, with high probability. So, in order to complete the proof, it is enough to show that 
\begin{align}
\sup_{\s\in[0,\mathfrak{t}_{\mathbf{Av}}]}\sup_{\t\in[0,\mathrm{T}]}\sup_{\x\in\llbracket0,\N^{\mathrm{D}}\rrbracket}|\boldsymbol{\mathcal{G}}[\mathbf{j}^{\N}_{\t+\s,\cdot}]_{\x}-\boldsymbol{\mathcal{G}}[\mathbf{j}^{\N}_{\t,\cdot}]_{\x}|\lesssim\N^{-\rho}\label{eq:currtregI2}
\end{align}
with high probability for some {\small$\rho>0$}. To this end, we use \eqref{eq:currI}-\eqref{eq:currII} to compute
\begin{align*}
\d\boldsymbol{\mathcal{G}}[\mathbf{j}^{\N}_{\t,\cdot}]_{\x}&=\boldsymbol{\mathcal{G}}[\mathbf{1}_{\cdot\in\llbracket1,\infty\rrbracket}\N^{\frac32}\grad^{\mathbf{X}}_{1}\mathscr{U}'[\bphi_{\t,\cdot}]]_{\x}\d\t+\boldsymbol{\mathcal{G}}[\N\mathbf{F}[\tau_{\cdot}\bphi_{\t}]]_{\x}\d\t+\boldsymbol{\mathcal{G}}_{\x,0}\cdot\N^{\frac32}\mathscr{U}'[\bphi_{\t,1}]\d\t+\boldsymbol{\mathcal{G}}[\sqrt{2}\N^{\frac12}\d\mathbf{b}_{\t,\cdot}]_{\x}.
\end{align*}
Now, we use summation-by-parts on {\small$\llbracket0,\infty\rrbracket$} as in the proof of Lemma \ref{lemma:sbp} to get
\begin{align*}
\boldsymbol{\mathcal{G}}[\mathbf{1}_{\cdot\in\llbracket1,\infty\rrbracket}\N^{\frac32}\grad^{\mathbf{X}}_{1}\mathscr{U}'[\bphi_{\t,\cdot}]]_{\x}=\sum_{\w\in\llbracket2,\infty\rrbracket}\grad^{\mathbf{X}}_{-1}\boldsymbol{\mathcal{G}}_{\x,\w}\cdot\N^{\frac32}\mathscr{U}'[\bphi_{\t,\w}]-\boldsymbol{\mathcal{G}}_{\x,1}\cdot\N^{\frac32}\mathscr{U}'[\bphi_{\t,1}].
\end{align*}
We clarify that the last term is troublesome on its own because of the highly singular factor of {\small$\N^{3/2}$}. However, the Neumann cancellation discussed in Section \ref{section:method} appears here once again. Indeed, when we add {\small$\boldsymbol{\mathcal{G}}_{\x,0}\cdot\N^{\frac32}\mathscr{U}'[\bphi_{\t,1}]$} to both sides of the previous display, the \abbr{RHS} changes by dropping the last term and changing summation indices from {\small$\w\in\llbracket2,\infty\rrbracket$} to {\small$\w\in\llbracket1,\infty\rrbracket$}. That is, we have 
\begin{align*}
\boldsymbol{\mathcal{G}}[\mathbf{1}_{\cdot\in\llbracket1,\infty\rrbracket}\N^{\frac32}\grad^{\mathbf{X}}_{1}\mathscr{U}'[\bphi_{\t,\cdot}]]_{\x}+\boldsymbol{\mathcal{G}}_{\x,0}\cdot\N^{\frac32}\mathscr{U}'[\bphi_{\t,1}]=\sum_{\w\in\llbracket1,\infty\rrbracket}\grad^{\mathbf{X}}_{-1}\boldsymbol{\mathcal{G}}_{\x,\w}\cdot\N^{\frac32}\mathscr{U}'[\bphi_{\t,\w}].
\end{align*}
Plugging this back into the \abbr{SDE} for {\small$\boldsymbol{\mathcal{G}}[\mathbf{j}^{\N}_{\t,\cdot}]_{\x}$} and integrating in time shows that 
\begin{align}
\sup_{\s\in[0,\mathfrak{t}_{\mathbf{Av}}]}\sup_{\t\in[0,\mathrm{T}]}\sup_{\x\in\llbracket0,\N^{\mathrm{D}}\rrbracket}|\boldsymbol{\mathcal{G}}[\mathbf{j}^{\N}_{\t+\s,\cdot}]_{\x}-\boldsymbol{\mathcal{G}}[\mathbf{j}^{\N}_{\t,\cdot}]_{\x}|&\lesssim\sup_{\s\in[0,\mathfrak{t}_{\mathbf{Av}}]}\sup_{\t\in[0,\mathrm{T}]}\sup_{\x\in\llbracket0,\N^{\mathrm{D}}\rrbracket}(\mathrm{I}_{\s,\t,\x}+\mathrm{II}_{\s,\t,\x}+\mathrm{III}_{\s,\t,\x}),\label{eq:currtregI3a}
\end{align}
where we have introduced the notation
\begin{align}
\mathrm{I}_{\s,\t,\x}&:=\Big|\int_{0}^{\s}\sum_{\w\in\llbracket1,\infty\rrbracket}\grad^{\mathbf{X}}_{-1}\boldsymbol{\mathcal{G}}_{\x,\w}\cdot\N^{\frac32}\mathscr{U}'[\bphi_{\t+\r,\w}]\d\r\Big|,\label{eq:currtregI3b}\\
\mathrm{II}_{\s,\t,\x}&:=\Big|\int_{0}^{\s}\sum_{\w\in\llbracket0,\infty\rrbracket}\boldsymbol{\mathcal{G}}_{\x,\w}\cdot\N\mathbf{F}[\tau_{\w}\bphi_{\t+\r}]\d\r\Big|,\label{eq:currtregI3c}\\
\mathrm{III}_{\s,\t,\x}&:=\Big|\int_{0}^{\s}\sum_{\w\in\llbracket0,\infty\rrbracket}\boldsymbol{\mathcal{G}}_{\x,\w}\cdot\N^{\frac12}\d\mathbf{b}_{\t+\r,\w}\Big|.\label{eq:currtregI3d}
\end{align}
We will now control each term above. By Assumption \ref{assump:potential}, we know that {\small$|\mathscr{U}'[\bphi_{\t+\r,\w}]|\lesssim1+|\bphi_{\t+\r,\w}|$}. Thus, by Lemma \ref{lemma:phiap}, there is an event of high probability such that for all {\small$\t\in[0,\mathrm{T}],\s\in[0,\mathfrak{t}_{\mathbf{Av}}],\r\in[0,\s],\x\in\llbracket0,\N^{\mathrm{D}}\rrbracket$}, and {\small$\w\in\llbracket1,\infty\rrbracket$} for which {\small$\grad^{\mathbf{X}}_{-1}\boldsymbol{\mathcal{G}}_{\x,\w}\neq0$}, we have {\small$|\mathscr{U}'[\bphi_{\t+\r,\w}]|\lesssim\N^{\delta}$} for {\small$\delta>0$} small. We now combine this with the estimate {\small$|\grad^{\mathbf{X}}_{-1}\boldsymbol{\mathcal{G}}_{\x,\w}|\lesssim\mathfrak{l}_{\N}^{-2}\mathbf{1}[|\x-\w|\lesssim\mathfrak{l}_{\N}]$} and {\small$\mathfrak{l}_{\N}=\N^{1-\delta_{\mathbf{S}}}$} to deduce that {\small$|\mathrm{I}_{\s,\t,\x}|\lesssim\mathfrak{t}_{\mathbf{Av}}\mathfrak{l}_{\N}^{-1}\N^{3/2}\lesssim\N^{-1/6+\delta_{\mathbf{S}}}$}, since {\small$\mathfrak{t}_{\mathbf{Av}}\lesssim\N^{-2/3}$}; this bound holds with high probability for all {\small$(\s,\t,\x)\in[0,\mathfrak{t}_{\mathbf{Av}}]\times[0,\mathrm{T}]\times\llbracket0,\N^{\mathrm{D}}\rrbracket$} simultaneously. Let us now control \eqref{eq:currtregI3d}. For fixed {\small$(\s,\t,\x)$}, the term inside the absolute value in \eqref{eq:currtregI3d} is Gaussian with mean {\small$0$} (since the {\small$\boldsymbol{\mathcal{G}}$}-kernel is deterministic) and variance {\small$\lesssim\s\cdot\N\cdot\|\boldsymbol{\mathcal{G}}_{\x,\cdot}\|_{\ell^{2}(\llbracket0,\infty\rrbracket)}^{2}\lesssim\mathfrak{t}_{\mathbf{Av}}\cdot\N\cdot\mathfrak{l}_{\N}^{-1}\lesssim\N^{-2/3+\delta_{\mathbf{S}}}$}. Thus, for any {\small$(\s,\t,\x)$}, we know {\small$|\mathrm{III}_{\s,\t,\x}|\lesssim\N^{-1/2}$} with probability {\small$\lesssim\exp(-\upsilon\N^{\omega})$} for some {\small$\omega,\upsilon>0$} by standard Gaussian concentration inequalities. Taking a union bound over all {\small$\x\in\llbracket0,\N^{\mathrm{D}}\rrbracket$} and then over all {\small$(\s,\t)\in[0,\mathfrak{t}_{\mathbf{Av}}]\times[0,\mathrm{T}]$}, first over a very fine discretization and then extending by short-time continuity as in the proof of Lemma \ref{lemma:phiap}, we then deduce control on {\small$|\mathrm{III}_{\s,\t,\x}|$} with high probability simultaneously for all {\small$(\s,\t,\x)\in[0,\mathfrak{t}_{\mathbf{Av}}]\times[0,\mathrm{T}]\times\llbracket0,\N^{\mathrm{D}}\rrbracket$}. As a conclusion of this paragraph, we deduce (for {\small$\delta_{\mathbf{S}}>0$} small enough)
\begin{align}
\sup_{\s\in[0,\mathfrak{t}_{\mathbf{Av}}]}\sup_{\t\in[0,\mathrm{T}]}\sup_{\x\in\llbracket0,\N^{\mathrm{D}}\rrbracket}(\mathrm{I}_{\s,\t,\x}+\mathrm{III}_{\s,\t,\x})\lesssim\N^{-\frac17}.\label{eq:currtregI4}
\end{align}
We are left with \eqref{eq:currtregI3c}. For this, let us use a similar argument as in the proof of \eqref{eq:cltxI2}. In particular, by \eqref{eq:nl} and {\small$\E^{0}\mathscr{U}'[\bphi_{\x}]=0$} for any {\small$\x\in\llbracket1,\infty\rrbracket$}, we have {\small$\E^{0}\mathbf{F}[\tau_{\w}\bphi]=0$} for {\small$\w\in\llbracket0,\infty\rrbracket$}. (Recall also {\small$\mathscr{U}'[\bphi_{\x}]=0$} by convention if {\small$\x\leq0$}.) Moreover, because {\small$\mathscr{U}'$} is uniformly Lipschitz, the formula \eqref{eq:nl} implies {\small$\mathbf{F}[\tau_{\w}\bphi]\lesssim\N^{\mathrm{C}\e}$} for some {\small$\mathrm{C}\lesssim1$} if {\small$|\bphi_{\y}|\lesssim\N^{\e}$} for all {\small$|\y-\w|\lesssim1$}. Thus, by Lemma \ref{lemma:phiap}, we can replace {\small$\mathbf{F}[\tau_{\w}\bphi_{\t+\r}]$} by {\small$\mathbf{F}[\tau_{\w}\bphi_{\t+\r}]\mathbf{1}(|\mathbf{F}[\tau_{\w}\bphi_{\t+\r}]|\lesssim\N^{\mathrm{C}\e})$} with {\small$\e>0$} small but fixed; this replacement holds with high probability simultaneously for all {\small$(\s,\t,\x,\w)$}-points at hand. Next, we claim that {\small$|\E^{0}\mathbf{F}[\tau_{\w}\bphi]\mathbf{1}(|\mathbf{F}[\tau_{\w}\bphi]|\lesssim\N^{\mathrm{C}\e})|\lesssim_{\mathrm{D}}\N^{-\mathrm{D}}$} for any {\small$\mathrm{D}>0$}. To see this, we use {\small$\E^{0}\mathbf{F}=0$} once again to get {\small$|\E^{0}\mathbf{F}[\tau_{\w}\bphi]\mathbf{1}(|\mathbf{F}[\tau_{\w}\bphi]|\lesssim\N^{\mathrm{C}\e})|=|\E^{0}\mathbf{F}[\tau_{\w}\bphi]\mathbf{1}(|\mathbf{F}[\tau_{\w}\bphi]|\gtrsim\N^{\mathrm{C}\e})|$}, at which point we can use the Cauchy-Schwarz inequality to get {\small$|\E^{0}\mathbf{F}[\tau_{\w}\bphi]\mathbf{1}(|\mathbf{F}[\tau_{\w}\bphi]|\gtrsim\N^{\mathrm{C}\e})|\lesssim(\E^{0}\mathbf{F}[\tau_{\w}\bphi]^{2})^{1/2}\mathbb{P}^{0}(|\mathbf{F}[\tau_{\w}\bphi]|\gtrsim\N^{\mathrm{C}\e})^{1/2}$}. Now recall \eqref{eq:nl} and that {\small$\mathscr{U}'$} is uniformly Lipschitz by Assumption \ref{assump:potential}, so that {\small$\E^{0}\mathbf{F}[\tau_{\w}\bphi]^{2}\lesssim1$}, since {\small$\bphi_{\x}$} are uniformly sub-Gaussian under {\small$\mathbb{P}^{0}$}. The same reasoning gives {\small$\mathbb{P}^{0}(|\mathbf{F}[\tau_{\w}\bphi]|\gtrsim\N^{\mathrm{C}\e})\leq\mathbb{P}^{0}(\max_{|\y-\w|\lesssim1}|\bphi_{\y}|\gtrsim\N^{\mathrm{c}\e})\lesssim_{\mathrm{D}}\N^{-\mathrm{D}}$} for any {\small$\mathrm{D}>0$}, where {\small$\mathrm{c}>0$} is a fixed constant. The claim thus follows. In particular, with high probability over all {\small$(\s,\t,\x)\in[0,\mathfrak{t}_{\mathbf{Av}}]\times[0,\mathrm{T}]\times\llbracket0,\N^{\mathrm{D}}\rrbracket$}, we ultimately deduce that
\begin{align}
\mathrm{II}_{\s,\t,\x}\lesssim_{\mathrm{D}}\Big|\int_{0}^{\s}\sum_{\w\in\llbracket0,\infty\rrbracket}\boldsymbol{\mathcal{G}}_{\x,\w}\cdot\N\mathfrak{a}_{\w}[\bphi_{\t+\r}]\d\r\Big|+\N^{-\mathrm{D}}\label{eq:currtregI5}
\end{align}
for any {\small$\mathrm{D}>0$}, in which {\small$\mathfrak{a}_{\w}:\R^{\llbracket1,\infty\rrbracket}\to\R$} satisfies {\small$\E^{0}\mathfrak{a}_{\w}=0$} and {\small$|\mathfrak{a}_{\w}|\lesssim\N^{\mathrm{C}\e}$} for {\small$\mathrm{C}\lesssim1$} and {\small$\e>0$} small. Moreover, we know that {\small$\mathfrak{a}_{\w}[\bphi]$} depends only on {\small$\bphi_{\y}$} with {\small$|\y-\w|\lesssim1$}, a consequence of locality of the function {\small$\mathbf{F}$} (see \eqref{eq:nl}). By the same martingale argument given in the proof of Lemma \ref{lemma:cltx} to bound the first term on the \abbr{RHS} of \eqref{eq:cltxI3a}, we then deduce that for any {\small$\t\in[0,\mathrm{T}],\s\in[0,\mathfrak{t}_{\mathbf{Av}}],\r\in[0,\s]$}, and {\small$\x\in\llbracket0,\N^{\mathrm{D}}\rrbracket$}, we have
\begin{align*}
\mathbb{P}^{0}\Big(\Big|\sum_{\w\in\llbracket0,\infty\rrbracket}\boldsymbol{\mathcal{G}}_{\x,\w}\cdot\N\mathfrak{a}_{\w}[\bphi_{\t+\r}]\Big|\geq\mathrm{K}\N^{1+\mathrm{C}\e}\|\boldsymbol{\mathcal{G}}_{\x,\cdot}\|_{\ell^{2}(\llbracket0,\infty\rrbracket)}\Big)\lesssim\exp(-\upsilon\mathrm{K}^{2})
\end{align*}
for any {\small$\mathrm{K}>0$} and some {\small$\upsilon>0$}. Take {\small$\mathrm{K}=\N^{\e}$}, so that the \abbr{RHS} of the previous estimate is exponentially small in {\small$\N$}. We take a union bound over all {\small$\x\in\llbracket0,\N^{\mathrm{D}}\rrbracket$}, and all {\small$\t\in[0,\mathrm{T}],\s\in[0,\mathfrak{t}_{\mathbf{Av}}],\r\in[0,\s]$} (first over a fine discretization of these time-intervals and then extending to the full time-interval with another short-time continuity argument). This, along with {\small$\|\boldsymbol{\mathcal{G}}_{\x,\cdot}\|_{\ell^{2}(\llbracket0,\infty\rrbracket)}\lesssim\mathfrak{l}_{\N}^{-1/2}\lesssim\N^{-1/2+\delta_{\mathbf{S}}/2}$}, yields
\begin{align*}
\mathbb{P}^{0}\Big(\sup_{\s\in[0,\mathfrak{t}_{\mathbf{Av}}]}\sup_{\t\in[0,\mathrm{T}]}\sup_{\r\in[0,\s]}\max_{\x\in\llbracket0,\N^{\mathrm{D}}\rrbracket}\Big|\sum_{\w\in\llbracket0,\infty\rrbracket}\boldsymbol{\mathcal{G}}_{\x,\w}\cdot\N\mathfrak{a}_{\w}[\bphi_{\t+\r}]\Big|\gtrsim\N^{1+\mathrm{C}\e+\e}\N^{-\frac12+\frac12\delta_{\mathbf{S}}}\Big)\lesssim\exp(-\upsilon\N^{\frac12\e}).
\end{align*}
By the change-of-measure argument at the end of the proof of Lemma \ref{lemma:phiap}, we deduce that the same holds if we replace {\small$\mathbb{P}^{0}$} on the \abbr{LHS} by {\small$\mathbb{P}$} and {\small$\upsilon$} by {\small$\upsilon/2$} above. If we combine this with \eqref{eq:currtregI5} and recall that {\small$\mathfrak{t}_{\mathbf{Av}}\lesssim\N^{-2/3}$}, then with high probability, we obtain the following for any {\small$\mathrm{D}>0$} if {\small$\e,\delta_{\mathbf{S}}>0$} are small:
\begin{align*}
\sup_{\s\in[0,\mathfrak{t}_{\mathbf{Av}}]}\sup_{\t\in[0,\mathrm{T}]}\sup_{\x\in\llbracket0,\N^{\mathrm{D}}\rrbracket}|\mathrm{II}_{\s,\t,\x}|\lesssim_{\mathrm{D}}\mathfrak{t}_{\mathbf{Av}}\N^{1+\mathrm{C}\e+\e}\N^{-\frac12+\frac12\delta_{\mathbf{S}}}+\N^{-\mathrm{D}}\lesssim\N^{-\frac18}.
\end{align*}
Combining this with \eqref{eq:currtregI4} and \eqref{eq:currtregI3a} yields \eqref{eq:currtregI2}, which finishes the proof (as noted before \eqref{eq:currtregI2}).
\end{proof}
We now upgrade \eqref{eq:cltxI} to an estimate on {\small$\mathbf{Av}^{\mathbf{T},\mathbf{X},\mathscr{Q}}$}-space-time averages. We clarify that the result below does \emph{not} produce additional cancellations in time-averaging; it essentially integrates the estimate \eqref{eq:cltxI} in time.
\begin{lemma}\label{lemma:cltxt?}
\fsp Take the setting of Lemma \ref{lemma:cltx}. The estimate \eqref{eq:cltxI} is true if we replace {\small$\mathbf{Av}^{\mathbf{X},\mathscr{Q}}$} with {\small$\mathbf{Av}^{\mathbf{T},\mathbf{X},\mathscr{Q}}$}. Also, for any edge-admissible {\small$\mathfrak{b}$}, we have (with high probability) that {\small$\max_{\t\in[0,\mathrm{T}]}\max_{\x\in\llbracket0,\mathrm{O}(1)\rrbracket}|\mathbf{Av}^{\mathbf{T},\mathfrak{b}}_{\t,\x}|\lesssim\N^{\delta}$}.
\end{lemma}
\begin{proof}
By construction in Definition \ref{definition:bulkaverage} and \eqref{eq:ch}, we have 
\begin{align}
\mathbf{Av}^{\mathbf{T},\mathbf{X},\mathscr{Q}}_{\t,\x}&=\mathfrak{t}_{\mathbf{Av}}^{-1}{\int_{0}^{\mathfrak{t}_{\mathbf{Av}}}}\mathbf{Av}^{\mathbf{X},\mathscr{Q}}_{\t-\s,\x}\cdot\exp(\beta_{2}(\mathbf{j}^{\N}_{\t-\s,\x_{\pm}}-\mathbf{j}^{\N}_{\t,\x_{\pm}})+\lambda\mathscr{R}_{\lambda}\s)\d\s.\nonumber
\end{align}
By Lemma \ref{lemma:currtreg} and {\small$\mathscr{R}_{\lambda}\lesssim1$}, with high probability the exponential on the \abbr{RHS} is {\small$\lesssim1$}. Now use \eqref{eq:cltxI} to obtain the desired estimate on {\small$\mathbf{Av}^{\mathbf{T},\mathbf{X},\mathscr{Q}}$}. For {\small$\mathbf{Av}^{\mathbf{T},\mathfrak{b}}$}, the same argument works, but we replace {\small$\mathbf{Av}^{\mathbf{X},\mathscr{Q}}_{\t-\s}$} in the line above by {\small$\mathfrak{b}[\bphi_{\t-\s}]$}, which is {\small$\lesssim\N^{\mathrm{C}\delta}$} for all {\small$\t-\s$}-values of interest simultaneously with high probability. This follows by the polynomial bound \eqref{eq:bulkadmissibleI} for edge-admissible functions (see Definition \ref{definition:edgeadmissible}) and Lemma \ref{lemma:phiap}.
\end{proof}
\section{Stochastic estimates II: Boltzmann-Gibbs estimates}\label{section:stochII}
We will now establish estimates for the {\small$\mathbf{Av}^{\mathbf{T},\mathbf{X},\mathscr{Q}}$}- and {\small$\mathbf{Av}^{\mathbf{T},\mathfrak{b}}$}-terms from Definitions \ref{definition:bulkaverage} and \ref{definition:edgeaverage}, respectively, and that appear in \eqref{eq:sdeI1}-\eqref{eq:sdeI5}. The nature of these estimates is substantially different and more difficult than not only what was done in Section \ref{section:stochI}. See Section \ref{section:method} for an intuitive discussion of what is done in this section.

Ultimately, the results to be established in this section are as follows. (We state both before we prove either.) We start with an estimate for {\small$\mathbf{Av}^{\mathbf{T},\mathbf{X},\mathscr{Q}}$}-terms from Definition \ref{definition:bulkaverage}.
\begin{prop}\label{prop:stochmain1}
\fsp Fix any {\small$\mathrm{D},\mathrm{T}>0$}, as well as any collection {\small$\mathscr{Q}$} as in Lemma \ref{lemma:cltx}. There is a fixed {\small$\beta_{\mathrm{CLT}}>0$} and a path-space event {\small$\mathscr{E}$} of high probability such that the following estimate holds:
\begin{align}
\sup_{\x\in\llbracket0,\N^{\mathrm{D}}\rrbracket}\int_{0}^{\mathrm{T}}\E(\mathbf{1}_{\mathscr{E}}|\mathbf{Av}^{\mathbf{T},\mathbf{X},\mathscr{Q}}_{\t,\x}|^{2})\d\t\lesssim\N^{-2-\beta_{\mathrm{CLT}}}.\label{eq:stochmain1I}
\end{align}
We clarify that the event {\small$\mathscr{E}$} is independent of the {\small$(\mathrm{T},\t,\x)$}-variables on the \abbr{LHS}.
\end{prop}
We clarify that \eqref{eq:stochmain1I} is an estimate for space-time averages with respect to a single space-time scale which is specified in Definition \ref{definition:bulkaverage}. However, in the proof of Proposition \ref{prop:stochmain1}, we will need to decompose this space-time average according to another scale parameter that encodes the distance from the edge {\color{black}as we discussed in Section \ref{section:method}}. The second main result is for the {\small$\mathbf{Av}^{\mathbf{T},\mathfrak{b}}$}-terms from Definition \ref{definition:edgeaverage}.
\begin{prop}\label{prop:stochmain2}
\fsp Fix any {\small$\mathrm{T}>0$}, as well as any edge-admissible function {\small$\mathfrak{b}$} (defined in Definition \ref{definition:edgeadmissible}). There is a fixed constant {\small$\beta_{\partial}>0$} and a path-space event {\small$\mathscr{E}_{\partial}$} of high probability such that
\begin{align}
\max_{\x\in\llbracket0,\mathrm{O}(1)\rrbracket}\int_{0}^{\mathrm{T}}\E(\mathbf{1}_{\mathscr{E}_{\partial}}|\mathbf{Av}^{\mathbf{T},\mathfrak{b}}_{\t,\x}|^{2})\d\t\lesssim\N^{-\beta_{\partial}}.\label{eq:stochmain2I}
\end{align}
We clarify that the event {\small$\mathscr{E}_{\partial}$} is independent of the {\small$(\mathrm{T},\t,\x)$}-variables on the \abbr{LHS}.
\end{prop}
The rest of this section is dedicated to the proofs of Propositions \ref{prop:stochmain1} and \ref{prop:stochmain2}. However, since these arguments have no technical bearing on the rest of the paper (beyond the results in Propositions \ref{prop:stochmain1} and \ref{prop:stochmain2} they produce), the reader is welcome to skip to Section \ref{section:hkzeta}. Also, see Section \ref{subsection:stochmain1} for the proof of Proposition \ref{prop:stochmain1}, and Section \ref{subsection:stochmain2} for the proof of Proposition \ref{prop:stochmain2}.

\emph{Finally, we before we start with proofs, we declare that throughout the rest of this section, whenever we refer to a statement holding with high probability, we mean that the high probability event on which it holds does not depend on {\small$(\t,\x)\in[0,\mathrm{T}]\times\llbracket0,\N^{\mathrm{D}}\rrbracket$}-coordinates in \eqref{eq:stochmain1I}; thus the event at hand has an indicator function that can be absorbed by either {\small$\mathbf{1}_{\mathscr{E}}$} or {\small$\mathbf{1}_{\mathscr{E}_{\partial}}$}, depending on if we are proving Proposition \ref{prop:stochmain1} or \ref{prop:stochmain2}.} We will try to be as clear as possible about this below, but make this declaration transparent here anyway to avoid potential confusions.
\subsection{Argument for Proposition \ref{prop:stochmain1}}\label{subsection:stochmain1}
This subsection is structured to present key ingredients needed to establish Proposition \ref{prop:stochmain1}, whose proofs are then deferred to the end of this subsection. We chose this format to help clarify the architecture of the argument and avoid obstructing it with technical details.

\emph{Throughout the argument below, we will often denote small constants by {\small$\e,\delta,\gamma,\ldots$} and similar notation. We will adopt the convention that {\small$\delta_{\mathbf{S}}>0$} (see Definition \ref{definition:heat}) is the smallest parameter. After that, {\small$\e$} is the ``smallest" parameter, while others are chosen in terms of {\small$\e,\delta_{\mathbf{S}}$}. However, they are all positive, and none of them are {\small$\mathrm{o}(1)$}.} 

We start with preliminary estimates for and modifications of the {\small$\mathbf{Av}^{\mathbf{T},\mathbf{X},\mathscr{Q}}$}-quantity on the \abbr{LHS} of \eqref{eq:stochmain2I}. First, we assume the collection {\small$\mathscr{Q}$} is of backwards-oriented bulk-admissible functions. For forwards-oriented ones, the same argument applies upon reflection. Recall from Definition \ref{definition:bulkaverage} that
\begin{align}
\mathbf{Av}^{\mathbf{T},\mathbf{X},\mathscr{Q}}_{\t,\x}=\mathfrak{t}_{\mathbf{Av}}^{-1}\int_{0}^{\mathfrak{t}_{\mathbf{Av}}}\sum_{\w\in\mathbb{I}_{\x}}\boldsymbol{\mathcal{H}}_{\x,\w}\cdot\mathfrak{q}_{\w}[\bphi_{\t-\s}]\cdot\mathbf{Z}^{\N}_{\t-\s,\w}(\mathbf{Z}^{\N}_{\t,\x_{+}})^{-1}\d\s,\nonumber
\end{align}
Let us now clarify the notation above.
\begin{itemize}
\item Here, the {\small$\mathfrak{q}_{\w}$} are bulk-admissible and backwards-oriented (see Definition \ref{definition:bulkadmissible}).
\item The set {\small$\mathbb{I}_{\x}\subseteq\llbracket0,\infty\rrbracket$} is an interval which contains {\small$\x$} and has size {\small$\lesssim\N^{1-3\delta_{\mathbf{S}}/2}(\log\N)^{\mathrm{O}(1)}$}.
\item The point {\small$\x_{+}$} is strictly to the right of {\small$\mathbb{I}_{\x}$}, and that {\small$|\x_{+}-\w|\lesssim(\log\N)^{3}\N^{1-3\delta_{\mathbf{S}}/2}$} for any {\small$\w\in\mathbb{I}_{\x}$}. 
\item We recall that {\small$\mathfrak{t}_{\mathbf{Av}}=\N^{-2/3-\mathrm{C}_{0}\delta_{\mathbf{S}}}$} for {\small$\mathrm{C}_{0}>0$} large but fixed, again by Definition \ref{definition:bulkaverage}.
\item Lastly, {\small$\boldsymbol{\mathcal{H}}_{\x,\w}$}, as the {\small$\mathbf{H}^{\N,0}$}-kernel at time {\small$\N^{-2}\N^{2-3\delta_{\mathbf{S}}}=\N^{-3\delta_{\mathbf{S}}}$} (see Definition \ref{definition:bulkaverage}), satisfies the pointwise and regularity bounds below for any {\small$\kappa>0$} by Proposition \ref{prop:hkestimates}, where {\small$\a\lesssim1$} is a non-negative integer, where {\small$\mathfrak{l}_{1},\ldots,\mathfrak{l}_{\a}\lesssim1$} are non-negative integers, and where {\small$\mathrm{C}\lesssim1$}:
\begin{align}
\Big|\Big(\prod_{\ell\in\llbracket1,\a\rrbracket}\grad^{\mathbf{X}}_{\mathfrak{l}_{\ell}}\Big)\boldsymbol{\mathcal{H}}_{\x,\w}\Big|\lesssim_{\kappa}\N^{-1-\a}\N^{\mathrm{C}\a\delta_{\mathbf{S}}}\exp(-\tfrac{\kappa|\x-\w|}{\N}).\label{eq:hestimate}
\end{align}
\end{itemize}
To start, let us modify the {\small$(\mathbf{Z}^{\N})^{-1}$}-factor above. First, we write {\small$(\mathbf{Z}^{\N}_{\t,\x_{+}})^{-1}=(\mathbf{Z}^{\N}_{\t-\mathfrak{t}_{\mathbf{Av}},\x_{+}})^{-1}\cdot\mathbf{Z}^{\N}_{\t-\mathfrak{t}_{\mathbf{Av}},\x_{+}}(\mathbf{Z}^{\N}_{\t,\x_{+}})^{-1}$}. Then, we apply the formula \eqref{eq:ch} for {\small$\mathbf{Z}^{\N}$} to write {\small$\mathbf{Z}^{\N}_{\t-\mathfrak{t}_{\mathbf{Av}},\x_{+}}(\mathbf{Z}^{\N}_{\t,\x_{+}})^{-1}=\exp\{\lambda(\mathbf{j}^{\N}_{\t-\mathfrak{t}_{\mathbf{Av}},\x_{+}}-\mathbf{j}^{\N}_{\t,\x_{+}})+\lambda\mathscr{R}_{\lambda}\mathfrak{t}_{\mathbf{Av}}\}$}. By \eqref{eq:rg} (and various a priori estimates on the potential {\small$\mathscr{U}$} from Assumption \ref{assump:potential}), we know that {\small$\lambda\mathscr{R}_{\lambda}\lesssim1$}, and thus {\small$\lambda\mathscr{R}_{\lambda}\mathfrak{t}_{\mathbf{Av}}\lesssim\N^{-\rho}$}. Moreover, by the time-regularity estimate in Lemma \ref{lemma:currtreg} for {\small$\mathbf{j}^{\N}$}, we get {\small$|\lambda(\mathbf{j}^{\N}_{\t-\mathfrak{t}_{\mathbf{Av}},\x_{+}}-\mathbf{j}^{\N}_{\t,\x_{+}})|\lesssim\N^{-\rho}$} with high probability (simultaneously over all {\small$(\t,\x)\in[0,\mathrm{T}]\times\llbracket0,\N^{\mathrm{D}}\rrbracket$} if {\small$\mathrm{T},\mathrm{D}\lesssim1$}). Thus, on this same high probability event, we have {\small$\mathbf{Z}^{\N}_{\t-\mathfrak{t}_{\mathbf{Av}},\x_{+}}(\mathbf{Z}^{\N}_{\t,\x_{+}})^{-1}\lesssim1$}. Ultimately, this yields
\begin{align}
|\mathbf{Av}^{\mathbf{T},\mathbf{X},\mathscr{Q}}_{\t,\x}|&\lesssim\Big|\mathfrak{t}_{\mathbf{Av}}^{-1}\int_{0}^{\mathfrak{t}_{\mathbf{Av}}}\sum_{\w\in\mathbb{I}_{\x}}\boldsymbol{\mathcal{H}}_{\x,\w}\cdot\mathfrak{q}_{\w}[\bphi_{\t-\s}]\cdot\mathbf{Z}^{\N}_{\t-\s,\w}(\mathbf{Z}^{\N}_{\t-\mathfrak{t}_{\mathbf{Av}},\x_{+}})^{-1}\d\s\Big|\nonumber\\
&=\Big|\mathfrak{t}_{\mathbf{Av}}^{-1}\int_{-\mathfrak{t}_{\mathbf{Av}}}^{0}\sum_{\w\in\mathbb{I}_{\x}}\boldsymbol{\mathcal{H}}_{\x,\w}\cdot\mathfrak{q}_{\w}[\bphi_{\t+\s}]\cdot\mathbf{Z}^{\N}_{\t+\s,\w}(\mathbf{Z}^{\N}_{\t-\mathfrak{t}_{\mathbf{Av}},\x_{+}})^{-1}\d\s\Big|,\nonumber
\end{align}
where the second line follows from a change-of-variables in the {\small$\d\s$}-integration (whose role is only to make it so that we look forwards in time). We clarify that, as in Definition \ref{definition:bulkaverage}, we drop all {\small$\s\in[-\mathfrak{t}_{\mathbf{Av}},0]$} for which {\small$\t+\s<0$}, as the processes above are not defined at negative times. \emph{Throughout the rest of this section, we can freely change {\small$\t-\mathfrak{t}_{\mathbf{Av}}$} in the last factor above to any other time in {\small$\t+[-\mathfrak{t}_{\mathbf{Av}},0]$} at the cost of a {\small$\mathrm{O}(1)$} multiplicative factor, justified by the same argument above}. 
\subsubsection{Multi-scale decompositions}
The next step that we will take is to further decompose the summation over {\small$\mathbb{I}_{\x}$} and the integration over {\small$[-\mathfrak{t}_{\mathbf{Av}},0]$} via the construction below. (See Section \ref{subsection:method2} for an intuitive discussion.)
\begin{notation}\label{notation:msblock}
\fsp Write {\small$\mathbb{I}_{\x}=\mathbb{I}_{\x,1}\cup\ldots\cup\mathbb{I}_{\x,\mathrm{K}}$}, where {\small$\mathbb{I}_{\x,1},\ldots,\mathbb{I}_{\x,\mathrm{K}}$} are pairwise disjoint intervals and satisfy:
\begin{enumerate}
\item First, we have {\small$\max\mathbb{I}_{\x,\k}+1=\min\mathbb{I}_{\x,\k+1}$} for any {\small$\k=1,\ldots,\mathrm{K}-1$}. 
\item Second, we have {\small$|\mathbb{I}_{\x,\k}|\asymp\N^{\gamma_{\k}}$}, where {\small$\gamma_{1}>0$} is small but fixed, and {\small$\gamma_{\k+1}=\gamma_{\k}+\e$} for any {\small$\k=1,\ldots,\mathrm{K}-2$} and for some small {\small$\e>0$}. The index {\small$\mathrm{K}$} is determined so that {\small$\gamma_{\mathrm{K}-1}\geq\frac12+\gamma$} for another small but fixed {\small$\gamma>0$}. (We emphasize that we necessarily have {\small$|\mathbb{I}_{\x,\mathrm{K}}|\asymp|\mathbb{I}_{\x}|$}, since {\small$|\mathbb{I}_{\x}|\gg\N^{\gamma_{\k}}$} for any {\small$\k\in\llbracket1,\mathrm{K}-1\rrbracket$}, and moreover {\small$\mathbb{I}_{\x,\mathrm{K}}:=\mathbb{I}_{\x}\setminus(\mathbb{I}_{\x,1}\cup\ldots\cup\mathbb{I}_{\x,\mathrm{K}-1})$}, so that {\small$\mathbb{I}_{\x,\mathrm{K}}$} needs to make up the ``majority" of {\small$\mathbb{I}_{\x}$}.)
\end{enumerate}
Next, for any {\small$\k\in\llbracket1,\mathrm{K}\rrbracket$}, write {\small$[-\mathfrak{t}_{\mathbf{Av}},0]=\mathbb{U}_{\k,1}\cup\ldots\cup\mathbb{U}_{\k,\mathrm{L}}$}, where {\small$\mathbb{U}_{\k,1},\ldots,\mathbb{U}_{\k,\mathrm{L}}$} are pairwise disjoint intervals and satisfy the following conditions (to be explained below):
\begin{enumerate}
\item First, the length {\small$|\mathbb{U}_{\k,\ell}|$} is independent of {\small$\ell\in\llbracket1,\mathrm{L}\rrbracket$}.
\item Second, we have {\small$|\mathbb{U}_{\k,\ell}|\asymp\N^{-2}\N^{2\gamma_{\k-1}}\N^{-\e}+\N^{-4/3}\N^{2\gamma_{\k-1}/3}\N^{-\e}$} for small {\small$\e>0$}, in the case {\small$\k\geq2$}. 
\item If {\small$\k=1$}, then we set {\small$|\mathbb{U}_{\k,\ell}|\asymp\N^{-3/2}$} for each {\small$\ell\in\llbracket1,\mathrm{L}\rrbracket$} if {\small$\min\mathbb{I}_{\x,1}\leq\N^{\gamma_{1}}$}.  On the other hand, if {\small$\min\mathbb{I}_{\x,1}>\N^{\gamma_{1}}$}, then we set {\small$|\mathbb{U}_{1,\ell}|\asymp|\mathbb{U}_{2,\ell}|$} for each {\small$\ell\in\llbracket1,\mathrm{L}\rrbracket$}.
\end{enumerate}
(We clarify that {\small$\mathrm{L}$} could depend on {\small$\k$}, though this dependence will not play any role in the analysis below, so we omit this dependence to avoid unnecessary notational distractions.)
\end{notation}
Intuitively, the decomposition in Notation \ref{notation:msblock} gives a ``multi-scale intermediate region" near the left-edge of {\small$\mathbb{I}_{\x}$}, {\color{black}similar to what is discussed in Section \ref{section:method}}. Moreover, for any {\small$\mathbb{I}_{\x,\k}$}, the associated time intervals {\small$\mathbb{U}_{\k,1},\ldots,\mathbb{U}_{\k,\mathrm{L}}$} are arranged with speed-of-propagation considerations in mind. In particular, as we show later when it becomes more directly relevant for our estimates, it turns out that the {\small$\mathbf{j}^{\N}$}-process has a spatial propagation length-scale of essentially {\small$\N|\mathbb{U}_{\k,\ell}|^{1/2}+\N^{2}|\mathbb{U}_{\k,\ell}|^{3/2}$} on the time-interval {\small$\mathbb{U}_{\k,\ell}$}; indeed, the first term is diffusive propagation, and the second term ultimately comes from the fact that the {\small$\mathbf{F}$}-nonlinearity in \eqref{eq:nl} is a ``singular transport operator". In any case, this length-scale is {\small$\ll\N^{\gamma_{\k-1}}$}, which is also the distance between {\small$\mathbb{I}_{\x,\k}$} and the left-edge of {\small$\mathbb{I}_{\x}$} if {\small$\k\geq2$} (which is the only part of {\small$\mathbb{I}_{\x}$} that can intersect the ``global edge" {\small$0\in\llbracket0,\infty\rrbracket$}). Therefore, analyzing dynamics on the space-time block {\small$\mathbb{U}_{\k,\ell}\times\mathbb{I}_{\x,\k}$} effectively avoids dealing with the global edge if {\small$\k\geq2$}. 

Let us now clarify the case {\small$\k=1$} (and the dichotomy in the size of {\small$\mathbb{U}_{1,\ell}$} sets). In the case of {\small$\min\mathbb{I}_{\x,1}\leq\N^{\gamma_{1}}$}, then {\small$\mathbb{I}_{\x,1}$} is close enough to the edge so that cancellations from space-time averaging come from the edge dynamics, and the time-scale {\small$\N^{-3/2}$} is large enough to obtain enough cancellations. In the case of {\small$\min\mathbb{I}_{\x,1}\geq\N^{\gamma_{1}}$}, then {\small$\mathbb{I}_{\x,1}$} is far enough from the edge that its analysis is essentially the same as for {\small$\k=2$}. In either case, the interval {\small$\mathbb{I}_{\x,1}$} is small enough so that our cancellations can be much weaker.

Equipped with Notation \ref{notation:msblock}, we can now further expand the previous display to get
\begin{align}
|\mathbf{Av}^{\mathbf{T},\mathbf{X},\mathscr{Q}}_{\t,\x}|&\lesssim\sum_{\k\in\llbracket1,\mathrm{K}\rrbracket}\mathrm{L}^{-1}\sum_{\ell\in\llbracket1,\mathrm{L}\rrbracket}\Big|\tfrac{1}{|\mathbb{U}_{\k,\ell}|}\int_{\mathbb{U}_{\k,\ell}}\sum_{\w\in\mathbb{I}_{\x,\k}}\boldsymbol{\mathcal{H}}_{\x,\w}\cdot\mathfrak{q}_{\w}[\bphi_{\t+\s}]\cdot\mathbf{Z}^{\N}_{\t+\s,\w}(\mathbf{Z}^{\N}_{\t-\mathfrak{t}_{\mathbf{Av}},\x_{+}})^{-1}\d\s\Big|.\label{eq:stochmain1I0a}
\end{align}
As before, we replace {\small$\t-\mathfrak{t}_{\mathbf{Av}}$} in the last factor above by {\small$\t+\inf\mathbb{U}_{\k,\ell}$}, as to place this factor at the time-boundary of the domain of integration above. In a similar spirit, we replace {\small$\x_{+}$} above by {\small$\x_{\k}:=\max\mathbb{I}_{\x,\k}$}. The cost in doing this replacement is to multiply by the factor of {\small$\mathbf{Z}^{\N}_{\t+\inf\mathbb{U}_{\k,\ell},\x_{\k}}(\mathbf{Z}^{\N}_{\t+\inf\mathbb{U}_{\k,\ell},\x_{+}})^{-1}=\exp\{\lambda(\mathbf{j}^{\N}_{\t+\inf\mathbb{U}_{\k,\ell},\x_{\k}}-\mathbf{j}^{\N}_{\t+\inf\mathbb{U}_{\k,\ell},\x_{+}})\}$}. By construction, the difference on the \abbr{RHS} is a sum over {\small$\y\in\llbracket\x_{\k}+1,\x_{+}\rrbracket$} of {\small$-\lambda\N^{-1/2}\bphi_{\t+\inf\mathbb{U}_{\k,\ell},\y}$}. By Lemma \ref{lemma:phiap}, this sum is {\small$\lesssim\N^{-1/2+\e}|\x_{\k}-\x_{+}|\lesssim\N^{-1/2+\e}(\log\N)^{3/2}\N^{1/2-3\delta_{\mathbf{S}}/2}\lesssim1$} with high probability, since {\small$|\x_{\k}-\x_{+}|\lesssim(\log\N)^{3}\N^{1-3\delta_{\mathbf{S}}}$} as noted near the beginning of Section \ref{subsection:stochmain1}. (Observe that we are only making statements about times in {\small$[0,\mathrm{T}]$} and points in {\small$\llbracket0,\N^{\mathrm{D}}\rrbracket$}.) (Throughout the rest of this section, we will similarly and freely make similar replacements in the point at which we evaluate the {\small$(\mathbf{Z}^{\N})^{-1}$} factor.) Ultimately, we get
\begin{align}
|\mathbf{Av}^{\mathbf{T},\mathbf{X},\mathscr{Q}}_{\t,\x}|&\lesssim\sum_{\k\in\llbracket1,\mathrm{K}\rrbracket}\mathrm{L}^{-1}\sum_{\ell\in\llbracket1,\mathrm{L}\rrbracket}\Big|\tfrac{1}{|\mathbb{U}_{\k,\ell}|}\int_{\mathbb{U}_{\k,\ell}}\sum_{\w\in\mathbb{I}_{\x,\k}}\boldsymbol{\mathcal{H}}_{\x,\w}\cdot\boldsymbol{\omega}^{\k,\ell}_{\s,\t,\w,\x}\d\s\Big|,\label{eq:stochmain1I1a}
\end{align}
where, for convenience in the analysis below, we have used the notation
\begin{align*}
\boldsymbol{\omega}^{\k,\ell}_{\s,\t,\w,\x}:=\mathfrak{q}_{\w}[\bphi_{\t+\s}]\cdot\mathbf{Z}^{\N}_{\t+\s,\w}(\mathbf{Z}^{\N}_{\t+\inf\mathbb{U}_{\k,\ell},\x_{\k}})^{-1}.
\end{align*}
%
\subsubsection{A priori cutoffs}
We now take \eqref{eq:stochmain1I1a} and modify the {\small$\boldsymbol{\omega}^{\k,\ell}$}-term therein in a way that makes it convenient to work with at a technical level but that also does nothing with high probability. First, we write
\begin{align*}
\boldsymbol{\omega}^{\k,\ell}_{\s,\t,\w,\x}&=\mathfrak{q}_{\w}[\bphi_{\t+\s}]\cdot\mathbf{Z}^{\N}_{\t+\s,\w}(\mathbf{Z}^{\N}_{\t+\s,\x_{\k}})^{-1}\cdot\mathbf{Z}^{\N}_{\t+\s,\x_{\k}}(\mathbf{Z}^{\N}_{\t+\inf\mathbb{U}_{\k,\ell},\x_{\k}})^{-1}\\
&=\mathfrak{q}_{\w}[\bphi_{\t+\s}]\cdot\exp\{\lambda(\mathbf{j}^{\N}_{\t+\s,\w}-\mathbf{j}^{\N}_{\t+\s,\x_{\k}})\}\cdot\exp\{\lambda(\mathbf{j}^{\N}_{\t+\s,\x_{\k}}-\mathbf{j}^{\N}_{\t+\inf\mathbb{U}_{\k,\ell},\x_{\k}})-\lambda\mathscr{R}_{\lambda}(\s-\inf\mathbb{U}_{\k,\ell})\},
\end{align*}
where the second line follows by the formula \eqref{eq:ch} for {\small$\mathbf{Z}^{\N}$}. Now, by the polynomial bound \eqref{eq:bulkadmissibleI} for admissible functions like {\small$\mathfrak{q}_{\w}$}, as well as the high probability bound on {\small$\bphi$}-spins from Lemma \ref{lemma:phiap}, we can bound the first factor in the last line above by {\small$\lesssim\N^{\e}$} for small {\small$\e>0$}. Similarly, the first exponential factor in the last line above is a sum over {\small$\y\in\llbracket\w+1,\x_{\k}\rrbracket$} of {\small$-\lambda\N^{-1/2}\bphi_{\t+\s,\y}$}, which, as explained earlier, is {\small$\lesssim1$} with high probability. Finally, also explained earlier, the last exponential factor is {\small$\exp(\mathrm{O}(\N^{-\rho}))=1+\mathrm{O}(\N^{-\rho})$} for some {\small$\rho>0$} fixed, again with high probability (because of the time-regularity estimate in Lemma \ref{lemma:currtreg}). In particular, with high probability, we can extend \eqref{eq:stochmain1I1a} to
\begin{align}
|\mathbf{Av}^{\mathbf{T},\mathbf{X},\mathscr{Q}}_{\t,\x}|&\lesssim\sum_{\k\in\llbracket1,\mathrm{K}\rrbracket}\mathrm{L}^{-1}\sum_{\ell\in\llbracket1,\mathrm{L}\rrbracket}\Big|\tfrac{1}{|\mathbb{U}_{\k,\ell}|}\int_{\mathbb{U}_{\k,\ell}}\sum_{\w\in\mathbb{I}_{\x,\k}}\boldsymbol{\mathcal{H}}_{\x,\w}\cdot\wt{\boldsymbol{\omega}}^{\k,\ell}_{\s,\t,\w,\x}\d\s\Big|,\label{eq:stochmain1I1b}
\end{align}
where {\small$\wt{\boldsymbol{\omega}}^{\k,\ell}$} is defined below, in which terms on the \abbr{RHS} are given precisely in Notation \ref{notation:cut} below:
\begin{align}
\wt{\boldsymbol{\omega}}^{\k,\ell}_{\s,\t,\w,\x}:=\wt{\mathfrak{q}}_{\w}[\bphi_{\t+\s}]\cdot\mathbf{G}_{\w}[\bphi_{\t+\s}]\cdot\mathbf{J}^{\k,\ell,\x}_{\s}.\label{eq:stochmain1I1bb}
\end{align}
%
\begin{notation}\label{notation:cut}
\fsp Fix {\small$\e>0$} small, and let {\small$\boldsymbol{\eta}(\cdot):\R\to\R$} be a function satisfying the following properties:
\begin{enumerate}
\item We have {\small$\boldsymbol{\eta}(\cdot)\equiv1$} on {\small$[-\N^{\e},\N^{\e}]$} and {\small$\boldsymbol{\eta}(\cdot)\equiv0$} outside {\small$[-2\N^{\e},2\N^{\e}]$}.
\item The function {\small$\boldsymbol{\eta}(\cdot)$} is smooth with all derivatives bounded uniformly and independently of {\small$\N$}.
\end{enumerate}
Let us also introduce another cutoff function {\small$\boldsymbol{\nu}:\R\to\R$} which instead satisfies the following for some {\small$\rho>0$}:
\begin{enumerate}
\item We have {\small$\boldsymbol{\nu}(\cdot)\equiv1$} on {\small$1+[-\N^{-\rho},\N^{-\rho}]$} and {\small$\boldsymbol{\nu}(\cdot)\equiv0$} outside {\small$1+[-2\N^{-\rho},2\N^{-\rho}]$}.
\item The function {\small$\boldsymbol{\nu}(\cdot)$} is smooth with all derivatives bounded uniformly and independently of {\small$\N$}.
\end{enumerate}
With this notation at hand, we define
\begin{align*}
\wt{\mathfrak{q}}_{\w}[\bphi_{\t+\s}]&:={\mathfrak{q}}_{\w}[\bphi_{\t+\s}]\cdot\boldsymbol{\eta}(\mathfrak{q}_{\w}[\bphi_{\t+\s}]),\\
\mathbf{G}_{\w}[\bphi_{\t+\s}]&:=\exp\{\lambda(\mathbf{j}^{\N}_{\t+\s,\w}-\mathbf{j}^{\N}_{\t+\s,\x_{\k}})\}\cdot\boldsymbol{\eta}(\exp\{\lambda(\mathbf{j}^{\N}_{\t+\s,\w}-\mathbf{j}^{\N}_{\t+\s,\x_{\k}})\}),\\
\mathbf{J}^{\k,\ell,\x}_{\s}&:=\exp\{\lambda(\mathbf{j}^{\N}_{\t+\s,\x_{\k}}-\mathbf{j}^{\N}_{\t+\inf\mathbb{U}_{\k,\ell},\x_{\k}})-\lambda\mathscr{R}_{\lambda}(\s-\inf\mathbb{U}_{\k,\ell})\}\\
&\times\boldsymbol{\nu}(\exp\{\lambda(\mathbf{j}^{\N}_{\t+\s,\x_{\k}}-\mathbf{j}^{\N}_{\t+\inf\mathbb{U}_{\k,\ell},\x_{\k}})-\lambda\mathscr{R}_{\lambda}(\s-\inf\mathbb{U}_{\k,\ell})\}).
\end{align*}
\end{notation}
Ultimately, by estimate \eqref{eq:stochmain1I1b} and the Cauchy-Schwarz inequality, we have (with high probability simultaneously over all {\small$(\t,\x)\in[0,\mathrm{O}(1)]\times\llbracket0,\mathrm{O}(\N^{\mathrm{D}})\rrbracket$}) that
\begin{align}
\int_{0}^{\mathrm{T}}|\mathbf{Av}^{\mathbf{T},\mathbf{X},\mathscr{Q}}_{\t,\x}|^{2}\d\t&\lesssim\sum_{\k\in\llbracket1,\mathrm{K}\rrbracket}\mathrm{L}^{-1}\sum_{\ell\in\llbracket1,\mathrm{L}\rrbracket}\int_{0}^{\mathrm{T}}\Big(\tfrac{1}{|\mathbb{U}_{\k,\ell}|}\int_{\mathbb{U}_{\k,\ell}}\sum_{\w\in\mathbb{I}_{\x,\k}}\boldsymbol{\mathcal{H}}_{\x,\w}\cdot\wt{\boldsymbol{\omega}}^{\k,\ell}_{\s,\t,\w,\x}\d\s\Big)^{2}\d\t.\label{eq:stochmain1I1c}
\end{align}
%
\subsubsection{Case I: {\small$\k\geq2$}}
We will first address the terms on the \abbr{RHS} of \eqref{eq:stochmain1I1c} above for which {\small$\k\geq2$}. As we discussed after Notation \ref{notation:msblock}, we will ultimately be able to control said terms by an analysis of the full-space version of the model \eqref{eq:currI}-\eqref{eq:currII} at hand. 
\begin{prop}\label{prop:case1}
\fsp Take any {\small$\k\in\llbracket2,\mathrm{K}\rrbracket$} and {\small$\ell\in\llbracket1,\mathrm{L}\rrbracket$}. We have the following estimate for some {\small$\beta_{\mathrm{CLT}}>0$} fixed, in which we use constructions from Notation \ref{notation:msblock} (for some small {\small$\e>0$}), and in which {\small$\mathscr{E}$} is a high probability event as in the statement of Proposition \ref{prop:stochmain1}:
\begin{align}
\E\Big(\mathbf{1}_{\mathscr{E}}\tfrac{1}{|\mathbb{U}_{\k,\ell}|}\int_{\mathbb{U}_{\k,\ell}}\sum_{\w\in\mathbb{I}_{\x,\k}}\boldsymbol{\mathcal{H}}_{\x,\w}\cdot\wt{\boldsymbol{\omega}}^{\k,\ell}_{\s,\t,\w,\x}\d\s\Big)^{2}&\lesssim_{}\N^{\mathrm{O}(\e)}\N^{-2}|\mathbb{U}_{\k,\ell}|^{-1}\|\boldsymbol{\mathcal{H}}_{\x,\cdot}\|_{\ell^{\infty}(\mathbb{I}_{\x,\k})}^{2}|\mathbb{I}_{\x,\k}|+\N^{-2-\beta_{\mathrm{CLT}}}.\label{eq:case1I}
\end{align}
\end{prop}
Again, we defer the proof of Proposition \ref{prop:case1} to Section \ref{subsection:proofs}. However, let us clarify Proposition \ref{prop:case1} intuitively. The factor {\small$\N^{\mathrm{O}(\e)}$} is entirely technical and can essentially be ignored because {\small$\e>0$} is small and our estimates will be quantitative in powers of {\small$\N$}. The other factors in the first term on the \abbr{RHS} of \eqref{eq:case1I} essentially assert \abbr{CLT}-type cancellations in the space-time integral on the \abbr{LHS} (recall that the speed of \eqref{eq:phiI}-\eqref{eq:phiII} is order {\small$\N^{2}$}, hence faster and sharper cancellations). We clarify that the last term on the \abbr{RHS} of \eqref{eq:case1I} contains contributions of the \abbr{LHS} on the low-probability event where \abbr{CLT} cancellations fail, as well as other ``lower-order corrections".
\subsubsection{Case II: {\small$\k=1$}}
The analysis in this case differs from that of case \abbr{I}. To start, recall the formula for {\small$\wt{\boldsymbol{\omega}}^{1,\ell}$} from right before Notation \ref{notation:cut}. The last term {\small$\mathbf{J}^{1,\ell,\x}$} is, by construction, deterministically equal to {\small$1+\mathrm{O}(\N^{-\rho})$}. So, if we replace this term by {\small$1$}, the cost is an additive one of {\small$\lesssim\|\boldsymbol{\mathcal{H}}_{\x,\cdot}\|_{\ell^{\infty}(\Z)}|\mathbb{I}_{\x,1}|\N^{\mathrm{O}(\e)}\N^{-\rho}\lesssim\N^{-1+\mathrm{C}\delta_{\mathbf{S}}+\e+\gamma_{1}}\N^{-\rho}\lesssim\N^{-1-\rho/2}$}, since {\small$\gamma_{1},\e,\delta_{\mathbf{S}}>0$} are small, while {\small$\rho>0$} from Notation \ref{notation:cut} is fixed. Thus, deterministically, we have
\begin{align}
\tfrac{1}{|\mathbb{U}_{1,\ell}|}\int_{\mathbb{U}_{1,\ell}}\sum_{\w\in\mathbb{I}_{\x,1}}\boldsymbol{\mathcal{H}}_{\x,\w}\cdot\wt{\boldsymbol{\omega}}^{1,\ell}_{\s,\t,\w,\x}\d\s&=\tfrac{1}{|\mathbb{U}_{1,\ell}|}\int_{\mathbb{U}_{1,\ell}}\sum_{\w\in\mathbb{I}_{\x,1}}\boldsymbol{\mathcal{H}}_{\x,\w}\cdot\wt{\mathfrak{q}}_{\w}[\bphi_{\t+\s}]\cdot\mathbf{G}_{\w}[\bphi_{\t+\s}]\d\s+\mathrm{O}(\N^{-1-\frac12\rho}).\label{eq:stochmain1I3}
\end{align}
We again recall the terms above from Notation \ref{notation:cut}. Since {\small$\|\boldsymbol{\mathcal{H}}_{\x,\cdot}\|_{\ell^{\infty}(\Z)}|\mathbb{I}_{\x,1}|\lesssim\N^{-1+\mathrm{C}\delta_{\mathbf{S}}+\gamma_{1}}$} and {\small$|\wt{\mathfrak{q}}|\lesssim\N^{\mathrm{O}(\e)}$}, we only need power-saving from time-integration of the form {\small$\N^{-\rho}$} for some {\small$\rho>0$} fixed. Ultimately, we arrive at:
\begin{prop}\label{prop:case2}
\fsp There exists a fixed constant {\small$\beta_{\mathrm{CLT}}>0$} such that for some high probability event {\small$\mathscr{E}$} as in the statement of Proposition \ref{prop:stochmain1}, we have
\begin{align}
\E\Big(\mathbf{1}_{\mathscr{E}}\tfrac{1}{|\mathbb{U}_{1,\ell}|}\int_{\mathbb{U}_{1,\ell}}\sum_{\w\in\mathbb{I}_{\x,1}}\boldsymbol{\mathcal{H}}_{\x,\w}\cdot\wt{\mathfrak{q}}_{\w}[\bphi_{\t+\s}]\cdot\mathbf{G}_{\w}[\bphi_{\t+\s}]\d\s\Big)^{2}\lesssim\N^{-2-\beta_{\mathrm{CLT}}}.\label{eq:case2I}
\end{align}
\end{prop}
\subsubsection{Putting it together}
Propositions \ref{prop:case1} and \ref{prop:case2} reduce Proposition \ref{prop:stochmain1} to some power-counting in the scaling parameter {\small$\N$}. We will now perform this power-counting.
\begin{proof}[Proof of Proposition \ref{prop:stochmain1}]
Consider the estimate \eqref{eq:stochmain1I1c}. By combining this with Propositions \ref{prop:case1} and \ref{prop:case2}, as well as \eqref{eq:stochmain1I3}, we get that for some high probability event {\small$\mathscr{E}$} (independent of {\small$(\t,\x)\in[0,\mathrm{T}]\times\llbracket0,\N^{\mathrm{D}}\rrbracket$}-variables below), the following estimate holds, where {\small$\beta_{\mathrm{CLT}}>0$} is a fixed constant and {\small$\e>0$} is small:
\begin{align*}
\int_{0}^{\mathrm{T}}\E(\mathbf{1}_{\mathscr{E}}|\mathbf{Av}^{\mathbf{T},\mathbf{X},\mathscr{Q}}_{\t,\x}|^{2})\d\t&\lesssim\sup_{\k\in\llbracket2,\mathrm{K}\rrbracket}\sup_{\ell\in\llbracket1,\mathrm{L}\rrbracket}\N^{\e}\N^{-2}|\mathbb{U}_{\k,\ell}|^{-1}\|\boldsymbol{\mathcal{H}}_{\x,\cdot}\|_{\ell^{\infty}(\mathbb{I}_{\x,\k})}^{2}|\mathbb{I}_{\x,\k}|+\N^{-2-\beta_{\mathrm{CLT}}}.
\end{align*}
For the first term on the \abbr{RHS}, we use the global bound {\small$\|\boldsymbol{\mathcal{H}}_{\x,\cdot}\|_{\ell^{\infty}(\llbracket0,\infty\rrbracket)}\lesssim\N^{-1+\mathrm{C}\delta_{\mathbf{S}}}$}, where {\small$\mathrm{C}\lesssim1$} and {\small$\delta_{\mathbf{S}}>0$} is the small constant from Definition \ref{definition:heat}. We then use size estimates for {\small$\mathbb{I}_{\x,\k},\mathbb{U}_{\k,\ell}$}-domains and {\small$\gamma_{\k}$}-exponents from Notation \ref{notation:msblock} to get the following for {\small$\k\in\llbracket2,\mathrm{K}-1\rrbracket$}:
\begin{align*}
&\N^{\e}\N^{-2}|\mathbb{U}_{\k,\ell}|^{-1}\|\boldsymbol{\mathcal{H}}_{\x,\cdot}\|_{\ell^{\infty}(\mathbb{I}_{\x,\k})}^{2}|\mathbb{I}_{\x,\k}|\lesssim\N^{\e}\N^{-2}|\mathbb{U}_{\k,\ell}|^{-1}\N^{-2+\mathrm{O}(\delta_{\mathbf{S}})}|\mathbb{I}_{\x,\k}|\\
&\lesssim\N^{\e+\mathrm{O}(\delta_{\mathbf{S}})}\cdot\N^{-4}\cdot(\N^{2}\N^{-2\gamma_{\k-1}}+\N^{\frac43}\N^{-\frac23\gamma_{\k-1}})\cdot\N^{\gamma_{\k}}\lesssim\N^{\e+\mathrm{O}(\delta_{\mathbf{S}})}\cdot(\N^{-2-\gamma_{\k}}+\N^{-\frac83+\frac13\gamma_{\k}})\lesssim\N^{-2-\beta_{\mathrm{CLT}}}.
\end{align*}
Above, {\small$\beta_{\mathrm{CLT}}>0$} is a possibly updated by still fixed constant. We clarify that the previous display requires that {\small$\k\neq\mathrm{K}$}, since it is only for {\small$\k\neq\mathrm{K}$} do we have {\small$\gamma_{\k}=\gamma_{\k-1}+\e$}. Moreover, the previous display holds because {\small$\e,\delta_{\mathbf{S}}$} are any small constants of our choosing. For the index {\small$\k=\mathrm{K}$}, what changes is that {\small$\N^{\gamma_{\mathrm{K}}}$} is now {\small$\lesssim\N^{\gamma_{\mathrm{K}-1}}\N^{1/2-\gamma}$}, because {\small$\N^{\gamma_{\mathrm{K}}}\lesssim|\mathbb{I}_{\x,\mathrm{K}}|\lesssim|\mathbb{I}_{\x}|\lesssim\N$}, and {\small$\N^{\gamma_{\k-1}}\gtrsim\N^{1/2+\gamma}$} by construction in Notation \ref{notation:msblock}. Moreover, we also have {\small$\N^{\gamma_{\mathrm{K}-1}}\lesssim\N^{1/2+\gamma+\e}$}, again by Notation \ref{notation:msblock}. Thus, we instead have
\begin{align*}
\N^{\e}\N^{-2}|\mathbb{U}_{\mathrm{K},\ell}|^{-1}\|\boldsymbol{\mathcal{H}}_{\x,\cdot}\|_{\ell^{\infty}(\mathbb{I}_{\x,\mathrm{K}})}^{2}|\mathbb{I}_{\x,\mathrm{K}}|&\lesssim\N^{\e}\N^{-2}|\mathbb{U}_{\mathrm{K},\ell}|^{-1}\N^{-2+\mathrm{O}(\delta_{\mathbf{S}})}|\mathbb{I}_{\x,\mathrm{K}}|\\
&\lesssim\N^{\e+\mathrm{O}(\delta_{\mathbf{S}})}\cdot\N^{-4}\cdot(\N^{2}\N^{-2\gamma_{\mathrm{K}-1}}+\N^{\frac43}\N^{-\frac23\gamma_{\mathrm{K}-1}})\cdot\N^{\gamma_{\mathrm{K}}}\\
&\lesssim\N^{-\gamma+\frac12+\e+\mathrm{O}(\delta_{\mathbf{S}})}\cdot(\N^{-2-\gamma_{\mathrm{K}-1}}+\N^{-\frac83+\frac13\gamma_{\mathrm{K}-1}})\lesssim\N^{-2-\beta_{\mathrm{CLT}}},
\end{align*}
where now {\small$\beta_{\mathrm{CLT}}\gtrsim\gamma$} in the previous display. We now combine the previous three displays to deduce the desired estimate \eqref{eq:stochmain1I}.
\end{proof}
\subsection{Wrapping up the proofs}\label{subsection:proofs}
The goal is to now establish Propositions \ref{prop:case1} and \ref{prop:case2}. 
\subsubsection{Speed of propagation}
As discussed earlier in this section, if we restrict the {\small$\mathbf{j}^{\N}$}-process to a spatial block {\small$\mathbb{I}$} and a time-interval {\small$\mathbb{U}$}, then information can only propagate a {\small$\mathbb{U}$}-dependent length from {\small$\mathbb{I}$} (with high probability). To establish this precisely, we will couple {\small$\mathbf{j}^{\N}$} to a full-space version of \eqref{eq:currI}-\eqref{eq:currII}, introduced below.
\begin{definition}\label{definition:speed}
\fsp Fix an interval {\small$\mathbb{I}\subseteq\llbracket0,\infty\rrbracket$} and a time {\small$\mathfrak{t}\geq0$}. Suppose that for some (small) {\small$\e>0$}, we have
\begin{align}
\min\mathbb{I}\geq\N^{\e}+\N^{\e}\N\mathfrak{t}^{\frac12}+\N^{\e}\N^{2}\mathfrak{t}^{\frac32}
\end{align}
(The extra {\small$\N^{\e}+$} is technical.) We define {\small$\mathbf{j}^{\Z}$} as the solution to the following \abbr{SDE} for {\small$(\s,\y)\in[0,\infty)\times\Z$}:
\begin{align*}
\d\mathbf{j}^{\Z}_{\s,\y}=\N^{\frac32}\grad^{\mathbf{X}}_{1}\mathscr{U}'[\bphi^{\Z}_{\s,\y}]\d\s+\N\mathbf{F}[\tau_{\y}\bphi^{\Z}_{\s}]\d\s+\sqrt{2}\N^{\frac12}\d\mathbf{w}_{\s,\y}.
\end{align*}
Here, we used the notation {\small$\bphi^{\Z}=-\N^{1/2}\grad^{\mathbf{X}}_{-1}\mathbf{j}^{\Z}$}. The meaning of {\small$\mathbf{F}[\tau_{\y}\bphi^{\Z}_{\s}]$} is obtained by the formula \eqref{eq:nl}. The {\small$\mathbf{w}$}-terms are independent standard Brownian motions, whose relation to {\small$\mathbf{b}$}-terms in \eqref{eq:currI}-\eqref{eq:currII} will be determined. Initial data of {\small$\mathbf{j}^{\Z}$} at time {\small$\s=0$} is also determined by context.
\end{definition}
We will now state the main speed of propagation result as follows.
\begin{prop}\label{prop:speed}
\fsp Take the context of Definition \ref{definition:speed} and fix any {\small$\mathrm{D}>0$}. Consider a time-interval {\small$\mathbb{U}$} of finite-length {\small$|\mathbb{U}|=\mathfrak{t}\lesssim1$}. Suppose that the initial data {\small$\mathbf{j}^{\Z}$} at time {\small$\s=0$} is given as follows:
\begin{itemize}
\item We have {\small$\mathbf{j}^{\Z}_{0,\y}=\mathbf{j}^{\N}_{\inf\mathbb{U},\y}$} for every {\small$\y\in\llbracket0,\infty\rrbracket$}. This fixes {\small$\bphi^{\Z}_{0,\y}$} for all {\small$\y\in\llbracket1,\infty\rrbracket$}.
\item The law of {\small$(\bphi^{\Z}_{0,\y})_{\y\in\llbracket-\infty,0\rrbracket}$} is distributed via the projection of {\small$\mathbb{P}^{0}$} onto {\small$\R^{\llbracket-\infty,0\rrbracket}$} (independently of {\small$\mathbf{j}^{\N}_{0,\cdot}$}).
\end{itemize}
Also, suppose that {\small$\mathbf{w}_{\s,\y}=\mathbf{b}_{\s+\inf\mathbb{U},\y}$} for all {\small$\y\in\llbracket1,\infty\rrbracket$}. With probability {\small$1-\mathrm{O}_{\mathrm{D}}(\N^{-\mathrm{D}})$}, we have 
\begin{align}
\sup_{\s\in[0,\mathfrak{t}]}\max_{\y\in\mathbb{I}}|\mathbf{j}^{\Z}_{\s,\y}-\mathbf{j}^{\N}_{\s+\inf\mathbb{U},\y}|\lesssim_{\mathrm{D}}\N^{-\mathrm{D}}.\label{eq:speedI}
\end{align}
\end{prop}
The linearization of \eqref{eq:currI}-\eqref{eq:currII} is a divergence-form parabolic \abbr{PDE} with singular first-order transport operator (corresponding to the {\small$\mathbf{F}$}-nonlinearity therein). Thus, to establish Proposition \ref{prop:speed}, we will employ the following Nash-Davies-type parabolic estimate for this equation.
\begin{prop}\label{prop:nashdavies}
\fsp Let {\small$\boldsymbol{\Gamma}^{\mathrm{full}}_{\s,\t,\x,\y}$}, as a function of {\small$(\s,\t,\x,\y)\in[0,\infty)^{2}\times\Z^{2}$} restricted to {\small$\s\leq\t$}, satisfy the equation below for {\small$\t\geq\s$} and {\small$\x\in\Z$} (in which {\small$1\lesssim\mathbf{c}\lesssim1$} and {\small$|\mathbf{v}|\lesssim\N^{\e_{0}}$} uniformly for some small {\small$\e_{0}>0$}):
\begin{align}
\partial_{\t}\boldsymbol{\Gamma}^{\mathrm{full}}_{\s,\t,\x,\y}=-\N^{2}\grad^{\mathbf{X}}_{1}(\mathbf{c}_{\t,\x}\grad^{\mathbf{X}}_{-1}\boldsymbol{\Gamma}^{\mathrm{full}}_{\s,\t,\x,\y})+\sum_{|\w|\lesssim1}\N^{\frac32}\mathbf{v}^{\w}_{\t,\x}\grad^{\mathbf{X}}_{\w}\boldsymbol{\Gamma}^{\mathrm{full}}_{\s,\t,\x,\y} \quad\text{and}\quad \boldsymbol{\Gamma}^{\mathrm{full}}_{\s,\s,\x,\y}=\mathbf{1}_{\x=\y}. \label{eq:nashI}
\end{align}
All discrete gradients above act on {\small$\x$}. Now, fix {\small$(\s,\t,\x,\y)\in[0,\infty)^{2}\times\Z^{2}$} such that {\small$\s\leq\t$} and {\small$|\t-\s|\lesssim1$}. Moreover, set {\small$\delta_{\s,\t}:=\max(\N^{-2},|\t-\s|)$}. Then, there exists constants {\small$\mathrm{C},\upsilon>0$} (depending only on bounds for {\small$\mathbf{c},\mathbf{v}$}) so that 
\begin{align}
|\boldsymbol{\Gamma}^{\mathrm{full}}_{\s,\t,\x,\y}|\lesssim\mathrm{e}^{\mathrm{C}\N^{1+\mathrm{C}\e_{0}}|\t-\s|}\cdot\exp(-\tfrac{\upsilon|\x-\y|}{\N\delta_{\s,\t}^{1/2}}).\label{eq:nashII}
\end{align}
\end{prop}
\begin{remark}
\fsp As noted at the end of Section \ref{subsection:method2}, because Proposition \ref{prop:nashdavies} plays a crucial role in the analysis of this section, we will provide a very detailed analysis which does not use anything special about the lattice in the spirit of universality.
\end{remark}
\begin{remark}\label{remark:nash}
\fsp We point out that the \abbr{RHS} of \eqref{eq:nashII} is ``missing" a factor of {\small$\N^{-1}\delta_{\s,\t}^{-1/2}$}, in that the estimate would perhaps remains true even with include this factor on the \abbr{RHS} above. Showing this seems to be very much within reach after some minor cosmetic adjustments to the proof of \eqref{eq:nashII} given below, though this factor will not be of much use for our purposes. 
\end{remark}
\begin{remark}\label{remark:nashII}
\fsp We also emphasize that the exponent on the \abbr{RHS} in the first factor is {\small$\lesssim\N^{1+\mathrm{O}(\e_{0})}|\t-\s|$}; in particular, it is \emph{not} {\small$\lesssim\N^{3/2+\e_{0}}|\t-\s|$}, which is perhaps a natural guess given the size of the coefficients {\small$\N^{3/2}\mathbf{v}^{\w}$} in \eqref{eq:nashI} and the Gronwall inequality. Thus, it is crucial that we use the gradient structure of the first-order operator in \eqref{eq:nashI}.
\end{remark}
The first step is to handle the case {\small$|\t-\s|\lesssim\N^{-2}$}, where the \abbr{RHS} of the bound \eqref{eq:nashII} has the form {\small$\lesssim\exp[-\upsilon|\x-\y|]$}. In this case, because the first-order operator in \eqref{eq:nashI} has strength {\small$\lesssim\N^{3/2+\e_{0}}\ll\N^{2}$}, a standard Duhamel-type argument (along with the estimate \eqref{eq:nashII} in the case {\small$\mathbf{v}^{\w}\equiv0$}) suffices.
\begin{lemma}\label{lemma:nashst}
\fsp Proposition \ref{prop:nashdavies} holds under the assumption that {\small$|\t-\s|\lesssim\N^{-2}$}.
\end{lemma}
\begin{proof}
We assume that {\small$\s=0$} without loss of generality, since we can always shift the coefficients in time without changing the validity of the assumed bounds. We also let {\small$\boldsymbol{\Gamma}^{\mathrm{full},\mathbf{c}}_{\s,\t,\x,\y}$} denote the solution to
\begin{align}
\partial_{\t}\boldsymbol{\Gamma}^{\mathrm{full},\mathbf{c}}_{\s,\t,\x,\y}=-\N^{2}\grad^{\mathbf{X}}_{1}(\mathbf{c}_{\t,\x}\grad^{\mathbf{X}}_{-1}\boldsymbol{\Gamma}^{\mathrm{full},\mathbf{c}}_{\s,\t,\x,\y}) \quad\text{and}\quad \boldsymbol{\Gamma}^{\mathrm{full},\mathbf{c}}_{\s,\s,\x,\y}=\mathbf{1}_{\x=\y}, \nonumber
\end{align}
in which all operators act on {\small$\x$} on the \abbr{RHS}, and in which {\small$\t\geq\s$}. By Proposition B.3 in \cite{GOS01}, we know that {\small$\boldsymbol{\Gamma}^{\mathrm{full},\mathbf{c}}$} satisfies the estimate \eqref{eq:nashII}. Moreover, by the Duhamel formula, we also have the formula
\begin{align*}
\boldsymbol{\Gamma}^{\mathrm{full}}_{0,\t,\x,\y}=\boldsymbol{\Gamma}^{\mathrm{full},\mathbf{c}}_{0,\t,\x,\y}+\int_{0}^{\t}\sum_{\z\in\Z}\boldsymbol{\Gamma}^{\mathrm{full},\mathbf{c}}_{\s,\t,\x,\z}\cdot\sum_{|\w|\lesssim1}\N^{\frac32}\mathbf{v}^{\w}_{\s,\z}\grad^{\mathbf{X}}_{\w}\boldsymbol{\Gamma}^{\mathrm{full}}_{0,\s,\z,\y}\d\s.
\end{align*}
We let {\small$\mu>0$} be small but fixed. We also observe that {\small$\exp[\mu|\x-\y|]\lesssim\exp[\mu|\x-\z|]\exp[\mu|(\z+\w)-\y|]$} for any {\small$|\w|\lesssim1$}. Thus, we have
\begin{align*}
\mathrm{e}^{\mu|\x-\y|}|\boldsymbol{\Gamma}^{\mathrm{full}}_{0,\t,\x,\y}|&\lesssim1+\int_{0}^{\t}\sum_{\z\in\Z}\mathrm{e}^{\upsilon|\x-\z|}|\boldsymbol{\Gamma}^{\mathrm{full},\mathbf{c}}_{\s,\t,\x,\z}|\cdot\N^{\frac32+\e_{0}}\|\mathrm{e}^{\mu|\cdot-\y|}\boldsymbol{\Gamma}^{\mathrm{full}}_{0,\s,\cdot,\y}\|_{\ell^{\infty}(\Z)}\d\s\\
&\lesssim1+\int_{0}^{\t}\sum_{\z\in\Z}\mathrm{e}^{-\frac12\mu|\x-\z|}\cdot\N^{\frac32+\e_{0}}\|\mathrm{e}^{\upsilon|\cdot-\y|}\boldsymbol{\Gamma}^{\mathrm{full}}_{0,\s,\cdot,\y}\|_{\ell^{\infty}(\Z)}\d\s\\
&\lesssim1+\int_{0}^{\t}\N^{\frac32+\e_{0}}\|\mathrm{e}^{\upsilon|\cdot-\y|}\boldsymbol{\Gamma}^{\mathrm{full}}_{0,\s,\cdot,\y}\|_{\ell^{\infty}(\Z)}\d\s.
\end{align*}
Observe that the last line of this display does not depend on the {\small$\x$}-variable on the \abbr{LHS} of the first line therein. So, we may replace the first term in the previous display by its supremum over {\small$\x\in\Z$}, at which point we can use the Gronwall inequality to deduce that {\small$\|\exp[\mu|\cdot-\y|]\boldsymbol{\Gamma}^{\mathrm{full}}_{0,\t,\cdot,\y}\|_{\ell^{\infty}(\Z)}\lesssim\exp[\mathrm{O}(\N^{3/2+\e_{0}}\t)]\lesssim1$} since {\small$\t\lesssim\N^{-2}$}.
\end{proof}
We now state a preliminary {\small$\ell^{2}$}-estimate for {\small$\boldsymbol{\Gamma}^{\mathrm{full}}$}. 
\begin{lemma}\label{lemma:nashl2}
\fsp Retain the setting of Proposition \ref{prop:nashdavies}, and define 
\begin{align*}
\boldsymbol{\Gamma}^{\mathrm{full},\bpsi}_{\s,\t,\x,\y}:=\mathrm{e}^{-\bpsi(\x)}\boldsymbol{\Gamma}^{\mathrm{full}}_{\s,\t,\x,\y}\mathrm{e}^{\bpsi(\y)},
\end{align*}
where {\small$\bpsi:\Z\to\R$} is a uniformly bounded function satisfying {\small$\|\grad^{\mathbf{X}}_{1}\bpsi\|_{\ell^{\infty}(\Z)}\lesssim1$}. Then, we have 
\begin{align}
\|\boldsymbol{\Gamma}^{\mathrm{full},\bpsi}_{\s,\t,\x,\cdot}\|_{\ell^{2}(\Z)}^{2}\lesssim\mathrm{e}^{\mathrm{O}(\N^{1+\e_{0}})|\t-\s|}\mathrm{e}^{\mathrm{O}(\N^{2}\Gamma(\bpsi)^{2})|\t-\s|}\mathrm{e}^{\mathrm{O}(\N^{\frac32+\e_{0}}\Gamma(\bpsi))|\t-\s|},\label{eq:nashl2I}
\end{align}
where {\small${\Gamma}(\bpsi)^{2}\lesssim\|\grad^{\mathbf{X}}_{1}\bpsi\|_{\ell^{\infty}(\Z)}^{2}$}.
\end{lemma}
\begin{proof}
Without loss of generality, we derive \eqref{eq:nashl2I} with {\small$\s=0$}; the assumptions on {\small$\mathbf{c},\mathbf{v}$}-coefficients are preserved under time shifts, so the general case follows as well. (However, we will use the time-variable {\small$\s$} as a parameter in{\small$[0,\t]$}.) It will be convenient to use the notation
\begin{align*}
\mathscr{S}_{\t,\x}:=-\N^{2}\grad^{\mathbf{X}}_{1}(\mathbf{c}_{\t,\x}\grad^{\mathbf{X}}_{-1}) \quad\text{and}\quad \mathscr{D}_{\t,\x}:=\sum_{|\w|\lesssim1}\N^{\frac32}\mathbf{v}^{\w}_{\t,\x}\grad^{\mathbf{X}}_{\w}.
\end{align*}
Lastly, we will require a formula for the evolution {\small$\partial_{-\s}\boldsymbol{\Gamma}^{\mathrm{full}}_{\s,\t,\x,\y}$}, for reasons explained in Remark \ref{remark:nashbackwards}. We note that \eqref{eq:nashI} is an equation in the forwards time-variable and backwards spatial-variable. So, the evolution {\small$\partial_{-\s}\boldsymbol{\Gamma}^{\mathrm{full}}_{\s,\t,\x,\y}$} is given by the adjoint of the operator on the \abbr{RHS} of the \abbr{PDE} in \eqref{eq:nashI} acting instead on the {\small$\y$}-variable; see \cite{FS86}, for instance (it is also the standard duality between Kolmogorov forwards- and backwards-equations). So,
\begin{align}
\partial_{-\s}\boldsymbol{\Gamma}^{\mathrm{full}}_{\s,\t,\x,\y}=\mathscr{S}_{\s,\y}\boldsymbol{\Gamma}^{\mathrm{full}}_{\s,\t,\x,\y}+\mathscr{D}^{\ast}_{\s,\y}\boldsymbol{\Gamma}^{\mathrm{full}}_{\s,\t,\x,\y}.\label{eq:nashl2I0}
\end{align}
The adjoint is taken with respect to Lebesgue measure on {\small$\Z$}, and we note that {\small$\mathscr{S}_{\s,\y}$} is self-adjoint. We also clarify that the need for {\small$\partial_{-\s}$} instead of {\small$\partial_{\s}$} is because we consider backwards-in-time evolution. We record the identity {\small$\partial_{-\s}(\boldsymbol{\Gamma}^{\mathrm{full},\bpsi}_{\s,\t,\x,\y})^{2}=2\boldsymbol{\Gamma}^{\mathrm{full},\bpsi}_{\s,\t,\x,\y}\partial_{-\s}\boldsymbol{\Gamma}^{\mathrm{full},\bpsi}_{\s,\t,\x,\y}$} and use this (along with the \abbr{PDE} above and the formula for {\small$\boldsymbol{\Gamma}^{\mathrm{full},\bpsi}$}):
\begin{align}
\tfrac{\d}{\d(-\s)}\|\boldsymbol{\Gamma}^{\mathrm{full},\bpsi}_{\s,\t,\x,\cdot}\|_{\ell^{2}(\Z)}^{2}=2\sum_{\y\in\Z}\boldsymbol{\Gamma}^{\mathrm{full},\bpsi}_{\s,\t,\x,\y}\cdot\mathrm{e}^{\bpsi(\y)}\mathscr{S}_{\s,\y}\boldsymbol{\Gamma}^{\mathrm{full}}_{\s,\t,\x,\y}\mathrm{e}^{-\bpsi(\x)}+2\sum_{\y\in\Z}\boldsymbol{\Gamma}^{\mathrm{full},\bpsi}_{\s,\t,\x,\y}\cdot\mathrm{e}^{\bpsi(\y)}\mathscr{D}_{\s,\y}^{\ast}\boldsymbol{\Gamma}^{\mathrm{full}}_{\s,\t,\x,\y}\mathrm{e}^{-\bpsi(\x)}.\label{eq:nashl2I1}
\end{align}
As noted in the proof of Proposition B.3 in \cite{GOS01} (and with a heuristic given below), the {\small$\mathscr{S}$} quantity satisfies
\begin{align}
2\sum_{\y\in\Z}\boldsymbol{\Gamma}^{\mathrm{full},\bpsi}_{\s,\t,\x,\y}\cdot\mathrm{e}^{\bpsi(\y)}\mathscr{S}_{\s,\y}\boldsymbol{\Gamma}^{\mathrm{full}}_{\s,\t,\x,\y}\mathrm{e}^{-\bpsi(\x)}\leq-\mathrm{c}_{1}\N^{2}\|\grad^{\mathbf{X}}_{1}\boldsymbol{\Gamma}^{\mathrm{full},\bpsi}_{\s,\t,\x,\cdot}\|_{\ell^{2}(\Z)}^{2}+\mathrm{c}_{2}\N^{2}\Gamma(\bpsi)^{2}\|\boldsymbol{\Gamma}^{\mathrm{full},\bpsi}_{\s,\t,\x,\cdot}\|_{\ell^{2}(\Z)}^{2},\label{eq:nashl2I2}
\end{align}
with {\small$\Gamma(\bpsi)^{2}\lesssim\|\grad^{\mathbf{X}}_{1}\bpsi\|_{\ell^{\infty}(\Z)}^{2}$} as in the lemma. (Throughout this proof, the {\small$\mathrm{c}$}-constants are fixed and bounded away from {\small$0$}.) Intuitively, we can move {\small$\exp[\bpsi(\y)]$} on the \abbr{LHS} of \eqref{eq:nashl2I2} past the {\small$\mathscr{S}_{\s,\y}$}-operator up to terms that have a factor of a gradient acting on {\small$\exp[\bpsi(\y)]$}. However, {\small$|\grad^{\mathbf{X}}_{\pm1}\exp[\bpsi(\y)]|\lesssim{\Gamma}(\bpsi)\exp[\bpsi(\y)]$}, since {\small$\|\grad^{\mathbf{X}}_{1}\bpsi\|_{\ell^{\infty}(\Z)}\lesssim1$}. Thus, terms with a factor of {\small$\grad^{\mathbf{X}}_{\pm1}\exp[\bpsi(\y)]$} have the form {\small$\boldsymbol{\Gamma}^{\mathrm{full},\bpsi}_{\s,\t,\x,\y}\cdot\mathrm{O}(\N\Gamma(\bpsi))\N\grad^{\mathbf{X}}_{-1}\boldsymbol{\Gamma}^{\mathrm{full},\bpsi}_{\s,\t,\x,\y}$} summed over {\small$\y\in\Z$}, which we use Cauchy-Schwarz to estimate from above by {\small$\lesssim\eta\N^{2}\|\grad^{\mathbf{X}}_{1}\boldsymbol{\Gamma}^{\mathrm{full},\bpsi}_{\s,\t,\x,\cdot}\|_{\ell^{2}(\Z)}^{2}+\N^{2}\Gamma(\bpsi)^{2}\|\boldsymbol{\Gamma}^{\mathrm{full}}_{\s,\t,\x,\cdot}\|_{\ell^{2}(\Z)}^{2}$} for small but fixed {\small$\eta>0$}. We are left with {\small$\boldsymbol{\Gamma}^{\mathrm{full},\bpsi}_{\s,\t,\x,\y}\cdot\mathscr{S}_{\s,\y}\boldsymbol{\Gamma}^{\mathrm{full},\bpsi}_{\s,\t,\x,\y}$} summed over {\small$\y\in\Z$}, for which we use the divergence-form structure of {\small$\mathscr{S}_{\s,\y}$} and bound above by {\small$-\mathrm{c}\N^{2}\|\grad^{\mathbf{X}}_{1}\boldsymbol{\Gamma}^{\mathrm{full},\bpsi}_{\s,\t,\x,\cdot}\|_{\ell^{2}(\Z)}^{2}$} for some fixed constant {\small$\mathrm{c}>0$}. Ultimately, this produces \eqref{eq:nashl2I2}. We will now estimate the {\small$\mathscr{D}^{\ast}$}-term. We claim, with explanation given afterwards, that
\begin{align}
2\sum_{\y\in\Z}\boldsymbol{\Gamma}^{\mathrm{full},\bpsi}_{\s,\t,\x,\y}\cdot\mathrm{e}^{\bpsi(\y)}\mathscr{D}_{\s,\y}^{\ast}\boldsymbol{\Gamma}^{\mathrm{full}}_{\s,\t,\x,\y}\mathrm{e}^{-\bpsi(\x)}&=2\sum_{\y\in\Z}\mathscr{D}_{\s,\y}(\boldsymbol{\Gamma}^{\mathrm{full},\bpsi}_{\s,\t,\x,\y}\cdot\mathrm{e}^{\bpsi(\y)})\cdot\boldsymbol{\Gamma}^{\mathrm{full}}_{\s,\t,\x,\y}\mathrm{e}^{-\bpsi(\x)}\label{eq:nashl2I3}\\
&\leq2\sum_{\y\in\Z}\mathscr{D}_{\s,\y}(\boldsymbol{\Gamma}^{\mathrm{full},\bpsi}_{\s,\t,\x,\y})\cdot\boldsymbol{\Gamma}^{\mathrm{full},\bpsi}_{\s,\t,\x,\y}+\mathrm{c}_{3}\N^{\frac32+\e_{0}}\Gamma(\bpsi)\|\boldsymbol{\Gamma}^{\mathrm{full},\bpsi}_{\s,\t,\x,\cdot}\|_{\ell^{2}(\Z)}^{2},\nonumber
\end{align}
The first line follows by the definition of the adjoint. The second line is justified as follows. We use the (discrete) Leibniz rule. This gives {\small$\mathscr{D}_{\s,\y}(\boldsymbol{\Gamma}^{\mathrm{full},\bpsi}_{\s,\t,\x,\y}\exp[\bpsi(\y)])=\mathscr{D}_{\s,\y}\boldsymbol{\Gamma}^{\mathrm{full},\bpsi}_{\s,\t,\x,\y}\cdot\exp[\bpsi(\y)]+\mathscr{Q}$}, where the {\small$\mathscr{Q}$} is explained shortly. For the first term in this identity, we obtain {\small$\mathscr{D}_{\s,\y}\boldsymbol{\Gamma}^{\mathrm{full},\bpsi}_{\s,\t,\x,\y}\cdot\exp[\bpsi(\y)]\cdot\boldsymbol{\Gamma}^{\mathrm{full}}_{\s,\t,\x,\y}\exp[-\bpsi(\x)]=\mathscr{D}_{\s,\y}\boldsymbol{\Gamma}^{\mathrm{full},\bpsi}_{\s,\t,\x,\y}\cdot\boldsymbol{\Gamma}^{\mathrm{full},\bpsi}_{\s,\t,\x,\y}$}. This explains the first term in the second line. Now, the quantity {\small$\mathscr{Q}$} has the following form:
\begin{itemize}
\item The discrete Leibniz rule states that {\small$\grad^{\mathbf{X}}_{\mathfrak{l}}(\mathsf{f}\mathsf{g})_{\x}=\mathsf{f}_{\x}\grad^{\mathbf{X}}_{\mathfrak{l}}\mathsf{g}_{\x}+\mathsf{g}_{\x+\mathfrak{l}}\grad^{\mathbf{X}}_{\mathfrak{l}}\mathsf{f}_{\x}$}. Therefore, we have 
\begin{align*}
\mathscr{Q}=\sum_{|\w|\lesssim1}\N^{3/2}\mathbf{v}^{\w}_{\s,\y}\boldsymbol{\Gamma}^{\mathrm{full},\bpsi}_{\s,\t,\x,\y+\w}\grad^{\mathbf{X}}_{\w}\exp[\bpsi(\y)].
\end{align*}
\item We multiply the display above by {\small$\boldsymbol{\Gamma}^{\mathrm{full}}_{\s,\t,\x,\y}\exp[-\bpsi(\x)]$}. We also use {\small$\grad^{\mathbf{X}}_{\w}\exp[\bpsi(\y)]\lesssim\Gamma(\bpsi)\exp[\bpsi(\y)]$}, since {\small$|\grad^{\mathbf{X}}_{1}\bpsi\|_{\ell^{\infty}(\Z)}\lesssim1$}. Thus, besides the first term in the last line of \eqref{eq:nashl2I3}, the other contribution of the \abbr{RHS} of the first line therein is given by
\begin{align*}
2\sum_{|\w|\lesssim1}\sum_{\y\in\Z}\mathrm{O}(\N^{\frac32+\e_{0}}\Gamma(\bpsi))\boldsymbol{\Gamma}^{\mathrm{full},\bpsi}_{\s,\t,\x,\y+\w}\mathrm{e}^{\bpsi(\y)}\boldsymbol{\Gamma}^{\mathrm{full}}_{\s,\t,\x,\y}\mathrm{e}^{-\bpsi(\x)}.
\end{align*}
We recognize {\small$\exp[\bpsi(\y)]\boldsymbol{\Gamma}^{\mathrm{full}}_{\s,\t,\x,\y}\exp[-\bpsi(\x)]=\boldsymbol{\Gamma}^{\mathrm{full},\bpsi}_{\s,\t,\x,\y}$}. Then, we use the Schwarz inequality with respect to the {\small$\y$}-summation to estimate the previous display by {\small$\lesssim\N^{3/2+\e_{0}}\Gamma(\bpsi)\|\boldsymbol{\Gamma}^{\mathrm{full},\bpsi}_{\s,\t,\x,\cdot+\w}\|_{\ell^{2}(\Z)}\|\boldsymbol{\Gamma}^{\mathrm{full},\bpsi}_{\s,\t,\x,\cdot}\|_{\ell^{2}(\Z)}$}. These last two {\small$\ell^{2}$}-norms are equal. Therefore the second line of \eqref{eq:nashl2I3} follows.
\end{itemize}
Thus, we ultimately have 
\begin{align*}
\tfrac{\d}{\d(-\s)}\|\boldsymbol{\Gamma}^{\mathrm{full},\bpsi}_{\s,\t,\x,\cdot}\|_{\ell^{2}(\Z)}^{2}\leq-\mathrm{c}_{1}\N^{2}\|\grad^{\mathbf{X}}_{1}\boldsymbol{\Gamma}^{\mathrm{full},\bpsi}_{\s,\t,\x,\cdot}\|_{\ell^{2}(\Z)}^{2}+&\mathrm{c}_{4}(\N^{2}\Gamma(\bpsi)^{2}+\N^{\frac32+\e_{0}}\Gamma(\bpsi))\|\boldsymbol{\Gamma}^{\mathrm{full},\bpsi}_{\s,\t,\x,\cdot}\|_{\ell^{2}(\Z)}^{2}\\
&+2\sum_{\y\in\Z}\mathscr{D}_{\s,\y}(\boldsymbol{\Gamma}^{\mathrm{full},\bpsi}_{\s,\t,\x,\y})\cdot\boldsymbol{\Gamma}^{\mathrm{full},\bpsi}_{\s,\t,\x,\y}.
\end{align*}
We note that {\small$\mathscr{D}_{\s,\y}$} is a finite linear combination of discrete gradients (with length-scale {\small$\lesssim1$}) scaled by {\small$\lesssim\N^{3/2+\e_{0}}$}. Thus, for the last term, we use the Schwarz inequality and elementary manipulations to get the following with {\small$\delta>0$} small but fixed:
\begin{align}
2\sum_{\y\in\Z}\mathscr{D}_{\s,\y}(\boldsymbol{\Gamma}^{\mathrm{full},\bpsi}_{\s,\t,\x,\y})\cdot\boldsymbol{\Gamma}^{\mathrm{full},\bpsi}_{\s,\t,\x,\y}&\lesssim\mathrm{C}_{\delta}\N^{1+2\e_{0}}\|\boldsymbol{\Gamma}^{\mathrm{full},\bpsi}_{\s,\t,\x,\cdot}\|_{\ell^{2}(\Z)}^{2}+\delta\N^{2}\sup_{|\w|\lesssim1}\|\grad^{\mathbf{X}}_{\w}\boldsymbol{\Gamma}^{\mathrm{full},\bpsi}_{\s,\t,\x,\cdot}\|_{\ell^{2}(\Z)}^{2}\nonumber\\
&\lesssim\mathrm{C}_{\delta}\N^{1+2\e_{0}}\|\boldsymbol{\Gamma}^{\mathrm{full},\bpsi}_{\s,\t,\x,\cdot}\|_{\ell^{2}(\Z)}^{2}+\delta\N^{2}\|\grad^{\mathbf{X}}_{1}\boldsymbol{\Gamma}^{\mathrm{full},\bpsi}_{\s,\t,\x,\cdot}\|_{\ell^{2}(\Z)}^{2}.\label{eq:nashl2I4}
\end{align}
The previous two displays now yield
\begin{align}
\tfrac{\d}{\d(-\s)}\|\boldsymbol{\Gamma}^{\mathrm{full},\bpsi}_{\s,\t,\x,\cdot}\|_{\ell^{2}(\Z)}^{2}&\leq-\tfrac12\mathrm{c}_{1}\|\grad^{\mathbf{X}}_{1}\boldsymbol{\Gamma}^{\mathrm{full},\bpsi}_{\s,\t,\x,\cdot}\|_{\ell^{2}(\Z)}^{2}+\mathrm{O}(\N^{2}\Gamma(\bpsi)+\N^{\frac32+\e_{0}}\Gamma(\bpsi)+\N^{1+2\e_{0}})\|\boldsymbol{\Gamma}^{\mathrm{full},\bpsi}_{\s,\t,\x,\cdot}\|_{\ell^{2}(\Z)}^{2}\label{eq:nashl2I5}
\end{align}
if we take {\small$\delta>0$} small enough (but still independent of {\small$\N$}). After dropping the first term on the \abbr{RHS}, the Gronwall inequality now gives us the following for any {\small$\t\in[0,\mathfrak{t}]$} after we integrate over {\small$\s\in[0,\t]$} (backwards in time):
\begin{align}
\|\boldsymbol{\Gamma}^{\mathrm{full},\bpsi}_{0,\t,\x,\cdot}\|_{\ell^{2}(\Z)}^{2}\lesssim\mathrm{e}^{\mathrm{C}\t(\N^{2}\Gamma(\bpsi)+\N^{\frac32+\e_{0}}\Gamma(\bpsi)+\N^{1+2\e_{0}})}\|\boldsymbol{\Gamma}^{\mathrm{full},\bpsi}_{\t,\t,\x,\cdot}\|_{\ell^{2}(\Z)}^{2}\lesssim\mathrm{e}^{\mathrm{C}\t(\N^{2}\Gamma(\bpsi)+\N^{\frac32+\e_{0}}\Gamma(\bpsi)+\N^{1+2\e_{0}})},\label{eq:nashl2I6}
\end{align}
We note that the last estimate follows because {\small$\boldsymbol{\Gamma}^{\mathrm{full},\bpsi}_{\t,\t,\x,\y}=\mathbf{1}_{\x=\y}$}, and thus {\small$\|\boldsymbol{\Gamma}^{\mathrm{full},\bpsi}_{\t,\t,\x,\cdot}\|_{\ell^{2}(\Z)}=1$}, by construction. This completes the proof of \eqref{eq:nashl2I}, so we are done. 
\end{proof}
\begin{remark}\label{remark:nashbackwards}
\fsp We clarify the possible awkward nature of appealing to the adjoint equation for \eqref{eq:nashI} instead of using directly \eqref{eq:nashI} itself. It turns out that in generalizing the argument above to the proof of Proposition \ref{prop:nashdavies} given below, it will be much more convenient to work with the adjoint \abbr{PDE}. Indeed, for the proof of Proposition \ref{prop:nashdavies} below, we will control higher {\small$\ell^{2p}$}-norms instead of just {\small$p=1$}, and allowing the {\small$\mathscr{D}$}-operator to hit the correct power of the exponentially titled heat kernel ultimately turns out to be technically convenient.
\end{remark}
We are now ready to prove Proposition \ref{prop:nashdavies}. The main idea, which is due to Davies (see \cite{FS86,D89}) is to bootstrap the {\small$\ell^{2}$}-estimate into an {\small$\ell^{\infty}$}-estimate by means of a generalized energy estimate (from {\small$\ell^{2}$}-norms as in Lemma \ref{lemma:nashl2} to {\small$\ell^{p}$}-norms for {\small$p\geq2$}) and the Nash inequality. (The main benefit of this argument is that it also succeeds in the continuum, not just the lattice \cite{D89}, while being fairly short. It is thus preferred in the spirit of universality.)
\begin{proof}[Proof of Proposition \ref{prop:nashdavies}]
By Lemma \ref{lemma:nashst}, we can assume {\small$\t\gtrsim\N^{-2}$}. Fix {\small$p\geq1$}, and recall the notation in Lemma \ref{lemma:nashl2}. By the \abbr{PDE} \eqref{eq:nashl2I0}, we get the following identity (analogous to \eqref{eq:nashl2I1}):
\begin{align}
&\tfrac{\d}{\d(-\s)}\|\boldsymbol{\Gamma}^{\mathrm{full},\bpsi}_{\s,\t,\x,\cdot}\|_{\ell^{2p}(\Z)}^{2p}=2p\sum_{\y\in\Z}(\boldsymbol{\Gamma}^{\mathrm{full},\bpsi}_{\s,\t,\x,\y})^{2p-1}\partial_{-\s}\boldsymbol{\Gamma}^{\mathrm{full},\bpsi}_{\s,\t,\x,\y}\label{eq:nashII1}\\
&=2p\sum_{\y\in\Z}(\boldsymbol{\Gamma}^{\mathrm{full},\bpsi}_{\s,\t,\x,\y})^{2p-1}\cdot\mathrm{e}^{\bpsi(\y)}\mathscr{S}_{\s,\y}\boldsymbol{\Gamma}^{\mathrm{full}}_{\s,\t,\x,\y}\mathrm{e}^{-\bpsi(\x)}+2p\sum_{\y\in\Z}(\boldsymbol{\Gamma}^{\mathrm{full},\bpsi}_{\s,\t,\x,\y})^{2p-1}\cdot\mathrm{e}^{\bpsi(\y)}\mathscr{D}_{\s,\y}^{\ast}\boldsymbol{\Gamma}^{\mathrm{full}}_{\s,\t,\x,\y}\mathrm{e}^{-\bpsi(\x)}.\nonumber
\end{align}
By Theorem 3.9 in \cite{CKS87}, we can estimate the first term in the second line of \eqref{eq:nashII1} as follows (for which we give a brief heuristic afterwards to explain the powers of {\small$p$} appearing below):
\begin{align}
&\sum_{\y\in\Z}(\boldsymbol{\Gamma}^{\mathrm{full},\bpsi}_{\s,\t,\x,\y})^{2p-1}\cdot\mathrm{e}^{\bpsi(\y)}\mathscr{S}_{\s,\y}\boldsymbol{\Gamma}^{\mathrm{full}}_{\s,\t,\x,\y}\mathrm{e}^{-\bpsi(\x)}\label{eq:nashII2}\\
&=-\N^{2}\sum_{\y\in\Z}\mathbf{c}_{\s,\y}\grad^{\mathbf{X}}_{-1}[\mathrm{e}^{\bpsi(\y)}(\boldsymbol{\Gamma}^{\mathrm{full},\bpsi}_{\s,\t,\x,\y})^{2p-1}]\cdot\grad^{\mathbf{X}}_{-1}\boldsymbol{\Gamma}^{\mathrm{full}}_{\s,\t,\x,\y}\mathrm{e}^{-\bpsi(\x)}\nonumber\\
&\leq-\tfrac{1}{p}\mathrm{c}\N^{2}\|\grad^{\mathbf{X}}_{-1}[(\boldsymbol{\Gamma}^{\mathrm{full},\bpsi}_{\s,\t,\x,\cdot})^{p}]\|_{\ell^{2}(\Z)}^{2}+\mathrm{C}p\N^{2}\Gamma(\bpsi)^{2}\|\boldsymbol{\Gamma}^{\mathrm{full},\bpsi}_{\s,\t,\x,\cdot}\|_{\ell^{2p}(\Z)}^{2p}.\nonumber
\end{align}
This first identity follows from the divergence-form structure of {\small$\mathscr{S}$} (i.e. summation-by-parts, which implies that the adjoint of {\small$\grad^{\mathbf{X}}_{\pm1}$} is equal to {\small$\grad^{\mathbf{X}}_{\mp1}$}). The inequality above follows from Theorem 3.9 in \cite{CKS87}, as mentioned earlier. (Intuitively, we can remove the {\small$\exp[\bpsi(\y)]$}-factor from the first {\small$\grad^{\mathbf{X}}_{-1}$}-operator in the second line above, and then we can move it inside the second {\small$\grad^{\mathbf{X}}_{-1}$}-operator therein. Because we are moving past two gradient operators, the cost is quadratic in {\small$\Gamma(\bpsi)$}, like in the last term above. For the scaling in {\small$p\geq1$}, we note that (a discrete version of) the chain rule suggests that the scaling in the second line of \eqref{eq:nashII2} is of order {\small$p$}. This is true also for the last term in \eqref{eq:nashII2}, and, again by (a discrete version of) the chain rule, it is true for the first term in the last line of \eqref{eq:nashII2} as well. Finally, all terms above are homogeneous in {\small$\boldsymbol{\Gamma}^{\mathrm{full},\bpsi}$} of degree {\small$2p$}, explaining the exponents in the previous display. Again, see Theorem 3.9 in \cite{CKS87} for a detailed proof; we do not reproduce the same argument here.)

Now take the last term in the second line of \eqref{eq:nashII1}. We start with the following generalization of the estimates from \eqref{eq:nashl2I3} and \eqref{eq:nashl2I4} to general {\small$2p\geq2$}, which follows by the same reasoning as those estimates did:
\begin{align}
&\sum_{\y\in\Z}(\boldsymbol{\Gamma}^{\mathrm{full},\bpsi}_{\s,\t,\x,\y})^{2p-1}\cdot\mathrm{e}^{\bpsi(\y)}\mathscr{D}^{\ast}_{\s,\y}\boldsymbol{\Gamma}^{\mathrm{full}}_{\s,\t,\x,\y}\mathrm{e}^{-\bpsi(\x)}=\sum_{\y\in\Z}\mathscr{D}_{\s,\y}[(\boldsymbol{\Gamma}^{\mathrm{full},\bpsi}_{\s,\t,\x,\y})^{2p-1}\cdot\mathrm{e}^{\bpsi(\y)}]\cdot\boldsymbol{\Gamma}^{\mathrm{full}}_{\s,\t,\x,\y}\mathrm{e}^{-\bpsi(\x)}\label{eq:nashII3a}\\
&\leq\sum_{\y\in\Z}\mathscr{D}_{\s,\y}[(\boldsymbol{\Gamma}^{\mathrm{full},\bpsi}_{\s,\t,\x,\y})^{2p-1}]\cdot\boldsymbol{\Gamma}^{\mathrm{full},\bpsi}_{\s,\t,\x,\y}+\mathrm{O}(\N^{\frac32+\e_{0}}\Gamma(\bpsi))\|\boldsymbol{\Gamma}^{\mathrm{full},\bpsi}_{\s,\t,\x,\cdot}\|_{\ell^{2p}(\Z)}^{2p}.\nonumber
\end{align}
We will now estimate the first term in the second line of \eqref{eq:nashII3a} in terms of the {\small$\mathrm{H}^{1}$}-norm in the last line of \eqref{eq:nashII2}. We start by carefully computing as follows (for which we give an explanation afterwards):
\begin{align}
\mathscr{D}_{\s,\y}[(\boldsymbol{\Gamma}^{\mathrm{full},\bpsi}_{\s,\t,\x,\y})^{2p-1}]\cdot\boldsymbol{\Gamma}^{\mathrm{full},\bpsi}_{\s,\t,\x,\y}&=\N^{\frac32}\sum_{|\w|\lesssim1}\mathbf{v}^{\w}_{\s,\y}\grad^{\mathbf{X}}_{\w}[((\boldsymbol{\Gamma}^{\mathrm{full},\bpsi}_{\s,\t,\x,\y})^{p})^{\frac{2p-1}{p}}]\cdot\boldsymbol{\Gamma}^{\mathrm{full},\bpsi}_{\s,\t,\x,\y}\label{eq:nashII3aa}\\
&\leq\N^{\frac32}\sum_{|\w|\lesssim1}|\mathbf{v}^{\w}_{\s,\y}||\grad^{\mathbf{X}}_{\w}[(\boldsymbol{\Gamma}^{\mathrm{full},\bpsi}_{\s,\t,\x,\y})^{p}]|\cdot(|\boldsymbol{\Gamma}^{\mathrm{full},\bpsi}_{\s,\t,\x,\y+\w}|^{\frac{p-1}{p}}+|\boldsymbol{\Gamma}^{\mathrm{full},\bpsi}_{\s,\t,\x,\y}|^{\frac{p-1}{p}})\cdot|\boldsymbol{\Gamma}^{\mathrm{full},\bpsi}_{\s,\t,\x,\y}|.\nonumber\\
&\lesssim\tfrac{1}{p}\eta\N^{2}\sum_{|\w|\lesssim1}|\grad^{\mathbf{X}}_{\w}[(\boldsymbol{\Gamma}^{\mathrm{full},\bpsi}_{\s,\t,\x,\y})^{p}]|^{2}+p\N^{1+\mathrm{O}(\e_{0})}\sum_{|\w|\lesssim1}|\boldsymbol{\Gamma}^{\mathrm{full},\bpsi}_{\s,\t,\x,\y+\w}|^{2p}.\nonumber
\end{align}
The first line is immediate. The second line follows from the inequality {\small$|\a^{q}-\b^{q}|\lesssim|\a-\b|\cdot(|\a|^{q-1}+|\b|^{q-1})$} as long as {\small$1\leq q\lesssim1$} (which is true for {\small$q=(2p-1)/p$} given {\small$p\geq1$}). The third line is true for every {\small$\eta>0$} small but fixed, and it follows by the generalized Young's inequality (to more than {\small$2$} factors). We note that any term of the form {\small$|\boldsymbol{\Gamma}^{\mathrm{full},\bpsi}_{\s,\t,\x,\y}|^{2p}$} is absorbed into the last term above. (As a sanity check, observe that every term in the last line is homogeneous in {\small$\boldsymbol{\Gamma}^{\mathrm{full},\bpsi}$} of degree {\small$2p$}, just like the entire sum in the second line above. Finally, the factors of {\small$\eta\N^{2}$} and {\small$\N^{1+\mathrm{O}(\e_{0})}$} can be understood by, in the second line above, taking a factor of {\small$\eta^{1/2}\N$} from {\small$\N^{3/2}$} therein and putting it into {\small$\grad^{\mathbf{X}}_{\w}[(\boldsymbol{\Gamma}^{\mathrm{full},\bpsi}_{\s,\t,\x,\y})^{p}]|$}.) Let us now sum over {\small$\y\in\Z$} in the previous display. In the last line, we recover the estimate {\small$\eta\N^{2}\sup_{|\w|\lesssim1}\|\grad^{\mathbf{X}}_{\w}[(\boldsymbol{\Gamma}^{\mathrm{full},\bpsi}_{\s,\t,\x,\cdot})^{p}]\|_{\ell^{2}(\Z)}^{2}\lesssim\eta\N^{2}\|\grad^{\mathbf{X}}_{-1}[(\boldsymbol{\Gamma}^{\mathrm{full},\bpsi}_{\s,\t,\x,\cdot})^{p}]\|_{\ell^{2}(\Z)}^{2}$}, since {\small$\grad^{\mathbf{X}}_{\w}$} is a sum of spatially shifted {\small$\grad^{\mathbf{X}}_{-1}$}-operators, and that the spatial shifts disappear after taking the {\small$\ell^{2}(\Z)$}-norm. By a similar token, the last term in the last line above turns into {\small$\lesssim\N^{1+\mathrm{O}(\e_{0})}\|(\boldsymbol{\Gamma}^{\mathrm{full},\bpsi}_{\s,\t,\x,\cdot})^{p}\|_{\ell^{2}(\Z)}^{2}=\N^{1+\mathrm{O}(\e_{0})}\|\boldsymbol{\Gamma}^{\mathrm{full},\bpsi}_{\s,\t,\x,\cdot}\|_{\ell^{2p}(\Z)}^{2p}$} after summing over {\small$\y\in\Z$}. So, if we combine the previous four displays, then we have (for possibly different but still fixed {\small$\mathrm{c},\mathrm{C}>0$})
\begin{align}
\tfrac{\d}{\d(-\s)}\|\boldsymbol{\Gamma}^{\mathrm{full},\bpsi}_{\s,\t,\x,\cdot}\|_{\ell^{2p}(\Z)}^{2p}&\leq-\mathrm{c}\N^{2}\|\grad^{\mathbf{X}}_{-1}[(\boldsymbol{\Gamma}^{\mathrm{full},\bpsi}_{\s,\t,\x,\cdot})^{p}]\|_{\ell^{2}(\Z)}^{2}+\mathrm{C}p^{2}(\N^{2}\Gamma(\bpsi)^{2}+\N^{\frac32+\e_{0}}\Gamma(\bpsi)+\N^{1+\mathrm{O}(\e_{0})})\|\boldsymbol{\Gamma}^{\mathrm{full},\bpsi}_{\s,\t,\x,\cdot}\|_{\ell^{2p}(\Z)}^{2p}.\nonumber
\end{align}
Instead of dropping the first term on the \abbr{RHS} above, we will now use the Nash inequality (see (3.18) in \cite{CKS87}); that is, we have {\small$\|\grad^{\mathbf{X}}_{-1}[(\boldsymbol{\Gamma}^{\mathrm{full},\bpsi}_{\s,\t,\x,\cdot})^{p}]\|_{\ell^{2}(\Z)}\gtrsim\|\boldsymbol{\Gamma}^{\mathrm{full},\bpsi}_{\s,\t,\x,\cdot}\|_{\ell^{2p}(\Z)}^{3p}\|\boldsymbol{\Gamma}^{\mathrm{full},\bpsi}_{\s,\t,\x,\cdot}\|_{\ell^{p}(\Z)}^{-2p}$}. Plugging this into the previous display and following the proof of (3.20) in \cite{CKS87} then yields the inequality
\begin{align*}
\tfrac{\d}{\d(-\s)}\|\boldsymbol{\Gamma}^{\mathrm{full},\bpsi}_{\s,\t,\x,\cdot}\|_{\ell^{2p}(\Z)}\leq-\tfrac{\mathrm{c}}{p}\N^{2}\|\boldsymbol{\Gamma}^{\mathrm{full},\bpsi}_{\s,\t,\x,\cdot}\|_{\ell^{2p}(\Z)}^{1+4p}\|\boldsymbol{\Gamma}^{\mathrm{full},\bpsi}_{\s,\t,\x,\cdot}\|_{\ell^{p}(\Z)}^{-4p}+\mathrm{C}p(\N^{2}\Gamma(\bpsi)^{2}+\N^{\frac32+\e_{0}}\Gamma(\bpsi)+\N^{1+\mathrm{O}(\e_{0})})\|\boldsymbol{\Gamma}^{\mathrm{full},\bpsi}_{\s,\t,\x,\cdot}\|_{\ell^{2p}(\Z)}.
\end{align*}
Observe the differential inequality above controls the {\small$\ell^{2p}(\Z)$}-norm in terms of the {\small$\ell^{p}(\Z)$}-norm, with an additional exponential growth coming from the last term above and the Gronwall inequality. Thus, by (1.5) in \cite{FS86}, we can bootstrap {\small$p\mapsto2p$} and control the {\small$\ell^{\infty}(\Z)$}-norm by the {\small$\ell^{2}(\Z)$}-norm, which is estimated in Lemma \ref{lemma:nashl2}. We get
\begin{align}
\|\boldsymbol{\Gamma}^{\mathrm{full},\bpsi}_{0,\t,\x,\cdot}\|_{\ell^{\infty}(\Z)}&\lesssim\N^{-\frac12}\t^{-\frac14}\exp[\mathrm{O}(\N^{2}\Gamma(\bpsi)^{2}\t+\N^{\frac32+\e_{0}}\Gamma(\bpsi)\t+\N^{1+\mathrm{O}(\e_{0})}\t)]\|\boldsymbol{\Gamma}^{\mathrm{full},\bpsi}_{0,\t,\x,\cdot}\|_{\ell^{2}(\Z)}\nonumber\\
&\lesssim\N^{-\frac12}\t^{-\frac14}\exp[\mathrm{O}(\N^{2}\Gamma(\bpsi)^{2}\t+\N^{\frac32+\e_{0}}\Gamma(\bpsi)\t+\N^{1+\mathrm{O}(\e_{0})}\t)].\label{eq:nashII3b}
\end{align}
From \eqref{eq:nashII3b} and {\small$\t\gtrsim\N^{-2}$}, we obtain (for any {\small$\x,\y\in\Z$}) that the following estimate holds for any {\small$\bpsi:\Z\to\R$} such that {\small$\|\grad^{\mathbf{X}}_{1}\bpsi\|_{\ell^{\infty}(\Z)}\lesssim1$}:
\begin{align*}
|\boldsymbol{\Gamma}^{\mathrm{full}}_{0,\t,\x,\y}|\lesssim\exp[\bpsi(\x)-\bpsi(\y)+\mathrm{O}(\N^{2}\Gamma(\bpsi)^{2}\t+\N^{\frac32+\e_{0}}\Gamma(\bpsi)\t+\N^{1+\mathrm{O}(\e_{0})}\t)]
\end{align*}
We now optimize over the choice of {\small$\bpsi$}. We first choose {\small$\bpsi\equiv0$}, which yields {\small$\|\boldsymbol{\Gamma}^{\mathrm{full}}_{0,\t,\cdot,\cdot}\|_{\ell^{\infty}(\Z\times\Z)}\lesssim\exp[\mathrm{O}(\N^{1+\e_{0}}\t)]$}. This, in particular, implies the desired estimate \eqref{eq:nashII} for {\small$\x=\y$}. On the other hand, take any {\small$\x_{0},\y_{0}\in\Z$} such that {\small$\x_{0}\neq\y_{0}$}, and suppose {\small$|\x_{0}-\y_{0}|=\x_{0}-\y_{0}$} (without loss of generality). We now choose {\small$\bpsi(\x)=\x\cdot(\N\t^{1/2})^{-1}$} and note that {\small$\t\gtrsim\N^{-2}$} implies the necessary assumption {\small$\|\grad^{\mathbf{X}}_{1}\bpsi\|_{\ell^{\infty}(\Z)}\lesssim1$}. In fact, we also have that {\small$\N^{2}\Gamma(\bpsi)^{2}\t\lesssim1$} and {\small$\N^{3/2+\e_{0}}\Gamma(\bpsi)\t\lesssim\N^{1/2+\e_{0}}\t^{1/2}$}, that latter of which is {\small$\lesssim1+\N^{1+\mathrm{O}(\e_{0})}\t$} because of the Schwarz inequality. So the previous display gives
\begin{align}
|\boldsymbol{\Gamma}^{\mathrm{full}}_{0,\t,\x_{0},\y_{0}}|&\lesssim\mathrm{e}^{\mathrm{O}(\N^{1+\mathrm{O}(\e_{0})}\t)}\cdot\exp[-\tfrac{(\x_{0}-\y_{0})}{\N\t^{1/2}}]=\mathrm{e}^{\mathrm{O}(\N^{1+\mathrm{O}(\e_{0})}\t)}\cdot\exp[-\tfrac{|\x_{0}-\y_{0}|}{\N\t^{1/2}}],\label{eq:nashII3c}
\end{align}
Since {\small$\x_{0},\y_{0}\in\Z$} are arbitrary distinct points, the argument is complete.
\end{proof}
\begin{proof}[Proof of Proposition \ref{prop:speed}]
Without loss of generality, assume that {\small$\inf\mathbb{U}=0$}. Indeed, we can shift the argument below in time by shifting the initial data of {\small$\mathbf{j}^{\N}$} in time.

We set {\small$\mathbf{g}:=\mathbf{j}^{\Z}-\mathbf{j}^{\N}$}, so that {\small$\mathbf{g}$} vanishes at the initial time {\small$\s=0$} by construction in Proposition \ref{prop:speed}. We also set {\small$\mathbb{L}:=\llbracket\mathfrak{l},\infty\rrbracket$}, where {\small$\mathfrak{l}=\lfloor\min\mathbb{I}/2\rfloor$}, so that {\small$\mathbb{L}$} lives ``between" {\small$\mathbb{I}$} and {\small$\llbracket0,\infty\rrbracket$}. We observe that on the spatial domain {\small$\mathbb{L}$}, the \abbr{SDE}s for {\small$\mathbf{j}^{\N},\mathbf{j}^{\Z}$} from \eqref{eq:currI} and Definition \ref{definition:speed}, respectively, are the same, since there are no edge dynamics to deal with on {\small$\mathbb{L}$}. Thus, we can obtain an \abbr{SDE} for the restriction of {\small$\mathbf{g}$} to {\small$\mathbb{L}$} by linearizing the \abbr{SDE} \eqref{eq:currI}. To be precise, we use \eqref{eq:currI} and Definition \ref{definition:speed} to get the following for any {\small$(\s,\y)\in[0,\mathfrak{t}]\times\mathbb{L}$}, which we explain after:
\begin{align}
\d\mathbf{g}_{\s,\y}=-\N^{2}\grad^{\mathbf{X}}_{1}(\mathscr{U}''[\boldsymbol{\eta}_{\s,\y}]\grad^{\mathbf{X}}_{-1}\mathbf{g}_{\s,\y})\d\s+\N(\mathbf{F}[\tau_{\y}\bphi^{\Z}_{\s}]-\mathbf{F}[\tau_{\y}\bphi_{\s}])\d\s.\nonumber
\end{align}
Indeed, the Brownian motions in \eqref{eq:currI} and Definition \ref{definition:speed} are coupled and therefore cancel. Moreover, {\small$\boldsymbol{\eta}$} denotes a random field (that may be correlated in a nontrivial way to {\small$\mathbf{j}^{\N},\mathbf{j}^{\Z}$}) that is determined by the mean-value theorem to compute {\small$\mathscr{U}'[\bphi^{\Z}]-\mathscr{U}'[\bphi]=\mathscr{U}''[\boldsymbol{\eta}](\bphi^{\Z}-\bphi)=-\N^{1/2}\mathscr{U}''[\boldsymbol{\eta}]\grad^{\mathbf{X}}_{-1}\mathbf{g}$}. Next, we again use the mean-value theorem to compute the last term on the \abbr{RHS} above. By the formula \eqref{eq:nl} for the {\small$\mathbf{F}$}-nonlinearity, we have
\begin{align*}
\N(\mathbf{F}[\tau_{\y}\bphi^{\Z}_{\s}]-\mathbf{F}[\tau_{\y}\bphi_{\s}])=\sum_{|\w|\lesssim1}\mathfrak{f}_{\w,\y}[\bphi^{\Z}_{\s},\bphi_{\s}]\cdot\N^{\frac32}\grad^{\mathbf{X}}_{\w}\mathbf{g}_{\s,\y},
\end{align*}
where {\small$\mathfrak{f}_{\w,\y}[\bphi^{\Z}_{\s},\bphi_{\s}]$} satisfies the estimate below, which follows by the potential estimates in Assumption \ref{assump:potential}:
\begin{align*}
|\mathfrak{f}_{\w,\y}[\bphi^{\Z}_{\s},\bphi_{\s}]|&\lesssim1+\sum_{|\z|\lesssim1}(|\bphi^{\Z}_{\s,\y+\z}|+|\bphi_{\s,\y+\z}|)^{\mathrm{O}(1)}.
\end{align*}
We clarify that in principle, the \abbr{RHS} of the formula for {\small$\N(\mathbf{F}[\tau_{\y}\bphi^{\Z}_{\s}]-\mathbf{F}[\tau_{\y}\bphi_{\s}])$} should have {\small$\N(\bphi^{\Z}_{\s,\y+\w}-\bphi^{\N}_{\s,\y+\w})$} in place of {\small$\N^{3/2}\grad^{\mathbf{X}}_{\w}\mathbf{g}_{\s,\y}$}. However, we can rewrite {\small$\N(\bphi^{\Z}_{\s,\y+\w}-\bphi^{\N}_{\s,\y+\w})=\N^{3/2}(\mathbf{j}^{\Z}_{\s,\y+\w}-\mathbf{j}^{\N}_{\s,\y+\w})-\N^{3/2}(\mathbf{j}^{\Z}_{\s,\y+\w-1}-\mathbf{j}^{\N}_{\s,\y+\w-1})=\N^{3/2}(\mathbf{g}_{\s,\y+\w}-\mathbf{g}_{\s,\y+\w-1})=\N^{3/2}(\mathbf{g}_{\s,\y+\w}-\mathbf{g}_{\s,\y})-\N^{3/2}(\mathbf{g}_{\s,\y+\w-1}-\mathbf{g}_{\s,\y})$}, and upon rearranging the coefficients {\small$\mathfrak{f}_{\w,\y}$} in the formula for {\small$\N(\mathbf{F}[\tau_{\y}\bphi^{\Z}_{\s}]-\mathbf{F}[\tau_{\y}\bphi_{\s}])$}, we arrive at the form above.

Now, set {\small$\mathbf{k}_{\s,\y}:=\mathbf{g}_{\s,\y}\cdot\mathbf{1}[\y\in\mathbb{K}]$}, where {\small$\mathbb{K}:=\llbracket\mathfrak{k},\infty\rrbracket$} with {\small$\mathfrak{k}=\lfloor\frac23\cdot\min\mathbb{I}\rfloor$} lives between {\small$\mathbb{I},\mathbb{L}$}. Because {\small$\mathbb{K}\subseteq\mathbb{L}$}, the equation above for {\small$\mathbf{g}$} on {\small$[0,\mathfrak{t}]\times\mathbb{L}$} also holds on {\small$[0,\mathfrak{t}]\times\mathbb{K}$}. Moreover, the gradients vanish on {\small$\mathbf{1}[\y\in\mathbb{K}]$} unless {\small$\y$} is within distance {\small$\mathrm{O}(1)$} from the boundary of {\small$\mathbb{K}$}, which we write as {\small$\mathbf{1}[\y\approx\partial\mathbb{K}]=1$} for convenience only. Similarly, we have {\small$\mathbf{g}_{\s,\y+\w}\mathbf{1}[\y\in\mathbb{K}]=\mathbf{k}_{\s,\y+\w}$} unless {\small$\mathbf{1}[\y\approx\partial\mathbb{K}]=1$}. Thus, given any {\small$(\s,\y)\in[0,\mathfrak{t}]\times\Z$} (note that {\small$\mathbf{k}_{\s,\y}$} extends to this domain via extension-by-zero), we have 
\begin{align}
\partial_{\s}\mathbf{k}_{\s,\y}=-\N^{2}\grad^{\mathbf{X}}_{1}(\mathscr{U}''[\boldsymbol{\eta}_{\s,\y}]\grad^{\mathbf{X}}_{-1}\mathbf{k}_{\s,\y})+\N^{\frac32}\sum_{|\w|\lesssim1}\mathfrak{f}_{\w,\y}[\bphi^{\Z}_{\s},\bphi_{\s}]\cdot\grad^{\mathbf{X}}_{\w}\mathbf{k}_{\s,\y}+\N^{2}\mathbf{1}_{\y\approx\partial\mathbb{K}}\mathfrak{f}_{\y}[\bphi^{\Z}_{\s},\bphi_{\s}]\cdot\sup_{|\w|\lesssim1}|\mathbf{g}_{\s,\y+\w}|.\nonumber
\end{align}
Here, the function {\small$\mathfrak{f}_{\y}$} admits the same estimate as the {\small$\mathfrak{f}_{\w,\y}$}-coefficient above. We will now perform the following conditioning, which restricts us to events on which we have ``good estimates".
\begin{itemize}
\item We have {\small$|\bphi_{\s,\u}|\lesssim\N^{\e_{\star}}$} for all {\small$(\s,\u)\in[0,\mathfrak{t}]\times\mathbb{L}$}. This holds with probability {\small$1-\mathrm{O}_{\mathrm{D}}(\N^{-\mathrm{D}})$} (for arbitrarily small but fixed {\small$\e_{\star}>0$} and for any {\small$\mathrm{D}>0$}) because of Lemma \ref{lemma:phiap}.
\item We have {\small$|\bphi^{\Z}_{\s,\u}|\lesssim\N^{\e_{\star}}$} for all {\small$(\s,\u)\in[0,\mathfrak{t}]\times\mathbb{L}$}. This follows for a similar reason, because the initial data of {\small$\bphi^{\Z}$} agrees with that of {\small$\bphi$} and thus has density with respect to {\small$\mathbb{P}^{0}$} which satisfies the estimate \eqref{eq:dataI} as well. Since {\small$\bphi^{\Z}$} also has {\small$\mathbb{P}^{0}$} (or rather its extension to {\small$\R^{\Z}$}) as an invariant measure, the proof of Lemma \ref{lemma:phiap} also applies for the {\small$\bphi^{\Z}$}-process (even if we allow for the intervals to live in {\small$\Z$}). {\color{black}We refer to Section A.3 in \cite{Y26PTRF} for the fact that the extension of {\small$\mathbb{P}^{0}$} onto {\small$\R^{\Z}$} is an invariant measure for the {\small$\bphi^{\Z}$}-process.}
\item We get, from the bullet points above, global estimates  of {\small$\lesssim\N^{\e_{0}}$} on the {\small$\mathfrak{f}_{\w,\y}[\bphi^{\Z}_{\s},\bphi_{\s}]$}-coefficients in the previous \abbr{PDE} for any small but fixed {\small$\e_{0}>0$}. This holds with probability {\small$1-\mathrm{O}_{\mathrm{D}}(\N^{-\mathrm{D}})$} for any {\small$\mathrm{D}$}. Indeed, the function {\small$\mathbf{k}$} and its gradients are supported on {\small$\mathbb{K}\subseteq\mathbb{L}$}, so said coefficients in the \abbr{PDE} above can be restricted to {\small$\y\in\mathbb{L}$}, on which we can apply the previous two bullet points. The same bound holds for {\small$\mathfrak{f}_{\y}[\bphi^{\Z}_{\s},\bphi_{\s}]$}-coefficients.
\item We have {\small$|\mathbf{g}_{\s,\u}|\lesssim\N^{\mathrm{O}(1)}$} for all {\small$(\s,\u)\in[0,\mathfrak{t}]\times\mathbb{L}$} such that {\small$\u\approx\partial\mathbb{K}$}. This holds with probability {\small$1-\mathrm{O}_{\mathrm{D}}(\N^{-\mathrm{D}})$} for any {\small$\mathrm{D}>0$}. This follows because it is true for {\small$\mathbf{j}^{\Z},\mathbf{j}^{\N}$} separately (since these satisfy \abbr{SDE}s whose coefficients are polynomially bounded in {\small$\bphi^{\Z}_{\s,\u},\bphi_{\s,\u}$}-variables for {\small$(\s,\u)\in\mathbb{U}\times\Z$}; see \eqref{eq:currI} and Definition \ref{definition:speed}).
\end{itemize}

We now let {\small$\boldsymbol{\Gamma}^{\mathrm{full}}$} denote the heat kernel from Proposition \ref{prop:nashdavies} with {\small$\mathbf{c}_{\s,\y}=\mathscr{U}''[\boldsymbol{\eta}_{\s,\y}]$} (see Assumption \ref{assump:potential} for the necessary bounds on {\small$\mathscr{U}''$}) and with {\small$\mathbf{v}^{\w}_{\s,\y}=\mathfrak{f}_{\w,\y}[\bphi^{\Z}_{\s},\bphi_{\s}]$}. Using the \abbr{PDE} above for {\small$\mathbf{k}$} and the Duhamel formula, we have the following display:
\begin{align}
\mathbf{k}_{\s,\y}=\N^{\mathrm{O}(1)}\int_{0}^{\s}\sum_{\z\approx\partial\mathbb{K}}\boldsymbol{\Gamma}^{\mathrm{full}}_{\r,\s,\y,\z}.\label{eq:speedI1}
\end{align}
This holds for all {\small$(\s,\y)\in[0,\mathfrak{t}]\times\Z$} simultaneously with probability {\small$1-\mathrm{O}_{\mathrm{D}}(\N^{-\mathrm{D}})$}, where {\small$\mathrm{D}>0$} is a large but fixed constant. We clarify that there is no initial data term above because {\small$\mathbf{k}_{0,\cdot}$} vanishes globally, as {\small$\mathbf{k}_{0,\cdot}$} is a restriction of {\small$\mathbf{g}_{0,\cdot}$} to {\small$\mathbb{K}$}, and {\small$\mathbf{g}=\mathbf{j}^{\Z}-\mathbf{j}^{\N}$} is assumed to have vanishing initial data on {\small$\Z$}.

If we fix {\small$\y\in\mathbb{I}$}, then for any {\small$\z\approx\partial\mathbb{K}$}, we have {\small$|\y-\z|\gtrsim\N^{\e}(1+\N\mathfrak{t}^{1/2}+\N^{2}\mathfrak{t}^{3/2})$}. Thus, by Proposition \ref{prop:nashdavies}, the {\small$\boldsymbol{\Gamma}^{\mathrm{full}}$}-terms on the \abbr{RHS} of \eqref{eq:speedI1} are {\small$\lesssim\exp[\mathrm{O}(\N^{1+\mathrm{O}(\e_{0})}\s)]\exp[-\upsilon\N^{\e}(1+\N\mathfrak{t}^{1/2}+\N^{2}\mathfrak{t}^{3/2})(\N\delta_{\r,\s})^{-1/2}]$} with {\small$\upsilon>0$} and {\small$\delta_{\r,\s}\leq\max(\N^{-2},\s)$} like in Proposition \ref{prop:nashdavies}. This product of exponentials is {\small$\lesssim\exp[-\upsilon'\N^{\e/2}]\lesssim_{\mathrm{D}}\N^{-\mathrm{D}}$} for any {\small$\mathrm{D}>0$} and some {\small$\upsilon'>0$}, if we take {\small$\e_{0}$} much smaller than {\small$\e$} and use {\small$\s\leq\mathfrak{t}$}. The \abbr{RHS} of \eqref{eq:speedI1} sums {\small$\mathrm{O}(1)$}-many of these exponentially small terms, so this finishes the proof.
\end{proof}
\subsubsection{Proofs of Propositions \ref{prop:case1} and \ref{prop:case2}}
We are now prepared to derive the remaining two ingredients towards Proposition \ref{prop:stochmain1}. Let us present the proof of Proposition \ref{prop:case1} first.
\begin{proof}[Proof of Proposition \ref{prop:case1}]
Let us first assume that the initial data for {\small$\t\mapsto\bphi_{\t}$} is sampled via the measure {\small$\mathbb{P}^{0}$}. We remove this assumption at the end of this proof. Recall the notation from \eqref{eq:stochmain1I1bb}. We claim that for any path-space event {\small$\mathscr{E}$} (to be determined shortly), we have the following for any {\small$\mathrm{D}>0$}:
\begin{align*}
&\E\Big(\mathbf{1}_{\mathscr{E}}\tfrac{1}{|\mathbb{U}_{\k,\ell}|}\int_{\mathbb{U}_{\k,\ell}}\sum_{\w\in\mathbb{I}_{\x,\k}}\boldsymbol{\mathcal{H}}_{\x,\w}\cdot\wt{\boldsymbol{\omega}}^{\k,\ell}_{\s,\t,\w,\x}\d\s\Big)^{2}\\
&\lesssim\E\Big(\mathbf{1}_{\mathscr{E}}\tfrac{1}{|\mathbb{U}_{\k,\ell}|}\int_{\mathbb{U}_{\k,\ell}}\sum_{\w\in\mathbb{I}_{\x,\k}}\boldsymbol{\mathcal{H}}_{\x,\w}\cdot\wt{\boldsymbol{\omega}}^{\k,\ell}_{\s,\t,\w,\x}\cdot\boldsymbol{\eta}(\|\mathbf{D}\mathfrak{q}_{\w}[\bphi_{\t+\s}]\|_{\infty})\d\s\Big)^{2}+\mathrm{O}_{\mathrm{D}}(\N^{-\mathrm{D}})
\end{align*}
Here, {\small$\boldsymbol{\eta}(\cdot)$} is from Notation \ref{notation:cut}, {\small$\mathbf{D}\mathfrak{q}_{\w}[\bphi]$} denotes the vector whose entries, parameterized by {\small$\x\in\llbracket1,\infty\rrbracket$}, are given by {\small$\partial_{\bphi_{\x}}\mathfrak{q}[\bphi]$}, where {\small$\|\cdot\|_{\infty}$} is the max-entry norm. Indeed, this display follows since {\small$|\wt{\boldsymbol{\omega}}^{\k,\ell}|\lesssim\N^{\mathrm{O}(1)}$} deterministically by construction in \eqref{eq:stochmain1I1bb} and Notation \ref{notation:cut}, whereas {\small$\boldsymbol{\eta}(\|\mathbf{D}\mathfrak{q}_{\w}[\bphi_{\t+\s}]\|_{\infty})=1$} with probability {\small$1-\mathrm{O}_{\mathrm{D}}(\N^{-\mathrm{D}})$} for any {\small$\mathrm{D}>0$}. (To prove this last claim, the gradient estimate for bulk-admissible functions in \eqref{eq:bulkadmissibleI} yields a polynomial estimate for {\small$\|\mathbf{D}\mathfrak{q}_{\w}[\bphi_{\t+\s}]\|_{\infty}$}. Then, we use the {\small$\bphi_{\t}$}-estimate in Lemma \ref{lemma:phiap} to control this polynomial upper bound. This implies that {\small$\|\mathbf{D}\mathfrak{q}_{\w}[\bphi_{\t+\s}]\|_{\infty}\lesssim\N^{\e_{0}}$} for small {\small$\e_{0}>0$}. Now, recall from Notation \ref{notation:cut} that  {\small$\boldsymbol{\eta}(\cdot)\equiv1$} for {\small$|\cdot|\leq\N^{\e}$} for a small but fixed {\small$\e>0$}.)

Now, recall from Notation \ref{notation:msblock} that {\small$\min\mathbb{I}_{\x,\k}\gtrsim\N^{\gamma_{\k-1}}$} for {\small$\k\in\llbracket2,\mathrm{K}\rrbracket$}. Indeed, it is separated from the left-edge of {\small$\mathbb{I}_{\x}\subseteq\llbracket0,\infty\rrbracket$} by a distance of {\small$\gtrsim\N^{\gamma_{\k-1}}$} according to Notation \ref{notation:msblock} if {\small$\k\in\llbracket2,\mathrm{K}\rrbracket$}. Moreover, we recall {\small$|\mathbb{U}_{\k,\ell}|$} from Notation \ref{notation:msblock}, and we note that {\small$\N^{\e}(\N|\mathbb{U}_{\k,\ell}|^{1/2}+\N^{2}|\mathbb{U}_{\k,\ell}|^{3/2}+1)\gtrsim\N^{\gamma_{\k-1}}$}. This allows us to use Proposition \ref{prop:speed} in the following manner. Recall the notation from Definition \ref{definition:speed}. We claim that 
\begin{align*}
&\E\Big(\mathbf{1}_{\mathscr{E}}\tfrac{1}{|\mathbb{U}_{\k,\ell}|}\int_{\mathbb{U}_{\k,\ell}}\sum_{\w\in\mathbb{I}_{\x,\k}}\boldsymbol{\mathcal{H}}_{\x,\w}\cdot\wt{\boldsymbol{\omega}}^{\k,\ell}_{\s,\t,\w,\x}\cdot\boldsymbol{\eta}(\|\mathbf{D}\mathfrak{q}_{\w}[\bphi_{\t+\s}]\|_{\infty})\d\s\Big)^{2}\\
&\lesssim\E\Big(\mathbf{1}_{\mathscr{E}}\tfrac{1}{|\mathbb{U}_{\k,\ell}|}\int_{0}^{|\mathbb{U}_{\k,\ell}|}\sum_{\w\in\mathbb{I}_{\x,\k}}\boldsymbol{\mathcal{H}}_{\x,\w}\cdot\wt{\boldsymbol{\omega}}^{\k,\ell,\Z}_{\s,\t,\w,\x}\d\s\Big)^{2}+\mathrm{O}_{\mathrm{D}}(\N^{-\mathrm{D}}),
\end{align*}
where, with notation to be explained shortly, we set
\begin{align}
\cdot\wt{\boldsymbol{\omega}}^{\k,\ell,\Z}_{\s,\t,\w,\x}:=\wt{\mathfrak{q}}_{\w}[\bphi^{\Z}_{\s}]\cdot\mathbf{G}_{\w}[\bphi^{\Z}_{\s}]\cdot\mathbf{J}^{\k,\ell,\x,\Z}_{\s}.
\end{align}
The process {\small$\bphi^{\Z}$} was constructed in Definition \ref{definition:speed}, and its initial data {\small$\bphi^{\Z}_{0}\in\R^{\Z}$} is distributed via {\small$\mathbb{P}^{0}$}. (Indeed, its initial data is constructed to agree with that of {\small$\s\mapsto\bphi_{\t+\s}$} restricted to {\small$\llbracket1,\infty\rrbracket$}, whereas its restriction to {\small$\llbracket-\infty,0\rrbracket$} is independently distributed according to the {\small$\R^{\llbracket-\infty,0\rrbracket}$}-marginal of {\small$\mathbb{P}^{0}$}.) The process {\small$\mathbf{J}^{\k,\ell,\x,\Z}$} is given by the formula for {\small$\mathbf{J}^{\k,\ell,\x}_{\s}$} from Notation \ref{notation:cut} but replacing {\small$\mathbf{j}^{\N}_{\t+\cdot,\cdot}$} therein by {\small$\mathbf{j}^{\Z}$}. The estimate above follows since the term {\small$\wt{\boldsymbol{\omega}}^{\k,\ell}_{\s,\t,\w,\x}$} therein is a smooth function of {\small$\mathbf{j}^{\N}_{\t+\s,\z}$} for {\small$\z$} in some neighborhood of {\small$\w$} and {\small$\x_{\k}=\max\mathbb{I}_{\x,\k}$} of radius {\small$\mathrm{O}(1)$}; moreover, the derivatives of this function {\small$\wt{\boldsymbol{\omega}}^{\k,\ell}_{\s,\t,\w,\x}$} are {\small$\lesssim\N^{\mathrm{O}(1)}$} since we have a factor of {\small$\boldsymbol{\eta}(\|\mathbf{D}\mathfrak{q}_{\w}[\bphi_{\t+\s}]\|_{\infty})$} on the \abbr{LHS} above. So, we can replace {\small$\wt{\boldsymbol{\omega}}^{\k,\ell}$} with {\small$\wt{\boldsymbol{\omega}}^{\k,\ell,\Z}$} up to an error of {\small$\mathrm{O}_{\mathrm{D}}(\N^{-\mathrm{D}})$} for any {\small$\mathrm{D}>0$}, again by Proposition \ref{prop:speed}. Finally, by the same reasoning which led to the first estimate in this proof, we can remove the factor of {\small$\boldsymbol{\eta}(\|\mathbf{D}\mathfrak{q}_{\w}[\bphi_{\t+\s}]\|_{\infty})$} at the cost of another {\small$\mathrm{O}_{\mathrm{D}}(\N^{-\mathrm{D}})$} error. 

Now, to complete the proof, we cite Proposition 6.6 in \cite{Y26PTRF}, which produces \abbr{CLT}-cancellations in space-time for the \abbr{RHS} of the previous estimate. That is, we have the following with explanation given after:
\begin{align}
\E\Big(\mathbf{1}_{\mathscr{E}}\tfrac{1}{|\mathbb{U}_{\k,\ell}|}\int_{0}^{|\mathbb{U}_{\k,\ell}|}\sum_{\w\in\mathbb{I}_{\x,\k}}\boldsymbol{\mathcal{H}}_{\x,\w}\cdot\wt{\boldsymbol{\omega}}^{\k,\ell,\Z}_{\s,\t,\w,\x}\d\s\Big)^{2}\lesssim_{}\N^{\mathrm{O}(\e)}\N^{-2}|\mathbb{U}_{\k,\ell}|^{-1}\|\boldsymbol{\mathcal{H}}_{\x,\cdot}\|_{\ell^{\infty}(\mathbb{I}_{\x,\k})}^{2}|\mathbb{I}_{\x,\k}|+\N^{-3+\mathrm{O}(\e)}.\label{eq:star}
\end{align}
A few clarifications are in order here:
\begin{itemize}
\item To fit in the framework of Proposition 6.6 of \cite{Y26PTRF}, we first replace {\small$\wt{\mathfrak{q}}_{\w}$} in {\small$\boldsymbol{\omega}^{\k,\ell,\Z}$} by the bulk-admissible function {\small$\mathfrak{q}_{\w}$}. As explained above, this comes at a cost of {\small$\lesssim_{\mathrm{D}}\N^{-\mathrm{D}}$} for any {\small$\mathrm{D}>0$}. 
\item Similarly, to fit in the framework of Proposition 6.6 of \cite{Y26PTRF}, we view the space-time integral on the \abbr{LHS} above as the average over {\small$(\s,\w)\in[0,|\mathbb{U}_{\k,\ell}|]\times\mathbb{I}_{\x,\k}$} of the quantity {\small$|\mathbb{I}_{\x,\k}|\boldsymbol{\mathcal{H}}_{\x,\w}\cdot\wt{\boldsymbol{\omega}}^{\k,\ell,\Z}_{\s,\t,\w,\x}$}. In particular, we bake the factor of {\small$|\mathbb{I}_{\x,\k}|\boldsymbol{\mathcal{H}}_{\x,\w}$} into {\small$\wt{\mathfrak{q}}_{\w}[\bphi^{\Z}_{\s}]$} in {\small$\wt{\boldsymbol{\omega}}^{\k,\ell,\Z}_{\s,\t,\w,\x}$}; this does not change the property of {\small$\wt{\mathfrak{q}}_{\w}$} being bulk admissible (besides the fact that its size changes by a factor of {\small$|\mathbb{I}_{\x,\k}|\boldsymbol{\mathcal{H}}_{\x,\w}$}).
\item It is shown in Lemma 6.8 of \cite{Y26PTRF} that if {\small$\mathfrak{w}_{\w}$} is a bulk-admissible family which is also backwards-oriented (see Definition \ref{definition:bulkadmissible}), then we have the following for some high probability path-space event {\small$\mathscr{E}$} (like in the statement of Proposition \ref{prop:stochmain1}):
\begin{align*}
\E^{0}\Big(\mathbf{1}_{\mathscr{E}}\tfrac{1}{|\mathbb{U}_{\k,\ell}|}\int_{0}^{|\mathbb{U}_{\k,\ell}|}\tfrac{1}{|\mathbb{I}_{\x,\k}|}\sum_{\w\in\mathbb{I}_{\x,\k}}\mathfrak{w}_{\w}[\bphi^{\Z}_{\s}]\cdot\mathbf{G}_{\w}[\bphi^{\Z}_{\s}]\cdot\mathbf{J}^{\k,\ell,\x,\Z}_{\s}\d\s\Big)^{2}\lesssim\N^{\mathrm{O}(\e)}\N^{-2}|\mathbb{U}_{\k,\ell}|^{-1}|\mathbb{I}_{\x,\k}|^{-1}\sup_{\w\in\mathbb{I}_{\x,\k}}\E^{0}|\mathfrak{w}_{\w}|^{2}+\N^{-3+\mathrm{O}(\e)}.
\end{align*}
Above, {\small$\E^{0}$} denotes the expectation with respect to either the measure {\small$\mathbb{P}^{0}$} itself or with respect to the law of the {\small$\bphi^{\Z},\mathbf{j}^{\Z}$}-processes, assuming {\small$\bphi^{\Z}_{0}\sim\mathbb{P}^{0}$}. (This is shown under an additional ``fluctuation" assumption {\small$\E^{\sigma}\mathfrak{w}_{\w}=0$} for all {\small$\sigma\in\R$}, but the proof of Proposition 6.6 in \cite{Y26PTRF} removes this assumption by replacing {\small$\mathfrak{w}_{\w}$} with a finite sum of fluctuations plus a term which is {\small$\lesssim\N^{-3+\mathrm{O}(\e)}$}.) Finally, to conclude the estimate \eqref{eq:star}, it suffices to choose {\small$\mathfrak{w}_{\w}=\mathfrak{q}_{\w}\cdot|\mathbb{I}_{\x,\k}|\cdot\boldsymbol{\mathcal{H}}_{\x,\w}$} and use {\small$\E^{0}|\mathfrak{q}_{\w}|^{2}\lesssim1$}, the last of which follows by \eqref{eq:bulkadmissibleI}.
\end{itemize}
Combining the three estimates in this argument yields the desired estimate \eqref{eq:case1I} in the case of {\small$\mathbb{P}^{0}$} initial data. For more general data satisfying \eqref{eq:dataI0}, we again change measure at a multiplicative cost of {\small$\N^{\gamma_{\mathrm{data}}}$}. 
\end{proof}
We now prove Proposition \ref{prop:case2}, which requires a specialized analysis at the edge.
\begin{proof}[Proof of Proposition \ref{prop:case2}]
We will consider two separate cases, the first of which corresponds to the case {\small$\min\mathbb{I}_{\x,1}\geq\N^{\gamma_{1}}$} as in Notation \ref{notation:msblock}. In this case, we have {\small$|\mathbb{U}_{1,\ell}|=|\mathbb{U}_{2,\ell}|$}, and we know that the distance between {\small$\mathbb{I}_{\x,1}$} and the edge {\small$0\in\llbracket0,\infty\rrbracket$} is {\small$\gtrsim\N^{\gamma_{1}}$}. This is the same scenario as in Proposition \ref{prop:case1} for {\small$\k=2$} (see Notation \ref{notation:msblock}), except the block {\small$\mathbb{I}_{\x,1}$} on which we sum on the \abbr{LHS} of the desired bound \eqref{eq:case2I} is smaller than {\small$\mathbb{I}_{\x,\k}$} for {\small$\k=2$}. In particular, we can borrow Proposition \ref{prop:case1} for the case {\small$\k=2$} to get the estimate
\begin{align*}
\E\Big(\mathbf{1}_{\mathscr{E}}\tfrac{1}{|\mathbb{U}_{1,\ell}|}\int_{\mathbb{U}_{1,\ell}}\sum_{\w\in\mathbb{I}_{\x,1}}\boldsymbol{\mathcal{H}}_{\x,\w}\cdot\wt{\mathfrak{q}}_{\w}[\bphi_{\t+\s}]\cdot\mathbf{G}_{\w}[\bphi_{\t+\s}]\d\s\Big)^{2}\lesssim\N^{\mathrm{O}(\e)}\N^{-2}|\mathbb{U}_{1,\ell}|^{-1}|\mathbb{I}_{\x,1}|\cdot\|\boldsymbol{\mathcal{H}}_{\x,\cdot}\|_{\ell^{\infty}(\mathbb{I}_{\x,1})}^{2}+\N^{-2-\beta_{\mathrm{CLT}}}.
\end{align*}
(Strictly speaking, Proposition \ref{prop:case1} requires there to be a factor of {\small$\mathbf{J}^{\k,\ell,\x}_{\s}$} inside the space-time integration on the \abbr{LHS}. However, one can always reinsert this factor up to an error of {\small$\N^{-2-\beta_{\mathrm{CLT}}}$} because of \eqref{eq:stochmain1I3}. In any case, the fluctuations which lead to the cancellation estimate on the \abbr{RHS} come from {\small$\wt{\mathfrak{q}}_{\w}$}, not this missing factor of {\small$\mathbf{J}^{\k,\ell,\x}_{\s}$}.) We now use {\small$|\mathbb{I}_{\x,1}|\asymp\N^{\gamma_{1}}$} and {\small$|\mathbb{U}_{1,\ell}|=|\mathbb{U}_{2,\ell}|\gtrsim\N^{-\e}\N^{-2}\N^{2\gamma_{1}}$} from Notation \ref{notation:msblock} to deduce that the first term on the \abbr{RHS} above is {\small$\lesssim\N^{-2}\N^{-\gamma_{1}+\mathrm{O}(\e)+\mathrm{O}(\delta_{\mathbf{S}})}\lesssim\N^{-2-\beta_{\mathrm{CLT}}}$}, since {\small$\e,\delta_{\mathbf{S}}>0$} are small. This provides the desired estimate \eqref{eq:case2I} in the case {\small$\min\mathbb{I}_{\x,1}\geq\N^{\gamma_{1}}$}. So, for the rest of this proof, we will assume that {\small$\min\mathbb{I}_{\x,1}\leq\N^{\gamma_{1}}$}, in which case {\small$|\mathbb{U}_{1,\ell}|\asymp\N^{-3/2}$}. For this, we again replace {\small$\E$} on the \abbr{LHS} of the desired estimate \eqref{eq:case2I} by {\small$\E^{0}$}, namely expectation with respect to the law of the process {\small$\s\mapsto\bphi_{\t+\s}$} with initial data {\small$\bphi_{\t+\inf\mathbb{U}_{1,\ell}}\sim\mathbb{P}^{0}$}. By \eqref{eq:dataI0}, this requires us to only insert a factor of {\small$\N^{\gamma_{\mathrm{data}}}$}, which is absorbed by {\small$\N^{-\beta_{\mathrm{CLT}}}$} on the \abbr{RHS} of \eqref{eq:case2I} anyways since {\small$\gamma_{\mathrm{data}}>0$} is small. Next, let us replace {\small$\wt{\mathfrak{q}}_{\w}$} on the \abbr{LHS} of \eqref{eq:case2I} by {\small$\mathfrak{q}_{\w}$}; by the paragraph before \eqref{eq:stochmain1I1b}, this replacement is justified on some high probability path-space event {\small$\mathscr{E}$} as in the statement of Proposition \ref{prop:stochmain1}.

We now observe that the quantity that we integrate on the \abbr{LHS} of the desired estimate \eqref{eq:case2I} depends only on the process {\small$\s\mapsto\bphi_{\t+\s}$}; it has {\small$\mathbb{P}^{0}$} as a stationary measure, and the symmetric part of its infinitesimal generator is given by the following (as shown in Section \ref{subsection:generator}):
\begin{align*}
\mathscr{L}_{\mathrm{S}}:=-\N^{2}\sum_{\y\in\llbracket1,\infty\rrbracket}\grad^{\mathbf{X}}_{1}\mathscr{U}'[\bphi_{\y}](\partial_{\bphi_{\y+1}}-\partial_{\bphi_{\y}})+\N^{2}\sum_{\y\in\llbracket1,\infty\rrbracket}(\partial_{\bphi_{\y+1}}-\partial_{\bphi_{\y}})^{2}+\N^{2}\partial_{\bphi_{1}}^{2}.
\end{align*}
Thus, by the Kipnis-Varadhan inequality (see Lemma 2.4 of \cite{KLO12}), we have 
\begin{align*}
\E^{0}\Big(\tfrac{1}{|\mathbb{U}_{1,\ell}|}\int_{\mathbb{U}_{1,\ell}}\sum_{\w\in\mathbb{I}_{\x,1}}\boldsymbol{\mathcal{H}}_{\x,\w}\cdot{\mathfrak{q}}_{\w}[\bphi_{\t+\s}]\cdot\mathbf{G}_{\w}[\bphi_{\t+\s}]\d\s\Big)^{2}&\lesssim\tfrac{1}{|\mathbb{U}_{1,\ell}|^{2}}\int_{\mathbb{U}_{1,\ell}}\Big\|\sum_{\w\in\mathbb{I}_{\x,1}}\boldsymbol{\mathcal{H}}_{\x,\w}\cdot{\mathfrak{q}}_{\w}[\bphi]\cdot\mathbf{G}_{\w}[\bphi]\Big\|_{\mathrm{H}^{-1}}^{2}\d\s\\
&\lesssim\tfrac{1}{|\mathbb{U}_{1,\ell}|}\Big\|\sum_{\w\in\mathbb{I}_{\x,1}}\boldsymbol{\mathcal{H}}_{\x,\w}\cdot{\mathfrak{q}}_{\w}[\bphi]\cdot\mathbf{G}_{\w}[\bphi]\Big\|_{\mathrm{H}^{-1}}^{2}=:\tfrac{1}{|\mathbb{U}_{1,\ell}|}\|\mathfrak{f}\|_{\mathrm{H}^{-1}}^{2},
\end{align*}
where the second estimate holds since the integrand on the \abbr{RHS} of the first line is independent of {\small$\s\in\mathbb{U}_{1,\ell}$}, and
\begin{align*}
&\|\mathfrak{f}\|_{\mathrm{H}^{-1}}^{2}:=\sup_{\mathfrak{g}\in\mathscr{C}^{\infty}(\R^{\llbracket1,\infty\rrbracket},\R)}\Big\{2\E^{0}(\mathfrak{f}[\bphi]\mathfrak{g}[\bphi])+\E^{0}(\mathfrak{g}[\bphi]\mathscr{L}_{\mathrm{S}}\mathfrak{g}[\bphi])\Big\}.
\end{align*}
Note that {\small$\E^{0}\mathfrak{f}[\bphi]=0$}, since {\small$\E^{0}{\mathfrak{q}}_{\w}[\bphi]=0$} (recall that it is a bulk-admissible function), since {\small${\mathfrak{q}}_{\w}[\bphi],\mathbf{G}_{\w}[\bphi]$} depend on {\small$\bphi_{\y}$}-spins for disjoint sets of {\small$\y$}-points (see Notation \ref{notation:cut} and recall the backwards-orientation of {\small${\mathfrak{q}}_{\w}$}), and since {\small$\mathbb{P}^{0}$} is a product measure on {\small$\R^{\llbracket1,\infty\rrbracket}$}. Moreover, {\small$\mathscr{L}_{\mathrm{S}}$} vanishes on constant functions. So, we can restrict to functions {\small$\mathfrak{g}$} which satisfy {\small$\E^{0}\mathfrak{g}=0$}. Next, a standard integration-by-parts (see the proof of Lemma 6.8 in \cite{Y26PTRF}) shows that 
\begin{align*}
\E^{0}(\mathfrak{g}[\bphi]\mathscr{L}_{\mathrm{S}}\mathfrak{g}[\bphi])&=-\N^{2}\sum_{\y\in\llbracket1,\infty\rrbracket}\E^{0}(|(\partial_{\bphi_{\y+1}}-\partial_{\bphi_{\y}})\mathfrak{g}[\bphi]|^{2})-\N^{2}\E^{0}(|\partial_{\bphi_{1}}\mathfrak{g}[\bphi]|^{2})\\
&\leq-\N^{2}\sum_{\y,\y+1\in\mathbb{K}}\E^{0}(|(\partial_{\bphi_{\y+1}}-\partial_{\bphi_{\y}})\mathfrak{g}[\bphi]|^{2})-\N^{2}\E^{0}(|\partial_{\bphi_{1}}\mathfrak{g}[\bphi]|^{2}),
\end{align*}
where {\small$\mathbb{K}\subseteq\llbracket1,\infty\rrbracket$} denotes an interval which contains {\small$1,\mathbb{I}_{\x,1}$}, has length {\small$\lesssim\N^{\gamma_{1}}$}, and is such that {\small$\mathfrak{f}[\bphi]$} depends only on {\small$\bphi_{\y}$} for {\small$\y\in\mathbb{K}$}.; such an interval exists because {\small$|\mathbb{I}_{\x,1}|\asymp\N^{\gamma_{1}}$}, because {\small$\min\mathbb{I}_{\x,1}\lesssim\N^{\gamma_{1}}$}, and because {\small$\mathfrak{q}_{\w}[\bphi],\mathbf{G}_{\w}[\bphi]$} depend only on {\small$\bphi_{\y}$} for {\small$\mathrm{dist}(\y,\mathbb{I}_{\x,1})\lesssim1$} if {\small$\w\in\mathbb{I}_{\x,1}$}. (Recall the locality properties of bulk-admissible functions in Definition \ref{definition:bulkadmissible} and the definition of {\small$\mathbf{G}_{\w}$} from Notation \ref{notation:cut}.) Because {\small$\mathfrak{f}[\bphi]$} depends only on {\small$\bphi_{\y}$} for {\small$\y\in\mathbb{K}$}, we have {\small$\E^{0}(\mathfrak{f}\mathfrak{g})=\E^{0}(\mathfrak{f}\E^{0}_{\mathbb{K}}\mathfrak{g})$}, where {\small$\E^{0}_{\mathbb{K}}$} is the expectation conditioning on {\small$\bphi_{\y}$} for each {\small$\y\in\mathbb{K}$}. Moreover, with a standard convexity argument (see the proof of Lemma 3.5 in \cite{DGP17}), the last line of the previous display can only increase if we replace {\small$\mathfrak{g}\mapsto\E^{0}_{\mathbb{K}}\mathfrak{g}$}. Ultimately, we deduce that 
\begin{align*}
\|\mathfrak{f}\|_{\mathrm{H}^{-1}}^{2}&\leq\sup_{\mathfrak{g}\in\mathscr{C}^{\infty}(\R^{\mathbb{K}},\R)}\Big\{2\E^{0}(\mathfrak{f}[\bphi]\mathfrak{g}[\bphi])-\N^{2}\mathfrak{D}_{\mathbb{K}}(\mathfrak{g})\Big\},
\end{align*}
in which {\small$\mathfrak{D}_{\mathbb{K}}(\mathfrak{g})=\sum_{\y,\y+1\i\mathbb{K}}\E^{0}(|(\partial_{\bphi_{\y+1}}-\partial_{\bphi_{\y}})\mathfrak{g}|^{2})+\E^{0}(|\partial_{\bphi_{1}}\mathfrak{g}|^{2})$} is a convenient and standard shorthand for the ``Dirichlet form". We now claim that the following inequality holds for any {\small$\eta>0$} small but fixed:
\begin{align}
|2\E^{0}(\mathfrak{f}[\bphi]\mathfrak{g}[\bphi])|\lesssim\eta\N^{2}\mathfrak{D}_{\mathbb{K}}(\mathfrak{g})+\N^{-2}|\mathbb{K}|^{2}\E^{0}(\mathfrak{f}^{2}).\label{eq:case2I1}
\end{align}
In short, \eqref{eq:case2I1} is a type of Schwarz inequality, and we will show it at the end of this argument. For now, we will take \eqref{eq:case2I1} for granted. Combining the previous two displays and recalling {\small$\mathfrak{f}$} gives
\begin{align*}
\|\mathfrak{f}\|_{\mathrm{H}^{-1}}^{2}\lesssim\N^{-2}|\mathbb{K}|^{2}\E^{0}\Big(\sum_{\w\in\mathbb{I}_{\x,1}}\boldsymbol{\mathcal{H}}_{\x,\w}\cdot{\mathfrak{q}}_{\w}[\bphi]\cdot\mathbf{G}_{\w}[\bphi]\Big)^{2}\lesssim\N^{-2}\N^{4\gamma_{1}}\N^{\mathrm{O}(\e)}\|\boldsymbol{\mathcal{H}}_{\x,\cdot}\|_{\ell^{\infty}(\mathbb{I}_{\x,1})}^{2}\lesssim\N^{-4+4\gamma_{1}+\mathrm{O}(\e)+\mathrm{O}(\delta_{\mathbf{S}})},
\end{align*}
where the last estimate follows because bulk-admissible functions are polynomially bounded and thus have finite second moments, and {\small$|\mathbf{G}_{\w}|\lesssim\N^{\e}$} by construction in Notation \ref{notation:cut}. Ultimately, from this, we obtain
\begin{align*}
\E^{0}\Big(\tfrac{1}{|\mathbb{U}_{1,\ell}|}\int_{\mathbb{U}_{1,\ell}}\sum_{\w\in\mathbb{I}_{\x,1}}\boldsymbol{\mathcal{H}}_{\x,\w}\cdot{\mathfrak{q}}_{\w}[\bphi_{\t+\s}]\cdot\mathbf{G}_{\w}[\bphi_{\t+\s}]\d\s\Big)^{2}&\lesssim\N^{-4+4\gamma_{1}+\mathrm{O}(\e)+\mathrm{O}(\delta_{\mathbf{S}})}|\mathbb{U}_{1,\ell}|^{-1}\\
&\lesssim\N^{-4+4\gamma_{1}+\mathrm{O}(\e)+\mathrm{O}(\delta_{\mathbf{S}})}\N^{\frac32}\lesssim\N^{-2-\beta_{\mathrm{CLT}}}
\end{align*}
This implies the desired estimate \eqref{eq:case2I}, so the proof is complete as soon as we show \eqref{eq:case2I1}. For this, we first use the Schwarz inequality to get {\small$|2\E^{0}(\mathfrak{f}[\bphi]\mathfrak{g}[\bphi])|\lesssim\eta\N^{2}|\mathbb{K}|^{-2}\E^{0}(|\mathfrak{g}|^{2})+\N^{-2}|\mathbb{K}|^{2}\E^{0}(|\mathfrak{f}|^{2}$} for {\small$\eta>0$} small but fixed. Thus, it suffices to show {\small$\E^{0}(|\mathfrak{g}|^{2})\lesssim|\mathbb{K}|^{2}\mathfrak{D}_{\mathbb{K}}(\mathfrak{g})$}. We claim {\small$\E^{0}(|\mathfrak{g}|^{2})\lesssim\wt{\mathfrak{D}}_{\mathbb{K}}(\mathfrak{g})$}, where {\small$\wt{\mathfrak{D}}_{\mathbb{K}}(\mathfrak{g})=\sum_{\y\in\mathbb{K}}\E^{0}(|\partial_{\bphi_{\y}}\mathfrak{g}|^{2})$} is the site-wise Dirichlet form. This follows from standard Bakry-Emery theory. Indeed, the functional {\small$\wt{\mathfrak{D}}_{\mathbb{K}}$} is the Dirichlet form for independent Langevin dynamics on {\small$\R^{\mathbb{K}}$}, for which we have a positive spectral gap. It therefore suffices to show that {\small$\wt{\mathfrak{D}}_{\mathbb{K}}(\mathfrak{g})\lesssim|\mathbb{K}|^{2}\mathfrak{D}_{\mathbb{K}}(\mathfrak{g})$}. This follows from the standard Poincar\'{e} inequality on {\small$\mathbb{K}$}, namely that {\small$\|\mathsf{f}\|_{\ell^{2}(\mathbb{K})}^{2}\lesssim|\mathsf{f}_{1}|^{2}+|\mathbb{K}|^{2}\|\grad^{\mathbf{X}}_{1}\mathsf{f}\|_{\ell^{2}(\mathbb{K})}^{2}$}, applied to {\small$\mathsf{f}_{\y}:=\partial_{\bphi_{\y}}\mathfrak{g}$}. This completes the proof.
\end{proof}
\subsection{Proof of Proposition \ref{prop:stochmain2}}\label{subsection:stochmain2}
The argument is very similar to the proof of Proposition \ref{prop:case2} given immediately above, specifically the latter case therein. Thus, we will inherit the ideas and notation therein.
\begin{proof}[Proof of Proposition \ref{prop:stochmain2}]
As {\small$\mathfrak{b}$} is edge-admissible (see Definition \ref{definition:edgeadmissible}), it satisfies the bound \eqref{eq:bulkadmissibleI}. In particular, by Lemma \ref{lemma:phiap}, the \abbr{RHS} of \eqref{eq:bulkadmissibleI} is {\small$\lesssim\N^{\e}$} for any fixed {\small$\e>0$} with high probability. Moreover, by Lemma \ref{lemma:currtreg}, we know that {\small$\mathbf{Z}^{\N}_{\t-\s,\x}(\mathbf{Z}^{\N}_{\t,\x})^{-1}=1+\mathrm{O}(\N^{-\rho})$} uniformly in {\small$(\t,\x)\in[0,\mathrm{T}]\times\llbracket0,\mathrm{O}(1)\rrbracket$} and {\small$\s\in[0,\mathfrak{t}_{\partial}]:=[0,\N^{-3/2}]$} (see Definition \ref{definition:edgeaverage}) with high probability. Thus, similar to \eqref{eq:stochmain1I3}, for some high probability path-space event {\small$\mathscr{E}$} as in the statement of Proposition \ref{prop:stochmain2}, we have the following estimate for some {\small$\rho>0$}:
\begin{align}
\E(\mathbf{1}_{\mathscr{E}}|\mathbf{Av}^{\mathbf{T},\mathfrak{b}}_{\t,\x}|^{2})&=\E\Big(\mathbf{1}_{\mathscr{E}}\cdot\mathfrak{t}_{\partial}^{-1}\int_{0}^{\mathfrak{t}_{\partial}}\mathfrak{b}[\bphi_{\t-\s}]\cdot\mathbf{Z}^{\N}_{\t-\s,\x}(\mathbf{Z}^{\N}_{\t,\x})^{-1}\d\s\Big)^{2}\lesssim\E\Big(\mathfrak{t}_{\partial}^{-1}\int_{0}^{\mathfrak{t}_{\partial}}\mathfrak{b}[\bphi_{\t-\s}]\d\s\Big)^{2}+\mathrm{O}(\N^{-\rho}).\nonumber
\end{align}
By a standard change-of-variables, we can replace {\small$\mathfrak{b}[\bphi_{\t-\s}]\mapsto\mathfrak{b}[\bphi_{\t-\mathfrak{t}_{\partial}+\s}]$} in the integral on the far \abbr{RHS}. Next, we can change measure from {\small$\E$} to {\small$\E^{0}$} if we give up a factor of {\small$\N^{\gamma_{\mathrm{data}}}$}, where {\small$\E^{0}$} means the expectation with respect to the path-space law of {\small$\bphi_{\t-\mathfrak{t}_{\partial}+\cdot}$} whose initial data is distributed according to {\small$\mathbb{P}^{0}$} from Definition \ref{definition:gc}. Therefore, by the Kipnis-Varadhan inequality (as used in the proof of Proposition \ref{prop:case2} above), we have the estimate
\begin{align*}
\E\Big(\mathfrak{t}_{\partial}^{-1}\int_{0}^{\mathfrak{t}_{\partial}}\mathfrak{b}[\bphi_{\t-\s}]\d\s\Big)^{2}\lesssim\N^{\gamma_{\mathrm{data}}}\mathfrak{t}_{\partial}^{-1}\|\mathfrak{b}\|_{\mathrm{H}^{-1}}^{2}.
\end{align*}
As in the proof of Proposition \ref{prop:case2}, if {\small$\E^{0}\mathfrak{b}=0$} (which is true for the edge-admissible functions in Definition \ref{definition:edgeadmissible}), then we have {\small$\|\mathfrak{b}\|_{\mathrm{H}^{-1}}^{2}\lesssim\N^{-2}|\mathbb{K}|^{2}\E^{0}|\mathfrak{b}|^{2}$}, in which {\small$\mathbb{K}$} is any interval containing both {\small$0$} and all {\small$\y\in\llbracket0,\infty\rrbracket$} such that {\small$\mathfrak{b}[\bphi]$} depends on {\small$\bphi_{\y}$}. Because edge-admissible terms have a local support (see Definition \ref{definition:edgeadmissible}), we know {\small$|\mathbb{K}|\lesssim1$}. Because said functions also satisfy a polynomial estimate \eqref{eq:bulkadmissibleI}, we obtain {\small$\E^{0}|\mathfrak{b}|^{2}\lesssim1$}. Therefore, the \abbr{RHS} of the previous display is {\small$\lesssim\N^{\gamma_{\mathrm{data}}}\N^{-2}\mathfrak{t}_{\partial}^{-1}\lesssim\N^{-1/2+\gamma_{\mathrm{data}}}$}, since {\small$\mathfrak{t}_{\partial}\gtrsim\N^{-3/2}$} (see Definition \ref{definition:edgeaverage}). Combining this with the previous two displays completes the proof.
\end{proof}
%
%
%
\section{Proof of Proposition \ref{prop:hkzeta}}\label{section:hkzeta}
The section is organized as follows. First, we will modify the stochastic heat kernel \abbr{SDE} \eqref{eq:hkzetaI1}-\eqref{eq:hkzetaI6} for {\small$\mathbf{K}^{\N,\zeta}$} in a way that does not change it with high probability but makes it technically more accessible to analyze. Then, we decompose the (modified) {\small$\mathbf{K}^{\N,\zeta}$}-kernel in terms of another stochastic kernel, denoted by {\small$\mathbf{L}^{\N,\zeta,\sim}$}, that solves an \abbr{SDE} that only includes terms in \eqref{eq:hkzetaI1}, \eqref{eq:hkzetaI4}, and \eqref{eq:hkzetaI6}. The motivation here is that these terms are estimated in a very different topology compared to the other terms in \eqref{eq:hkzetaI1}-\eqref{eq:hkzetaI6}. For example, the noise in \eqref{eq:hkzetaI1} is analyzed using high moment estimates (as is the case for the limit \abbr{SHE} \eqref{eq:sheIa}-\eqref{eq:sheIb}). On the other hand, the {\small$\mathbf{Av}^{\mathbf{T},\mathbf{X},\mathscr{Q}}$}-terms in \eqref{eq:hkzetaI3} admit only second moment estimates in a fairly analytically weak sense (see Proposition \ref{prop:stochmain1}).

First, however, before we begin, we introduce the following stopping time that will be part of our analysis by means of providing various a priori estimates:
\begin{align}
\mathfrak{t}_{\mathrm{ap}}:=1&\wedge\inf\Big\{\s\in[0,1]:\max_{\x\in\llbracket0,\N^{2+2\zeta+\zeta_{\mathrm{large}}}\rrbracket}|\bphi_{\s,\x}|\geq\N^{\delta_{\mathbf{S}}}\Big\}\label{eq:tap}\\
&\wedge\inf\Big\{\s\in[0,1]:\max_{\ell\in\llbracket1,\mathrm{L}_{1}\rrbracket}\max_{\x\in\llbracket0,\N^{2+2\zeta+\zeta_{\mathrm{large}}}\rrbracket}|\mathbf{Av}^{\mathbf{X},\mathscr{Q}_{\ell}}_{\s,\x}|\geq\N^{-\frac12+\delta_{\mathbf{S}}}\Big\}\nonumber\\
&\wedge\inf\Big\{\s\in[0,1]:\max_{\ell\in\llbracket1,\mathrm{L}_{1}\rrbracket}\max_{\x\in\llbracket0,\N^{2+2\zeta+\zeta_{\mathrm{large}}}\rrbracket}|\mathbf{Av}^{\mathbf{T},\mathbf{X},\mathscr{Q}_{\ell}}_{\s,\x}|\geq\N^{-\frac12+\delta_{\mathbf{S}}}\Big\}\nonumber\\
&\wedge\inf\Big\{\s\in[0,1]:\max_{\ell\in\llbracket1,\mathrm{L}_{2}\rrbracket}\max_{\x\in\llbracket0,\mathrm{O}(1)\rrbracket}|\mathbf{Av}^{\mathbf{T},\mathfrak{b}_{\ell}}_{\s,\x}|\geq\N^{\mathrm{C}\delta_{\mathbf{S}}}\Big\}\nonumber\\
&\wedge\inf\Big\{\s\in[0,1]:\max_{\x\in\llbracket0,\N^{2+2\zeta+2\zeta_{\mathrm{large}}}\rrbracket}|\mathbf{R}^{\N}_{\s,\x}-\mathbf{R}^{\N,\wedge}_{\s,\x}|\neq0\Big\}.\nonumber
\end{align}
Above, we have {\small$\mathrm{C}\lesssim1$}. The use of {\small$\N^{2+2\zeta+\zeta_{\mathrm{large}}}$} is technical; it is there to guarantee we cover the (spatial support) of {\small$\boldsymbol{\chi}^{(\zeta)}$}. We clarify that {\small$\mathfrak{t}_{\mathrm{ap}}$} is, indeed, a stopping time, because the time-averages in the third and fourth lines are backwards in time and are thus adapted to the natural filtration associated to \eqref{eq:currI}-\eqref{eq:currII}.

Above, recall {\small$\delta_{\mathbf{S}}>0$} as the small constant from Definition \ref{definition:heat}, as well as {\small$\zeta_{\mathrm{large}}=\mathrm{O}(1)$} and {\small$\mathbf{R}^{\N,\wedge}$} from \eqref{eq:rwedge}. We also recall the collections {\small$\mathscr{Q}_{\ell}$} of bulk-admissible functions from Proposition \ref{prop:sde}, used in the stochastic heat kernel \abbr{SDE} \eqref{eq:hkzetaI1}-\eqref{eq:hkzetaI6}. As an immediate consequence of Lemmas \ref{lemma:phiap}, \ref{lemma:rratio}, \ref{lemma:cltx}, and \ref{lemma:cltxt?}, we have
\begin{align}
\mathbb{P}(\mathfrak{t}_{\mathrm{ap}}=1)=1-\mathrm{o}(1).\label{eq:tap2}
\end{align}
%
\subsection{Technical modification}
We let {\small$\mathbf{K}^{\N,\zeta,\sim}_{\s,\t,\x,\y}$} be a function of {\small$(\s,\t,\x,\y)\in[0,\infty)^{2}\times\llbracket0,\infty\rrbracket^{2}$}, restricted to times {\small$\s\leq\t$}, which solves the following modified version of \eqref{eq:hkzetaI1}-\eqref{eq:hkzetaI6} (as an \abbr{SDE} in the {\small$(\t,\x)$}-variables, so that {\small$(\s,\y)$} are treated as fixed), in which the modifications are marked in red font for clarity:
\begin{align}
\d\mathbf{K}^{\N,\zeta,\sim}_{\s,\t,\x,\y}&=\mathscr{T}_{\N}\mathbf{K}^{\N,\zeta,\sim}_{\s,\t,\x,\y}\d\t+[\mathbf{e}^{\mathfrak{t}_{\N}\mathscr{T}_{\N}}(\mathbf{1}^{(\x)}_{\cdot}{\color{blue}\boldsymbol{\chi}^{(\zeta_{\mathrm{large}})}_{\cdot}}\cdot\sqrt{2}\lambda\N^{\frac12}{\color{blue}\mathbf{R}^{\N,\wedge}_{\t,\cdot}}\mathbf{K}^{\N,\zeta,\sim}_{\s,\t,\cdot,\y}\d\mathbf{b}_{\t,\cdot})]_{\x}+{\ocolor{red}\mathbf{1}_{\mathfrak{t}_{\mathrm{ap}}\geq\t}}\mathrm{Err}[{\color{blue}\boldsymbol{\chi}^{(\zeta_{})}_{}\mathbf{R}^{\N,\wedge}}\mathbf{K}^{\N,\zeta,\sim}_{\s,\cdot,\cdot,\y}]_{\t,\x}\d\t\nonumber\\
&+{\ocolor{red}\mathbf{1}_{\mathfrak{t}_{\mathrm{ap}}\geq\t}}\sum_{\ell\in\llbracket1,\mathrm{L}_{1}\rrbracket}{}\sum_{\w\in\llbracket0,\infty\rrbracket}\boldsymbol{\mathcal{K}}^{(\ell)}_{\x,\w}\cdot{\color{blue}\boldsymbol{\chi}^{(\zeta_{})}_{\w}}\cdot\N\cdot\mathbf{Av}^{\mathbf{T},\mathbf{X},\mathscr{Q}_{\ell}}_{\t,\w}{\color{blue}\mathbf{R}^{\N,\wedge}_{\t,\w_{\pm}}}\mathbf{K}^{\N,\zeta,\sim}_{\s,\t,\w_{\pm},\y}\d\t\nonumber\\
&+\sum_{\ell\in\llbracket1,\mathrm{L}_{1}\rrbracket}{}\grad^{\mathbf{T},\mathrm{av}}_{\mathfrak{t}_{\mathbf{Av}}}\Big({\ocolor{red}\mathbf{1}_{\mathfrak{t}_{\mathrm{ap}}\geq\t}}\sum_{\w\in\llbracket0,\infty\rrbracket}\boldsymbol{\mathcal{K}}^{(\ell)}_{\x,\w}\cdot{\color{blue}\boldsymbol{\chi}^{(\zeta_{})}_{\w}}\cdot\N\cdot\mathbf{Av}^{\mathbf{X},\mathscr{Q}_{\ell}}_{\t,\w}{\color{blue}\mathbf{R}^{\N,\wedge}_{\t,\w_{\pm}}}\mathbf{K}^{\N,\zeta,\sim}_{\s,\t,\w_{\pm},\y}\Big)\d\t\nonumber\\
&+{\ocolor{red}\mathbf{1}_{\mathfrak{t}_{\mathrm{ap}}\geq\t}}\sum_{\ell\in\llbracket1,\mathrm{L}_{2}\rrbracket}\sum_{\w\in\llbracket0,\infty\rrbracket}\boldsymbol{\mathcal{K}}^{\partial,\ell}_{\x,\w}\cdot\N\mathbf{Av}^{\mathbf{T},\mathfrak{b}_{\ell}}_{\t,\w}{\color{blue}\mathbf{R}^{\N,\wedge}_{\t,\w}}\mathbf{K}^{\N,\zeta,\sim}_{\s,\t,\w,\y}\d\t\nonumber\\
&+\sum_{\ell\in\llbracket1,\mathrm{L}_{2}\rrbracket}\grad^{\mathbf{T},\mathrm{av}}_{\mathfrak{t}_{\partial}}\Big({\ocolor{red}\mathbf{1}_{\mathfrak{t}_{\mathrm{ap}}\geq\t}}\sum_{\w\in\llbracket0,\infty\rrbracket}\boldsymbol{\mathcal{K}}^{\partial,\ell}_{\x,\w}\cdot\N\mathfrak{b}_{\ell}[\bphi_{\t}]{\color{blue}\mathbf{R}^{\N,\wedge}_{\t,\w}}\mathbf{K}^{\N,\zeta,\sim}_{\s,\t,\w,\y}\Big)\d\t.\label{eq:hkzetasim}
\end{align}
The initial data is also {\small$\mathbf{K}^{\N,\zeta,\sim}_{\s,\s,\x,\y}=\mathbf{1}_{\x=\y}$}. In words, we modify the original stochastic heat kernel \abbr{SDE} \eqref{eq:hkzetaI1}-\eqref{eq:hkzetaI6} by adding the stopping time {\small$\mathfrak{t}_{\mathrm{ap}}$} in front of essentially any term for which we require a priori stochastic estimates. However, by \eqref{eq:tap2}, these modifications are harmless with high probability. Indeed, we have the following.
\begin{lemma}\label{lemma:hkzetasim}
\fsp With high probability, we have {\small$\mathbf{K}^{\N,\zeta}_{\s,\t,\x,\y}=\mathbf{K}^{\N,\zeta,\sim}_{\s,\t,\x,\y}$} for all {\small$0\leq\s\leq\t\leq1$} and {\small$\x,\y\in\llbracket0,\infty\rrbracket$}.
\end{lemma}
\begin{proof}
The difference {\small$\mathbf{K}^{\N,\zeta}_{\s,\t,\x,\y}-\mathbf{K}^{\N,\zeta,\sim}_{\s,\t,\x,\y}$}, as a process in {\small$\t\geq\s$}, solves the linear \abbr{SDE} \eqref{eq:hkzetasim} until the stopping time {\small$\mathfrak{t}_{\mathrm{ap}}$} (since both terms in this difference solve this equation until time {\small$\mathfrak{t}_{\mathrm{ap}}$}, and the equation itself is linear). Also, its initial data at {\small$\t=\s$} vanishes identically. Thus, by uniqueness of solutions to the well-posed linear \abbr{SDE} (e.g. via Picard iteration), this difference must vanish until time {\small$\mathfrak{t}_{\mathrm{ap}}$}, so the claim follows by \eqref{eq:tap2}.
\end{proof}
\subsection{Decomposition into auxiliary kernels}
We now expand {\small$\mathbf{K}^{\N,\zeta,\sim}$} around a kernel {\small$\mathbf{L}^{\N,\zeta,\sim}$}; see the beginning of this section for motivation in doing so. Let {\small$\mathbf{L}^{\N,\zeta,\sim}_{\s,\t,\x,\y}$} be a function of {\small$(\s,\t,\x,\y)\in[0,\infty)^{2}\times\llbracket0,\infty\rrbracket^{2}$}, restricted to {\small$\s\leq\t$}, which solves the \abbr{SDE} below in {\small$(\t,\x)$}-variables obtained by taking the first, third, and fifth lines of \eqref{eq:hkzetasim}:
\begin{align}
\d\mathbf{L}^{\N,\zeta,\sim}_{\s,\t,\x,\y}&=\mathscr{T}_{\N}\mathbf{L}^{\N,\zeta,\sim}_{\s,\t,\x,\y}\d\t+[\mathbf{e}^{\mathfrak{t}_{\N}\mathscr{T}_{\N}}(\mathbf{1}^{(\x)}_{\cdot}{\color{blue}\boldsymbol{\chi}^{(\zeta_{\mathrm{large}})}_{\cdot}}\cdot\sqrt{2}\lambda\N^{\frac12}{\color{blue}\mathbf{R}^{\N,\wedge}_{\t,\cdot}}\mathbf{L}^{\N,\zeta,\sim}_{\s,\t,\cdot,\y}\d\mathbf{b}_{\t,\cdot})]_{\x}+{\ocolor{red}\mathbf{1}_{\mathfrak{t}_{\mathrm{ap}}\geq\t}}\mathrm{Err}[{\color{blue}\boldsymbol{\chi}^{(\zeta_{})}_{}\mathbf{R}^{\N,\wedge}}\mathbf{L}^{\N,\zeta,\sim}_{\s,\cdot,\cdot,\y}]_{\t,\x}\d\t\nonumber\\
&+\sum_{\ell\in\llbracket1,\mathrm{L}_{1}\rrbracket}{}\grad^{\mathbf{T},\mathrm{av}}_{\mathfrak{t}_{\mathbf{Av}}}\Big({\ocolor{red}\mathbf{1}_{\mathfrak{t}_{\mathrm{ap}}\geq\t}}\sum_{\w\in\llbracket0,\infty\rrbracket}\boldsymbol{\mathcal{K}}^{(\ell)}_{\x,\w}\cdot{\color{blue}\boldsymbol{\chi}^{(\zeta_{})}_{\w}}\cdot\N\cdot\mathbf{Av}^{\mathbf{X},\mathscr{Q}_{\ell}}_{\t,\w}{\color{blue}\mathbf{R}^{\N,\wedge}_{\t,\w_{\pm}}}\mathbf{L}^{\N,\zeta,\sim}_{\s,\t,\w_{\pm},\y}\Big)\d\t\nonumber\\
&+\sum_{\ell\in\llbracket1,\mathrm{L}_{2}\rrbracket}\grad^{\mathbf{T},\mathrm{av}}_{\mathfrak{t}_{\partial}}\Big({\ocolor{red}\mathbf{1}_{\mathfrak{t}_{\mathrm{ap}}\geq\t}}\sum_{\w\in\llbracket0,\infty\rrbracket}\boldsymbol{\mathcal{K}}^{\partial,\ell}_{\x,\w}\cdot\N\mathfrak{b}_{\ell}[\bphi_{\t}]{\color{blue}\mathbf{R}^{\N,\wedge}_{\t,\w}}\mathbf{L}^{\N,\zeta,\sim}_{\s,\t,\w,\y}\Big)\d\t.\label{eq:lzetasim}
\end{align}
We take the ``fundamental" initial data at {\small$\t=\s$} given by {\small$\mathbf{L}^{\N,\zeta,\sim}_{\s,\s,\x,\y}=\mathbf{1}_{\x=\y}$}.

The following result, which is essentially a Duhamel formula, expands {\small$\mathbf{K}^{\N,\zeta,\sim}$} around {\small$\mathbf{L}^{\N,\zeta,\sim}$}.
\begin{lemma}\label{lemma:duhamel}
\fsp \fsp For any {\small$(\s,\t,\x,\y)\in[0,\infty)^{2}\times\llbracket0,\infty\rrbracket^{2}$} satisfying {\small$\s\leq\t$}, we have
\begin{align}
\mathbf{K}^{\N,\zeta,\sim}_{\s,\t,\x,\y}&=\mathbf{L}^{\N,\zeta,\sim}_{\s,\t,\x,\y}+\int_{\s}^{\t}\sum_{\w\in\llbracket0,\infty\rrbracket}\mathbf{L}^{\N,\zeta,\sim}_{\r,\t,\x,\w}\cdot\boldsymbol{\Psi}_{\r}[\mathbf{K}^{\N,\zeta,\sim}_{\s,\r,\cdot,\y}]_{\w}\d\r+\int_{\s}^{\t}\sum_{\w\in\llbracket0,\infty\rrbracket}\mathbf{L}^{\N,\zeta,\sim}_{\r,\t,\x,\w}\cdot\boldsymbol{\Phi}_{\r}[\mathbf{K}^{\N,\zeta,\sim}_{\s,\r,\cdot,\y}]_{\w}\d\r.\label{eq:duhamelI1}
\end{align}
For the sake of a convenient and clean presentation, we have used the notation
\begin{align}
\boldsymbol{\Psi}_{\r}[\mathbf{K}^{\N,\zeta,\sim}_{\s,\r,\cdot,\y}]_{\w}&:={\ocolor{red}\mathbf{1}_{\mathfrak{t}_{\mathrm{ap}}\geq\r}}\sum_{\ell\in\llbracket1,\mathrm{L}_{1}\rrbracket}{}\sum_{\z\in\llbracket0,\infty\rrbracket}\boldsymbol{\mathcal{K}}^{(\ell)}_{\w,\z}\cdot{\color{blue}\boldsymbol{\chi}^{(\zeta_{})}_{\z}}\cdot\N\cdot\mathbf{Av}^{\mathbf{T},\mathbf{X},\mathscr{Q}_{\ell}}_{\r,\z}{\color{blue}\mathbf{R}^{\N,\wedge}_{\r,\z_{\pm}}}\mathbf{K}^{\N,\zeta,\sim}_{\s,\r,\z_{\pm},\y},\label{eq:duhamelII2}\\
\boldsymbol{\Phi}_{\r}[\mathbf{K}^{\N,\zeta,\sim}_{\s,\r,\cdot,\y}]_{\w}&:={\ocolor{red}\mathbf{1}_{\mathfrak{t}_{\mathrm{ap}}\geq\r}}\sum_{\ell\in\llbracket1,\mathrm{L}_{2}\rrbracket}\sum_{\z\in\llbracket0,\infty\rrbracket}\boldsymbol{\mathcal{K}}^{\partial,\ell}_{\w,\z}\cdot\N\mathbf{Av}^{\mathbf{T},\mathfrak{b}_{\ell}}_{\r,\z}{\color{blue}\mathbf{R}^{\N,\wedge}_{\r,\z}}\mathbf{K}^{\N,\zeta,\sim}_{\s,\r,\z,\y}.\label{eq:duhamelII1}
\end{align}
\end{lemma}
\begin{proof}
A direct inspection shows that the \abbr{LHS} and \abbr{RHS} of \eqref{eq:duhamelI1} satisfy the same \abbr{SDE} for {\small$\s\leq\t$} and all {\small$\x,\y\in\llbracket0,\infty\rrbracket$} with the same initial condition of {\small$\mathbf{1}_{\x=\y}$} at {\small$\s=\t$}. (We note that the integrals on the \abbr{RHS} of \eqref{eq:duhamelI1} vanish for {\small$\s=\t$}.) Thus, by uniqueness of solutions to said \abbr{SDE}, \eqref{eq:duhamelI1} follows. 
\end{proof}
\subsection{Estimates}
We now present estimates for the {\small$\mathbf{L}^{\N,\zeta,\sim}$} and {\small$\mathbf{K}^{\N,\zeta,\sim}$}-kernels, in that order. We defer the proofs of these lemmas to the last subsection of this section after we have deduced Proposition \ref{prop:hkzeta}.

We first record the following result concerning the {\small$\mathbf{L}^{\N,\zeta,\sim}$}-kernel; we clarify its statement afterwards.
\begin{lemma}\label{lemma:lzeta}
\fsp Fix any {\small$\kappa\in\R$} and {\small$\delta>0$}. There exists a high probability event on which
\begin{align}
\sup_{0\leq\s\leq\t\leq1}\sup_{\x\in\llbracket0,\infty\rrbracket}\sum_{\y\in\llbracket0,\infty\rrbracket}\exp(\tfrac{2\kappa|\x-\y|}{\N})|\mathbf{H}^{\N,\mathbf{A}}_{\s,\t,\x,\y}-\mathbf{L}^{\N,\zeta,\sim}_{\s,\t,\x,\y}|^{2}\lesssim_{\kappa,\delta}\N^{-1+\delta}.\label{eq:lzetaI}
\end{align}
Above, we used the {\small$\mathbf{H}^{\N,\mathbf{A}}$}-kernel from Definition \ref{definition:heat} (where {\small$\mathbf{A}$} is from \eqref{eq:mainII}). On the same event, we have 
\begin{align}
\sup_{0\leq\s\leq\t\leq1}\sup_{\x\in\llbracket0,\infty\rrbracket}\sum_{\y\in\llbracket0,\infty\rrbracket}\exp(\tfrac{\kappa|\x-\y|}{\N})|\mathbf{L}^{\N,\zeta,\sim}_{\s,\t,\x,\y}|&\lesssim_{\kappa,\delta}\N^{\frac12\delta}.\label{eq:lzetaII}
\end{align}
\end{lemma}
The estimate \eqref{eq:lzetaI} states that {\small$\mathbf{L}^{\N,\zeta,\sim}$} is, in some sense, a perturbation of the {\small$\mathbf{H}^{\N,\mathbf{A}}$}-kernel. Indeed, in view of {Proposition \ref{prop:hkestimates}}, if we remove the {\small$\mathbf{L}^{\N,\zeta,\sim}$}-kernel from the \abbr{LHS} of \eqref{eq:lzetaI}, then (up to the {\small$\delta$} in the exponent) the \abbr{RHS} of \eqref{eq:lzetaI} would require another factor that blows up as {\small$\t\to\s$}. Equipped with Lemma \ref{lemma:lzeta}, the formula \eqref{eq:duhamelI1}, and stochastic estimates from {\color{black}Section \ref{section:stochII}}, we will show the following ``short-time" estimate for the {\small$\mathbf{K}^{\N,\zeta,\sim}$}-kernel.
\begin{lemma}\label{lemma:kzeta}
\fsp Take any {\small$\delta>0$}. There are deterministic constants {\small$\beta_{0},\kappa>0$} (that are independent of {\small$\N$}) along with a high probability event on which the following estimate holds:
\begin{align}
\sup_{\substack{0\leq\s\leq\t\leq1\\\t-\s\leq\N^{\beta_{0}}\N^{-\zeta}}}\sup_{\x\in\llbracket0,\infty\rrbracket}\sum_{\y\in\llbracket0,\infty\rrbracket}\exp(\tfrac{\kappa|\x-\y|}{\N})|\mathbf{K}^{\N,\zeta,\sim}_{\s,\t,\x,\y}|\lesssim_{\delta}\N^{\delta}.\label{eq:kzetaI}
\end{align}
\end{lemma}
We emphasize the restriction {\small$\t-\s\leq\N^{\beta_{0}}\N^{-\zeta}$} on the \abbr{LHS} of \eqref{eq:kzetaI}; typically, we think of {\small$\beta_{0}>0$} as small and {\small$\zeta>0$} is ``not small", so that this restriction turns \eqref{eq:kzetaI} into short-time estimate. (On the other hand, if {\small$\beta_{0}\geq\zeta$}, then the restriction is redundant since we already consider only {\small$0\leq\s\leq\t\leq1$}.) 

To remove the short-time restriction in Lemma \ref{lemma:kzeta}, we will bootstrap by the Chapman-Kolmogorov equation (or semigroup property) for the {\small$\mathbf{K}^{\N,\zeta,\sim}$}-kernel. This ultimately yields the following estimate.
\begin{lemma}\label{lemma:kzeta2}
\fsp Take the same setting as Lemma \ref{lemma:kzeta}. On the same high probability event therein, we have 
\begin{align}
\sup_{0\leq\s\leq\t\leq1}\sup_{\x\in\llbracket0,\infty\rrbracket}\sum_{\y\in\llbracket0,\infty\rrbracket}\exp(\tfrac{\kappa|\x-\y|}{\N})|\mathbf{K}^{\N,\zeta,\sim}_{\s,\t,\x,\y}|\lesssim_{\delta}\exp(\N^{-\frac12\beta_{0}}\N^{\zeta})\N^{\delta}.\label{eq:kzeta2I}
\end{align}
\end{lemma}
\begin{proof}[Proof of Proposition \ref{prop:hkzeta}]
Combine Lemmas \ref{lemma:hkzetasim} and \ref{lemma:kzeta}.
\end{proof}
\subsection{Proofs of estimates}
Let us now show Lemmas \ref{lemma:kzeta}, \ref{lemma:kzeta2}, and \ref{lemma:lzeta} in that order. We defer the proof of Lemma \ref{lemma:lzeta} to the end of this section since it is a standard analysis via moments of (perturbations of) the \abbr{SHE}.
\begin{proof}[Proof of Lemma \ref{lemma:kzeta}]
Our estimates will be with respect to the following control parameter (which is exactly the object we want to control in this lemma):
\begin{align}
\boldsymbol{\Theta}_{\s,\t,\kappa}:=\sup_{\r\in[\s,\t]}\|\|\exp(\tfrac{\kappa|\x-\y|}{\N})\mathbf{K}^{\N,\zeta,\sim}_{\s,\r,\x,\y}\|_{\ell^{1}_{\y}(\llbracket0,\infty\rrbracket)}\|_{\ell^{\infty}_{\x}(\llbracket0,\infty\rrbracket)}.\label{eq:kzetaI0}
\end{align}
To be clear, the norm on the \abbr{RHS} is given by first taking the {\small$\ell^{1}(\llbracket0,\infty\rrbracket)$}-norm in the {\small$\y$}-variable, and then taking the {\small$\ell^{\infty}(\llbracket0,\infty\rrbracket$)}-norm in the {\small$\x$}-variable. For convenience, we may omit the domain {\small$\llbracket0,\infty\rrbracket$} from this norm notation. It will also be convenient to introduce the following inner-product, in which {\small$\mathsf{f},\mathsf{g}:\llbracket0,\infty\rrbracket\to\R$}:
\begin{align}
\langle\mathsf{f},\mathsf{g}\rangle:=\sum_{\w\in\llbracket0,\infty\rrbracket}\mathsf{f}_{\w}\mathsf{g}_{\w}.
\end{align}
Before we proceed, we present the following estimate, which we clarify afterwards:
\begin{align}
\int_{0}^{1}\sum_{\x\in\llbracket0,\mathrm{C}\N^{1+\zeta}\rrbracket}\sum_{\ell\in\llbracket1,\mathrm{L}_{1}\rrbracket}|\mathbf{Av}^{\mathbf{T},\mathbf{X},\mathscr{Q}_{\ell}}_{\t,\x}|^{2}\d\t\lesssim\N^{-2-\beta_{\bullet}}\N^{1+\zeta}.\label{eq:kzetaI1}
\end{align}
Above, the constant {\small$\mathrm{C}=\mathrm{O}(1)$} is large, whereas {\small$\beta_{\bullet}>0$} is small (but still independent of {\small$\N$}). Let us also recall {\small$\mathbf{Av}^{\mathbf{T},\mathbf{X},\mathscr{Q}_{\ell}}$} from Definition \ref{definition:bulkaverage}. We claim that \eqref{eq:kzetaI1} holds with high probability. This follows from {\color{black}Proposition \ref{prop:stochmain1}} and the Markov inequality (along with the fact that {\small$\mathrm{L}_{1}=\mathrm{O}(1)$}). Next, we introduce the estimate
\begin{align}
\sup_{\x\in\llbracket0,\mathrm{O}(1)\rrbracket}\int_{0}^{1}\sum_{\ell\in\llbracket1,\mathrm{L}_{2}\rrbracket}|\mathbf{Av}^{\mathbf{T},\mathfrak{b}_{\ell}}_{\t,\x}|^{2}\d\t\lesssim\N^{-\beta_{\bullet}},\label{eq:kzetaI1b}
\end{align}
which holds with high probability for {\small$\beta_{\bullet}>0$} small as well by Proposition \ref{prop:stochmain2} and the Markov inequality (along with {\small$\mathrm{L}_{2}=\mathrm{O}(1)$}). Now, throughout the entire argument below, assume that \eqref{eq:kzetaI1}-\eqref{eq:kzetaI1b} and \eqref{eq:lzetaI}-\eqref{eq:lzetaII} hold. The rest of the proof is deterministic. Since \eqref{eq:kzetaI1} and \eqref{eq:lzetaI}-\eqref{eq:lzetaII} hold with high probability, the final conclusion of this proof will also hold with high probability.

We start with the last term in \eqref{eq:duhamelI1}. Observe that, by \eqref{eq:lzetaII} and \eqref{eq:duhamelII2}, this term has the form
\begin{align}
&\int_{\s}^{\t}\sum_{\w\in\llbracket0,\infty\rrbracket}\mathbf{L}^{\N,\zeta,\sim}_{\r,\t,\x,\w}\cdot\boldsymbol{\Psi}_{\r}[\mathbf{K}^{\N,\zeta,\sim}_{\s,\r,\cdot,\y}]_{\w}\d\r
&=\sum_{\ell\in\llbracket1,\mathrm{L}_{1}\rrbracket}{}\int_{\s}^{\t}\langle\mathscr{A}^{(\ell)}_{\r,\t,\x,\cdot},\N\cdot\mathbf{Av}^{\mathbf{T},\mathbf{X},\mathscr{Q}_{\ell}}_{\r,\cdot}\mathbf{R}^{\N,\wedge}_{\r,\cdot_{\pm}}\mathbf{K}^{\N,\zeta,\sim}_{\s,\r,\cdot_{\pm},\y}\rangle\d\r,\label{eq:kzetaI2}
\end{align}
where the kernel {\small$\mathscr{A}^{(\ell)}$} is given by the formula
\begin{align}
\mathscr{A}^{(\ell)}_{\r,\t,\x,\z}:=\mathbf{1}_{\mathfrak{t}_{\mathrm{ap}}\geq\r}\sum_{\w\in\llbracket0,\infty\rrbracket}\mathbf{L}^{\N,\zeta,\sim}_{\r,\t,\x,\w}\cdot\boldsymbol{\mathcal{K}}^{(\ell)}_{\w,\z}\boldsymbol{\chi}^{(\zeta)}_{\z}.\nonumber
\end{align}
We now claim, with explanation given afterwards, that
\begin{align}
&\int_{\s}^{\t}\|\exp(\tfrac{\kappa|\x-\y|}{\N})\langle\mathscr{A}^{(\ell)}_{\r,\t,\x,\cdot},\N\cdot\mathbf{Av}^{\mathbf{T},\mathbf{X},\mathscr{Q}_{\ell}}_{\r,\cdot}\mathbf{R}^{\N,\wedge}_{\r,\cdot_{\pm}}\mathbf{K}^{\N,\zeta,\sim}_{\s,\r,\cdot_{\pm},\y}\rangle\|_{\ell^{1}_{\y}}\d\r\label{eq:kzetaI3}\\
&\lesssim_{\kappa}\int_{\s}^{\t}\langle\exp(\tfrac{\kappa|\x-\cdot|}{\N})|\mathscr{A}^{(\ell)}_{\r,\t,\x,\cdot}|,\N\cdot|\mathbf{Av}^{\mathbf{T},\mathbf{X},\mathscr{Q}_{\ell}}_{\r,\cdot}|\cdot\|\exp(\tfrac{\kappa|\cdot_{\pm}-\y|}{\N})|\mathbf{K}^{\N,\zeta,\sim}_{\s,\r,\cdot_{\pm},\y}|\|_{\ell^{1}_{\y}}\rangle\d\r\lesssim\mathrm{I}^{\frac12}\mathrm{II}^{\frac12}, \nonumber
\end{align}
where
\begin{align*}
\mathrm{I}&:=\int_{\s}^{\t}\langle\exp(\tfrac{\kappa|\x-\cdot|}{\N})|\mathscr{A}^{(\ell)}_{\r,\t,\x,\cdot}|,\N^{2}|\mathbf{Av}^{\mathbf{T},\mathbf{X},\mathscr{Q}_{\ell}}_{\r,\cdot}|^{2}\rangle\d\r,\\
\mathrm{II}&:=\int_{\s}^{\t}\langle\exp(\tfrac{\kappa|\x-\cdot|}{\N})|\mathscr{A}^{(\ell)}_{\r,\t,\x,\cdot}|,\|\exp(\tfrac{\kappa|\cdot_{\pm}-\y|}{\N})|\mathbf{K}^{\N,\zeta,\sim}_{\s,\r,\cdot_{\pm},\y}|\|_{\ell^{1}_{\y}}^{2}\rangle\d\r.
\end{align*}
For the first step, we use the estimate {\small$\exp(\frac{\kappa|\x-\y|}{\N})\lesssim_{\kappa}\exp(\frac{\kappa|\x-\w|}{\N})\exp(\frac{\kappa|\w_{\pm}-\y|}{\N})$}, which follows by the triangle inequality and the fact that {\small$|\w-\w_{\pm}|\lesssim\N$} for any {\small$\w\in\llbracket0,\infty\rrbracket$} by construction in Definition \ref{definition:bulkaverage}. We also use that {\small$\mathbf{R}^{\N,\wedge}\lesssim1$} (see \eqref{eq:rwedge}), as well as the triangle inequality for the {\small$\ell^{1}_{\y}$}-norm, since the only term which depends on {\small$\y$} above is the {\small$\mathbf{K}^{\N,\zeta,\sim}$}-kernel and its exponential factor. The last bound follows by the Cauchy-Schwarz inequality with respect to both the {\small$\d\r$}-integration and the {\small$\langle\cdot,\cdot\rangle$}-summation, with respect to the weight {\small$\exp(\frac{\kappa|\x-\cdot|}{\N})|\mathscr{A}^{(\ell)}_{\r,\t,\x,\cdot}|$}. In order to control {\small$\mathrm{I},\mathrm{II}$}, we recall the {\small$\mathbf{L}^{\N,\zeta,\sim}$}-estimate from \eqref{eq:lzetaII} and the {\small$\boldsymbol{\mathcal{K}}^{(\ell)}$}-estimate from \eqref{eq:sdeIII}. This gives us two estimates. 
\begin{itemize}
\item First is that {\small$\|\exp(\frac{\kappa|\x-\cdot|}{\N})\mathscr{A}^{(\ell)}_{\r,\t,\x,\cdot}\|_{\ell^{1}(\llbracket0,\infty\rrbracket)}\lesssim_{\kappa}\N^{\mathrm{C}\delta_{\mathbf{S}}}$} for some {\small$\mathrm{C}\lesssim1$}. 
\item Second is the pointwise estimate {\small$\exp(\frac{\kappa|\x-\z|}{\N})|\mathscr{A}^{(\ell)}_{\r,\t,\x,\z}|\lesssim\N^{-1+\mathrm{C}\delta_{\mathbf{S}}}\mathbf{1}_{\z\in\llbracket0,\mathrm{O}(\N^{1+\zeta})\rrbracket}$} for all {\small$\z\in\llbracket0,\infty\rrbracket$}, where the indicator comes from the support of {\small$\boldsymbol{\chi}^{(\zeta)}$}.
\end{itemize}
Using the first of these two estimates, we deduce that {\small$\mathrm{II}\lesssim_{\kappa}\N^{\mathrm{C}\delta_{\mathbf{S}}}|\t-\s|\boldsymbol{\Theta}_{\s,\t,\kappa}^{2}$}. By the second and \eqref{eq:kzetaI1}, we have instead that {\small$\mathrm{I}\lesssim_{\kappa}\N^{-1+\mathrm{C}\delta_{\mathbf{S}}}\N^{2}\N^{-2-\beta_{\bullet}}\N^{1+\zeta}\lesssim\N^{-\beta_{\bullet}+\zeta+\mathrm{C}\delta_{\mathbf{S}}}$}. Plugging this into \eqref{eq:kzetaI3} and \eqref{eq:kzetaI2} gives
\begin{align*}
\Big\|\exp(\tfrac{\kappa|\x-\y|}{\N})\int_{\s}^{\t}\sum_{\w\in\llbracket0,\infty\rrbracket}\mathbf{L}^{\N,\zeta,\sim}_{\r,\t,\x,\w}\cdot\boldsymbol{\Psi}_{\r}[\mathbf{K}^{\N,\zeta,\sim}_{\s,\r,\cdot,\y}]_{\w}\d\r\Big\|_{\ell^{1}_{\y}}\lesssim_{\kappa}\N^{\mathrm{C}\delta_{\mathbf{S}}}\N^{-\frac12\beta_{\bullet}+\frac12\zeta}|\t-\s|^{\frac12}\boldsymbol{\Theta}_{\s,\t,\kappa}\lesssim\N^{\mathrm{C}\delta_{\mathbf{S}}}\N^{-\frac12\beta_{\bullet}+\frac12\beta_{0}}\boldsymbol{\Theta}_{\s,\t,\kappa},
\end{align*}
where the last bound follows by our assumption that {\small$|\t-\s|\leq\N^{\beta_{0}}\N^{-\zeta}$}. The same would hold if we replaced {\small$\t$} on the \abbr{LHS} by any {\small$\tau\in[\s,\t]$}, since the far \abbr{RHS} is monotone non-decreasing in {\small$\t$}, so we can replace {\small$\t\mapsto\tau$} on the \abbr{LHS} and take a supremum over {\small$\tau\in[\s,\t]$}. Similarly, note that the \abbr{RHS} is independent of the {\small$\x$}-variable on the \abbr{LHS}, so we can place a supremum over {\small$\x\in\llbracket0,\infty\rrbracket$} on the \abbr{LHS} as well. Ultimately, if we take {\small$\beta_{0},\delta_{\mathbf{S}}>0$} small, we get
\begin{align}
\sup_{\tau\in[\s,\t]}\Big\|\Big\|\exp(\tfrac{\kappa|\x-\y|}{\N})\cdot\int_{\s}^{\tau}\sum_{\w\in\llbracket0,\infty\rrbracket}\mathbf{L}^{\N,\zeta,\sim}_{\r,\tau,\x,\w}\cdot\boldsymbol{\Psi}_{\r}[\mathbf{K}^{\N,\zeta,\sim}_{\s,\r,\cdot,\y}]_{\w}\d\r\Big\|_{\ell^{1}_{\y}}\Big\|_{\ell^{\infty}_{\x}}\lesssim_{\kappa}\N^{-\beta_{\#}}\boldsymbol{\Theta}_{\s,\t,\kappa}\label{eq:kzetaI4}
\end{align}
for some {\small$\beta_{\#}>0$}. We now study the second term on the \abbr{RHS} of \eqref{eq:duhamelI1}. Similar to \eqref{eq:kzetaI2}, we have 
\begin{align}
\int_{\s}^{\t}\sum_{\w\in\llbracket0,\infty\rrbracket}\mathbf{L}^{\N,\zeta,\sim}_{\r,\t,\x,\w}\cdot\boldsymbol{\Phi}_{\r}[\mathbf{K}^{\N,\zeta,\sim}_{\s,\r,\cdot,\y}]_{\w}\d\r&=\sum_{\ell\in\llbracket1,\mathrm{L}_{2}\rrbracket}\int_{\s}^{\t}\langle\mathscr{B}^{\partial,\ell}_{\r,\t,\x,\cdot},\N\mathbf{Av}^{\mathbf{T},\mathfrak{b}_{\ell}}_{\r,\cdot}\mathbf{R}^{\N,\wedge}_{\r,\cdot}\mathbf{K}^{\N,\zeta,\sim}_{\s,\r,\cdot,\y}\rangle\d\r,\label{eq:kzetaI5}
\end{align}
where the kernel {\small$\mathscr{B}^{\partial,\ell}$} is given by the following formula:
\begin{align}
\mathscr{B}^{\partial,\ell}_{\r,\t,\x,\z}&:=\mathbf{1}_{\mathfrak{t}_{\mathrm{ap}}\geq\r}\sum_{\w\in\llbracket0,\infty\rrbracket}\mathbf{L}^{\N,\zeta,\sim}_{\r,\t,\x,\w}\cdot\boldsymbol{\mathcal{K}}^{\partial,\ell}_{\w,\z}.\nonumber
\end{align}
The same argument which gave \eqref{eq:kzetaI3} also provides the following estimate (if we formally replace {\small$\mathscr{A}^{(\ell)}\mapsto\mathscr{B}^{\partial,\ell}$} and {\small$\mathbf{Av}^{\mathbf{T},\mathbf{X},\mathscr{Q}_{\ell}}\mapsto\mathbf{Av}^{\mathbf{T},\mathfrak{b}_{\ell}}$}):
\begin{align}
&\int_{\s}^{\t}\|\exp(\tfrac{\kappa|\x-\y|}{\N})\langle\mathscr{B}^{\partial,\ell}_{\r,\t,\x,\cdot},\N\mathbf{Av}^{\mathbf{T},\mathfrak{b}_{\ell}}_{\r,\cdot}\mathbf{R}^{\N,\wedge}_{\r,\cdot}\mathbf{K}^{\N,\zeta,\sim}_{\s,\r,\cdot_{\pm},\y}\rangle\|_{\ell^{1}_{\y}}\d\r\label{eq:kzetaI6}\\
&\lesssim_{\kappa}(\int_{\s}^{\t}\langle\exp(\tfrac{\kappa|\x-\cdot|}{\N})|\mathscr{B}^{\partial,\ell}_{\r,\t,\x,\cdot}|,\N^{2}|\mathbf{Av}^{\mathbf{T},\mathfrak{b}_{\ell}}_{\r,\cdot}|^{2}\rangle\d\r)^{\frac12}(\int_{\s}^{\t}\langle\exp(\tfrac{\kappa|\x-\cdot|}{\N})|\mathscr{B}^{\partial,\ell}_{\r,\t,\x,\cdot}|,\|\exp(\tfrac{\kappa|\cdot_{\pm}-\y|}{\N})|\mathbf{K}^{\N,\zeta,\sim}_{\s,\r,\cdot_{\pm},\y}|\|_{\ell^{1}_{\y}}^{2}\rangle\d\r)^{\frac12}.\nonumber
\end{align}
By the estimate \eqref{eq:sdeIV} on {\small$\boldsymbol{\mathcal{K}}^{\partial,\ell}$}, as well as the estimate \eqref{eq:lzetaII} for {\small$\mathbf{L}^{\N,\zeta,\sim}$} that we conditioned on, we similarly have the kernel estimate {\small$\|\exp(\frac{\kappa|\x-\cdot|}{\N})\mathscr{B}^{\partial,\ell}_{\r,\t,\x,\cdot}\|_{\ell^{1}(\llbracket0,\infty\rrbracket)}\lesssim_{\kappa}\N^{\mathrm{C}\delta_{\mathbf{S}}}$} for some {\small$\mathrm{C}\lesssim1$}. On the other hand, since \eqref{eq:sdeIV} has a restriction of the forwards spatial variable to {\small$\{0,1\}\subseteq\llbracket0,\infty\rrbracket$}, we get {\small$\exp(\frac{\kappa|\x-\z|}{\N})|\mathscr{B}^{\partial,\ell}_{\r,\t,\x,\z}|\lesssim\N^{-1+\mathrm{C}\delta_{\mathbf{S}}}\mathbf{1}_{\w\in\llbracket0,1\rrbracket}$} for all {\small$\z\in\llbracket0,\infty\rrbracket$}. Thus, the second factor in the second line of \eqref{eq:kzetaI6} is {\small$\lesssim\N^{\mathrm{C}\delta_{\mathbf{S}}/2}|\t-\s|^{1/2}\boldsymbol{\Theta}_{\s,\t,\kappa}$}, while the first factor therein is {\small$\lesssim\N^{\mathrm{C}\delta_{\mathbf{S}}/2}\N^{-\beta_{\bullet}/2}$} due to \eqref{eq:kzetaI1b}. Thus, the \abbr{LHS} of \eqref{eq:kzetaI6} is {\small$\lesssim\N^{\mathrm{C}\delta_{\mathbf{S}}}\N^{-\beta_{\bullet}/2}\boldsymbol{\Theta}_{\s,\t,\kappa}\lesssim\N^{-\beta_{\bullet}/3}\boldsymbol{\Theta}_{\s,\t,\kappa}$} if {\small$\delta_{\mathbf{S}}>0$} is small, where we have bounded {\small$|\t-\s|\lesssim1$} by assumption. Ultimately, similar to \eqref{eq:kzetaI4}, from this and \eqref{eq:kzetaI5}, we obtain the following estimate (for some possibly different {\small$\beta_{\#}>0$}):
\begin{align}
\sup_{\tau\in[\s,\t]}\Big\|\Big\|\exp(\tfrac{\kappa|\x-\y|}{\N})\cdot\int_{\s}^{\tau}\sum_{\w\in\llbracket0,\infty\rrbracket}\mathbf{L}^{\N,\zeta,\sim}_{\r,\tau,\x,\w}\cdot\boldsymbol{\Phi}_{\r}[\mathbf{K}^{\N,\zeta,\sim}_{\s,\r,\cdot,\y}]_{\w}\d\r\Big\|_{\ell^{1}_{\y}}\Big\|_{\ell^{\infty}_{\x}}\lesssim_{\kappa}\N^{-\beta_{\#}}\boldsymbol{\Theta}_{\s,\t,\kappa}.\label{eq:kzetaI7}
\end{align}
Now, by \eqref{eq:lzetaII}, we obtain that {\small$\sup_{\r\in[\s,\t]}\|\|\exp(\frac{\kappa|\x-\y|}{\N})\mathbf{L}^{\N,\zeta,\sim}_{\s,\r,\x,\y}\|_{\ell^{1}_{\y}}\|_{\ell^{\infty}_{\x}}\lesssim_{\kappa,\delta}\N^{\delta}$} for any {\small$\kappa>0$} and {\small$\delta>0$} small. We combine this with \eqref{eq:kzetaI4}, the previous display \eqref{eq:kzetaI7}, and \eqref{eq:duhamelI1}. This provides
\begin{align}
\boldsymbol{\Theta}_{\s,\t,\kappa}\lesssim_{\kappa}\N^{\delta}+\N^{-\beta_{\#}}\boldsymbol{\Theta}_{\s,\t,\kappa},\label{eq:kzetaI8}
\end{align}
at which point the result follows by moving the last term to the \abbr{LHS} and dividing out by {\small$1-\mathrm{O}(\N^{-\beta_{\#}})\gtrsim1$}.
\end{proof}
\begin{proof}[Proof of Lemma \ref{lemma:kzeta2}]
For any {\small$\s\leq\r\leq\t$}, we have the following Chapman-Kolmogorov equation, which follows since both sides of the identity below solve the same linear \abbr{SDE} for {\small$\t\geq\r$} with the same initial data at {\small$\t=\r$}:
\begin{align*}
\mathbf{K}^{\N,\zeta,\sim}_{\s,\t,\x,\y}=\sum_{\w\in\llbracket0,\infty\rrbracket}\mathbf{K}^{\N,\zeta,\sim}_{\r,\t,\x,\w}\mathbf{K}^{\N,\zeta,\sim}_{\s,\r,\w,\y}.
\end{align*}
At this point, we follow exactly the proof of Lemma 7.5 in \cite{Y26PTRF}. That is, we iterate the above equation to obtain a representation of the \abbr{LHS} in terms of heat kernels whose time-variables are separated by {\small$\leq\N^{\beta_{0}}\N^{-\zeta}$}; this requires {\small$\lesssim1+\N^{-\beta_{0}}\N^{\zeta}$}-many heat kernel copies. Each heat kernel yields {\small$\lesssim_{\delta}\N^{\delta}$} (with a sub-exponential decay in space). So in total, we collect a factor of {\small$\lesssim(\mathrm{C}_{\delta}\N^{\delta})^{1+\N^{-\beta_{0}}\N^{\zeta}}\lesssim\exp(\mathrm{C}\N^{-\beta_{0}}\N^{\zeta}\log\N)\N^{\delta}\lesssim\exp(\N^{-\beta_{0}/2}\N^{\zeta})\N^{\delta}$}.
\end{proof}
\begin{proof}[Proof of Lemma \ref{lemma:lzeta}]
We first claim that \eqref{eq:lzetaII} follows by \eqref{eq:lzetaI}, {\color{black} heat kernel estimates from Proposition \ref{prop:hkestimates}}, and some technical gymnastics. Indeed, we have 
\begin{align*}
&\sum_{\y\in\llbracket0,\infty\rrbracket}\exp(\tfrac{\kappa|\x-\y|}{\N})|\mathbf{L}^{\N,\zeta,\sim}_{\s,\t,\x,\y}|\leq\sum_{\y\in\llbracket0,\infty\rrbracket}\exp(\tfrac{\kappa|\x-\y|}{\N})|\mathbf{H}^{\N,\mathbf{A}}_{\s,\t,\x,\y}|+\sum_{\y\in\llbracket0,\infty\rrbracket}\exp(\tfrac{\kappa|\x-\y|}{\N})|\mathbf{H}^{\N,\mathbf{A}}_{\s,\t,\x,\y}-\mathbf{L}^{\N,\zeta,\sim}_{\s,\t,\x,\y}|\\
&\lesssim_{\kappa}1+\Big(\sum_{\y\in\llbracket0,\infty\rrbracket}\exp(\tfrac{-\kappa|\x-\y|}{\N})\Big)^{\frac12}\Big(\sum_{\y\in\llbracket0,\infty\rrbracket}\exp(\tfrac{2\kappa|\x-\y|}{\N})|\mathbf{H}^{\N,\mathbf{A}}_{\s,\t,\x,\y}-\mathbf{L}^{\N,\zeta,\sim}_{\s,\t,\x,\y}|^{2}\Big)^{\frac12}\lesssim_{\kappa}1+\N^{\frac12}\N^{-\frac12+\frac12\delta}\lesssim\N^{\frac12\delta},
\end{align*}
where we used the Cauchy-Schwarz inequality to obtain the second line above. Thus, we are left with \eqref{eq:lzetaI}. For this, we introduce the following control parameter, in which {\small$\kappa,p\geq1$} are large but fixed, in which {\small$0\leq\s\leq\t\leq1$}, and in which {\small$\|\cdot\|_{\mathrm{L}^{p}(\mathbb{P})}:=(\E|\cdot|^{p})^{1/p}$} means {\small$p$}-th root of the {\small$p$}-th moment:
\begin{align}
\boldsymbol{\Pi}^{\kappa,p}_{\s,\t}:=|\t-\s|^{\frac12}\sup_{\x\in\llbracket0,\infty\rrbracket}\sum_{\y\in\llbracket0,\infty\rrbracket}\exp(\tfrac{2\kappa|\x-\y|}{\N})\||\mathbf{L}^{\N,\zeta,\sim}_{\s,\t,\x,\y}|^{2}\|_{\mathrm{L}^{p}(\mathbb{P})}.\label{eq:lzeta1}
\end{align}
By the Duhamel formula and \eqref{eq:lzetasim}, we have 
\begin{align}
\mathbf{L}^{\N,\zeta,\sim}_{\s,\t,\x,\y}&=\mathbf{H}^{\N,\mathbf{A}}_{\s,\t,\x,\y}+\boldsymbol{\Lambda}^{(1)}_{\s,\t,\x,\y}+\boldsymbol{\Lambda}^{(2)}_{\s,\t,\x,\y}+\boldsymbol{\Lambda}^{(3)}_{\s,\t,\x,\y}+\boldsymbol{\Lambda}^{(4)}_{\s,\t,\x,\y},\label{eq:lzeta2a}
\end{align}
where
\begin{align*}
\boldsymbol{\Lambda}^{(1)}_{\s,\t,\x,\y}&:=\int_{\s}^{\t}\sum_{\w\in\llbracket0,\infty\rrbracket}\mathbf{H}^{\N,\mathbf{A}}_{\r,\t,\x,\w}\cdot[\mathbf{e}^{\mathfrak{t}_{\N}\mathscr{T}_{\N}}(\mathbf{1}^{(\w)}_{\cdot}\boldsymbol{\chi}^{(\zeta_{\mathrm{large}})}_{\cdot}\cdot\sqrt{2}\lambda\N^{\frac12}\mathbf{R}^{\N,\wedge}_{\r,\cdot}\mathbf{L}^{\N,\zeta,\sim}_{\s,\r,\cdot,\y}\d\mathbf{b}_{\r,\cdot})]_{\w}\d\r\\
\boldsymbol{\Lambda}^{(2)}_{\s,\t,\x,\y}&:=\int_{\s}^{\t}\sum_{\w\in\llbracket0,\infty\rrbracket}\mathbf{H}^{\N,\mathbf{A}}_{\r,\t,\x,\w}\cdot\mathbf{1}_{\mathfrak{t}_{\mathrm{ap}}\geq\r}\mathrm{Err}[\boldsymbol{\chi}^{(\zeta)}\mathbf{R}^{\N,\wedge}\mathbf{L}^{\N,\zeta,\sim}_{\s,\cdot,\cdot,\y}]_{\r,\w}\d\r\\
\boldsymbol{\Lambda}^{(3)}_{\s,\t,\x,\y}&:=\int_{\s}^{\t}\sum_{\w\in\llbracket0,\infty\rrbracket}\mathbf{H}^{\N,\mathbf{A}}_{\r,\t,\x,\w}\cdot\sum_{\ell\in\llbracket1,\mathrm{L}_{1}\rrbracket}{}\grad^{\mathbf{T},\mathrm{av}}_{\mathfrak{t}_{\mathbf{Av}}}\Big({\mathbf{1}_{\mathfrak{t}_{\mathrm{ap}}\geq\r}}\sum_{\z\in\llbracket0,\infty\rrbracket}\boldsymbol{\mathcal{K}}^{(\ell)}_{\w,\z}\cdot{\boldsymbol{\chi}^{(\zeta_{})}_{\z}}\cdot\N\cdot\mathbf{Av}^{\mathbf{X},\mathscr{Q}_{\ell}}_{\r,\z}{\mathbf{R}^{\N,\wedge}_{\t,\z_{\pm}}}\mathbf{L}^{\N,\zeta,\sim}_{\s,\r,\z_{\pm},\y}\Big)\d\r\\
\boldsymbol{\Lambda}^{(4)}_{\s,\t,\x,\y}&:=\int_{\s}^{\t}\sum_{\w\in\llbracket0,\infty\rrbracket}\mathbf{H}^{\N,\mathbf{A}}_{\r,\t,\x,\w}\cdot\sum_{\ell\in\llbracket1,\mathrm{L}_{2}\rrbracket}\grad^{\mathbf{T},\mathrm{av}}_{\mathfrak{t}_{\partial}}\Big({\mathbf{1}_{\mathfrak{t}_{\mathrm{ap}}\geq\r}}\sum_{\z\in\llbracket0,\infty\rrbracket}\boldsymbol{\mathcal{K}}^{\partial,\ell}_{\w,\z}\cdot\N\mathfrak{b}_{\ell}[\bphi_{\r}]{\mathbf{R}^{\N,\wedge}_{\r,\z}}\mathbf{L}^{\N,\zeta,\sim}_{\s,\r,\z,\y}\Big)\d\r.
\end{align*}
Our goal is to control each term in \eqref{eq:lzeta2a}. We claim, with justification given after, that for some {\small$\delta>0$}, we have
\begin{align}
|\t-\s|^{\frac12}\sup_{\x\in\llbracket0,\infty\rrbracket}\sum_{\y\in\llbracket0,\infty\rrbracket}\exp(\tfrac{2\kappa|\x-\y|}{\N})\||\mathbf{H}^{\N,\mathbf{A}}_{\s,\t,\x,\y}|^{2}\|_{\mathrm{L}^{p}(\mathbb{P})}&\lesssim_{\kappa,p}\N^{-1},\label{eq:lzeta3a}\\
|\t-\s|^{\frac12}\sup_{\x\in\llbracket0,\infty\rrbracket}\sum_{\y\in\llbracket0,\infty\rrbracket}\exp(\tfrac{2\kappa|\x-\y|}{\N})\||\boldsymbol{\Lambda}^{(1)}_{\s,\t,\x,\y}|^{2}\|_{\mathrm{L}^{p}(\mathbb{P})}&\lesssim_{\kappa,p}|\t-\s|^{\frac12}\int_{\s}^{\t}|\t-\r|^{-\frac12}|\r-\s|^{-\frac12}\boldsymbol{\Pi}^{\kappa,p}_{\s,\r}\d\r,\label{eq:lzeta3b}\\
|\t-\s|^{\frac12}\sup_{\x\in\llbracket0,\infty\rrbracket}\sum_{\y\in\llbracket0,\infty\rrbracket}\exp(\tfrac{2\kappa|\x-\y|}{\N})\||\boldsymbol{\Lambda}^{(3)}_{\s,\t,\x,\y}|^{2}\|_{\mathrm{L}^{p}(\mathbb{P})}&\lesssim_{\kappa,p}\N^{-\delta}\sup_{\r\in[\s,\t]}\boldsymbol{\Pi}^{\kappa,p}_{\s,\r},\label{eq:lzeta3c}\\
|\t-\s|^{\frac12}\sup_{\x\in\llbracket0,\infty\rrbracket}\sum_{\y\in\llbracket0,\infty\rrbracket}\exp(\tfrac{2\kappa|\x-\y|}{\N})\||\boldsymbol{\Lambda}^{(4)}_{\s,\t,\x,\y}|^{2}\|_{\mathrm{L}^{p}(\mathbb{P})}&\lesssim_{\kappa,p}\N^{-\delta}\sup_{\r\in[\s,\t]}\boldsymbol{\Pi}^{\kappa,p}_{\s,\r}.\label{eq:lzeta3d}
\end{align}
The first estimate \eqref{eq:lzeta3a} follows from standard heat kernel estimates as in Proposition \ref{prop:hkestimates}. The second estimate \eqref{eq:lzeta3b} follows by the proof of (7.37) in \cite{Y26PTRF}, since all that is used therein is a collection of standard pointwise heat kernel estimates for the full-space analogue of {\small$\mathbf{H}^{\N,\mathbf{A}}$}, which also hold for {\small$\mathbf{H}^{\N,\mathbf{A}}$}-kernels as well by Proposition \ref{prop:hkestimates}. (In a nutshell, the \abbr{BDG} inequality allows us to control moments of {\small$\boldsymbol{\Lambda}^{(1)}$} in terms of its quadratic variation. The factor of {\small$|\t-\s|^{-1/2}$} on the \abbr{RHS} of \eqref{eq:lzeta3b} comes from squaring the {\small$\mathbf{H}^{\N,\mathbf{A}}_{\r,\t,\x,\w}$}-term in {\small$\boldsymbol{\Lambda}^{(1)}$}, for which we have estimates from Proposition \ref{prop:hkestimates} that also cancel the {\small$\N^{1/2}$}-factor in {\small$\boldsymbol{\Lambda}^{(1)}$} as well. The factor {\small$|\r-\s|^{-1/2}$} in \eqref{eq:lzeta3b} is there because the {\small$\boldsymbol{\Pi}^{\kappa,p}_{\s,\r}$}-term carries with it a factor of {\small$|\r-\s|^{1/2}$}; see \eqref{eq:lzeta1}.) The estimate \eqref{eq:lzeta3c} follows by the proof of (7.48) in \cite{Y26PTRF}. (In short, because of the stopping time {\small$\mathfrak{t}_{\mathrm{ap}}$}, we have an a priori estimate on {\small$\mathbf{Av}^{\mathbf{X},\mathscr{Q}_{\ell}}$} in {\small$\boldsymbol{\Lambda}^{(3)}$} of {\small$\lesssim\N^{-1/2+\mathrm{C}\delta_{\mathbf{S}}}$}; see \eqref{eq:tap}. Moreover, we also gain a factor of {\small$\mathfrak{t}_{\mathbf{Av}}\lesssim\N^{-2/3}$} [see Definition \ref{definition:bulkaverage}] because of the time-gradient and time-regularity of {\small$\mathbf{H}^{\N,\mathbf{A}}$} from Proposition \ref{prop:hkestimates}. Therefore, the total power-counting in {\small$\boldsymbol{\Lambda}^{(3)}$} is {\small$\lesssim\N^{-1/2-2/3+\mathrm{C}\delta_{\mathbf{S}}}\N\lesssim\N^{-1/8}$}. Finally, the {\small$\mathbf{H}^{\N,\mathbf{A}}$}- and {\small$\boldsymbol{\mathcal{K}}^{(\ell)}$}-kernels are averaging kernels in space up to a factor of {\small$\N^{\mathrm{C}\delta_{\mathbf{S}}}$} for {\small$\mathrm{C}\lesssim1$} because of \eqref{eq:sdeIII} and Proposition \ref{prop:hkestimates}.) Finally, \eqref{eq:lzeta3d} follows by the same reasoning as what gives \eqref{eq:lzeta3c}. Indeed, because of the a priori estimates from {\small$\mathfrak{t}_{\mathrm{ap}}$} (again, see \eqref{eq:tap}), the {\small$\mathfrak{b}_{\ell}$}-term is {\small$\lesssim\N^{\mathrm{C}\delta_{\mathbf{S}}}$} for some {\small$\mathrm{C}\lesssim1$}. We also gain a factor of {\small$\mathfrak{t}_{\partial}\lesssim\N^{-3/2}$} (see Definition \ref{definition:edgeaverage}) from the time-gradient, so the power-counting in {\small$\boldsymbol{\Lambda}^{(4)}$} is ultimately {\small$\lesssim\N^{1+\mathrm{C}\delta_{\mathbf{S}}}\N^{-3/2}\lesssim\N^{-1/6}$}. (Thus, power-counting is even ``better" in {\small$\boldsymbol{\Lambda}^{(4)}$} than {\small$\boldsymbol{\Lambda}^{(3)}$}.) Finally, the {\small$\mathbf{H}^{\N,\mathbf{A}}$} and {\small$\boldsymbol{\mathcal{K}}^{\partial,\ell}$} kernels are averaging kernels in space up to a factor of {\small$\N^{\mathrm{C}\delta_{\mathbf{S}}}$} for {\small$\mathrm{C}\lesssim1$} because of \eqref{eq:sdeIV} and Proposition \ref{prop:hkestimates}.

We are left with {\small$\boldsymbol{\Lambda}^{(2)}$}. Take \eqref{eq:sdeII1}-\eqref{eq:sdeII2}, but replace {\small$\mathbf{R}^{\N}$} by {\small$\mathbf{R}^{\N,\wedge}$}, and replace {\small$\mathbf{S}^{\N}$} by {\small$\boldsymbol{\chi}^{(\zeta)}\mathbf{L}^{\N,\zeta,\sim}_{\s,\r,\cdot,\y}$}. This gives
\begin{align*}
\mathbf{1}_{\mathfrak{t}_{\mathrm{ap}}\geq\r}|\mathrm{Err}[\boldsymbol{\chi}^{(\zeta)}\mathbf{R}^{\N,\wedge}\mathbf{L}^{\N,\zeta,\sim}_{\s,\cdot,\cdot,\y}]_{\r,\w}|\lesssim_{\omega}\sum_{\z\in\llbracket0,\infty\rrbracket}(\N^{-\frac65}+\N^{-\frac13}\mathbf{1}_{\z\in\llbracket0,1\rrbracket})\cdot\exp(\tfrac{-\omega|\w-\z|}{\N})\cdot\N^{\mathrm{C}\delta_{\mathbf{S}}}\cdot|\mathbf{L}^{\N,\zeta,\sim}_{\s,\r,\z,\y}|,
\end{align*}
where {\small$\omega>0$} is arbitrary, and where the {\small$\bphi$}-dependent factors from \eqref{eq:sdeII2} are replaced by {\small$\lesssim\N^{\mathrm{C}\delta_{\mathbf{S}}}$} because of the {\small$\mathfrak{t}_{\mathrm{ap}}$}-stopping time above (see \eqref{eq:tap}). From this and the proof of (7.38) in \cite{Y26PTRF}, we deduce that \eqref{eq:lzeta3d} holds for {\small$\boldsymbol{\Lambda}^{(2)}$} in place of {\small$\boldsymbol{\Lambda}^{(4)}$}. (Indeed, in the previous display, the effective power-counting is {\small$\lesssim\N^{-1/5+\mathrm{C}\delta_{\mathbf{S}}}+\N^{-1/3+\mathrm{C}\delta_{\mathbf{S}}}\lesssim\N^{-1/6}$} since {\small$\mathrm{C}\lesssim1$} and {\small$\delta_{\mathbf{S}}>0$} is small.) Combining this estimate for {\small$\boldsymbol{\Lambda}^{(2)}$} with \eqref{eq:lzeta2a} and \eqref{eq:lzeta3a}-\eqref{eq:lzeta3d} yields
\begin{align}
\boldsymbol{\Pi}^{\kappa,p}_{\s,\t}\lesssim_{\kappa,p}\N^{-1}+|\t-\s|^{\frac12}\int_{\s}^{\t}|\t-\r|^{-\frac12}|\r-\s|^{-\frac12}\boldsymbol{\Pi}^{\kappa,p}_{\s,\r}\d\r+\N^{-\delta}\sup_{\r\in[\s,\t]}\boldsymbol{\Pi}^{\kappa,p}_{\s,\r}.\nonumber
\end{align}
By the Gronwall inequality, from the previous display, we then obtain (for some {\small$\mathrm{C}_{\kappa,p}\lesssim_{\kappa,p}1$})
\begin{align*}
\boldsymbol{\Pi}^{\kappa,p}_{\s,\t}\lesssim\sup_{\tau\in[\s,\t]}\exp(\mathrm{C}_{\kappa,p}|\tau-\s|^{\frac12}\int_{\s}^{\tau}|\tau-\r|^{-\frac12}|\r-\s|^{-\frac12}\d\r)\cdot(\N^{-1}+\N^{-\delta}\sup_{\r\in[\s,\t]}\boldsymbol{\Pi}^{\kappa,p}_{\s,\r})\lesssim_{\kappa,p}\N^{-1}+\N^{-\delta}\sup_{\r\in[\s,\t]}\boldsymbol{\Pi}^{\kappa,p}_{\s,\r},
\end{align*}
where the last bound follows from an elementary integral estimate (split the integration domain {\small$[\s,\tau]$} in half, and on each half, control one of {\small$|\tau-\r|^{-1/2}|\r-\s|^{-1/2}$} uniformly and integrate the other). Because the far \abbr{RHS} of the previous estimate is monotone non-decreasing in {\small$\t\geq\s$}, we obtain from this an estimate on {\small$\sup_{\r\in[\s,\t]}\boldsymbol{\Pi}^{\kappa,p}_{\s,\r}$}:
\begin{align*}
\sup_{\r\in[\s,\t]}\boldsymbol{\Pi}^{\kappa,p}_{\s,\r}\lesssim_{\kappa,p}\N^{-1}+\N^{-\delta}\sup_{\r\in[\s,\t]}\boldsymbol{\Pi}^{\kappa,p}_{\s,\r}.
\end{align*}
From the above, we get {\small$\sup_{\r\in[\s,\t]}\boldsymbol{\Pi}^{\kappa,p}_{\s,\r}\lesssim_{\kappa,p}\N^{-1}$}. To show \eqref{eq:lzetaI} (and thus the lemma), we will need control not on {\small$\mathbf{L}^{\N,\zeta,\sim}$}, but rather {\small$\mathbf{L}^{\N,\zeta,\sim}-\mathbf{H}^{\N,\mathbf{A}}$}. To this end, we again use \eqref{eq:lzeta2a} along with \eqref{eq:lzeta3b}-\eqref{eq:lzeta3d} (and the extension of \eqref{eq:lzeta3d} to {\small$\boldsymbol{\Lambda}^{(2)}$}) to deduce that the \abbr{LHS} of \eqref{eq:lzetaI} is
\begin{align*}
\lesssim_{\kappa,p}\sup_{0\leq\s\leq\t\leq1}\Big(|\t-\s|^{\frac12}\int_{\s}^{\t}|\t-\r|^{-\frac12}|\r-\s|^{-\frac12}\boldsymbol{\Pi}^{\kappa,p}_{\s,\r}\d\r+\N^{-\delta}\sup_{\r\in[\s,\t]}\boldsymbol{\Pi}^{\kappa,p}_{\s,\r}\Big)\lesssim_{\kappa,p}\N^{-1}.
\end{align*}
This completes the proof of \eqref{eq:lzetaI}, so we are done.
\end{proof}
%
%
%
\section{Proof of Lemma \ref{lemma:smallzeta}}\label{section:smallzeta}
\subsection{Proof of \eqref{eq:smallzetaI}}
It suffices to assume that {\small$\zeta_{2}=\zeta_{1}-\e$} for some small but fixed {\small$\e>0$}. Indeed, for any other {\small$\zeta_{1},\zeta_{2}$}, consider any sequence {\small$\wt{\zeta}_{1},\ldots,\wt{\zeta}_{\k}$}, with {\small$\k\lesssim1$}, such that {\small$\wt{\zeta}_{1}=\zeta_{1}$}, that {\small$\wt{\zeta}_{\k}=\zeta_{2}$}, and {\small$\wt{\zeta}_{\i+1}=\wt{\zeta}_{\i}-\e$}. We would then have \eqref{eq:smallzetaI} with {\small$\zeta_{1},\zeta_{2}$} replaced by {\small$\wt{\zeta}_{\i},\wt{\zeta}_{\i+1}$} for each {\small$\i=1,\ldots,\k-1$}, at which point \eqref{eq:smallzetaI} as written follows by the triangle inequality (and that the intersection of finitely many high probability events also holds with high probability).

Thus, suppose that {\small$\zeta_{2}=\zeta_{1}-\e$}, and consider the difference {\small$\mathbf{D}^{\N}:=\mathbf{S}^{\N,\zeta_{1}}-\mathbf{S}^{\N,\zeta_{2}}$}; recall our assumption that {\small$\zeta_{1}>\zeta_{2}$}. By \eqref{eq:sdezetaI1}-\eqref{eq:sdezetaI6}, we know {\small$\mathbf{S}^{\N,\zeta_{1}}$} and {\small$\mathbf{S}^{\N,\zeta_{2}}$} satisfy the same linear \abbr{SDE}, with the same initial data, except the \abbr{SDE} for {\small$\mathbf{S}^{\N,\zeta_{1}}$} admits additional terms that are supported in space on {\small$\llbracket\mathrm{c}\N^{1+\zeta_{2}},\mathrm{C}\N^{1+\zeta_{1}}\rrbracket$} for some {\small$0<\mathrm{c},\mathrm{C}\lesssim1$}. Therefore, we can write {\small$\mathbf{D}^{\N}$} in terms of the stochastic heat kernel {\small$\mathbf{K}^{\N,\zeta_{2}}$} for the {\small$\mathbf{S}^{\N,\zeta_{2}}$} \abbr{SDE} integrated against these additional terms (which are linear in {\small$\mathbf{S}^{\N,\zeta_{1}}$} and supported on {\small$\llbracket\mathrm{c}\N^{1+\zeta_{2}},\mathrm{C}\N^{1+\zeta_{1}}\rrbracket$}). To be precise, we have 
\begin{align}
\mathbf{D}^{\N}_{\t,\x}=\int_{0}^{\t}\sum_{\w\in\llbracket\mathrm{c}\N^{1+\zeta_{2}},\mathrm{C}\N^{1+\zeta_{1}}\rrbracket}\mathbf{K}^{\N,\zeta_{2}}_{\s,\t,\x,\w}\cdot\mathscr{G}[\mathbf{S}^{\N,\zeta_{1}}]_{\s,\w}\d\s,\label{eq:smallzetaI1}
\end{align}
where {\small$\mathscr{G}[\mathbf{S}^{\N,\zeta_{1}}]$} is a linear operator in {\small$\mathbf{S}^{\N,\zeta_{1}}$} which satisfies the following deterministic estimate:
\begin{align}
\mathbf{1}_{\mathfrak{t}_{\mathrm{ap}}\geq\s}|\mathscr{G}[\mathbf{S}^{\N,\zeta_{1}}]_{\s,\w}|&\lesssim\N^{\mathrm{O}(1)}\sup_{\r\in[0,\s]}\sup_{|\y-\w|\lesssim\N}|\mathbf{S}^{\N,\zeta_{1}}_{\r,\y}|.\label{eq:smallzetaI2}
\end{align}
We recall {\small$\mathfrak{t}_{\mathrm{ap}}$} from \eqref{eq:tap} (in which we take {\small$\zeta=\zeta_{1}$}). The estimate above comes from the bound \eqref{eq:sdeII1}-\eqref{eq:sdeII2} on the {\small$\mathrm{Err}$}-operator in \eqref{eq:sdezetaI1} and the a priori inputs from \eqref{eq:tap} to control the other coefficients in \eqref{eq:sdezetaI3}-\eqref{eq:sdezetaI6}. Now, we use the representation \eqref{eq:hkzetaI7} for {\small$\mathbf{S}^{\N,\zeta_{1}}$}, as well as the {\small$\mathbf{K}^{\N,\zeta_{1}}$}-heat kernel estimate \eqref{eq:hkzetaI} and the high probability initial data estimate \eqref{eq:databoundI}. This gives, for some large but fixed {\small$0<\kappa\lesssim1$} and for some {\small$\beta_{\star}>0$}, the bound
\begin{align}
\exp(-\tfrac{\kappa|\y|}{\N})|\mathbf{S}^{\N,\zeta_{1}}_{\r,\y}|&\lesssim\sum_{\w\in\llbracket0,\infty\rrbracket}\exp(-\tfrac{\kappa|\y|}{\N})|\mathbf{K}^{\N,\zeta_{1}}_{0,\r,\y,\w}|\cdot|\mathbf{S}^{\N,\zeta_{1}}_{0,\y}|\lesssim\sum_{\w\in\llbracket0,\infty\rrbracket}\exp(-\tfrac{\kappa|\y|}{\N})\exp(\tfrac{\kappa|\w|}{\N})|\mathbf{K}^{\N,\zeta_{1}}_{0,\r,\y,\w}|\nonumber\\
&\lesssim\sum_{\w\in\llbracket0,\infty\rrbracket}\exp(\tfrac{\kappa|\y-\w|}{\N})|\mathbf{K}^{\N,\zeta_{1}}_{0,\r,\y,\w}|\lesssim\exp(\N^{-\beta_{\star}}\N^{\zeta_{1}})\N^{\delta}
\end{align}
with high probability simultaneously over all {\small$\r\in[0,1]$} and {\small$\y\in\llbracket0,\infty\rrbracket$}. Let us now combine \eqref{eq:smallzetaI1} and \eqref{eq:smallzetaI2} with the previous display, while noting that by \eqref{eq:tap2}, we can insert {\small$\mathbf{1}_{\mathfrak{t}_{\mathrm{ap}}\geq\s}$} inside the integral in \eqref{eq:smallzetaI1} for free with high probability. This provides, for {\small$\x\in\llbracket0,\mathrm{O}(\N)\rrbracket$}, the bound
\begin{align*}
\exp(-\tfrac{\kappa|\x|}{\N})|\mathbf{D}^{\N}_{\t,\x}|&\lesssim\exp(\N^{-\beta_{\star}}\N^{\zeta_{1}})\N^{\mathrm{O}(1)}\int_{0}^{\t}\sum_{\w\in\llbracket\mathrm{c}\N^{1+\zeta_{2}},\mathrm{C}\N^{1+\zeta_{1}}\rrbracket}\exp(-\tfrac{\kappa|\x|}{\N})|\mathbf{K}^{\N,\zeta_{2}}_{\s,\t,\x,\w}|\cdot\exp(\tfrac{\kappa|\w|}{\N})\d\s\\
&\lesssim\exp(\N^{-\beta_{\star}}\N^{\zeta_{1}})\N^{\mathrm{O}(1)}\int_{0}^{\t}\sum_{\w\in\llbracket\mathrm{c}\N^{1+\zeta_{2}},\mathrm{C}\N^{1+\zeta_{1}}\rrbracket}\exp(\tfrac{\kappa|\x-\w|}{\N})|\mathbf{K}^{\N,\zeta_{2}}_{\s,\t,\x,\w}|\d\s.
\end{align*}
In the last line, write {\small$\exp(\frac{\kappa|\x-\w|}{\N})=\exp(-\frac{\kappa|\x-\w|}{\N})\exp(\frac{2\kappa|\x-\w|}{\N})$} and bound the first factor by {\small$\lesssim\exp(-\upsilon\N^{\zeta_{2}})$} for some {\small$\upsilon\gtrsim1$}, because {\small$\x\in\llbracket0,\mathrm{O}(\N)\rrbracket$} and {\small$\w\gtrsim\N^{1+\zeta_{2}}$}. Then, use the heat kernel estimate \eqref{eq:hkzetaI} with {\small$2\kappa$} instead of {\small$\kappa>0$} therein. Ultimately, we establish that {\small$\exp(-\frac{\kappa|\x|}{\N})|\mathbf{D}^{\N}_{\t,\x}|\lesssim\exp(\N^{-\beta_{\star}}\N^{\zeta_{1}})\N^{\mathrm{O}(1)}\exp(-\upsilon\N^{\zeta_{2}})$}, which is exponentially small in {\small$\N$} if we take {\small$\zeta_{2}=\zeta_{1}-\e$} with small enough but fixed {\small$\e>0$} depending on {\small$\beta_{\star}$}. We clarify that this estimate holds with high probability simultaneously over all {\small$\t\in[0,1]$} and {\small$\x\in\llbracket0,\mathrm{O}(\N)\rrbracket$} given the global nature of the inputs \eqref{eq:hkzetaI} and \eqref{eq:databoundI} used to establish it. Finally, because {\small$\kappa\lesssim1$} and {\small$\x\in\llbracket0,\mathrm{O}(\N)\rrbracket$}, we know that {\small$\exp(-\frac{\kappa|\x|}{\N})\gtrsim1$}. Thus, the desired estimate \eqref{eq:smallzetaI} follows.
\subsection{Proof of \eqref{eq:smallzetaII}}
Suppose that, in the language of Lemma \ref{lemma:kzeta}, we know {\small$\zeta\in(0,\beta_{0}/2)$}. Then, the constraint in Lemma \ref{lemma:kzeta} on times {\small$\s,\t$} is only that {\small$0\leq\s\leq\t\leq1$}. In particular, we can use \eqref{eq:kzetaI8} in the proof of Lemma \ref{lemma:kzeta} to obtain {\small$\boldsymbol{\Theta}_{\s,\t,\kappa}\lesssim\N^{\delta}$} for any {\small$0\leq\s\leq\t\leq1$} simultaneously with high probability, where {\small$\boldsymbol{\Theta}_{\s,\t,\kappa}$} is defined in \eqref{eq:kzetaI0}. With this, we now claim that, recalling {\small$\mathbf{K}^{\N,\zeta,\sim}$} and {\small$\mathbf{L}^{\N,\zeta,\sim}$} from \eqref{eq:hkzetasim} and \eqref{eq:lzetasim}, respectively, we have
\begin{align}
\sup_{0\leq\s\leq\t\leq1}\sup_{\x\in\llbracket0,\infty\rrbracket}\sum_{\y\in\llbracket0,\infty\rrbracket}\exp(\tfrac{\kappa|\x-\y|}{\N})|\mathbf{K}^{\N,\zeta,\sim}_{\s,\t,\x,\y}-\mathbf{L}^{\N,\zeta,\sim}_{\s,\t,\x,\y}|\lesssim_{\kappa}\N^{-\beta_{1}}\label{eq:smallzetaII1}
\end{align}
with high probability for some fixed {\small$\beta_{1}>0$} and any fixed {\small$\kappa>0$}. This follows from the Duhamel formula \eqref{eq:duhamelI1}, \eqref{eq:kzetaI7}, and the estimate {\small$\boldsymbol{\Theta}_{\s,\t,\kappa}\lesssim\N^{\delta}$} for small {\small$\delta>0$} from \eqref{eq:kzetaI8}. Moreover, by Lemma \ref{lemma:hkzetasim}, the same holds for {\small$\mathbf{K}^{\N,\zeta}$} in place of {\small$\mathbf{K}^{\N,\zeta,\sim}$}. So, if we define the function
\begin{align*}
\mathbf{W}^{\N}_{\t,\x}=\sum_{\y\in\llbracket0,\infty\rrbracket}\mathbf{L}^{\N,\zeta,\sim}_{0,\t,\x,\y}\mathbf{S}^{\N}_{0,\y},
\end{align*}
then by \eqref{eq:smallzetaII1} and the initial data estimate \eqref{eq:databoundI}, we have the following with high probability:
\begin{align}
\sup_{\t\in[0,1]}\sup_{\x\in\llbracket0,\mathrm{O}(\N)\rrbracket}|\mathbf{S}^{\N,\zeta}_{\t,\x}-\mathbf{W}^{\N}_{\t,\x}|\lesssim\N^{-\frac12\beta_{1}}.\label{eq:smallzetaII2}
\end{align}
Now, suppose that {\small$\mathbf{Y}^{\N}$} is the solution to the \abbr{SDE} below, obtained by taking \eqref{eq:sdezetaI1}-\eqref{eq:sdezetaI6} and including only the first two terms on the \abbr{RHS} of the first line (i.e. the \abbr{SHE}-type terms), and with with initial data {\small$\mathbf{Y}^{\N}_{0,\cdot}=\mathbf{S}^{\N}_{0,\cdot}$}:
\begin{align}
\d\mathbf{Y}^{\N}_{\t,\x}=\mathscr{T}_{\N}\mathbf{Y}^{\N}_{\t,\x}\d\t+[\mathbf{e}^{\mathfrak{t}_{\N}\mathscr{T}_{\N}}(\mathbf{1}^{(\x)}_{\cdot}\boldsymbol{\chi}^{(\zeta_{\mathrm{large}})}_{\cdot}\cdot\sqrt{2}\lambda\N^{\frac12}\mathbf{R}^{\N,\wedge}_{\t,\cdot}\mathbf{Y}^{\N}_{\t,\cdot}\d\mathbf{b}_{\t,\cdot}]_{\x}.\label{eq:smallzetaII3}
\end{align}
Sine \eqref{eq:smallzetaII3} is a lattice version of the \abbr{SHE} \eqref{eq:sheIa}-\eqref{eq:sheIb}, with additional averaging and bounded coefficients, standard methods (e.g. those in Section 3 of \cite{P19}) combined with heat kernel estimates for {\small$\mathbf{H}^{\N,\mathbf{A}}$} from Proposition \ref{prop:hkestimates} give the following (analogous to \eqref{eq:mainI}) for any {\small$p\geq1$} and {\small$\kappa_{p}>0$} large enough:
\begin{align}
\sup_{\t\in[0,1]}\sup_{\x\in\llbracket0,\infty\rrbracket}\exp(-\tfrac{\kappa_{p}|\x|}{\N})\cdot\E|\mathbf{Y}^{\N}_{\t,\x}|^{p}\lesssim_{p}1.\label{eq:smallzetaII3a}
\end{align}
We will now control {\small$\mathbf{W}^{\N}-\mathbf{Y}^{\N}$}. Note that {\small$\mathbf{W}^{\N}$}, by construction, satisfies the \abbr{SDE} \eqref{eq:lzetasim} after replacing each copy of {\small$\mathbf{L}^{\N,\zeta,\sim}_{\s.\cdot,\cdot,\y}$} by {\small$\mathbf{W}^{\N}$}, and with initial data {\small$\mathbf{S}^{\N}_{0,\cdot}$}. So, {\small$\mathbf{F}^{\N}:=\mathbf{W}^{\N}-\mathbf{Y}^{\N}$} solves the \abbr{SDE} below with zero initial data:
\begin{align}
\d\mathbf{F}^{\N}_{\t,\x}=\mathscr{T}_{\N}\mathbf{F}^{\N}_{\t,\x}\d\t&+[\mathbf{e}^{\mathfrak{t}_{\N}\mathscr{T}_{\N}}(\mathbf{1}^{(\x)}_{\cdot}{\boldsymbol{\chi}^{(\zeta_{\mathrm{large}})}_{\cdot}}\cdot\sqrt{2}\lambda\N^{\frac12}{\mathbf{R}^{\N,\wedge}_{\t,\cdot}}\mathbf{F}^{\N}_{\t,\cdot}\d\mathbf{b}_{\t,\cdot})]_{\x}+{\mathbf{1}_{\mathfrak{t}_{\mathrm{ap}}\geq\t}}\mathrm{Err}[{\boldsymbol{\chi}^{(\zeta_{})}_{}\mathbf{R}^{\N,\wedge}}\mathbf{W}^{\N}]_{\t,\x}\d\t\nonumber\\
&+\sum_{\ell\in\llbracket1,\mathrm{L}_{1}\rrbracket}{}\grad^{\mathbf{T},\mathrm{av}}_{\mathfrak{t}_{\mathbf{Av}}}\Big({\mathbf{1}_{\mathfrak{t}_{\mathrm{ap}}\geq\t}}\sum_{\w\in\llbracket0,\infty\rrbracket}\boldsymbol{\mathcal{K}}^{(\ell)}_{\x,\w}\cdot{\boldsymbol{\chi}^{(\zeta_{})}_{\w}}\cdot\N\cdot\mathbf{Av}^{\mathbf{X},\mathscr{Q}_{\ell}}_{\t,\w}{\mathbf{R}^{\N,\wedge}_{\t,\w_{\pm}}}\mathbf{W}^{\N}_{\t,\w_{\pm}}\Big)\d\t\nonumber\\
&+\sum_{\ell\in\llbracket1,\mathrm{L}_{2}\rrbracket}\grad^{\mathbf{T},\mathrm{av}}_{\mathfrak{t}_{\partial}}\Big({\mathbf{1}_{\mathfrak{t}_{\mathrm{ap}}\geq\t}}\sum_{\w\in\llbracket0,\infty\rrbracket}\boldsymbol{\mathcal{K}}^{\partial,\ell}_{\x,\w}\cdot\N\mathfrak{b}_{\ell}[\bphi_{\t}]{\mathbf{R}^{\N,\wedge}_{\t,\w}}\mathbf{W}^{\N}_{\t,\w}\Big)\d\t.\label{eq:smallzetaII4}
\end{align}
By the Duhamel formula, we now obtain from \eqref{eq:smallzetaII4} the following analogue of \eqref{eq:lzeta2a}:
\begin{align}
\mathbf{F}^{\N}_{\t,\x}&=\boldsymbol{\Lambda}^{(1)}_{\t,\x}+\ldots+\boldsymbol{\Lambda}^{(4)}_{\t,\x}.\label{eq:smallzetaII5}
\end{align}
Above, we have introduced the notation
\begin{align*}
\boldsymbol{\Lambda}^{(1)}_{\t,\x}&:=\int_{0}^{\t}\sum_{\w\in\llbracket0,\infty\rrbracket}\mathbf{H}^{\N,\mathbf{A}}_{\r,\t,\x,\w}\cdot[\mathbf{e}^{\mathfrak{t}_{\N}\mathscr{T}_{\N}}(\mathbf{1}^{(\w)}_{\cdot}\boldsymbol{\chi}^{(\zeta_{\mathrm{large}})}_{\cdot}\cdot\sqrt{2}\lambda\N^{\frac12}\mathbf{R}^{\N,\wedge}_{\r,\cdot}\mathbf{F}^{\N}_{\r,\cdot}\d\mathbf{b}_{\r,\cdot}]_{\w},\\
\boldsymbol{\Lambda}^{(2)}_{\t,\x}&:=\int_{0}^{\t}\sum_{\w\in\llbracket0,\infty\rrbracket}\mathbf{H}^{\N,\mathbf{A}}_{\r,\t,\x,\w}\cdot\mathbf{1}_{\mathfrak{t}_{\mathrm{ap}}\geq\r}\mathrm{Err}[\boldsymbol{\chi}^{(\zeta)}\mathbf{R}^{\N,\wedge}\mathbf{W}^{\N}]_{\r,\w}\d\r,\\
\boldsymbol{\Lambda}^{(3)}_{\t,\x}&:=\int_{0}^{\t}\sum_{\w\in\llbracket0,\infty\rrbracket}\mathbf{H}^{\N,\mathbf{A}}_{\r,\t,\x,\w}\cdot\sum_{\ell\in\llbracket1,\mathrm{L}_{1}\rrbracket}{}\grad^{\mathbf{T},\mathrm{av}}_{\mathfrak{t}_{\mathbf{Av}}}\Big({\mathbf{1}_{\mathfrak{t}_{\mathrm{ap}}\geq\r}}\sum_{\z\in\llbracket0,\infty\rrbracket}\boldsymbol{\mathcal{K}}^{(\ell)}_{\w,\z}\cdot{\boldsymbol{\chi}^{(\zeta_{})}_{\z}}\cdot\N\cdot\mathbf{Av}^{\mathbf{X},\mathscr{Q}_{\ell}}_{\r,\z}{\mathbf{R}^{\N,\wedge}_{\t,\z_{\pm}}}\mathbf{W}^{\N}_{\r,\z_{\pm}}\Big)\d\r,\\
\boldsymbol{\Lambda}^{(4)}_{\t,\x}&:=\int_{0}^{\t}\sum_{\w\in\llbracket0,\infty\rrbracket}\mathbf{H}^{\N,\mathbf{A}}_{\r,\t,\x,\w}\cdot\sum_{\ell\in\llbracket1,\mathrm{L}_{2}\rrbracket}\grad^{\mathbf{T},\mathrm{av}}_{\mathfrak{t}_{\partial}}\Big({\mathbf{1}_{\mathfrak{t}_{\mathrm{ap}}\geq\r}}\sum_{\z\in\llbracket0,\infty\rrbracket}\boldsymbol{\mathcal{K}}^{\partial,\ell}_{\w,\z}\cdot\N\mathfrak{b}_{\ell}[\bphi_{\r}]{\mathbf{R}^{\N,\wedge}_{\r,\z}}\mathbf{W}^{\N}_{\r,\z}\Big)\d\r.
\end{align*}
In order to analyze this expansion, we appeal to the control parameter below for {\small$\kappa,p\geq1$} large but fixed:
\begin{align*}
\boldsymbol{\Xi}^{\kappa,p,\mathbf{F}}_{\t}:=\sup_{\x\in\llbracket0,\infty\rrbracket}\exp(-\tfrac{2\kappa|\x|}{\N})\|\mathbf{F}^{\N}_{\t,\x}\|_{\mathrm{L}^{p}(\mathbb{P})}^{2}:=\sup_{\x\in\llbracket0,\infty\rrbracket}\exp(-\tfrac{2\kappa|\x|}{\N})(\E|\mathbf{F}^{\N}_{\t,\x}|^{p})^{\frac{2}{p}}.
\end{align*}
Let us define {\small$\boldsymbol{\Xi}^{\kappa,p,\mathbf{W}}$} similarly by replacing {\small$\mathbf{F}\mapsto\mathbf{W}$} everywhere. Compared to the control parameter in \eqref{eq:lzeta1}, there is no singular time-dependent factor here. Moreover, one can regard both {\small$\mathbf{F}^{\N},\mathbf{W}^{\N}$} as kernels trivially by the extensions {\small$\mathbf{F}^{\N}_{\s,\t,\x,\y}=\mathbf{F}^{\N}_{\t-\s,\x}\mathbf{1}_{\x=\y}$} and {\small$\mathbf{W}^{\N}_{\s,\t,\x,\y}=\mathbf{W}^{\N}_{\t-\s,\x}\mathbf{1}_{\x=\y}$} for any {\small$0\leq\s\leq\t$}. Thus, the bounds \eqref{eq:lzeta3b}-\eqref{eq:lzeta3d}, which hold also for {\small$\boldsymbol{\Lambda}^{(2)}$} therein, extend to the following for some {\small$\delta>0$} fixed:
\begin{align*}
\exp(-\tfrac{2\kappa|\x|}{\N})\|\boldsymbol{\Lambda}^{(1)}_{\t,\x}\|_{\mathrm{L}^{p}(\mathbb{P})}^{2}+\ldots+\exp(-\tfrac{2\kappa|\x|}{\N})\|\boldsymbol{\Lambda}^{(4)}_{\t,\x}\|_{\mathrm{L}^{p}(\mathbb{P})}^{2}\lesssim_{\kappa,p}\int_{0}^{\t}|\t-\r|^{-\frac12}\boldsymbol{\Xi}^{\kappa,p,\mathbf{F}}_{\r}\d\r+\N^{-\delta}\sup_{\r\in[0,\t]}\boldsymbol{\Xi}^{\kappa,p,\mathbf{W}}_{\r}.
\end{align*}
The \abbr{RHS} of the display above does not depend on {\small$\x$} on the \abbr{LHS}, so we can place a supremum over {\small$\x\in\llbracket0,\infty\rrbracket$} on the \abbr{LHS}. Since {\small$\exp(-\frac{2\kappa|\x|}{\N})\|\mathbf{F}^{\N}_{\t,\x}\|_{\mathrm{L}^{p}(\mathbb{P})}^{2}$} is controlled by the \abbr{LHS} of the previous display, we obtain the estimate 
\begin{align*}
\boldsymbol{\Xi}^{\kappa,p,\mathbf{F}}_{\t}\lesssim_{\kappa,p}\int_{0}^{\t}|\t-\r|^{-\frac12}\boldsymbol{\Xi}^{\kappa,p,\mathbf{F}}_{\r}\d\r+\N^{-\delta}\sup_{\r\in[0,\t]}\boldsymbol{\Xi}^{\kappa,p,\mathbf{W}}_{\r}.
\end{align*}
Thus, by the Gronwall inequality as in the proof of Lemma \ref{lemma:lzeta}, we obtain {\small$\boldsymbol{\Xi}^{\kappa,p,\mathbf{F}}_{\t}\lesssim_{\kappa,p}\N^{-\delta}\sup_{\r\in[0,\t]}\boldsymbol{\Xi}^{\kappa,p,\mathbf{W}}_{\r}$} for all {\small$\t\in[0,1]$}. Moreover, by \eqref{eq:smallzetaII3a}, we get that {\small$\boldsymbol{\Xi}^{\kappa,p,\mathbf{W}}_{\r}\lesssim\boldsymbol{\Xi}^{\kappa,p,\mathbf{F}}_{\r}+1$} for all {\small$\r\in[0,1]$} for any {\small$p\geq1$} and {\small$\kappa=\kappa_{p}>0$} large enough. Thus, we obtain {\small$\boldsymbol{\Xi}^{\kappa,p,\mathbf{F}}_{\t}\lesssim_{p,\kappa}\N^{-\delta}$} for such {\small$p,\kappa\geq1$}. By the Chebyshev inequality, we then deduce
\begin{align*}
\sup_{\t\in[0,1]}\sup_{\x\in\llbracket0,\mathrm{O}(\N)\rrbracket}\mathbb{P}(|\mathbf{F}^{\N}_{\t,\x}|\geq\N^{-\frac16\delta})\lesssim\N^{\frac16p\delta}\E|\mathbf{F}^{\N}_{\t,\x}|^{p}\lesssim\N^{\frac16p\delta}\N^{-\frac12p\delta}\lesssim\N^{-\frac13p\delta}.
\end{align*}
For any {\small$\mathrm{D}>0$} large but fixed, we can take {\small$p\geq1$} large enough (but still fixed) so that the far \abbr{RHS} is {\small$\lesssim_{\mathrm{D}}\N^{-\mathrm{D}}$}. If we take a union bound over all {\small$\x\in\llbracket0,\mathrm{O}(\N)\rrbracket$} and all {\small$\t\in[0,1]$}, first over a very fine discretization and then to all of {\small$[0,1]$} by a standard short-time continuity argument as in the proof of Lemma \ref{lemma:phiap}, then we obtain {\small$|\mathbf{F}^{\N}_{\t,\x}|\lesssim\N^{-\delta/6}$} simultaneously over all {\small$(\t,\x)\in[0,1]\times\llbracket0,\mathrm{O}(\N)\rrbracket$} with high probability. By combining this and \eqref{eq:smallzetaII2}, we deduce that to show the desired estimate \eqref{eq:smallzetaII}, it suffices to show that with high probability, we have {\small$|\mathbf{Y}^{\N}_{\t,\x}-\mathbf{Q}^{\N}_{\t,\x}|\lesssim\N^{-\gamma}$} uniformly for all {\small$(\t,\x)\in[0,1]\times\llbracket0,\mathrm{L}\N\rrbracket$}, for some {\small$\gamma>0$}, and with {\small$\mathrm{L}\lesssim1$} as in the statement of the lemma. This, however, is a consequence of standard stochastic analysis for \abbr{SHE}-type equations in dimension {\small$1+1$}. Indeed, it follows by standard methods (which give \eqref{eq:smallzetaII3a}) that {\small$\mathbf{Y}^{\N}$} has spatial H\"{o}lder regularity at length-scales of order {\small$\N$}. Moreover, the indicator {\small$\mathbf{1}^{(\x)}$} in \eqref{eq:smallzetaII3} can be ignored up to an exponentially small error, because it is supported where the {\small$\mathbf{H}^{\N,\mathbf{A}}$}-heat kernel therein is exponentially small in {\small$\N$}. Also, the {\small$\mathbf{e}^{\mathfrak{t}_{\N}\mathscr{T}_{\N}}$}-operator in \eqref{eq:smallzetaII3} is smoothing at a mesoscopic scale (of {\small$\ll\N$}, as {\small$\mathfrak{t}_{\N}\ll1$} as in Definition \ref{definition:heat}). Thus, we can also remove this operator from \eqref{eq:smallzetaII3} up to an error which vanishes in the large-{\small$\N$} limit (uniformly on {\small$[0,1]\times\llbracket0,\mathrm{O}(\N)\rrbracket$}). Finally, the function {\small$\boldsymbol{\chi}^{(\zeta_{\mathrm{large}})}$} in \eqref{eq:smallzetaII3} is identically equal to {\small$1$} on {\small$\llbracket0,\N^{1+\zeta}\rrbracket$}, whose length-scale is much larger than the spatial domain {\small$\llbracket0,\mathrm{O}(\N)\rrbracket$} of interest. Thus, up to another error that vanishes in the large-{\small$\N$} limit, we may remove this {\small$\boldsymbol{\chi}^{(\zeta_{\mathrm{large}})}$}-function in \eqref{eq:smallzetaII3}. Lastly, by construction in \eqref{eq:rwedge}, we have {\small$\mathbf{R}^{\N,\wedge}=1+\mathrm{o}(1)$} uniformly in space-time, so we can also remove this factor in \eqref{eq:smallzetaII3} up to another vanishing error. The \abbr{SDE} we end up with is that \eqref{eq:qsde} for {\small$\mathbf{Q}^{\N}$}, and the initial data between these two differ by another mesoscopic-scale {\small$\mathbf{H}^{\N,\mathbf{A}}$}-smoothing operator whose contribution vanishes in the large-{\small$\N$} limit due to spatial regularity of the initial data for {\small$\mathbf{Q}^{\N}$} (which we assumed in \eqref{eq:mainI}). Ultimately, we obtain {\small$|\mathbf{Y}^{\N}_{\t,\x}-\mathbf{Q}^{\N}_{\t,\x}|\lesssim\N^{-\gamma}$} uniformly for all {\small$(\t,\x)\in[0,1]\times\llbracket0,\mathrm{L}\N\rrbracket$}, for some {\small$\gamma>0$}, and with {\small$\mathrm{L}\lesssim1$} as in the statement of the lemma, all with high probability. As noted in this paragraph, this completes the proof. \qed
\appendix
\section{Auxiliary estimates and properties}
\subsection{Heat kernel estimates}
We now record a list of purely analytic estimates that are used throughout.
\begin{prop}\label{prop:hkestimates}
\fsp Recall the notation in Definition \ref{definition:heat}, and for any {\small$0\leq\s\leq\t$}, we set {\small$\delta_{\s,\t}:=\max(|\t-\s|,\N^{-2})$}. Throughout, fix any {\small$\kappa\in\R$}, any {\small$0\leq\s\leq\t\lesssim1$}, and any {\small$\x,\y\in\llbracket0,\infty\rrbracket$}. First, we have
\begin{align}
|\mathbf{H}^{\N,\mathbf{a}}_{\s,\t,\x,\y}|\lesssim_{\kappa,\mathbf{a}}\N^{-1}\delta_{\s,\t}^{-\frac12}\exp(-\tfrac{\kappa|\x-\y|}{\N\delta_{\s,\t}^{1/2}}).\label{eq:hkestimatesI}
\end{align}
In particular, we also have the following estimates for all {\small$\kappa>0$}, all {\small$0\leq\s\leq\t$}, and all {\small$\x\in\llbracket0,\infty\rrbracket$}:
\begin{align}
\sum_{\y\in\llbracket0,\infty\rrbracket}\exp(\tfrac{\kappa|\x-\y|}{\N})|\mathbf{H}^{\N,\mathbf{a}}_{\s,\t,\x,\y}|\lesssim_{\kappa,\mathbf{a}}1 \quad\text{and}\quad \sum_{\y\in\llbracket0,\infty\rrbracket}\exp(\tfrac{\kappa|\x-\y|}{\N})|\mathbf{H}^{\N,\mathbf{a}}_{\s,\t,\x,\y}|^{2}\lesssim_{\kappa,\mathbf{a}}\N^{-1}\delta_{\s,\t}^{-\frac12}.\label{eq:hkestimatesII}
\end{align}
We move to regularity estimates. Fix any tuples {\small$\vec{\a}=(\a_{1},\ldots,\a_{\n})$} and {\small$\vec{\ell}=(\ell_{1},\ldots,\ell_{\m})$} of integers, any {\small$0\leq\s\leq\t\lesssim1$}, and any {\small$\x,\y\in\llbracket0,\infty\rrbracket$} such that the following conditions are all satisfied:
\begin{enumerate}
\item We have {\small$\n,\m\lesssim1$} and {\small$\N|\t-\s|^{1/2}\gtrsim1$}.
\item We have {\small$|\a_{1}|,\ldots,|\a_{\n}|,|\ell_{1}|,\ldots,|\ell_{\m}|\lesssim\max(\N|\t-\s|^{1/2},1)$}.
\end{enumerate}
Then, for any {\small$\kappa>0$}, we have the following joint regularity estimate at length-scale {\small$\N\delta_{\s,\t}^{1/2}$}:
\begin{align}
|\grad^{\mathbf{X}}_{\vec{\a}}\grad^{\mathbf{X}}_{\vec{\ell}}\mathbf{H}^{\N,\mathbf{a}}_{\s,\t,\x,\y}|\lesssim_{\kappa,\mathbf{a}}\N^{-1}\delta_{\s,\t}^{-\frac12}\exp(-\tfrac{\kappa|\x-\y|}{\N\delta_{\s,\t}^{1/2}})\cdot\prod_{\i\in\llbracket1,\n\rrbracket}\N^{-1}\delta_{\s,\t}^{-\frac12}|\a_{\i}|\cdot\prod_{\j\in\llbracket1,\m\rrbracket}\N^{-1}\delta_{\s,\t}^{-\frac12}|\ell_{\j}|. \label{eq:hkestimatesIII}
\end{align}
Above, {\small$\grad^{\mathbf{X}}_{\vec{\a}}:=\grad^{\mathbf{X}}_{\a_{1}}\ldots\grad^{\mathbf{X}}_{\a_{\n}}$} acts on {\small$\x$}, and {\small$\grad^{\mathbf{X}}_{\vec{\ell}}$} (which is defined similarly) acts on {\small$\y$}. (We adopt the convention that we only restrict to choices of parameters for which the \abbr{LHS} depends only on {\small$\mathbf{H}^{\N,\mathbf{a}}_{\s,\t,\cdot,\cdot}$}-values on {\small$\llbracket0,\infty\rrbracket\times\llbracket0,\infty\rrbracket$}.) 

In particular, under the same assumptions as in \eqref{eq:hkestimatesIII}, we have 
\begin{align}
\sum_{\y\in\llbracket0,\infty\rrbracket}\exp(\tfrac{\kappa|\x-\y|}{\N})|\grad^{\mathbf{X}}_{\vec{\a}}\grad^{\mathbf{X}}_{\vec{\ell}}\mathbf{H}^{\N,\mathbf{a}}_{\s,\t,\x,\y}|\lesssim_{\kappa,\mathbf{a}}\prod_{\i\in\llbracket1,\n\rrbracket}\N^{-1}\delta_{\s,\t}^{-\frac12}|\a_{\i}|\cdot\prod_{\j\in\llbracket1,\m\rrbracket}\N^{-1}\delta_{\s,\t}^{-\frac12}|\ell_{\j}|.\label{eq:hkestimatesIV}
\end{align}
Finally, we also have the Chapman-Kolmogorov equation below for any {\small$0\leq\s\leq\r\leq\t$} and any {\small$\x,\y\in\llbracket0,\infty\rrbracket$}:
\begin{align}
\mathbf{H}^{\N,\mathbf{a}}_{\s,\t,\x,\y}=\sum_{\w\in\llbracket0,\infty\rrbracket}\mathbf{H}^{\N,\mathbf{a}}_{\r,\t,\x,\w}\mathbf{H}^{\N,\mathbf{a}}_{\s,\r,\w,\y}.\label{eq:hkestimatesV}
\end{align}
\end{prop}
\begin{proof}
We show \eqref{eq:hkestimatesI}-\eqref{eq:hkestimatesII}. The first of these is Proposition 3.1 in \cite{P19}. The second \eqref{eq:hkestimatesII} follows by the former \eqref{eq:hkestimatesI} (for a larger choice of {\small$\kappa>0$}). We clarify that the argument in \cite{P19} to establish \eqref{eq:hkestimatesI} is based on the fact that said estimate is true for the full-line heat kernel (see \cite{DT16}) plus some gymnastics based on a modified ``method of images" formula for the half-space heat kernel in terms of the full-space one. In particular, because the estimate \eqref{eq:hkestimatesIII} holds for the full-space heat kernel by a repeated application of the calculations in \cite{DT16}, the same gymnastics also yield \eqref{eq:hkestimatesIII} itself; since these gymnastics are elementary, we omit them. Next, the bound \eqref{eq:hkestimatesIV} follows from \eqref{eq:hkestimatesIII}, in the same way that \eqref{eq:hkestimatesII} follows from \eqref{eq:hkestimatesI}. Lastly, the identity \eqref{eq:hkestimatesV} follows since both sides solve the same linear evolution equation in {\small$(\t,\x)$} with the same ``initial" data at time {\small$\t=\r$}.
\end{proof}
\begin{lemma}\label{lemma:hknorm}
\fsp Recall the notation in Definition \ref{definition:heat}. Fix any {\small$0\leq\s\leq\t\lesssim1$}. We have that {\small$\mathbf{e}^{\t\mathscr{T}_{\N,\mathbf{a}}}:\ell^{2}(\llbracket0,\infty\rrbracket)\to\ell^{2}(\llbracket0,\infty\rrbracket)$} is a bounded operator whose norm is {\small$\lesssim_{\mathbf{a}}1$}. Moreover, if {\small$\mathbf{1}$} is the constant function on {\small$\llbracket0,\infty\rrbracket$} (whose value is equal to {\small$1$} everywhere in space), then we have 
\begin{align}
|\mathbf{e}^{\t\mathscr{T}_{\N,\mathbf{a}}}(\mathbf{1})-1|\lesssim_{\mathbf{a}}\t^{\frac12}.\label{eq:hknormI}
\end{align}
\end{lemma}
\begin{proof}
We start with the operator norm estimate. Recall the heat {\small$\mathscr{T}_{\N,\mathbf{a}}$}-semigroup from Definition \ref{definition:heat}. From this and \eqref{eq:hkestimatesI}-\eqref{eq:hkestimatesII} with {\small$\kappa=0$}, we get {\small$\|\mathbf{e}^{\t\mathscr{T}_{\N,\mathbf{a}}}\|_{\ell^{p}(\llbracket0,\infty\rrbracket)\to\ell^{p}(\llbracket0,\infty\rrbracket)}\lesssim_{\mathbf{a}}1$} for {\small$p=1,\infty$}. The operator norm estimate of interest now follows by interpolation. For the estimate \eqref{eq:hknormI}, see Proposition 3.7 in \cite{P19} (the time-parameter therein has not yet been scaled by {\small$\N^{2}$}, which is {\small$\e^{-2}$} therein). This completes the proof of the lemma.
\end{proof}
\begin{lemma}\label{lemma:kernelestimate}
\fsp Let us assume that {\small$\boldsymbol{\mathcal{S}}_{\cdot,\cdot}:\llbracket0,\infty\rrbracket^{2}\to\R$} satisfies \eqref{eq:sdepreI}. Recall also {\small$\mathfrak{l}_{\mathbf{Av}}=\N^{1-3\delta_{\mathbf{S}}/2}$} from Definition \ref{definition:bulkaverage}, where {\small$\delta_{\mathbf{S}}>0$} is small but fixed, and the operator {\small$\mathscr{X}:=\mathrm{Id}-\exp(\tau_{\mathbf{Av}}\mathscr{T}_{\N,0})$} in the proof of Lemma \ref{lemma:xav}. For any positive integer {\small$0\leq\m\lesssim1$}, any {\small$\kappa>0$}, and for some {\small$\mathrm{C}\lesssim1$}, we have the following for each {\small$\y,\w\in\llbracket0,\infty\rrbracket$}:
\begin{align}
|\mathscr{X}^{\m}\boldsymbol{\mathcal{S}}_{\y,\w}|&\lesssim_{\m,\kappa}\N^{-1+\mathrm{C}\delta_{\mathbf{S}}}\exp(-\tfrac{\kappa|\w-\y|}{\N}).\label{eq:kernelestimateI}
\end{align}
Moreover, for any {\small$\mathrm{D}>0$}, there exists {\small$\mathrm{M}=\mathrm{O}(1)$} large enough such that 
\begin{align}
|\mathscr{X}^{\mathrm{M}}\boldsymbol{\mathcal{S}}_{\y,\w}|&\lesssim_{\m,\kappa}\N^{-\mathrm{D}}\exp(-\tfrac{\kappa|\w-\y|}{\N}).\label{eq:kernelestimateII}
\end{align}
\end{lemma}
\begin{proof}
First, \eqref{eq:kernelestimateI} follows since it is true for {\small$\m=0$}, and the operator has an integral kernel which is an averaging kernel with exponential decay at macroscopic scale {\small$\N$} by Proposition \ref{prop:hkestimates}. Now, we prove \eqref{eq:kernelestimateII}. We claim that it suffices to prove it for {\small$\kappa=0$}. Indeed, we have
\begin{align*}
|\mathscr{X}^{\mathrm{M}}\boldsymbol{\mathcal{S}}_{\y,\w}|&=|\mathscr{X}^{\mathrm{M}}\boldsymbol{\mathcal{S}}_{\y,\w}|^{\frac12}|\mathscr{X}^{\mathrm{M}}\boldsymbol{\mathcal{S}}_{\y,\w}|^{\frac12}.
\end{align*}
For the first factor on the \abbr{RHS}, use \eqref{eq:kernelestimateI} for {\small$\m=\mathrm{M}$} and {\small$2\kappa$} instead of {\small$\kappa$}, and for the second term, we use \eqref{eq:kernelestimateII} for {\small$\kappa=0$} (and for {\small$2\mathrm{D}$} instead of {\small$\mathrm{D}$}). This would yield \eqref{eq:kernelestimateII} for the {\small$\kappa,\mathrm{D}$} of interest. Thus, restrict to {\small$\kappa=0$}.

By construction in Definition \ref{definition:heat}, we have {\small$\partial_{\t}[\mathbf{e}^{\t\mathscr{T}_{\N,0}}(\mathsf{f}_{\cdot})]_{\x}=(\N^{2}+\frac14\N\lambda)\Delta_{0}[\mathbf{e}^{\t\mathscr{T}_{\N,0}}(\mathsf{f}_{\cdot})]_{\x}$} for {\small$\mathsf{f}:\llbracket0,\infty\rrbracket\to\R$}. So, since the {\small$\mathscr{T}_{\N,0}$}-semigroup commutes with its generator {\small$\mathscr{T}_{\N}=\mathscr{T}_{\N,0}=\mathrm{O}(\N^{2})\Delta_{0}$}, for any {\small$\t\geq0$}, we have
\begin{align*}
|\mathscr{X}\mathsf{f}_{\x}|=|[\mathbf{e}^{\t\mathscr{T}_{\N,0}}(\mathsf{f}_{\cdot})]_{\x}-\mathsf{f}_{\x}|\lesssim\int_{0}^{\t}|\N^{2}\Delta_{0}[\mathbf{e}^{\s\mathscr{T}_{\N,0}}(\mathsf{f}_{\cdot})]_{\x}|\d\s\lesssim\t\sup_{\s\in[0,\t]}|[\mathbf{e}^{\s\mathscr{T}_{\N,0}}(\N^{2}\Delta_{0}\mathsf{f}_{\cdot})]_{\x}|.
\end{align*}
Next, we take the supremum over {\small$\x\in\llbracket0,\infty\rrbracket$} on both sides. Then, we plug in {\small$\t=\N^{-2}\mathfrak{l}_{\mathbf{Av}}^{2}$} with {\small$\mathfrak{l}_{\mathbf{Av}}=\N^{1-3\delta_{\mathbf{S}}/2}$} from Definition \ref{definition:bulkaverage}. This gives
\begin{align*}
\|\mathscr{X}\mathsf{f}\|_{\mathrm{L}^{\infty}(\llbracket0,\infty\rrbracket)}&\lesssim\N^{-2}\mathfrak{l}_{\mathbf{Av}}^{2}\sup_{\t\in[0,\N^{-2}\mathfrak{l}_{\mathbf{Av}}^{2}]}\|\mathbf{e}^{\t\mathscr{T}_{\N,0}}(\N^{2}\Delta_{0}\mathsf{f}_{\cdot})\|_{\mathrm{L}^{\infty}(\llbracket0,\infty\rrbracket)}\\
&\lesssim\N^{-2}\mathfrak{l}_{\mathbf{Av}}^{2}\|\N^{2}\Delta_{0}\mathsf{f}\|_{\mathrm{L}^{\infty}(\llbracket0,\infty\rrbracket)},
\end{align*}
where the last bound follows because {\small$\mathbf{e}^{\t\mathscr{T}_{\N,0}}$} is locally uniformly (in {\small$\t\geq0$}) bounded as an operator {\small$\mathrm{L}^{\infty}(\llbracket0,\infty\rrbracket)$}, a consequence of Proposition \ref{prop:hkestimates}. We iterate the previous estimate a total of {\small$\m\geq0$}-many times, and we plug in {\small$\mathsf{f}_{\cdot}=\boldsymbol{\mathcal{S}}_{\y,\cdot}$}. Then, we use the regularity estimate in \eqref{eq:sdepreI}, which says that {\small$\boldsymbol{\mathcal{S}}_{\y,\cdot}$} is an averaging kernel and thus has size {\small$\lesssim\N^{-1+\mathrm{C}\delta_{\mathbf{S}}}$} and regularity at scale {\small$\gtrsim\N^{1-\delta_{\mathbf{S}}}$}. Thus, we obtain (again since {\small$\mathfrak{l}_{\mathbf{Av}}=\N^{1-3\delta_{\mathbf{S}}/2}$})
\begin{align}
\|\mathscr{X}^{\m}\boldsymbol{\mathcal{S}}_{\y,\cdot}\|_{\mathrm{L}^{\infty}(\llbracket0,\infty\rrbracket)}\lesssim\N^{-2\m}\mathfrak{l}_{\mathbf{Av}}^{2\m}\|\N^{2\m}(\Delta_{0})^{\m}\boldsymbol{\mathcal{S}}_{\y,\cdot}\|_{\mathrm{L}^{\infty}(\llbracket0,\infty\rrbracket)}\lesssim\N^{2\m-3\m\delta_{\mathbf{S}}}\N^{-1-2\m+2\m\delta_{\mathbf{S}}+\mathrm{C}\delta_{\mathbf{S}}}\lesssim\N^{-\m\delta_{\mathbf{S}}+\mathrm{C}\delta_{\mathbf{S}}}\nonumber
\end{align}
for any {\small$0\leq\m\lesssim1$}. (We clarify the importance that {\small$\Delta_{0}$} commutes with the {\small$\mathscr{X}$}-operator.) For any {\small$\mathrm{D}>0$}, choose {\small$\m\lesssim1$} large enough so that {\small$-\m\delta_{\mathbf{S}}+\mathrm{C}\delta_{\mathbf{S}}\leq-\mathrm{D}$}. This yields \eqref{eq:kernelestimateII} for {\small$\kappa=0$}, which completes the proof.
\end{proof}
\subsection{Infinitesimal generators and invariant measures}\label{subsection:generator}
The infinitesimal generator of \eqref{eq:phiI}-\eqref{eq:phiII} is given by 
\begin{align}
\mathscr{L}:=\mathscr{L}_{\mathrm{S}}+\mathscr{L}_{\mathrm{A}}+\mathscr{L}_{\mathrm{G}},
\end{align}
in which {\small$\mathscr{L}_{\mathrm{S}},\mathscr{L}_{\mathrm{A}},\mathscr{L}_{\mathrm{G}}$} act on smooth local functions {\small$\R^{\llbracket1,\infty\rrbracket}\to\R$}, where ``local" refers to functions that depend only on a finite number of coordinates of any input {\small$\bphi\in\R^{\llbracket1,\infty\rrbracket}$}. These operators are given by
\begin{align}
\mathscr{L}_{\mathrm{S}}&:=\N^{2}\sum_{\x\in\llbracket1,\infty\rrbracket}\Delta_{\mathrm{Neu}}\mathscr{U}'[\bphi_{\x}]\cdot\partial_{\bphi_{\x}}+\N^{2}\sum_{\x\in\llbracket1,\infty\rrbracket}(\grad^{\mathbf{X}}_{1}\partial_{\bphi_{\x}})^{2}\\
\mathscr{L}_{\mathrm{A}}&:=\N^{\frac32}\sum_{\x\in\llbracket1,\infty\rrbracket}\wt{\mathbf{F}}[\tau_{\x}\bphi]\partial_{\bphi_{\x}} \quad\text{and}\quad \mathscr{L}_{\mathrm{G}}:=-\N^{2}\mathscr{U}'[\bphi_{1}]\partial_{\bphi_{1}}+\N^{2}\partial_{\bphi_{1}}^{2}.
\end{align}
Above, {\small$\Delta_{\mathrm{Neu}}=-\grad^{\mathbf{X}}_{-1}\grad^{\mathbf{X}}_{1}$} is the Neumann Laplacian on {\small$\llbracket1,\infty\rrbracket$}, which acts on functions {\small$\mathsf{f}:\llbracket1,\infty\rrbracket$} by {\small$\Delta_{\mathrm{Neu}}\mathsf{f}_{\x}=\mathsf{f}_{\x+1}+\mathsf{f}_{\x-1}-2\mathsf{f}_{\x}$} with the convention {\small$\mathsf{f}_{0}=\mathsf{f}_{1}$}. A standard summation-by-parts for the Neumann Laplacian shows
\begin{align*}
\mathscr{L}_{\mathrm{S}}&=-\N^{2}\sum_{\x\in\llbracket1,\infty\rrbracket}\grad^{\mathbf{X}}_{1}\mathscr{U}'[\bphi_{\x}]\grad^{\mathbf{X}}_{1}\partial_{\bphi_{\x}}+\N^{2}\sum_{\x\in\llbracket1,\infty\rrbracket}(\grad^{\mathbf{X}}_{1}\partial_{\bphi_{\x}})^{2}.
\end{align*}
A standard integration-by-parts calculation (like in \cite{DGP17}, for instance) shows that {\small$\E^{0}(\mathfrak{f}\mathscr{L}_{\mathrm{S}}\mathfrak{g})=\E^{0}(\mathfrak{g}\mathscr{L}_{\mathrm{S}}\mathfrak{f})$} for any smooth and local functions {\small$\mathfrak{f},\mathfrak{g}:\R^{\Z}\to\R$}. On the other hand, for the same class of test functions {\small$\mathfrak{f},\mathfrak{g}$}, we also have (see \cite{DGP17}, again) {\small$\E^{0}(\mathfrak{f}\mathscr{L}_{\mathrm{A}}\mathfrak{g})=-\E^{0}(\mathfrak{g}\mathscr{L}_{\mathrm{A}}\mathfrak{f})$}. Finally, another standard integration-by-parts calculation shows that {\small$\E^{0}(\mathfrak{f}\mathscr{L}_{\mathrm{G}}\mathfrak{g})=\E^{0}(\mathfrak{g}\mathscr{L}_{\mathrm{G}}\mathfrak{f})$} for the same class of {\small$\mathfrak{f},\mathfrak{g}$}; indeed, {\small$\mathscr{L}_{\mathrm{G}}$} is the generator for the Langevin flow at {\small$0\in\llbracket0,\infty\rrbracket$} for the marginal of the product measure {\small$\mathbb{P}^{0}$} at this point. Putting this altogether shows that {\small$\mathbb{P}^{0}$} is invariant for {\small$\mathscr{L}$}.
\subsection{An initial data estimate}
Let us now present an estimate on the initial data that holds with high probability uniformly on {\small$\llbracket0,\infty\rrbracket$}. The proof of the lemma below is identical to that of Lemma A.2 in \cite{Y26PTRF} (as it uses only the initial data moment estimates in \eqref{eq:mainI}, as well as the heat kernel estimate \eqref{eq:hkestimatesI}). 
\begin{lemma}\label{lemma:databound}
\fsp Fix any {\small$\delta>0$}. If {\small$\kappa>0$} is large enough, then with high probability, we have for all {\small$\x\in\llbracket0,\infty\rrbracket$} that 
\begin{align}
\exp(-\tfrac{\kappa|\x|}{\N})|\mathbf{Z}^{\N}_{0,\x}|+\exp(-\tfrac{\kappa|\x|}{\N})|\mathbf{S}^{\N}_{0,\x}|\lesssim_{\kappa,\delta}\N^{\delta}.\label{eq:databoundI}
\end{align}
\end{lemma}

%
%
%
\end{document}